\documentclass[a4paper,12pt]{report}

\usepackage{newtxtext,newtxmath}

\usepackage{latexsym}
\usepackage{amsmath}
\usepackage[dvipsnames]{xcolor}

\newcommand{\bref}[1]{(\ref{#1})}
\newcommand{\ucontents}[2]{\addcontentsline{toc}{#1}{\numberline{}{#2}}}

\newcommand{\ovln}[1]{\overline{#1}}
\newcommand{\twid}[1]{\widetilde{#1}}
\newcommand{\such}{:}
\newcommand{\without}{\setminus}
\newcommand{\epsln}{\varepsilon}
\newcommand{\id}{\mathrm{id}}

\usepackage{stmaryrd}
\usepackage{graphicx}

\newcommand{\oppairu}{\rightleftarrows}
\newcommand{\demphcolor}{\color[rgb]{0,0,0.8}}
\newcommand{\demph}[1]{\textbf{\textup{\demphcolor#1}}}
\newcommand{\dem}[1]{{\demphcolor#1}}
\newcommand{\iso}{\cong}
\newcommand{\of}{\circ}
\newcommand{\sub}{\subseteq}
\newcommand{\cell}[4]{\put(#1,#2){\makebox(0,0)[#3]{#4}}}

\newcommand{\rowofstars}{%
\bigskip
\hfill $*$ \hfill $*$ \hfill $*$ \hfill
\bigskip}

\newcommand{\toby}[1]{\stackrel{#1}{\longrightarrow}}
\newcommand{\GL}[2]{\mathrm{GL}_{#1}(#2)}
\DeclareMathOperator{\im}{im}
\newcommand{\from}{\colon}
\DeclareMathOperator{\spn}{span}
\DeclareMathOperator{\chr}{char}
\newcommand{\C}{\mathbb{C}}
\newcommand{\R}{\mathbb{R}}
\newcommand{\Z}{\mathbb{Z}}
\newcommand{\Q}{\mathbb{Q}}
\DeclareMathOperator{\ev}{ev}
\newcommand{\vc}[1]{\mathbf{#1}} 
\renewcommand{\emptyset}{\varnothing}

\usepackage{framed}

\colorlet{shadecolor}{blue}

\newenvironment{preex}[1]{%
  \colorlet{shadecolor}{LimeGreen!20}%
  \begin{snugshade}%
  \begin{list}{\hspace*{-10mm}%
  \setlength{\unitlength}{1mm}%
  \begin{picture}(10,0)%
    \cell{0}{2.5}{tl}{\includegraphics[width=10mm]{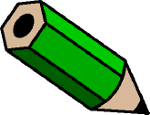}}%
  \end{picture}}%
  {\setlength{\leftmargin}{2\leftmargin}%
   \setlength{\labelsep}{0.25\leftmargin}}\item #1%
}%
{%
  \end{list}%
  \end{snugshade}%
}

\newenvironment{ex}[1]{%
  \begin{preex}{\textbf{Exercise \refstepcounter{thm}\thethm\label{#1}}}
}%
{%
  \end{preex}%
}

\newenvironment{prewarning}[1]{%
  \colorlet{shadecolor}{Red!20}%
  \begin{snugshade}%
  \begin{list}{\hspace*{-10mm}%
  \setlength{\unitlength}{1mm}%
  \begin{picture}(9,0)%
    \cell{0}{2.5}{tl}{\includegraphics[width=9mm]{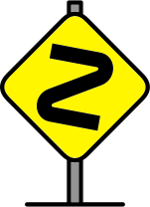}}%
  \end{picture}}%
  {\setlength{\leftmargin}{2\leftmargin}%
   \setlength{\labelsep}{0.25\leftmargin}}\item #1%
}%
{%
  \end{list}%
  \end{snugshade}%
}

\newenvironment{warning}[1]{%
  \begin{prewarning}{\textbf{Warning \refstepcounter{thm}\thethm\label{#1}}}
}%
{%
  \end{prewarning}%
}

\newenvironment{predigression}[1]{%
  \colorlet{shadecolor}{SkyBlue!20}%
  \begin{snugshade}%
  \begin{list}{\hspace*{-12mm}%
  \setlength{\unitlength}{1mm}%
  \begin{picture}(12,0)%
    \cell{0}{2.5}{tl}{\includegraphics[width=12mm]{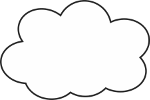}}%
  \end{picture}}%
  {\setlength{\leftmargin}{2\leftmargin}%
   \setlength{\labelsep}{0.25\leftmargin}\small}\item #1%
}%
{%
  \end{list}%
  \end{snugshade}%
}

\newenvironment{digression}[1]{%
  \begin{predigression}{\textbf{Digression \refstepcounter{thm}\thethm\label{#1}}}
}%
{%
  \end{predigression}%
}

\usepackage{mdframed}

\newcommand{\idl}{\mathbin{\trianglelefteqslant}}
\newcommand{\nsub}{\mathbin{\trianglelefteqslant}}

\renewcommand{\phi}{\varphi}
\newcommand{\idlgen}[1]{\langle #1 \rangle}
\newcommand{\gpgen}[1]{\langle #1 \rangle}
\newcommand{\dvd}{\mathrel{\mid}}
\newcommand{\ndvd}{\mathrel{\not\mid}}
\def\mathcal{\mathscr}
\newcommand{\F}{\mathbb{F}}

\newcommand{\lengths}{\setlength{\unitlength}{1mm}%
\setlength{\fboxsep}{0mm}}

\DeclareMathOperator{\Gal}{Gal}
\DeclareMathOperator{\codeg}{codeg}
\newcommand{\csuch}{\mathrel{\colon}}
\DeclareMathOperator{\SF}{SF}
\DeclareMathOperator{\Sym}{Sym}
\DeclareMathOperator{\Fix}{Fix}
\DeclareMathOperator{\Eq}{Eq}
\DeclareMathOperator{\Aut}{Aut}

\newcommand{\FF}{\mathcal{F}}
\newcommand{\GG}{\mathcal{G}}
\newcommand{\rad}{{\operatorname{rad}}}
\newcommand{\sol}{{\operatorname{sol}}}

\reversemarginpar
\newcommand{\video}[1]{%
\marginpar{%
\href{https://www.maths.ed.ac.uk/~tl/galois}{%
\colorbox{CadetBlue!0}{%
\parbox{28mm}{%
\centering
\includegraphics[width=12mm]{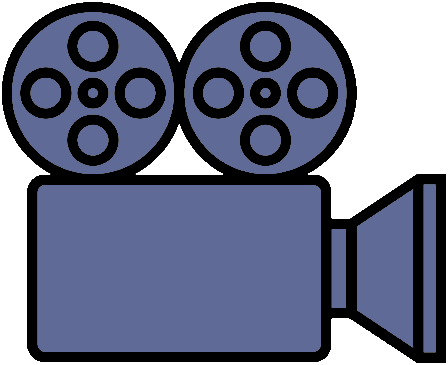}\\ 
\footnotesize\textit{\textcolor{black}{#1}}%
}%
}%
}%
}%
}

\usepackage[amsmath,thmmarks]{ntheorem}

\newtheorem{thm}{Theorem}[section]
\newtheorem{propn}[thm]{Proposition}
\newtheorem{lemma}[thm]{Lemma}
\newtheorem{cor}[thm]{Corollary}

\newmdtheoremenv[%
linewidth=0mm,
backgroundcolor=Goldenrod!30,
ntheorem=true]%
{bigthm}[thm]{Theorem}

\newmdtheoremenv[%
linecolor=Red,
linewidth=.5mm,
backgroundcolor=Goldenrod!50,
ntheorem=true]%
{megathm}[thm]{Theorem}

\theorembodyfont{\normalfont}

\newtheorem{defn}[thm]{Definition}
\newtheorem{example}[thm]{Example}
\newtheorem{examples}[thm]{Examples}
\newtheorem{remark}[thm]{Remark}

\theoremstyle{nonumberplain}
\theoremsymbol{\ensuremath{\square}}
\qedsymbol{\ensuremath{\square}}

\newtheorem{proof}{Proof}

\newcommand{\theoremtobeproved}{}
\newtheorem{pfoftheorem}{Proof of \theoremtobeproved}
\newenvironment{pfof}[1]
{
\renewcommand{\theoremtobeproved}{#1}
\begin{pfoftheorem}
}
{\end{pfoftheorem}}

\definecolor{myurlcolor}{rgb}{0.7,0,0}
\definecolor{mylinkcolor}{rgb}{0,0.4,0}
\usepackage{hyperref}
\hypersetup{%
colorlinks,
linkcolor=mylinkcolor,
citecolor=black,
urlcolor=mylinkcolor}

\usepackage{tocloft}
\usepackage[all]{xy}
\xyoption{arc}
\xyoption{rotate}

\begin{document}
\sloppy

\tableofcontents

\chapter*{Note to the reader}
\ucontents{chapter}{Note to the reader}

These are the course notes for Galois Theory, University of Edinburgh,
2022--23. For this arXiv version, I have made a
\href{https://www.maths.ed.ac.uk/~tl/galois}{web page} containing
additional resources such as videos and problem sheets.

\paragraph{Structure} Each chapter corresponds to one week of the
semester. You are expected to read Chapter $n$ before the lectures in Week
$n$, except for Chapter~1. I may make small changes to these notes as we go
along (e.g.\ to correct errors), so I recommend that you download a fresh
copy before you start each week's reading.

\begin{preex}{\textbf{Exercises}} 
looking like this are sprinkled through the notes. The idea is that you
try them immediately, before you continue reading. 

Most of them are meant to be quick and easy, much easier than assignment or
workshop questions. If you can do them, you can take it as a sign that
you're following successfully. For those that defeat you, talk with others
in the class, ask on Piazza, or ask me.

I promise you that if you make a habit of trying every exercise, you'll
enjoy the course more and understand it better than if you don't.
\end{preex}

\begin{predigression}{\textbf{Digressions}}
like this are optional and not examinable, but might interest you.
They're usually on points that \emph{I} find 
interesting, and often describe connections between Galois theory and
other parts of mathematics.
\end{predigression}
\video{Here you'll see titles of relevant videos, made two years ago when
the class was online. They are entirely optional but may help your
understanding.}

References to theorem numbers, page numbers, etc., are clickable links.

\paragraph{What to prioritize} You know by now that the most
important things in almost any course are the \emph{definitions} and the
results called \emph{Theorem}. But I also want to emphasize the
\emph{proofs}. This course presents a wonderful body of theory, and the
idea is that you learn it all, including the proofs that are its beating
heart. 

Less idealistically, the exam will test not only that you know the proofs,
but also something harder: that you \emph{understand} them. So the proofs
will need your attention and energy.

\paragraph{Compulsory prerequisites} To take this course, you must have already taken
these two courses:
\begin{itemize}
\item 
\href{http://www.drps.ed.ac.uk/22-23/dpt/cxmath10069.htm}{Honours
Algebra}: We'll need some abstract linear algebra, corresponding to
Chapter~1 of that course. We'll also need everything from Honours Algebra
about rings and polynomials (Chapter~3 there), including ideals, quotient
rings (factor rings), the universal property of quotient rings, and the
first isomorphism theorem for rings.

\item
\href{http://www.drps.ed.ac.uk/22-23/dpt/cxmath10079.htm}{Group Theory}:
From that course, we'll need fundamentals such as normal subgroups,
quotient groups, the universal property of quotient groups, and the first
isomorphism theorem for groups. You should know lots about the symmetric
groups $S_n$, alternating groups $A_n$, and cyclic groups $C_n$, as well as
a little about the dihedral groups $D_n$, and I hope you can list all of
the groups of order $< 8$ without having to think too hard.

Chapter~8 of Group Theory, on solvable groups, will be crucial for us. For
example, you'll need to understand what it means that $S_4$ is solvable but
$A_5$ is not.

We won't need anything on free groups, the Sylow theorems, or the
Jordan--H\"older theorem.
\end{itemize}
If you're a visiting or MSc student and didn't take those courses, please
contact me so that we can decide whether your background is suitable.

\paragraph{Mistakes} I'll be grateful to hear of mistakes 
in these notes (Tom.Leinster\mbox{@}\linebreak[5]ed.ac.uk), even if it's
something very small and even if you're not sure.

\chapter{Overview of Galois theory}
\label{ch:overview}

This chapter stands apart from all the others, 

Modern treatments of Galois theory take advantage of several well-developed
branches of algebra: the theories of groups, rings, fields, and vector
spaces. This is as it should be! However, assembling all the algebraic
apparatus will take us several weeks, during which it's easy to lose sight of
what it's all for.
\video{Introduction to Week 1}

Galois theory came from two basic insights:
\begin{itemize}
\item 
every polynomial has a symmetry group;

\item
this group determines whether the polynomial can be solved by radicals (in
a sense I'll define).
\end{itemize}
In this chapter, I'll explain these two ideas in as short and low-tech a
way as I can manage. In Chapter~\ref{ch:garf} we'll start again, beginning
the modern approach that will take up the rest of the course.  But I hope
that all through that long build-up, you'll keep in mind the fundamental
ideas you learn in this chapter.

\section{The view of $\C$ from $\Q$}

Imagine you lived several centuries ago, before the discovery of complex
numbers. Your whole mathematical world is the real numbers, and there is no
square root of $-1$. This situation frustrates you, and you decide to do
something about it.

So, you invent a new symbol $i$ (for `imaginary') and decree that $i^2 =
-1$. You still want to be able to do all the usual arithmetic operations
($+$, $\times$, etc.), and you want to keep all the rules that govern them
(associativity, commutativity, etc.). So you're also forced to introduce
new numbers such as $2 + 3\times i$, and you end up with what today
we call the complex numbers.

So far, so good. But then you notice something strange. When you invented
the complex numbers, you only intended to introduce one square root of
$-1$. But accidentally, you introduced a second one at the same time:
$-i$. (You wait centuries for a square root of $-1$, then two come along at
once.) Maybe that's not so strange in itself; after all, positive reals
have two square roots too. But then you realize something genuinely weird:
\begin{quote}
\emph{There's nothing you can do to distinguish $i$ from $-i$.}
\end{quote}
Try as you might, you can't find any reasonable statement that's true
for $i$ but not $-i$. For example, you notice that $i$ is a solution of
\[
z^3 - 3z^2 - 16z - 3 = \frac{17}{z},
\]
but then you realize that $-i$ is too. 

Of course, there are \emph{unreasonable} statements that are true for $i$
but not $-i$, such as `$z = i$'. We should restrict to statements that only
refer to the known world of real numbers. More precisely, let's consider
statements of the form
\[
\frac{p_1(z)}{p_2(z)} = \frac{p_3(z)}{p_4(z)},
\]
where $p_1, p_2, p_3, p_4$ are polynomials with \emph{real}
coefficients. Any such equation can be rearranged to give
\[
p(z) = 0,
\]
where again $p$ is a polynomial with real coefficients, so we might as well
just consider statements of that form. The point is that if $p(i) = 0$ then
$p(-i) = 0$. 

Let's make this formal. We could say that two complex numbers are
`indistinguishable when seen from $\R$' if they satisfy the same
polynomials over $\R$. But the official term is `conjugate':

\begin{defn}
\label{defn:indist-R}
Two complex numbers $z$ and $z'$ are \demph{conjugate over $\R$} if for
all polynomials $p$ with coefficients in $\R$,
\[
p(z) = 0 \iff p(z') = 0.
\]
\end{defn}

For example, $i$ and $-i$ are conjugate over $\R$. This follows from a more
general result, stating that conjugacy in this new sense is closely
related to complex conjugacy:

\begin{lemma}
\label{lemma:indist-R}
Let $z, z' \in \C$. Then $z$ and $z'$ are conjugate over $\R$ if and only
if $z' = z$ or $z' = \ovln{z}$.
\end{lemma}

\begin{proof}
`Only if': suppose that $z$ and $z'$ are conjugate over $\R$. Write $z = x
+ iy$ with $x, y \in \R$. Then $(z - x)^2 + y^2 = 0$. Since $x$ and $y$ are
real, conjugacy implies that $(z' - x)^2 + y^2 = 0$, so $z' - x = \pm iy$, so
$z' = x \pm iy$.

`If': obviously $z$ is conjugate to itself, so it's enough to prove that
$z$ is conjugate to $\ovln{z}$. I'll give two proofs. Each one teaches us a
lesson that will be valuable later.

\emph{First proof:} recall that complex conjugation satisfies
\[
\ovln{w_1 + w_2} = \ovln{w_1} + \ovln{w_2},
\qquad
\ovln{w_1 \cdot w_2} = \ovln{w_1} \cdot \ovln{w_2}
\]
for all $w_1, w_2 \in \C$. Also, $\ovln{a} = a$ for all $a \in \R$. It follows
by induction that for any polynomial $p$ over $\R$,
\[
\ovln{p(w)} = p(\ovln{w})
\]
for all $w \in \C$. So 
\[
p(z) = 0 \iff 
\ovln{p(z)} = \ovln{0} \iff
p(\ovln{z}) = 0.
\]

\emph{Second proof:} write $z = x + iy$ with $x, y \in \R$. Let $p$ be a
polynomial over $\R$ such that $p(z) = 0$. We will prove that $p(\ovln{z})
= 0$. This is trivial if $y = 0$, so suppose that $y \neq 0$.

Consider the real polynomial $m(t) = (t - x)^2 + y^2$. Then $m(z) =
0$. You know from Honours Algebra that 
\begin{align}
\label{eq:indist-div}
p(t) = m(t)q(t) + r(t)
\end{align}
for some real polynomials $q$ and $r$ with $\deg(r) < \deg(m) = 2$ (so $r$
is either a constant or of degree $1$). Putting $t = z$
in~\eqref{eq:indist-div} gives $r(z) = 0$. It's easy to see that this is
impossible unless $r$ is the zero polynomial (using the assumption that $y
\neq 0$). So $p(t) = m(t)q(t)$. But $m(\ovln{z}) = 0$, so $p(\ovln{z}) =
0$, as required.

We have just shown that for all polynomials $p$ over $\R$, if $p(z) = 0$
then $p(\ovln{z}) = 0$. Exchanging the roles of $z$ and $\ovln{z}$ proves
the converse. Hence $z$ and $\ovln{z}$ are conjugate over $\R$.
\end{proof}

\begin{ex}{ex:poly-ind}
Both proofs of `if' contain little gaps: `It follows by induction' in the
first proof, and `it's easy to see' in the second. Fill them.
\end{ex}

\begin{digression}{dig:pwr-series}
With complex analysis in mind, we could imagine a stricter definition of
conjugacy in which polynomials are replaced by arbitrary
convergent power series (still with coefficients in $\R$). This would allow
functions such as $\exp$, $\cos$ and $\sin$, and equations such
as $\exp(i\pi) = -1$.

But this apparently different definition of conjugacy is, in fact,
equivalent. A complex number is still conjugate to exactly itself and its
complex conjugate. (For example, $\exp(-i\pi) = -1$ too.) Do you see why?
\end{digression}

Lemma~\ref{lemma:indist-R} tells us that conjugacy over $\R$ is
rather simple. But the same idea becomes much more interesting if we
replace $\R$ by $\Q$. And in this course, we will mainly focus on
polynomials over $\Q$.

Define \demph{conjugacy over $\Q$} by replacing $\R$ by $\Q$ in
Definition~\ref{defn:indist-R}. Again, when you see the words `conjugate
over $\Q$', you can think to yourself `indistinguishable when seen from
$\Q$'. From now on, I will usually just say `conjugate', dropping the
`over $\Q$'.

\begin{example}
\label{eg:sqrt-2-indist}
I claim that $\sqrt{2}$ and $-\sqrt{2}$ are conjugate.  And I'll give you
two different proofs, closely analogous to the two proofs of the `if' part
of Lemma~\ref{lemma:indist-R}.

\emph{First proof:} write
\[
\Q(\sqrt{2})
=
\{ a + b\sqrt{2} \such a, b \in \Q\}.
\]
For $w \in \Q(\sqrt{2})$, there are \emph{unique} $a, b \in \Q$ such that
$w = a + b\sqrt{2}$, because $\sqrt{2}$ is irrational. So it is logically
valid to define 
\video{Example~\ref{eg:sqrt-2-indist}}
\[
\twid{w} = a - b\sqrt{2} \in \Q(\sqrt{2}).
\]
(Question: what did the uniqueness of $a$ and $b$ have to do with the
logical validity of that definition?) Now, $\Q(\sqrt{2})$ is closed under
addition and multiplication, and it is straightforward to check that
\[
\twid{w_1 + w_2} = \twid{w_1} + \twid{w_2},
\qquad
\twid{w_1 \cdot w_2} = \twid{w_1} \cdot \twid{w_2}
\]
for all $w_1, w_2 \in \Q(\sqrt{2})$. Also, $\twid{a} = a$ for all $a
\in \Q$. So just as in the proof of Lemma~\ref{lemma:indist-R}, it follows
that $w$ and $\twid{w}$ are conjugate for every $w \in \Q(\sqrt{2})$. In
particular, $\sqrt{2}$ is conjugate to (`indistinguishable from')
$-\sqrt{2}$.

\emph{Second proof:} let $p = p(t)$ be a polynomial with coefficients in
$\Q$ such that $p(\sqrt{2}) = 0$. You know from Honours Algebra that
\[
p(t) = (t^2 - 2)q(t) + r(t)
\]
for some polynomials $q(t)$ and $r(t)$ over $\Q$ with $\deg r <
2$. Putting $t = \sqrt{2}$ gives $r(\sqrt{2}) = 0$. But $\sqrt{2}$ is
irrational and $r(t)$ is of the form $at + b$ with $a, b \in \Q$, so $r$
must be the zero polynomial. Hence $p(t) = (t^2 - 2)q(t)$, giving
$p(-\sqrt{2}) = 0$. 

We have just shown that for all polynomials $p$ over $\Q$, if $p(\sqrt{2})
= 0$ then $p(-\sqrt{2}) = 0$. The same argument with the roles of
$\sqrt{2}$ and $-\sqrt{2}$ reversed proves the converse. Hence
$\pm\sqrt{2}$ are conjugate.
\end{example}

\begin{ex}{ex:rat-dist}
Let $z \in \Q$. Show that $z$ is not conjugate to $z'$ for any
complex number $z' \neq z$.
\end{ex}

One thing that makes conjugacy more subtle over $\Q$ than over
$\R$ is that over $\Q$, more than two numbers can be conjugate:

\begin{figure}
\centering
\lengths%
\begin{picture}(136,40)
\cell{68}{20}{c}{\includegraphics[height=40\unitlength]{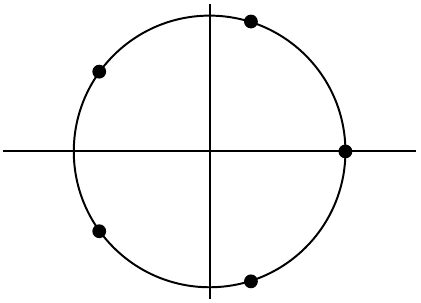}}
\cell{88}{22}{c}{$1$}
\cell{77}{39}{c}{$\omega$}
\cell{50}{32}{c}{$\omega^2$}
\cell{50}{8}{c}{$\omega^3$}
\cell{78}{2}{c}{$\omega^4$}
\end{picture}%
\caption{The $5$th roots of unity.}
\label{fig:fifth}
\end{figure}

\begin{example}
\label{eg:indist-fifth}
The $5$th roots of unity are 
\[
1, \omega, \omega^2, \omega^3, \omega^4,
\]
where $\omega = e^{2\pi i/5}$ (Figure~\ref{fig:fifth}). Now $1$ is not
conjugate to any of the rest, since it is a root of the
polynomial $t - 1$ and the others are not. (See also
Exercise~\ref{ex:rat-dist}.) But it turns out that $\omega, \omega^2,
\omega^3, \omega^4$ are all conjugate to each other.

Complex conjugate numbers are conjugate over $\R$, so they're certainly
conjugate over $\Q$. (If you've got a pair of complex numbers that you can't
tell apart using only the reals, you certainly can't tell them apart
using only the rationals.) Since $\omega^4 = 1/\omega = \ovln{\omega}$, it
follows that $\omega$ and $\omega^4$ are conjugate over $\Q$. By the same
argument, $\omega^2$ and $\omega^3$ are conjugate. What's not so obvious is
that $\omega$ and $\omega^2$ are conjugate. I know two proofs, which are
like the two proofs of Lemma~\ref{lemma:indist-R} and
Example~\ref{eg:sqrt-2-indist}. But we're not equipped to do either yet.
\end{example}

\begin{example}
More generally, let $p$ be any prime and put $\omega = e^{2\pi i/p}$. Then
$\omega, \omega^2, \ldots, \omega^{p - 1}$ are all conjugate to one another.
\end{example}

So far, we have asked when \emph{one} complex number can be distinguished
from another, using only polynomials over $\Q$. But what about more than
one?

\begin{defn}
\label{defn:indist-tuple}
Let $k \geq 0$ and let $(z_1, \ldots, z_k)$ and $(z'_1, \ldots, z'_k)$ be
$k$-tuples of complex numbers. Then $(z_1, \ldots, z_k)$ and $(z'_1,
\ldots, z'_k)$ are \demph{conjugate over $\Q$} if for all polynomials
$p(t_1, \ldots, t_k)$ over $\Q$ in $k$ variables,
\[
p(z_1, \ldots, z_k) = 0 
\iff
p(z'_1, \ldots, z'_k) = 0.
\]
\end{defn}

When $k = 1$, this is just the earlier definition of conjugacy.

\begin{ex}{ex:indist-many-one}
Suppose that $(z_1, \ldots, z_k)$ and $(z'_1, \ldots, z'_k)$ are
conjugate. Show that $z_i$ and $z'_i$ are conjugate, for each $i \in \{1,
\ldots, k\}$.
\end{ex}

\begin{example}
\label{eg:indist-conj}
For any $z_1, \ldots, z_k \in \C$, the $k$-tuples $(z_1, \ldots, z_k)$ and
$(\ovln{z_1}, \ldots, \ovln{z_k})$ are conjugate. For
let $p(t_1, \ldots, t_k)$ be a polynomial over $\Q$. Then
\[
\ovln{p(z_1, \ldots, z_k)} = p(\ovln{z_1}, \ldots, \ovln{z_k})
\]
since the coefficients of $p$ are real, by a similar argument to the one in
the first proof of Lemma~\ref{lemma:indist-R}. Hence
\[
p(z_1, \ldots, z_k) = 0
\iff
p(\ovln{z_1}, \ldots, \ovln{z_k}) = 0,
\]
which is what we had to prove.
\end{example}

\begin{example}
\label{eg:indist-4-5}
Let $\omega = e^{2\pi i/5}$, as in Example~\ref{eg:indist-fifth}. Then 
\[
(\omega, \omega^2, \omega^3, \omega^4)
\quad\text{and}\quad
(\omega^4, \omega^3, \omega^2, \omega)
\]
are conjugate, by Example~\ref{eg:indist-conj}. It can also be shown
that 
\[
(\omega, \omega^2, \omega^3, \omega^4)
\quad\text{and}\quad 
(\omega^2, \omega^4, \omega, \omega^3)
\]
are conjugate, although the proof is beyond us for now. But
\begin{align}
\label{eq:dist-four}
(\omega, \omega^2, \omega^3, \omega^4)
\quad\text{and}\quad
(\omega^2, \omega, \omega^3, \omega^4)
\end{align}
are \emph{not} conjugate, since if we put $p(t_1, t_2, t_3, t_4) = t_2 -
t_1^2$ then 
\[
p(\omega, \omega^2, \omega^3, \omega^4)
=
0
\neq
p(\omega^2, \omega, \omega^3, \omega^4).
\]
\end{example}

\begin{warning}{wg:conj-many}
The converse of Exercise~\ref{ex:indist-many-one} is false: just because
$z_i$ and $z'_i$ are conjugate for all $i$, it doesn't follow that
$(z_1, \ldots, z_k)$ and $(z'_1, \ldots, z'_k)$ are conjugate. For
we saw in Example~\ref{eg:indist-fifth} that $\omega$, $\omega^2$,
$\omega^3$ and $\omega^4$ are all conjugate to each other, but we
just saw that the 4-tuples~\eqref{eq:dist-four} are not conjugate.
\end{warning}

\section{Every polynomial has a symmetry group\ldots}

We are now ready to describe the first main idea of Galois theory: every
polynomial has a symmetry group. 

\begin{defn}
\label{defn:gg-conc}
Let $f$ be a polynomial with coefficients in $\Q$. Write $\alpha_1, \ldots,
\alpha_k$ for its distinct roots in $\C$. The \demph{Galois group} of $f$
is 
\[
\dem{\Gal(f)}
=
\{ \sigma \in S_k \such 
(\alpha_1, \ldots, \alpha_k) 
\text{ and }
(\alpha_{\sigma(1)}, \ldots, \alpha_{\sigma(k)})
\text{ are conjugate}
\}.
\]
\end{defn}

`Distinct roots' means that we ignore any repetition of roots: e.g.\ if
$f(t) = t^5(t - 1)^9$ then $k = 2$ and $\{\alpha_1, \alpha_2\} = \{0,
1\}$. 

\begin{ex}{ex:indist-subgp}
Show that $\Gal(f)$ is a subgroup of $S_k$. (This one is harder. Hint: if
you permute the variables of a polynomial, you get another polynomial.)
\end{ex}
\video{Exercise~\ref{ex:indist-subgp}}

\begin{digression}{dig:indist-Gal}
I brushed something under the carpet. The definition of $\Gal(f)$ depends
on the order in which the roots are listed. Different orderings gives
different subgroups of $S_k$. However, these subgroups are all
\emph{conjugate} to each other (conjugacy in the sense of group theory!),
and therefore isomorphic as abstract groups. So $\Gal(f)$ is well-defined
as an abstract group, independently of the choice of ordering.
\end{digression}

\begin{example}
Let $f$ be a polynomial over $\Q$ whose complex roots $\alpha_1,
\ldots, \alpha_k$ are all rational. If $\sigma \in \Gal(f)$ then
$\alpha_{\sigma(i)}$ and $\alpha_i$ are conjugate for each $i$, by
Exercise~\ref{ex:indist-many-one}. But since they are rational, that forces
$\alpha_{\sigma(i)} = \alpha_i$ (by Exercise~\ref{ex:rat-dist}), and since
$\alpha_1, \ldots, \alpha_k$ are distinct, $\sigma(i) = i$.  Hence $\sigma
= \id$. So the Galois group of $f$ is trivial.
\end{example}

\begin{example}
\label{eg:gg-conc-quad}
Let $f$ be a quadratic over $\Q$. If $f$ has rational roots then as we
have just seen, $\Gal(f)$ is trivial. If $f$ has two non-real roots then
they are complex conjugate, so $\Gal(f) = S_2$ by
Example~\ref{eg:indist-conj}. The remaining case is where $f$ has two distinct
roots that are real but not rational, and it can be shown that in that case
too, $\Gal(f) = S_2$.
\end{example}

\begin{warning}{wg:non-real}
On terminology: note that just now I said `non-real'. Sometimes people
casually say `complex' to mean `not real'. But try not to do this
yourself. It makes as little sense as saying `real' to mean `irrational',
or `rational' to mean `not an integer'.
\end{warning}

\begin{example}
Let $f(t) = t^4 + t^3 + t^2 + t + 1$. Then $(t - 1)f(t) = t^5 - 1$, so $f$
has roots $\omega, \omega^2, \omega^3, \omega^4$ where $\omega = e^{2\pi
i/5}$. We saw in Example~\ref{eg:indist-4-5} that
\[
\begin{pmatrix}
1       &2      &3      &4      \\
4       &3      &2      &1
\end{pmatrix},
\begin{pmatrix}
1       &2      &3      &4      \\
2       &4      &1      &3
\end{pmatrix}
\in 
\Gal(f),
\qquad
\begin{pmatrix}
1       &2      &3      &4      \\
2       &1      &3      &4
\end{pmatrix}
\not\in
\Gal(f).
\]
In fact, it can be shown that
\[
\Gal(f) 
=
\biggl\langle 
\begin{pmatrix}
1       &2      &3      &4      \\
2       &4      &1      &3
\end{pmatrix}
\biggr\rangle
\iso
C_4.
\]
\end{example}

\begin{example}
Let $f(t) = t^3 + bt^2 + ct + d$ be a cubic over $\Q$ with no rational
roots. Then
\[
\Gal(f)
\iso
\begin{cases}
A_3     &
\text{if }
\sqrt{-27d^2 + 18bcd - 4c^3 - 4b^3d + b^2c^2} \in \Q,   \\
S_3     &
\text{otherwise}.
\end{cases}
\]
This appears as Proposition~22.4 in Stewart, but is way beyond us for now.
Calculating Galois groups is hard.
\end{example}
\video{Galois groups, intuitively}

\section{\ldots which determines whether it can be solved}

Here we meet the second main idea of Galois theory: the Galois group of a
polynomial determines whether it can be solved. More exactly, it determines
whether the polynomial can be `solved by radicals'.

To explain what this means, let's begin with the quadratic formula. The
roots of a quadratic $at^2 + bt + c$ are
\[
\frac{-b \pm \sqrt{b^2 - 4ac}}{2a}.
\]
After much struggling, it was discovered that there is a similar formula
for cubics $at^3 + bt^2 + ct + d$: the roots are given by
\[
\hspace*{-19mm}
\frac{
\scriptstyle
\sqrt[3]{-27a^2d + 9abc - 2b^3 + 
3a 
\sqrt{3(27a^2d^2 - 18abcd + 4ac^3 + 4b^3d - b^2c^2)}}
\ +\ 
\sqrt[3]{-27a^2d + 9abc - 2b^3 - 
3a 
\sqrt{3(27a^2d^2 - 18abcd + 4ac^3 + 4b^3d - b^2c^2)}}
}
{
\scriptstyle
3\sqrt[3]{2}a
}.
\]
(No, you don't need to memorize that!) This is a complicated formula, and
there's also something strange about it. Any nonzero complex number has
three cube roots, and there are two $\sqrt[3]{}$ signs in the formula
(ignoring the $\sqrt[3]{2}$ in the denominator), so it looks as if the
formula gives \emph{nine} roots for the cubic. But a cubic can only have
three roots. What's going on?

It turns out that some of the nine aren't roots of
the cubic at all. You have to choose your cube roots carefully.
Section~1.4 of Stewart's book has much more on this point, as well as an
explanation of how the cubic formula was obtained. We won't be going into
this ourselves.

As Stewart also explains, there is a similar but even
more complicated formula for quartics (polynomials of degree $4$). 

\begin{digression}{dig:cubic-quartic}
Stewart doesn't actually write out the explicit formula for the cubic, let
alone the much worse one for the quartic. He just describes algorithms by
which they can be solved. But if you unwind the algorithm for the cubic,
you get the formula above. I have done this exercise once and do not
recommend it.
\end{digression}

Once mathematicians discovered how to solve quartics, they naturally looked
for a formula for quintics (polynomials of degree $5$). But it was
eventually proved by Abel and Ruffini, in the early 19th century, that
there is \emph{no} formula like the quadratic, cubic or quartic formula for
polynomials of degree $\geq 5$. A bit more precisely, there is no
formula for the roots in terms of the coefficients that uses only the usual
arithmetic operations ($+$, $-$, $\times$, $\div$) and $k$th roots (for
integers $k$).

Spectacular as this result was, Galois went further---and so will we.

Informally, let us say that a complex number is
\demph{radical}\label{p:radical} if it can be obtained from the rationals
using only the usual arithmetic operations and $k$th roots.
For example,
\[
\frac{
\frac{1}{2} + \sqrt[3]{\sqrt[7]{2} - \sqrt[2]{7}}
}
{
\sqrt[4]{6 + \sqrt[5]{\frac{2}{3}}}
}
\]
is radical, whichever square root, cube root, etc., we choose. A polynomial
over $\Q$ is \demph{solvable (or soluble) by radicals} if all of its
complex roots are radical.

\begin{example}
\label{eg:sr-2}
Every quadratic over $\Q$ is solvable by radicals. This follows from the
quadratic formula: $(-b \pm \sqrt{b^2 - 4ac})/2a$ is visibly a radical
number.
\end{example}

\begin{example}
\label{eg:sr-34}
Similarly, the cubic formula shows that every cubic over $\Q$ is solvable
by radicals. The same goes for quartics.
\end{example}

\begin{example}
\label{eg:sr-5}
\emph{Some} quintics are solvable by radicals. For instance, 
\[
(t - 1)(t - 2)(t - 3)(t - 4)(t - 5)
\]
is solvable by radicals, since all its roots are rational and, therefore,
radical. A bit less trivially, $(t - 123)^5 + 456$ is solvable by radicals,
since its roots are the five complex numbers $123 + \sqrt[5]{-456}$, which
are all radical.
\end{example}

What determines whether a polynomial is solvable by radicals? Galois's
amazing achievement was to answer this question completely:

\begin{bigthm}[Galois]
\label{thm:solv-conc}
Let $f$ be a polynomial over $\Q$. Then
\[
f \text{ is solvable by radicals}
\iff
\Gal(f) \text{ is a solvable group}.
\]
\end{bigthm}

\begin{example}
Definition~\ref{defn:gg-conc} implies that if $f$ has degree $n$ then
$\Gal(f)$ is isomorphic to a subgroup of $S_n$. You saw in Group Theory
that $S_4$ is solvable, and that every subgroup of a solvable group is
solvable. Hence the Galois group of any polynomial of degree $\leq 4$ is
solvable. It follows from Theorem~\ref{thm:solv-conc} that every polynomial
of degree $\leq 4$ is solvable by radicals.
\end{example}

\begin{example}
\label{eg:solv-conc-not}
Put $f(t) = t^5 - 6t + 3$. Later we'll show that $\Gal(f) = S_5$. You saw
in Group Theory that $S_5$ is \emph{not} solvable. Hence $f$ is not
solvable by radicals.
\end{example}

If there was a quintic formula then \emph{all} quintics would be solvable
by radicals, for the same reason as in Examples~\ref{eg:sr-2}
and~\ref{eg:sr-34}. But since this is not the case, there is no quintic
formula.

Galois's result is much sharper than Abel and Ruffini's. They proved that
there is no formula providing a solution by radicals of \emph{every}
quintic, whereas Galois found a way of determining \emph{which}
quintics (and higher) can be solved by radicals and which cannot.

\begin{digression}{dig:radical-numerics}
From the point of view of modern numerical computation, this is all a bit
odd. Computationally speaking, there is probably not much difference
between solving $t^5 + 3 = 0$ to $100$ decimal places (that is, finding
$\sqrt[5]{-3}$) and solving $t^5 -6t + 3 = 0$ to $100$ decimal places (that
is, solving a polynomial that isn't solvable by radicals).  Numerical
computation and abstract algebra have different ideas about what is easy
and what is hard!
\end{digression}

\rowofstars

This completes our overview of Galois theory. What's next?

Mathematics increasingly emphasizes \emph{abstraction} over
\emph{calculation}. Individual mathematicians' tastes vary, but the
historical trend is clear. In the case of Galois theory, this means dealing
with \emph{abstract algebraic structures}, principally fields, instead of
manipulating \emph{explicit algebraic expressions} such as polynomials. The
cubic formula already gave you a taste of how hairy that can get.

Developing Galois theory using abstract algebraic structures helps us to
see its connections to other parts of mathematics, and also has some fringe
benefits. For example, we'll solve some notorious geometry problems that
perplexed the ancient Greeks and remained unsolved for millennia. For that
and many other things, we'll need some of the theory of groups, rings and
fields---and that's what's next.

\chapter{Group actions, rings and fields}
\label{ch:garf}

We now start again. This chapter is a mixture of revision and material that
is likely to be new to you.  The revision is from Fundamentals of Pure
Mathematics, Honours Algebra, and Introduction to Number Theory (if you took
it, which I won't assume). Because much of it is revision, it's a longer
chapter than usual.

\video{Introduction to Week 2}

\section{Group actions}
\label{sec:actions}

Let's begin with a definition from Fundamentals of Pure Mathematics
(Figure~\ref{fig:action}).

\begin{figure}
\centering
\includegraphics[width=110mm]{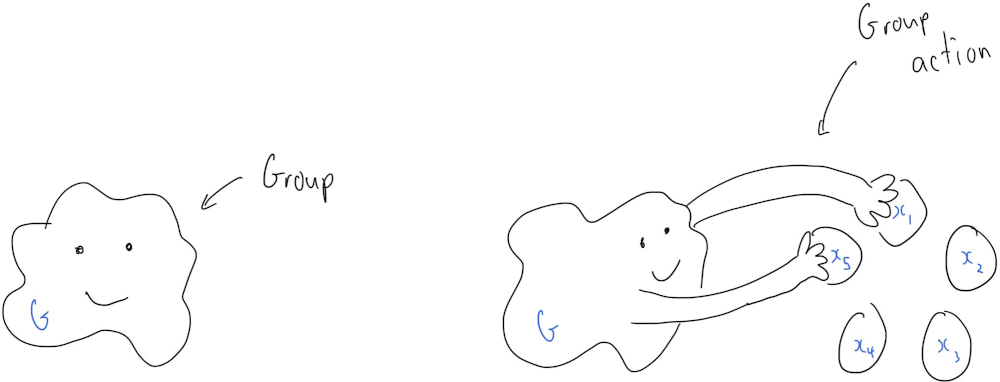}
\caption{Action of a group $G$ on a set $X$. (Image adapted from
\href{https://twitter.com/rowvector/status/1337174208284041221}{@rowvector}.)}
\label{fig:action}
\end{figure}

\begin{defn}
\label{defn:gp-action}
Let $G$ be a group and $X$ a set. An \demph{action} of $G$ on $X$ is a
function $G \times X \to X$, written as $(g, x) \mapsto gx$, such
that
\[
(gh)x = g(hx)
\]
for all $g, h \in G$ and $x \in X$, and
\[
1x = x
\]
for all $x \in X$. Here $1$ denotes the identity element of $G$.
\end{defn}

\begin{examples}
\label{egs:gp-actions}
\begin{enumerate}
\item 
\label{eg:ga-sym}
Let $X$ be a set. There is a group \demph{$\Sym(X)$} whose elements are the
bijections $X \to X$, with composition as the group operation and the
identity function $\id_X\from X \to X$ as the identity of the group. When
$X = \{1, \ldots, n\}$, this group is nothing but $S_n$.

There is an action of $\Sym(X)$ on $X$ defined by
\[
\begin{array}{ccc}
\Sym(X) \times X        &\to            &X      \\
(g, x)                  &\mapsto        &g(x).
\end{array}
\]
Acting on $X$ is what $\Sym(X)$ was born to do!

\item
\label{eg:ga-aut}
Similar examples can be given for many kinds of mathematical object, not
just sets. Generally, an \demph{automorphism} of an object $X$ is an
isomorphism $X \to X$ (preserving whatever structure $X$ has), and the
automorphisms of $X$ form a group \demph{$\Aut(X)$} under composition. It
acts on $X$ just as in~\bref{eg:ga-sym}: $gx = g(x)$, for $g \in \Aut(X)$
and $x \in X$.

For instance, when $X$ is a real vector space, the linear automorphisms
form a group $\Aut(X)$ which acts on the vector space $X$. When $X$ is
finite-dimensional, we can describe this action in more concrete
terms. Writing $n = \dim X$, the vector space $X$ is isomorphic to $\R^n$,
whose elements we will view as column vectors. The group $\Aut(X)$ is
isomorphic to the group of $n \times n$ real invertible matrices under
multiplication, usually called $\GL{n}{\R}$ (`general linear' group). Under
these isomorphisms, the action of $\Aut(X)$ on $X$ becomes
\[
\begin{array}{ccc}
\GL{n}{\R} \times \R^n    &\to            &\R^n   \\
(M, \vc{v})             &\mapsto        &M\vc{v},
\end{array}
\]
where $M\vc{v}$ is the usual matrix product.

\item
\label{eg:ga-cube}
Let $G$ be the 48-element group of isometries (rotations and reflections)
of a cube. Then $G$ acts on the 6-element set of faces of the cube: any
isometry maps faces to faces. It also acts in a similar way on the
12-element set of edges, the 8-element set of vertices, and a little less
obviously, the 4-element set of long diagonals. (The \demph{long diagonals}
are the lines between a vertex and its opposite, furthest-away, vertex.)

\item
\label{eg:ga-triv}
For any group $G$ and set $X$, the \demph{trivial action} of $G$ on $X$ is
given by $gx = x$ for all $g$ and $x$. Nothing moves anything!
\end{enumerate}
\end{examples}

Take an action of a group $G$ on a set $X$. Every group element $g$ gives
rise to a function
\[
\dem{\bar{g}} \from X \to X
\]
defined by
\[
\bar{g}(x) = gx.
\]
In fact, $\bar{g}$ is a bijection, because $\ovln{g^{-1}}$ is the inverse
function of $\bar{g}$. So $\bar{g} \in \Sym(X)$ for each $g \in
G$. For instance, consider the usual action of the isometry group $G$ of
the cube on the set $X$ of faces
(Example~\ref{egs:gp-actions}\bref{eg:ga-cube}). If $g$ is a particular
isometry, then $\bar{g}$ is whatever permutation of the set of faces 
the isometry induces.

We have just seen that whenever $G$ acts on $X$, every element $g$ of the
group $G$ gives rise to an element $\bar{g}$ of the group $\Sym(X)$. So, we
have defined a function
\[
\begin{array}{cccc}
\dem{\Sigma}\from     
                &G      &\to            &\Sym(X)        \\
                &g      &\mapsto        &\bar{g}.
\end{array}
\]
You can check that $\Sigma$ is a group homomorphism. 

\begin{ex}{ex:action-homm}
Check that $\bar{g}$ is a bijection for each $g \in G$. Also check that
$\Sigma$ is a homomorphism.
\end{ex}

In summary: any action of a group $G$ on $X$ gives rise
to a homomorphism $G \to \Sym(X)$, in a natural way.

\begin{examples}
\label{egs:act-homm}
\begin{enumerate}
\item 
\label{eg:ah-sym}
Let $X$ be a set, and consider the action of $\Sym(X)$ on $X$ described in
Example~\ref{egs:gp-actions}\bref{eg:ga-sym}. For each $g \in \Sym(X)$, the
function $\bar{g}\from X \to X$ is just $g$ itself. Hence the homomorphism
$\Sigma\from \Sym(X) \to \Sym(X)$ is the identity.

\item
\label{eg:ah-aut}
Similarly, take a real vector space $X$ and consider the action of
$\Aut(X)$ on $X$ described in
Example~\ref{egs:gp-actions}\bref{eg:ga-aut}. The resulting
homomorphism $\Sigma\from \Aut(X) \to \Sym(X)$ is the inclusion; that is,
$\Sigma(g) = g$ for all $g \in \Aut(X)$. (The domain of $\Sigma$ is the
group of \emph{linear} bijections $X \to X$, whereas the codomain is the
group of \emph{all} bijections $X \to X$.)

\item 
\label{eg:ah-cube}
Consider the usual action of the isometry group $G$ of the cube on the set
$X$ of edges (Example~\ref{egs:gp-actions}\bref{eg:ga-cube}). Since $X$ has
12 elements, $\Sym(X) \iso S_{12}$, and $\Sigma$ amounts to a homomorphism
$G \to S_{12}$.

\item
\label{eg:ah-triv}
The trivial action of a group $G$ on a set $X$
(Example~\ref{egs:gp-actions}\bref{eg:ga-triv}) corresponds to the trivial
homomorphism $G \to \Sym(X)$.
\end{enumerate}
\end{examples}

\begin{remark}
\label{rmk:act-homm-indices}
When $X$ is finite, we often choose an ordering of its elements, writing $X =
\{x_1, \ldots, x_k\}$. Then $\Sym(X) \iso S_k$ (assuming the $x_i$s are all
distinct). For each $g \in G$ and $i \in \{1, \ldots, k\}$, the element
$gx_i$ of $X$ must be equal to $x_j$ for some $j$. Write that $j$ as
$\sigma_g(i)$, so that
\[
gx_i = x_{\sigma_g(i)}.
\]
Then $\sigma_g \in S_k$, and the composite homomorphism
\[
G \toby{\Sigma} \Sym(X) \iso S_k
\]
is $g \mapsto \sigma_g$.
\end{remark}

\begin{digression}{dig:action-homm}
In fact, an action of $G$ on $X$ is \emph{the same thing as} a homomorphism
$G \to \Sym(X)$. What I mean is that there is a natural one-to-one
correspondence between actions of $G$ on $X$ and homomorphisms $G \to
\Sym(X)$. Some books even \emph{define} an action of $G$ on $X$ to be a
homomorphism $G \to \Sym(X)$.

In detail: we've just seen how an action of $G$ on $X$ gives rise to a
homomorphism $\Sigma\from G \to \Sym(X)$. In the other direction, take any
homomorphism $\Sigma\from G \to \Sym(X)$. Define a function $G \times X \to
X$ by 
\[
(g, x) \mapsto \bigl(\Sigma(g)\bigr)(x).
\]
(To make sense of the right-hand side: $\Sigma(g)$ is an element of the
group $\Sym(X)$, which is the set of bijections $X \to X$, so we can apply
the function $\Sigma(g)$ to the element $x$ to obtain another element
$(\Sigma(g))(x)$ of $X$.) You can check that this function $G \times X \to
X$ is an action of $G$ on $X$. So, we've now seen how to convert an action
into a homomorphism and vice versa. These two processes are mutually
inverse. Hence actions of $G$ on $X$ correspond one-to-one with homomorphisms
$G \to \Sym(X)$.

At the purely set-theoretic level (ignoring the group structures), the key
is that for any sets $A$, $B$ and $C$, there's a natural bijection
\[
C^{A \times B} \iso (C^B)^A.
\]
Here \demph{$C^B$} means the set of functions $B \to C$. The general proof
is very similar to what we've just done (where $A = G$ and $B = C = X$). In
words, a function $A \times B \to C$ can be seen as a way of assigning to
each element of $A$ a function $B \to C$. In a picture:
\[
\includegraphics[width=30mm]{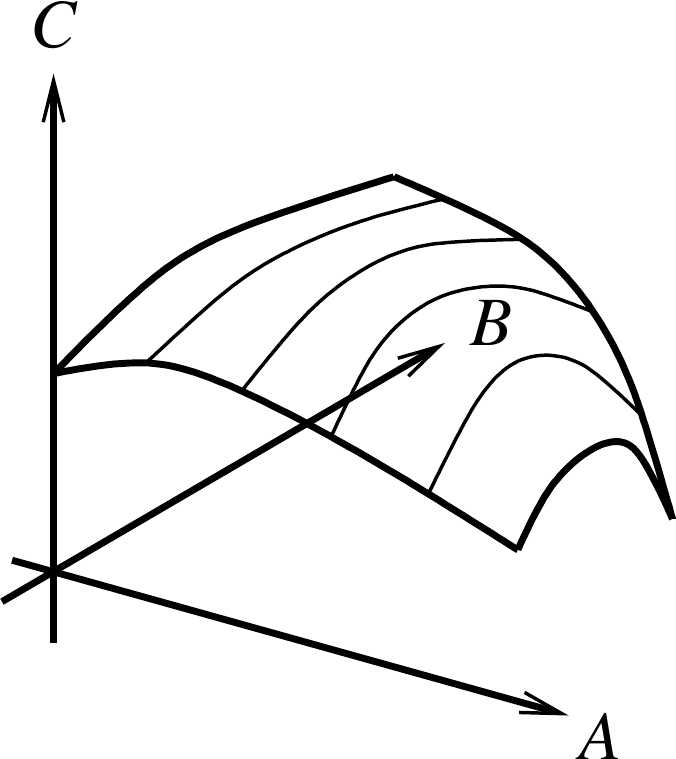}
\]
Here $A = B = C = \R$. By slicing up the surface as shown, a function $\R^2 \to
\R$ can be seen as a function from $\R$ to $\{\text{functions } \R \to \R\}$.
\end{digression}

\begin{defn}
An action of a group $G$ on a set $X$ is \demph{faithful} if for $g, h \in
G$,
\[
gx = hx \text{ for all } x \in X
\implies
g = h.
\]
\end{defn}

Faithfulness means that if two elements of the group \emph{do} the same,
they \emph{are} the same. Here are some other ways to express it.

\begin{lemma}
\label{lemma:ff-tfae}
For an action of a group $G$ on a set $X$, the following are equivalent:
\begin{enumerate}
\item 
\label{part:fft-f}
the action is faithful;

\item
\label{part:fft-g}
for $g \in G$, if $gx = x$ for all $x \in X$ then $g = 1$;

\item
\label{part:fft-inj}
the homomorphism $\Sigma\from G \to \Sym(X)$ is injective;

\item
\label{part:fft-ker}
$\ker \Sigma$ is trivial.
\end{enumerate}
\end{lemma}

\begin{proof}
Faithfulness states that whenever $g, h \in G$ with $\bar{g} = \bar{h}$,
then $g = h$. But $\Sigma(g) = \bar{g}$, so
\bref{part:fft-f}$\iff$\bref{part:fft-inj}. Similarly,
\bref{part:fft-g}$\iff$\bref{part:fft-ker}. Finally, it is a standard fact
that a homomorphism is injective if and only if its kernel
is trivial, so \bref{part:fft-inj}$\iff$\bref{part:fft-ker}. 
\end{proof}

Many common actions are faithful:

\begin{examples}
\label{egs:faithful}
\begin{enumerate}
\item 
The natural action of $\Sym(X)$ on a set $X$
(Examples~\ref{egs:gp-actions}\bref{eg:ga-sym}
and~\ref{egs:act-homm}\bref{eg:ah-sym}) is faithful, since the
corresponding homomorphism $\id\from \Sym(X) \to \Sym(X)$ is injective.

\item
Similarly, the natural action of $\Aut(X)$ on a vector space
(Examples~\ref{egs:gp-actions}\bref{eg:ga-aut} 
and~\ref{egs:act-homm}\bref{eg:ah-aut}) is faithful, since the
corresponding homomorphism $\Aut(X) \to \Sym(X)$ is injective.

\item
\label{eg:faithful-cube}
The action of the isometry group $G$ of the cube on the set of faces
(Examples~\ref{egs:gp-actions}\bref{eg:ga-cube}
and~\ref{egs:act-homm}\bref{eg:ah-cube})
is faithful, since an isometry is determined by its effect on faces. The
same is true for edges and vertices. 

But the action of $G$ on the $4$-element set $X$ of long diagonals is not
faithful: for $G$ has $48$ elements, whereas $\Sym(X)$ has only $4! = 24$
elements, so the homomorphism $\Sigma\from G \to \Sym(X)$ cannot be
injective.

\item
The trivial action of a group $G$ on a set $X$ is never faithful unless $G$
itself is trivial, since $gx = x$ for all $g \in G$ and $x \in X$.
\end{enumerate}
\end{examples}

\begin{ex}{ex:cube-faithful}
Example~\ref{egs:faithful}\bref{eg:faithful-cube} shows that the
action of the isometry cube $G$ of the cube on the set $X$ of long
diagonals is not faithful. By Lemma~\ref{lemma:ff-tfae}, there must be some
non-identity isometry of the cube that fixes all four long diagonals. In
fact, there is exactly one. What is it?
\end{ex}

When a group $G$ acts faithfully on a set $X$, there is a copy of $G$
sitting inside $\Sym(X)$ as a subgroup (a `faithful representation' of
$G$):

\begin{lemma}
\label{lemma:faith-rep}
Let $G$ be a group acting faithfully on a set $X$. Then $G$ is isomorphic
to the subgroup
\[
\im\Sigma = \{ \bar{g} \such g \in G\}
\]
of $\Sym(X)$, where $\Sigma\from G \to \Sym(X)$ and $\bar{g}$ are defined
as above.
\end{lemma}

\begin{proof}
By Lemma~\ref{lemma:ff-tfae}, $\Sigma$ is injective, and it is a general
group-theoretic fact that any injective homomorphism $\phi \from G \to H$
induces an isomorphism between $G$ and $\im\phi$. 
\end{proof}

\begin{example}
Consider the usual action of the isometry group $G$ of the cube on the
$8$-element set $X$ of vertices. As we have seen, this action is
faithful. Hence the associated homomorphism
\[
\begin{array}{cccc}
\Sigma\from     &G       &\to           &\Sym(X)       \\
                &g       &\mapsto       &\bar{g}
\end{array}
\]
induces an isomorphism between $G$ and the subgroup $\{ \bar{g} \such g \in
G\}$ of $\Sym(X)$. The subgroup consists of all permutations of the set of
vertices that come from some isometry. For instance, there is no isometry
that exchanges two vertices but leaves the rest fixed, so this subgroup
contains no 2-cycles.
\end{example}

\begin{remark}
\label{rmk:faith-indices}
How does Lemma~\ref{lemma:faith-rep} look when $X$ is a finite set with
elements $x_1, \ldots, x_k$? Then $\Sym(X) \iso S_k$, and as in
Remark~\ref{rmk:act-homm-indices}, we can write $gx_i =
x_{\sigma_g(i)}$. It follows from that lemma and remark
that $G$ is isomorphic to the subgroup $\{\sigma_g \such g \in G\}$ of
$S_k$ (which \emph{is} a subgroup). The isomorphism is given by $g \mapsto
\sigma_g$. 
\end{remark}

Faithfulness is about which elements of the group fix everything in the
set. We can also ask which elements of the set are fixed by everything in
the group---or more generally, by some prescribed set $S$ of group
elements. 

\begin{defn}
\label{defn:fixed-set}
Let $G$ be a group acting on a set $X$. Let $S \sub G$. The \demph{fixed
set} of $S$ is 
\[
\dem{\Fix(S)} = \{x \in X \such sx = x \text{ for all } s \in S\}.
\]
\end{defn}

Later, we'll need the following lemma. 

\begin{lemma}
\label{lemma:act-fix-conj}
Let $G$ be a group acting on a set $X$, let $S \sub G$, and let $g \in
G$. Then $\Fix(gSg^{-1}) = g\Fix(S)$.
\end{lemma}

Here $gSg^{-1} = \{gsg^{-1} \such s \in S\}$ and $g\Fix(S) = \{gx \such x
\in \Fix(S)\}$. 

\begin{proof}
For $x \in X$, we have
\begin{align*}
x \in \Fix(gSg^{-1})    &
\iff
gsg^{-1}x = x \text{ for all } s \in S        \\
&
\iff
sg^{-1}x = g^{-1}x \text{ for all } s \in S     \\
&
\iff
g^{-1}x \in \Fix(S)     \\
&
\iff
x \in g\Fix(S). 
\end{align*}
\end{proof}

\section{Rings}
\label{sec:rings}

We'll begin this part with some stuff you know---but with a twist.

In this course, the word \demph{ring} means commutative ring with $1$
(multiplicative identity). Noncommutative rings and rings without $1$ are
important in some parts of mathematics, but since we'll be focusing on
commutative rings with $1$, it will be easier to just call them `rings'.

\begin{example}
There are many ways of building new rings from old. One of the most
fundamental is that from any ring $R$, we can build the ring \demph{$R[t]$}
of polynomials over $R$.  We will define $R[t]$ formally and study it in
detail in Chapter~\ref{ch:polys}.
\end{example}

Given rings $R$ and $S$, a \demph{homomorphism} from $R$ to $S$ is a
function $\phi \from R \to S$ satisfying the equations
\begin{align*}
\phi(r + r')  &= \phi(r) + \phi(r'),      &
\phi(0)       &= 0,   &
\phi(-r)      &= -\phi(r),  \\
\phi(r r')    &= \phi(r) \phi(r'),        &
\phi(1)       &= 1 \text{ \color{Red!100} (note this!)}
\end{align*}
for all $r, r' \in R$. For example, complex conjugation is a homomorphism
$\C \to \C$. It is a very useful lemma that if 
\begin{align*}
\phi(r + r')  = \phi(r) + \phi(r'),    
\qquad
\phi(r r')    = \phi(r) \phi(r'),       
\qquad
\phi(1)       = 1
\end{align*}
for all $r, r' \in R$ then $\phi$ is a homomorphism. In other words, to
show that $\phi$ is a homomorphism, you only need to check it preserves
$+$, $\cdot$ and $1$; preservation of $0$ and negatives then comes for
free. But you \emph{do} need to check it preserves $1$. That doesn't follow
from the other conditions.

A \demph{subring} of a ring $R$ is a subset $S \sub R$ that contains $0$
{\color{Red!100}and $1$} and is closed under addition, multiplication and
negatives. Whenever $S$ is a subring of $R$, the inclusion $\iota \from S
\to R$ (defined by $\iota(s) = s$) is a homomorphism.

\begin{warning}{wg:1}
In Honours Algebra, rings had $1$s but homomorphisms were \emph{not}
required to preserve $1$. Similarly, subrings of $R$ had to have a $1$, but
it was \emph{not} required to be the same as the $1$ of $R$.

For example, take the ring $\C$, the noncommutative ring $M$ of $2 \times
2$ matrices over $\C$, and the function $\phi \from \C \to M$ defined by
\[
\phi(z) =
\begin{pmatrix}
z       &0      \\
0       &0
\end{pmatrix}.
\]
In the terminology of Honours Algebra, $\phi$ is a homomorphism and its
image $\im \phi$ is a subring of $M$. But in our terminology, $\phi$ is not
a homomorphism (as $\phi(1) \neq I$) and $\im \phi$ is not a subring of $M$
(as $I \not\in \im \phi$).
\end{warning}

\begin{lemma}
\label{lemma:int-subrings}
Let $R$ be a ring and let $\mathcal{S}$ be any set (perhaps infinite) of
subrings of $R$. Then their intersection $\bigcap_{S \in \mathcal{S}} S$ is
also a subring of $R$.
\end{lemma}

In contrast, in the Honours Algebra setup, even the intersection of
\emph{two} subrings need not be a subring.

\begin{proof}
Write $T = \bigcap_{S \in \mathcal{S}} S$. 

For each $S \in \mathcal{S}$, we have $0 \in S$ since $S$ is a
subring. Hence $0 \in T$ by definition of intersection.

Let $r, s \in T$. For each $S \in \mathcal{S}$, we have $r, s \in S$ by
definition of intersection, so $r + s \in S$ since $S$ is a subring. Hence
$r + s \in T$ by definition of intersection.

Similar arguments show that $r \in T \implies -r \in T$, that $1 \in T$,
and that $r, s \in T \implies rs \in T$.
\end{proof}

\begin{example}
\label{eg:Z-initial}
For any ring $R$, there is exactly one homomorphism $\Z \to R$. Here is a
sketch of the proof. 

To show there is \emph{at least} one homomorphism $\chi \from \Z \to R$, we
construct one. Define $\chi$ inductively on integers $n \geq 0$
by $\chi(0) = 0$ and $\chi(n + 1) = \chi(n) + 1_R$. Thus, 
\[
\chi(n) = 1_R + \cdots + 1_R.
\]
Define $\chi$ on negative integers $n$ by $\chi(n) = -\chi(-n)$. A series
of tedious checks shows that $\chi$ is indeed a ring homomorphism.

To show there is \emph{only} one homomorphism $\Z \to R$, let $\phi$ be
any homomorphism $\Z \to R$; we have to prove that $\phi =
\chi$. Certainly $\phi(0) = 0 = \chi(0)$. Next prove by induction on
$n$ that $\phi(n) = \chi(n)$ for nonnegative integers $n$. I leave the
details to you, but the crucial point is that \emph{because homomorphisms
preserve $1$}, we must have
\[
\phi(n + 1) = \phi(n) + \phi(1) = \phi(n) + 1_R
\]
for all $n \geq 0$. Once we have shown that $\phi$ and $\chi$ agree on the
nonnegative integers, it follows that for negative $n$,
\[
\phi(n) = -\phi(-n) = -\chi(-n) = \chi(n).
\]
Hence $\phi(n) = \chi(n)$ for all $n \in \Z$; that is, $\phi = \chi$. 

Usually we write $\chi(n)$ as \demph{$n \cdot 1_R$}, or simply as
\dem{$n$}
\video{The meaning of `$n \cdot 1$', and Exercise~\ref{ex:char-conn}}
if it is clear from the context that $n$ is to be interpreted as an element
of $R$. So for $n \geq 0$,
\[
n \cdot 1_R = \underbrace{1_R + \cdots + 1_R}_{n \text{ times}}.
\]
The dot in the expression `$n \cdot 1_R$' is not multiplication in
any ring, since $n \in \Z$ but $1_R \in R$. It's just notation.
\end{example}

Every ring homomorphism $\phi\from R \to S$ has an image \demph{$\im\phi$},
which is a subring of $S$, and a kernel \demph{$\ker\phi$}, which is an
ideal of $R$.

\begin{warning}{wg:ideal-subring}
Subrings are analogous to subgroups, and ideals are analogous to normal
subgroups. But whereas normal subgroups are a special kind of subgroup,
ideals are \emph{not} a special kind of subring! Subrings must contain $1$,
but most ideals don't.
\end{warning}

\begin{ex}{ex:ideal-subring}
Prove that the only subring of a ring $R$ that is also an ideal is $R$
itself. 
\end{ex}

Given an ideal $I \idl R$, we obtain the quotient ring or factor ring
\demph{$R/I$} and the canonical homomorphism $\dem{\pi_I} \from R \to R/I$,
which is surjective and has kernel $I$.
\video{Quotient rings}

As explained in Honours Algebra, the quotient ring together with the
canonical homomorphism has a `universal\label{p:univ-qt} property': given
any ring $S$ and any homomorphism $\phi \from R \to S$ satisfying $\ker
\phi \supseteq I$, there is exactly one homomorphism $\bar{\phi} \from R/I
\to S$ such that this diagram commutes:
\begin{align*}
\xymatrix{
R \ar[d]_{\pi_I} \ar[dr]^{\phi} &       \\
R/I \ar@{.>}[r]_{\bar{\phi}}    &S.
}
\end{align*}
(For a diagram to \demph{commute} means that whenever there are two
different paths from one object to another, the composites along the two
paths are equal. 
Here, it means that $\phi = \bar{\phi} \of \pi_I$.) The
first\label{p:fit} isomorphism theorem says that if $\phi$ is surjective
and has kernel \emph{equal} to $I$ then $\bar{\phi}$ is an isomorphism. So
$\pi_I \from R \to R/I$ is essentially the only surjective homomorphism out
of $R$ with kernel $I$.

\begin{digression}{dig:first-iso}
Loosely, the ideals of a ring $R$ correspond one-to-one with the surjective
homomorphisms out of $R$. This means four things:
\begin{itemize}
\item given an ideal $I \idl R$, we get a surjective homomorphism out of
$R$ (namely, $\pi_I\from R \to R/I$);

\item given a surjective homomorphism $\phi$ out of $R$, we
get an ideal of $R$ (namely, $\ker\phi$);

\item if we start with an ideal $I$ of $R$, take its associated surjective
homomorphism $\pi_I\from R \to R/I$, then take \emph{its} associated ideal,
we end up where we started (that is, $\ker(\pi_I) = I$);

\item if we start with a surjective homomorphism $\phi \from R \to S$, take
its associated ideal $\ker\phi$, then take \emph{its} associated surjective
homomorphism $\pi_{\ker\phi} \from R \to R/\ker\phi$, we end up where we
started (at least `up to isomorphism', in that we have the isomorphism
$\bar{\phi} \from R/\ker\phi \to S$ making the triangle commute). This is
the first isomorphism theorem.
\end{itemize}

Analogous stories can be told for groups and modules.
\end{digression}

An \demph{integral domain} is a ring $R$ such that $0_R \neq
1_R$ and for $r, r' \in R$,
\[
rr' = 0 \implies r = 0 \text{ or } r' = 0.
\]

\begin{ex}{ex:trivring}
The \demph{trivial ring} or \demph{zero ring} is the one-element set with
its only possible ring structure. Show that the only ring in which $0 = 1$
is the trivial ring.
\end{ex}

Equivalently, an integral domain is a nontrivial ring in which cancellation
is valid: $rs = r's$ implies $r = r'$ or $s = 0$.

\begin{warning}{wg:cant-cancel}
In an \emph{arbitrary} ring, you can't reliably cancel by nonzero elements. For
example, in the ring $\Z/\idlgen{6}$ of integers mod $6$, we have $1\times
2 = 4\times 2$ but $1 \neq 4$.
\end{warning}

\begin{digression}{dig:defn-id}
Why is the condition $0 \neq 1$ in the definition of integral domain?

My answer begins with a useful general point: the sum of no things
should always be interpreted as $0$. (The amount you pay in a shop is the
sum of the prices of the individual things. If you buy no things, you pay
\pounds 0.) This is ultimately because $0$ is the identity for addition.

Similarly, the product of no things should be interpreted as $1$. One
justification is that $1$ is the identity for multiplication. Another is
that if we want laws like $\exp(\sum x_i) = \prod \exp(x_i)$ to hold, and
if we believe that the sum of no things is $0$, then the product of no
things should be $1$. Or if we want every positive integer to be a
product of primes, we'd better say that $1$ is the product
of no primes. It's a convention to let us handle trivial cases smoothly.

Now consider the following condition on a ring $R$: for all $n \geq 0$ and
$r_1, \ldots, r_n \in R$,
\begin{align}
\label{eq:id-unbiased}
r_1 r_2 \cdots r_n = 0 
\implies 
\text{there exists } i \in \{1, \ldots, n\} 
\text{ such that } r_i = 0.
\end{align}
For $n = 2$, this is the main condition in the definition of integral
domain.  For $n = 0$, it says: if $1 = 0$ then there exists $i \in
\emptyset$ such that $r_i = 0$. But any statement beginning `there exists
$i \in \emptyset$' is false! So in the case $n = 0$,
condition~\bref{eq:id-unbiased} states that $1 \neq 0$. Hence `$1 \neq 0$'
is the 0-fold analogue of the main condition.

On the other hand, if~\bref{eq:id-unbiased} holds for $n = 0$ and $n = 2$
then a simple induction shows that it holds for all $n \geq 0$. Conclusion:
an integral domain can equivalently be defined as a ring in
which~\bref{eq:id-unbiased} holds for all $n \geq 0$.
\end{digression}

Let $Y$ be a subset of a ring $R$. The \demph{ideal $\idlgen{Y}$ generated
by $Y$} is defined as the intersection of all the ideals of $R$ containing
$Y$. You can show that any intersection of ideals is an ideal (much as for
subrings in Lemma~\ref{lemma:int-subrings}). So $\idlgen{Y}$ \emph{is} an
ideal. We can also characterize $\idlgen{Y}$ as the smallest ideal of $R$
containing $Y$. That is, $\idlgen{Y}$ is an ideal containing $Y$,
and if $I$ is another ideal containing $Y$ then $\idlgen{Y} \sub I$.

This definition of the ideal generated by $Y$ is top-down: we obtain
$\idlgen{Y}$ as the intersection of bigger ideals.  But there is also a
useful bottom-up description of $\idlgen{Y}$. Here it is when $Y$ is
finite.

\begin{lemma}
\label{lemma:idl-gen-exp}
Let $R$ be a ring and let $Y = \{r_1, \ldots, r_n\}$ be a finite
subset. Then
\[
\idlgen{Y}
=
\{ a_1 r_1 + \cdots + a_n r_n \such a_1, \ldots, a_n \in R \}.
\]
\end{lemma}

\begin{proof}
Write $I$ for the right-hand side. It is straightforward to check that 
$I$ is an ideal of $R$, and it contains $Y$ because, for instance, $r_1 = 1
r_1 + 0 r_2 + \cdots + 0 r_n$. 

Now let $J$ be any ideal of $R$ containing $Y$. Let $a_1, \ldots, a_n \in
R$. For each $i$, we have $r_i \in J$ since $J$ contains $Y$, and so $a_i
r_i \in J$ since $J$ is an ideal. Hence $\sum a_i r_i \in J$, again since
$J$ is an ideal. So $I \sub J$.

Hence $I$ is the smallest ideal of $R$ containing $Y$, that is, $I =
\idlgen{Y}$.
\end{proof}

\begin{digression}{dig:top-bottom}
A similar interplay between top-down and bottom-up appears in other parts of
mathematics. 

For example, in topology, the closure of a subset of a metric or
topological space is the intersection of all closed subsets containing
it. In linear algebra, the span of a subset of a vector space is the
intersection of all linear subspaces containing it. In group theory, the
subgroup generated by a subset of a group is the intersection of all
subgroups containing it.

These are all top-down definitions, but there are equivalent bottom-up
definitions, describing explicitly which elements belong to the
subset. Sometimes we're lucky and those descriptions are simple. For
instance, closures can easily be described in terms of limit points, and
spans are just sets of linear combinations. But sometimes it
gets more complicated. For example, the subgroup of a group $G$ generated
by a subset $Y$ can be described \emph{informally} as the set of elements
of $G$ that can be obtained from $Y$ by taking products and inverses, but
expressing that precisely is a little bit fiddly.

It's worth getting comfortable with the top-down style of definition, as
it works well in cases where the bottom-up approach is prohibitively
complicated, and we'll need it later. 
\end{digression}

When $Y = \{r_1, \ldots, r_n\}$, we write $\idlgen{Y}$ as
\demph{$\idlgen{r_1, \ldots, r_n}$} rather than $\idlgen{\{r_1, \ldots,
r_n\}}$.  In particular, when $n = 1$, Lemma~\ref{lemma:idl-gen-exp}
implies that
\[
\idlgen{r} = \{ ar \such a \in R \}.
\]
Ideals of the form $\idlgen{r}$ are called \demph{principal ideals}.  A
\demph{principal ideal domain} is an integral domain in which every ideal
is principal.

\begin{example}
\label{eg:Z-pid}
$\Z$ is a principal ideal domain. Indeed, if $I \idl \Z$ then either $I =
\{0\}$, in which case $I = \idlgen{0}$, or $I$ contains some positive
integer, in which case we can define $n$ to be the least positive integer
in $I$ and use the division algorithm to show that $I = \idlgen{n}$.
\end{example}

\begin{ex}{ex:Z-pid}
Fill in the details of Example~\ref{eg:Z-pid}.\\
\end{ex}

Let $r$ and $s$ be elements of a ring $R$. We say that $r$ \demph{divides}
$s$, and write \demph{$r \dvd s$}, if there exists $a \in R$ such that $s =
ar$. This condition is equivalent to $s \in \idlgen{r}$, and to $\idlgen{s}
\sub \idlgen{r}$. 

An element $u \in R$ is a \demph{unit} if it has a multiplicative inverse,
or equivalently if $\idlgen{u} = R$. The units form a group
\demph{$R^\times$} under multiplication. For instance, $\Z^\times = \{1,
-1\}$.

\begin{ex}{ex:associates}
Let $r$ and $s$ be elements of an integral domain. Show that $r\dvd s \dvd
r \iff \idlgen{r} = \idlgen{s} \iff s = ur$ for some unit $u$.
\end{ex}

Elements $r$ and $s$ of a ring are \demph{coprime} if for $a \in R$,
\[
a \dvd r \text{ and } a \dvd s \implies a \text{ is a unit}.
\]

\begin{propn}
\label{propn:bezout}
Let $R$ be a principal ideal domain and $r, s \in R$. Then
\begin{align*}
r \text{ and } s \text{ are coprime}    
\iff
ar + bs = 1 \text{ for some } a, b \in R.
\end{align*}
\end{propn}

\begin{proof}
$\Rightarrow$: suppose that $r$ and $s$ are coprime. Since $R$ is a
principal ideal domain, $\idlgen{r, s} = \idlgen{u}$ for some $u \in
R$. Since $r \in \idlgen{r, s} = \idlgen{u}$, we must have $u \dvd r$, and
similarly $u \dvd s$. But $r$ and $s$ are coprime, so $u$ is a unit. Hence
$1 \in \idlgen{u} = \idlgen{r, s}$. But by Lemma~\ref{lemma:idl-gen-exp},
\[
\idlgen{r, s}
=
\{ ar + bs \such a, b \in R\},
\]
and the result follows.

$\Leftarrow$: suppose that $ar + bs = 1$ for some $a, b \in R$. If $u \in
R$ with $u \dvd r$ and $u \dvd s$ then $u \dvd (ar + bs) = 1$, so $u$ is a
unit. Hence $r$ and $s$ are coprime.
\end{proof}

\section{Fields}
\label{sec:fields}

A \demph{field} is a ring $K$ in which $0 \neq 1$ and every nonzero element
is a unit. Equivalently, it is a ring such that $K^\times = K \without
\{0\}$.\label{p:mult-fd} Every field is an integral domain.

\begin{ex}{ex:eg-fd}
Write down all the examples of fields that you know.\\
\end{ex}

As we go on, we'll see several ways of making new fields out of old. Here's
the simplest.

\begin{example}
\label{eg:rat}
Let $K$ be a field. A \demph{rational expression} over $K$ is a ratio of
two polynomials
\[
\frac{f(t)}{g(t)},
\]
where $f(t), g(t) \in K[t]$ with $g \neq 0$. Two such expressions,
$f_1/g_1$ and $f_2/g_2$, are regarded as equal if $f_1 g_2 = f_2 g_1$ in
$K[t]$. So formally, a rational expression is an equivalence class of pairs
$(f, g)$ under the equivalence relation in the last sentence. The set
of rational expressions over $K$ is denoted by \demph{$K(t)$}.

Rational expressions are added, subtracted and multiplied in the ways you'd
expect, making $K(t)$ into a field. We will look at it more carefully in
Chapter~\ref{ch:polys}. 
\end{example}

A field $K$ has exactly two ideals: $\{0\}$ and $K$. For if $\{0\} \neq I
\idl K$ then $u \in I$ for some $u \neq 0$; but then $u$ is a unit, so
$\idlgen{u} = K$, so $I = K$.

\begin{lemma}
\label{lemma:fd-homm}
Every homomorphism between fields is injective.
\end{lemma}

A `homomorphism between fields' means a \emph{ring} homomorphism.

\begin{proof}
Let $\phi \from K \to L$ be a homomorphism between fields. Then $\ker\phi
\idl K$, so $\ker\phi$ is either $\{0\}$ or $K$. If $\ker\phi = K$ then
$\phi(1) = 0$; but $\phi(1) = 1$ by definition of homomorphism, so $0 = 1$
in $L$, contradicting the assumption that $L$ is a field. Hence $\ker\phi =
\{0\}$, that is, $\phi$ is injective.
\end{proof}

\begin{warning}{wg:fd-homm}
With the Honours Algebra definition of homomorphism,
Lemma~\ref{lemma:fd-homm} would be false, since the map with constant value
$0$ would be a homomorphism.
\end{warning}

\begin{ex}{ex:homm-inv}
Let $\phi \from K \to L$ be a homomorphism of fields and let $0 \neq a \in
K$. Prove that $\phi(a^{-1}) = \phi(a)^{-1}$. Why is $\phi(a)^{-1}$ defined?
\end{ex}

A \demph{subfield} of a field $K$ is a subring that is a field. 

\begin{lemma}
\label{lemma:subfd-im-inv}
Let $\phi \from K \to L$ be a homomorphism between fields.
\begin{enumerate}
\item 
\label{part:sii-im}
For any subfield $K'$ of $K$, the image $\phi K'$ is a subfield of $L$.

\item
\label{part:sii-inv}
For any subfield $L'$ of $L$, the preimage $\phi^{-1} L'$ is a
subfield of $K$.
\end{enumerate}
\end{lemma}

\begin{proof}
For~\bref{part:sii-im}, you know from Proposition~3.4.28 of Honours Algebra
that $\phi K'$ is a subring of $L$, and you can use
Exercise~\ref{ex:homm-inv} above to show that if $0 \neq b \in \phi K'$ then
$b^{-1} \in \phi K'$. The proof of~\bref{part:sii-inv} is similar.
\end{proof}

Whenever we have a collection of homomorphisms between the same pair of
fields, we get a subfield in the following way.

\begin{defn}
\label{defn:equalizer}
Let $X$ and $Y$ be sets, and let $S \sub \{\text{functions $X \to
Y$}\}$. The \demph{equalizer} of $S$ is 
\[
\dem{\Eq(S)} 
= 
\{ x \in X \such f(x) = g(x) \text{ for all } f, g \in S \}.
\] 
\end{defn}

In other words, it is the part of $X$ where all the functions in $S$ are
\emph{equal}.

\begin{lemma}
\label{lemma:eq-subfd}
Let $K$ and $L$ be fields, and let $S \sub \{\text{homomorphisms } K \to
L\}$. Then $\Eq(S)$ is a subfield of $K$.
\end{lemma}

\begin{proof}
We must show that $0, 1 \in \Eq(S)$, that if $a \in \Eq(S)$ then $-a \in
\Eq(S)$ and $1/a \in \Eq(S)$ (for $a \neq 0$), and that if $a, b \in
\Eq(S)$ then $a + b, ab \in \Eq(S)$.  I will show just the last of these,
leaving the rest to you. 

Suppose that $a, b \in \Eq(S)$. For all $\phi, \theta \in S$, we have
\[
\phi(ab)
=
\phi(a)\phi(b)
=
\theta(a)\theta(b)
=
\theta(ab),
\]
so $ab \in \Eq(S)$.
\end{proof}

\begin{example}
Let $K = L = \C$. Let $S = \{\id_\C, \kappa\}$, where $\kappa \from \C \to
\C$ is complex conjugation. Then 
\[
\Eq(S) = \{z \in \C \such z = \ovln{z} \} = \R,
\]
and $\R$ is indeed a subfield of $\C$.
\end{example}

Next we ask: when is $1 + \cdots + 1$ equal to $0$?

Let $R$ be any ring. By Example~\ref{eg:Z-initial}, there is a unique
homomorphism $\chi \from \Z \to R$. Its kernel is an ideal of the principal
ideal domain $\Z$. Hence $\ker\chi = \idlgen{n}$ for a unique integer $n
\geq 0$. This $n$ is called the \demph{characteristic} of $R$, and written
as \demph{$\chr R$}. So for $m \in \Z$, we have $m \cdot 1_R = 0$ if and
only if $m$ is a multiple of $\chr R$. Or equivalently,
\begin{align}
\label{eq:defn-char}
\chr R =
\begin{cases}
\text{the least $n > 0$ such that $n \cdot 1_R = 0_R$,}   & 
\text{if such an $n$ exists;}\\
0,      &
\text{otherwise}.
\end{cases}
\end{align}

The concept of characteristic is mostly used in the case of fields.

\begin{examples}
\label{egs:char-fds}
\begin{enumerate}
\item
$\Q$, $\R$ and $\C$ all have characteristic $0$. 

\item
For a prime number $p$, we write \demph{$\F_p$} for the field $\Z/\idlgen{p}$ of
integers modulo $p$. Then $\chr\F_p = p$.

\item
For any field $K$, the field $K(t)$ of rational expressions has the same
characteristic as $K$.
\end{enumerate}
\end{examples}

\begin{lemma}
\label{lemma:char-0p}
The characteristic of an integral domain is $0$ or a prime number.
\end{lemma}

\begin{proof}
Let $R$ be an integral domain and write $n = \chr R$. Suppose that $n > 0$;
we must prove that $n$ is prime.

Since $1 \neq 0$ in an integral domain, $n \neq 1$. (Remember that $1$ is
not a prime! So that step was necessary.) Now let $k, m > 0$ with $km =
n$. Writing $\chi$ for the unique homomorphism $\Z \to R$, we have 
\[
\chi(k) \chi(m) = \chi(km) = \chi(n) = 0,
\]
and $R$ is an integral domain, so $\chi(k) = 0$ or $\chi(m) = 0$. WLOG,
$\chi(k) = 0$. But $\ker\chi = \idlgen{n}$, so $n \dvd k$, so $k =
n$. Hence $n$ is prime.
\end{proof}

In particular, the characteristic of a field is always $0$ or a prime.  But
there is no way of mapping between fields of different characteristics:

\begin{lemma}
\label{lemma:same-char}
Let $\phi \from K \to L$ be a homomorphism of fields. Then $\chr K = \chr L$.
\end{lemma}

\begin{proof}
Write $\chi_K$ and $\chi_L$ for the unique homomorphisms from $\Z$ to $K$
and $L$, respectively. Since $\chi_L$ is the \emph{unique} homomorphism $\Z
\to L$, the triangle 
\[
\xymatrix{
                &\Z \ar[dl]_{\chi_K} \ar[dr]^{\chi_L}   &       \\
K \ar[rr]_{\phi}&                                       &L
}
\]
commutes. (Concretely, this says that $\phi(n \cdot 1_K) = n \cdot 1_L$
for all $n \in \Z$.) Hence $\ker(\phi \of \chi_K) = \ker \chi_L $. But
$\phi$ is injective by Lemma~\ref{lemma:fd-homm}, so $\ker(\phi \of \chi_K)
= \ker \chi_K $. Hence $\ker\chi_K = \ker\chi_L$, or equivalently, $\chr K
= \chr L$.
\end{proof}

For example, the inclusion $\Q \to \R$ is a homomorphism of fields, and
both have characteristic $0$.

\begin{ex}{ex:char-conn}
This proof of Lemma~\ref{lemma:same-char} is quite abstract. Find a more
concrete proof, taking equation~\eqref{eq:defn-char} as your definition of
characteristic. (You will still need the fact that $\phi$ is injective.)
\end{ex}
\video{The meaning of `$n \cdot 1$', and Exercise~\ref{ex:char-conn}}

The \demph{prime subfield} of $K$ is the intersection of all the subfields
of $K$. It is straightforward to show that any intersection
\label{p:int-subfds}
of subfields is a subfield (much as in
Lemma~\ref{lemma:int-subrings}). Hence the prime subfield \emph{is} a
subfield. It is the smallest subfield of $K$, in the sense that any other
subfield of $K$ contains it.

Concretely (`bottom-up'), the prime subfield of $K$ is
\begin{align*}
\biggl\{
\frac{m \cdot 1_K}{n \cdot 1_K} 
\such
m, n \in \Z \text{ with } n \cdot 1_K \neq 0
\biggr\}.
\end{align*}
To see this, first note that this set is a subfield of $K$. It is
the smallest subfield of $K$: for if $L$ is a subfield of $K$ then $1_K \in
L$ by definition of subfield, so $m \cdot 1_K \in L$ for all integers
$m$, so $(m \cdot 1_K)/(n \cdot 1_K) \in L$ for all integers $m$ and $n$
such that $n \cdot 1_K \neq 0$.

\begin{examples}
\label{egs:ps}
\begin{enumerate}
\item
The field $\Q$ has no proper subfields, so the prime subfield of $\Q$ is
$\Q$ itself. 

\item
Let $p$ be a prime. The field $\F_p$ has no proper subfields, so the prime
subfield of $\F_p$ is $\F_p$ itself.
\end{enumerate}
\end{examples}

\begin{ex}{ex:prime-sub}
What is the prime subfield of $\R$? Of $\C$?\\
\end{ex}

The prime subfields appearing in Examples~\ref{egs:ps} were $\Q$ and
$\F_p$. In fact, these are the \emph{only} prime subfields of anything: 

\begin{lemma}
\label{lemma:char-prime-subfd}
Let $K$ be a field.
\begin{enumerate}
\item 
\label{part:cps-0}
If $\chr K = 0$ then the prime subfield of $K$ is $\Q$.

\item
\label{part:cps-p}
If $\chr K = p > 0$ then the prime subfield of $K$ is $\F_p$.
\end{enumerate}
\end{lemma}

In the statement of this lemma, as so often in mathematics, the word `is'
means `is isomorphic to'. I hope you're comfortable with that by now.

\begin{proof}
For~\bref{part:cps-0}, suppose that $\chr K = 0$. By definition of
characteristic, $n \cdot 1_K \neq 0$ for all integers $n \geq 0$. One can
check that there is a well-defined homomorphism $\phi \from \Q \to K$
defined by $m/n \mapsto (m \cdot 1_K)/(n \cdot 1_K)$. (The check uses the
fact that $\chi\from n \mapsto n \cdot 1_K$ is a homomorphism.) Now $\phi$
is injective (being a homomorphism of fields), so $\im \phi \iso \Q$. But
$\im \phi$ is a subfield of $K$, and since $\Q$ has no proper subfields, it
is the prime subfield. 

For~\bref{part:cps-p}, suppose that $\chr K = p > 0$. By
Lemma~\ref{lemma:char-0p}, $p$ is prime. The unique homomorphism $\chi
\from \Z \to K$ has kernel $\idlgen{p}$, by definition. By the first
isomorphism theorem, $\im\chi \iso \Z/\idlgen{p} = \F_p$. But $\im\chi$ is
a subfield of $K$, and since $\F_p$ has no proper subfields, it is the
prime subfield. 
\end{proof}

\begin{lemma}
\label{lemma:ff-p}
Every finite field has positive characteristic.
\end{lemma}

\begin{proof}
By Lemma~\ref{lemma:char-prime-subfd}, a field of characteristic $0$
contains a copy of $\Q$ and is therefore infinite.
\end{proof}

\begin{warning}{wg:inf-p}
There are also \emph{infinite} fields of positive
characteristic.\linebreak[5] An example is the field $\F_p(t)$ of rational
expressions over $\F_p$.\\
\end{warning}

Square roots usually come in pairs: how many times in your life have you
written a $\pm$ sign before a $\sqrt{\mbox{\phantom{x}}}$? But in
characteristic $2$, plus and minus are the same, so the two square roots
become one.  We'll see that this pattern persists: $p$th roots behave
strangely in characteristic $p$. First, an important little lemma:

\begin{lemma}
\label{lemma:p-binom}
Let $p$ be a prime and $0 < i < p$. Then $p \dvd \binom{p}{i}$.
\end{lemma}

For example, the $7$th row of Pascal's triangle is $1, 7, 21, 35, 35, 21,
7, 1$, and the lemma predicts that $7$ divides all of these numbers apart
from the first and last.

\begin{proof}
We have $i!(p - i)! \binom{p}{i} = p!$. Now $p$ divides $p!$ but not $i!$
or $(p - i)!$ (since $p$ is prime and $0 < i < p$), so $p$ must divide
$\binom{p}{i}$.
\end{proof}

\begin{propn}
\label{propn:frob}
Let $p$ be a prime number and $R$ a ring of characteristic $p$.
\begin{enumerate}
\item 
\label{part:frob-ring}
The function
\[
\begin{array}{cccc}
\theta\from     &R      &\to            &R      \\
                &r      &\mapsto        &r^p
\end{array}
\]
is a homomorphism.

\item
\label{part:frob-fd}
If $R$ is a field then $\theta$ is injective.

\item
\label{part:frob-ff}
If $R$ is a finite field then $\theta$ is an automorphism of $R$.
\end{enumerate}
\end{propn}

\begin{proof}
For~\bref{part:frob-ring}, certainly $\theta$ preserves multiplication and
$1$. To show that $\theta$ preserves addition, let $r, s \in R$: then by 
Lemma~\ref{lemma:p-binom} and the hypothesis that $\chr R = p$, 
\[
\theta(r + s) 
=
(r + s)^p
=
\sum_{i = 0}^p \binom{p}{i} r^i s^{p -i}
=
r^p + s^p
=
\theta(r) + \theta(s).
\]
Now~\bref{part:frob-fd} follows since every homomorphism between fields is
injective, and~\bref{part:frob-ff} since every injection from a finite set
to itself is bijective.
\end{proof}

The homomorphism $\theta \from r \mapsto r^p$ is called the
\demph{Frobenius map}, or, in the case of finite fields, the
\demph{Frobenius automorphism}.

That $\theta$ is a homomorphism is a shocker. Writing $(x + y)^n = x^n +
y^n$ is a classic algebra mistake. But here, it's true!

\begin{example}
\label{eg:frob-Fp}
The Frobenius automorphism of $\F_p = \Z/\idlgen{p}$ is not very
interesting. When $G$ is a finite group of order $n$, Lagrange's theorem
implies that $g^n = 1$ for all $g \in G$. Applying this to the
multiplicative group $\F_p^\times = \F_p \without \{0\}$ gives $a^{p - 1} =
1$ whenever $0 \neq a \in \F_p$.  It follows that $a^p = a$ for all $a \in
\F_p$. That is, $\theta$ is the identity.  Everything is its own $p$th
root!
\end{example}

Unfortunately, we can't give any interesting examples of the Frobenius map
just now, because we have so few examples of fields. That will change
later.

\begin{cor}
\label{cor:pth-roots}
Let $p$ be a prime number.
\begin{enumerate}
\item 
\label{part:pr-fd}
In a field of characteristic $p$, every element has at most one $p$th root.

\item
\label{part:pr-ff}
In a finite field of characteristic $p$, every element has exactly one
$p$th root.
\end{enumerate}
\end{cor}

\begin{proof}
Part~\bref{part:pr-fd} says that the Frobenius map is injective, and
part~\bref{part:pr-ff} says that it is bijective, as
Proposition~\ref{propn:frob} states.
\end{proof}

\begin{examples}
\label{egs:ff-roots}
\begin{enumerate}
\item 
In a field of characteristic $2$, every element has at most one square
root.

\item
\label{eg:ffr-unity}
In $\C$, there are $p$ different $p$th roots of unity. But in a field of
characteristic $p$, there is only one: $1$ itself.

\item
Let $K$ be a field of characteristic $p$ and $a \in
K$. Corollary~\ref{cor:pth-roots}\bref{part:pr-fd} says that $a$ has
\emph{at most} one $p$th root. It may have none. For instance, you'll show
in Exercise~\ref{ex:pth-root-trans} that the element $t$ of $\F_p(t)$ has
no $p$th root.
\end{enumerate}
\end{examples}

Here's a construction that will let us manufacture many more examples of
fields.

An element $r$ of a ring $R$ is \demph{irreducible} if $r$ is not $0$ or a
\video{Building blocks}
unit, and if for $a, b \in R$,
\[
r = ab \implies a \text{ or } b \text{ is a unit}.
\]
For example, the irreducibles in $\Z$ are $\pm 2, \pm 3, \pm 5, \ldots$. An
element of a ring is \demph{reducible} if it is not $0$, a unit, or
irreducible. 

\begin{warning}{wg:unirred}
The $0$ and units of a ring count as neither reducible nor irreducible, in
much the same way that the integers $0$ and $1$ are neither prime nor
composite.
\end{warning}

\begin{ex}{ex:irred-fd}
What are the irreducible elements of a field?\\
\end{ex}

\begin{propn}
\label{propn:pid-irr}
Let $R$ be a principal ideal domain and $0 \neq r \in R$. Then
\[
r \text{ is irreducible}
\iff
R/\idlgen{r} \text{ is a field}.
\]
\end{propn}

\begin{proof}
Write $\pi$ for the canonical homomorphism $R \to R/\idlgen{r}$.

$\Rightarrow$: suppose that $r$ is irreducible. To show that
$1_{R/\idlgen{r}} \neq 0_{R/\idlgen{r}}$, note that since $r$ is not a
unit, $1_R \not\in\idlgen{r} = \ker\pi$, so
\[
1_{R/\idlgen{r}} = \pi(1_R) \neq 0_{R/\idlgen{r}}.
\]

Next we have to show that every nonzero element of $R/\idlgen{r}$ is a
unit, or equivalently that $\pi(s)$ is a unit whenever $s \in R$ with $s
\not\in \idlgen{r}$. We have $r \ndvd s$, and $r$ is irreducible, so $r$
and $s$ are coprime. Hence by Proposition~\ref{propn:bezout} and the
assumption that $R$ is a principal ideal domain, we can choose $a, b \in
R$ such that
\[
ar + bs = 1_R.
\]
Applying $\pi$ to each side gives
\[
\pi(a)\pi(r) + \pi(b)\pi(s) = 1_{R/\idlgen{r}}.
\]
But $\pi(r) = 0$, so $\pi(b)\pi(s) = 1$, so $\pi(s)$ is a unit.

$\Leftarrow$: suppose that $R/\idlgen{r}$ is a field. Then
$1_{R/\idlgen{r}} \neq 0_{R/\idlgen{r}}$, that is, $1_R \not\in
\ker\pi = \idlgen{r}$, that is, $r \ndvd 1_R$. Hence $r$ is not a
unit. 

Next we have to show that if $a, b \in R$ with $r = ab$ then $a$ or $b$ is
a unit. We have 
\[
0 = \pi(r) = \pi(a)\pi(b)
\]
and $R/\idlgen{r}$ is an integral domain, so WLOG $\pi(a) = 0$. Then $a
\in \ker\pi = \idlgen{r}$, so $a = rb'$ for some $b' \in R$. This gives 
\[
r = ab = rb'b.
\]
But $r \neq 0$ by hypothesis, and $R$ is an integral domain, so $b'b =
1$. Hence $b$ is a unit. 
\end{proof}

\begin{example}
\label{eg:Z-fd}
When $n$ is an integer, $\Z/\idlgen{n}$ is a field if and only if
$n$ is irreducible (that is, $\pm$ a prime number).
\end{example}

Proposition~\ref{propn:pid-irr} enables us to construct fields from
irreducible elements\ldots but irreducible elements \emph{of a principal
ideal domain}. Right now that's not much help, because we don't have many
examples of principal ideal domains. But we will do soon.

\chapter{Polynomials}
\label{ch:polys}

This chapter revisits and develops some themes you met in Honours
Algebra. Before you begin, it may help you to reread Section~3.3
(Polynomials) of the Honours Algebra notes.
\video{Introduction to Week~3}

\section{The ring of polynomials}
\label{sec:ring-polys}

You already know the definition of polynomial, but I want to make a point
by phrasing it in an unfamiliar way.

\begin{defn} 
\label{defn:poly}
Let $R$ be a ring. A \demph{polynomial over $R$} is an infinite sequence
$(a_0, a_1, a_2, \ldots)$ of elements of $R$ such that $\{i \such a_i \neq
0\}$ is finite. 
\end{defn}

The set of polynomials over $R$ forms a ring as follows:
\begin{align}
(a_0, a_1, \ldots) + (b_0, b_1, \ldots)         &
=
(a_0 + b_0, a_1 + b_1, \ldots), 
\label{eq:poly-add}
\\
(a_0, a_1, \ldots) \cdot (b_0, b_1, \ldots)         &
=
(c_0, c_1, \ldots)
\label{eq:poly-mult}
\\
\text{where }
c_k     &
= 
\sum_{i, j\csuch i + j = k} a_i b_j,
\label{eq:poly-mult-coeffs}
\end{align}
the zero of the ring is $(0, 0, \ldots)$, and the multiplicative identity
is $(1, 0, 0, \ldots)$.

Of course, we almost always write $(a_0, a_1, a_2, \ldots)$ as \demph{$a_0
+ a_1 t + a_2 t^2 + \cdots$}, or the same with some other symbol in place
of $t$. In that notation, formulas~\eqref{eq:poly-add}
and~\eqref{eq:poly-mult} look like the usual formulas for addition and
multiplication of polynomials. Nevertheless:
\begin{warning}{wg:poly-not-fn}
A polynomial is not a function!

A polynomial \emph{gives rise to} a function, as we'll recall in a
moment. But a polynomial itself is a purely formal object. To emphasize
this, we sometimes call the symbol $t$ an \demph{indeterminate} rather than
a `variable'.
\end{warning}

The set of polynomials over $R$ is written as \demph{$R[t]$} (or $R[u]$,
$R[x]$, etc.). Since $R[t]$ is itself a ring $S$, we can consider the ring
$S[u] = (R[t])[u]$, usually written as \demph{$R[t, u]$}. And then there's 
$\dem{R[t, u, v]} = (R[t, u])[v]$, and so on.
\video{Why study polynomials?}

Polynomials are typically written as either $f$ or $f(t)$, 
interchangeably. A polynomial $f = (a_0, a_1, \ldots)$ over $R$
gives rise to a function
\[
\begin{array}{ccc}
R       &\to            &R      \\
r       &\mapsto        &a_0 + a_1 r + a_2 r^2 + \cdots.
\end{array}
\]
(The sum on the right-hand side makes sense because only finitely many
$a_i$s are nonzero.) This function is usually denoted by $f$ too. But
calling it that is slightly dangerous, because:

\begin{warning}{wg:poly-ev}
Different polynomials can give rise to the same function. For example,
consider $t, t^2 \in \F_2[t]$. They are different
polynomials: going back to Definition~\ref{defn:poly}, they're alternative
notation for the sequences
\[
(0, 1, 0, 0, \ldots) \quad\text{and}\quad (0, 0, 1, 0, \ldots),
\]
which are plainly not the same. On the other hand, they induce the same
\emph{function} $\F_2 \to \F_2$, because $a = a^2$ for all (both) $a
\in \F_2$. 
\end{warning}

\begin{ex}{ex:fin-ring-poly}
Show that whenever $R$ is a finite nontrivial ring, it is possible to find
distinct polynomials over $R$ that induce the same function $R \to
R$. (Hint: are there finitely or infinitely many polynomials over $R$?
Functions $R \to R$?)
\end{ex}

\begin{remark}
In Example~\ref{eg:rat}, we met the field $K(t)$ of rational expressions
over a field $K$. People sometimes say `rational function' to mean
`rational expression'. But just as for polynomials, I want to emphasize
that \emph{rational expressions are not functions}. For instance, $1/(t -
1)$ is a totally respectable element of $K(t)$. You don't have to worry
about what happens `when $t = 1$', because $t$ is just a formal symbol (a
mark on a piece of paper), not a variable. And $1/(t - 1)$ is just a formal
expression, not a function.
\end{remark}

The ring of polynomials has a universal property: a homomorphism
from $R[t]$ to some other ring $B$ is determined by its effect on scalars
and on $t$ itself, in the following sense.
\video{The universal property of $R[t]$}

\begin{propn}[Universal property of the polynomial ring]
\label{propn:univ-poly}
Let $R$ and $B$ be rings. For every homomorphism $\phi \from R \to B$ and
every $b \in B$, there is exactly one homomorphism $\theta \from R[t] \to
B$ such that 
\begin{align}
\theta(a)       &
= 
\phi(a) \quad\text{for all } a \in R,     
\label{eq:up-main}
\\
\theta(t) &
=
b.
\label{eq:up-t}
\end{align}
\end{propn}

On the left-hand side of~\eqref{eq:up-main}, the `$a$' means the polynomial
$a + 0t + 0t^2 + \cdots$.

\begin{proof}
To show there is \emph{at most one} such $\theta$, take any homomorphism
$\theta\from R[t] \to B$ satisfying~\eqref{eq:up-main}
and~\eqref{eq:up-t}. Then for every polynomial $\sum_i a_i t^i$ over $R$,
\begin{align*}
\theta\Bigl( \sum_i a_i t^i \Bigr)      &
=
\sum_i \theta(a_i) \theta(t)^i  &
\text{since $\theta$ is a homomorphism} \\
&
=
\sum_i \phi(a_i) b^i    &
\text{by~\eqref{eq:up-main} and~\eqref{eq:up-t}}.
\end{align*}
So $\theta$ is uniquely determined.

To show there is \emph{at least one} such $\theta$, define a function
$\theta \from R[t] \to B$ by
\[
\theta\Bigl(\sum_i a_i t^i\Bigr) = \sum_i \phi(a_i) b^i
\]
($\sum_i a_i t^i \in R[t]$). Then $\theta$ clearly satisfies
conditions~\eqref{eq:up-main} and~\eqref{eq:up-t}. It remains to check that
$\theta$ is a homomorphism. I will do the worst part of this, which is
to check that $\theta$ preserves multiplication, and leave the rest to
you.

So, take polynomials $f(t) = \sum_i a_i t^i$ and $g(t) = \sum_j b_j
t^j$. Then $f(t)g(t) = \sum_k c_k t^k$, where $c_k$ is as defined in
equation~\eqref{eq:poly-mult-coeffs}. We have
\begin{align*}
\theta(fg) &
=
\theta\Bigl( \sum_k c_k t^k \Bigr)      \\
&
=
\sum_k \phi(c_k) b^k    &
\text{by definition of $\theta$}        \\
&
=
\sum_k \phi\Bigl( \sum_{i, j\csuch i + j = k} a_i b_j \Bigr) b^k     &
\text{by definition of $c_k$}   \\
&
=
\sum_k \sum_{i, j\csuch i + j = k} \phi(a_i) \phi(b_j) b^k     &
\text{since $\phi$ is a homomorphism}   \\
&
=
\sum_{i, j} \phi(a_i) \phi(b_j) b^{i + j}       \\
&
=
\Bigl( \sum_i \phi(a_i) b^i \Bigr) 
\Bigl( \sum_j \phi(b_j) b^j \Bigr)      \\
&
=
\theta(f) \theta(g)     &
\text{by definition of $\theta$}.
\end{align*}
\end{proof}

Here are three uses for the universal property of the ring of polynomials.
First:

\begin{defn}
\label{defn:ind-hom}
Let $\phi \from R \to S$ be a ring homomorphism. The \demph{induced
homomorphism} 
\[
\dem{\phi_*} \from R[t] \to S[t]
\]
is the unique homomorphism $R[t] \to S[t]$ such that $\phi_*(a) =
\phi(a)$ for all $a \in R$ and $\phi_*(t) = t$.
\end{defn}

The universal property
guarantees that there is one and only one homomorphism $\phi_*$ with
these properties. Concretely,
\[
\phi_*\Bigl(\sum_i a_i t^i\Bigr) = \sum_i \phi(a_i) t^i
\]
for all $\sum_i a_i t^i \in R[t]$.

Second, let $R$ be a ring and $r \in R$. By the universal property, there
is a unique homomorphism $\dem{\ev_r} \from R[t] \to R$ such that
$\ev_r(a) = a$ for all $a \in R$ and $\ev_r(t) = r$.  Concretely,
\[
\ev_r\Bigl(\sum_i a_i t^i\Bigr) = \sum_i a_i r^i
\]
for all $\sum_i a_i t^i \in R[t]$. This map $\ev_r$ is called
\demph{evaluation at $r$}.

(The notation $\sum a_i t^i$ for what is officially $(a_0, a_1, \ldots)$
makes it look \emph{obvious} that we can evaluate a polynomial at an
element, and that this gives a homomorphism: \emph{of course} $(f \cdot
g)(r) = f(r)g(r)$, for instance!  But that's only because of the notation:
there was actually something to prove here.)

Third, let $R$ be a ring and $c \in R$. For any $f(t) \in R[t]$, we can
`substitute\label{p:subst} $t = u + c$' to get a polynomial in $u$. What
exactly does this mean?  Formally, there is a unique homomorphism $\theta
\from R[t] \to R[u]$ such that $\theta(a) = a$ for all $a \in R$ and
$\theta(t) = u + c$. Concretely,
\[
\theta\Bigl(\sum_i a_i t^i\Bigr) = \sum_i a_i (u + c)^i.
\]
This particular substitution is invertible. Informally, the inverse is
`substitute $u = t - c$'. Formally, there is a unique homomorphism $\theta'
\from R[u] \to R[t]$ such that $\theta'(a) = a$ for all $a \in R$ and
$\theta'(u) = t - c$.  These maps $\theta$ and $\theta'$ carrying out the
substitutions are inverse to each other, as you can deduce from either
the universal property or the concrete descriptions. So, the substitution
maps
\begin{align}
\label{eq:subst}
\xymatrix{
R[t] \ar@<.5ex>[r]^-{\theta} &
R[u] \ar@<.5ex>[l]^-{\theta'}
}
\end{align}
define an isomorphism between $R[t]$ and $R[u]$. For example, since
isomorphism preserve irreducibility (and everything else that
matters!),\ $f(t)$ is irreducible\label{p:subst-irr} if and only if $f(t -
c)$ is irreducible.

\begin{ex}{ex:subst}
What happens to everything in the previous paragraph if we
substitute $t = u^2 + c$ instead?
\end{ex}

The rest of this section is about degree.

\begin{defn}
The \demph{degree}, \demph{$\deg(f)$}, of a nonzero polynomial $f(t) = \sum
a_i t^i$ is the largest $n \geq 0$ such that $a_n \neq 0$. By convention,
$\deg(0) = -\infty$, where \demph{$-\infty$} is a formal symbol which we
give the properties
\[
-\infty < n,
\qquad
(-\infty) + n = -\infty,
\qquad
(-\infty) + (-\infty) = -\infty
\]
for all integers $n$.
\end{defn}

\begin{digression}{dig:deg}
Defining $\deg(0)$ like this is helpful because it allows us to make
statements about \emph{all} polynomials without annoying
exceptions for the zero polynomial
(e.g.\ Lemma~\ref{lemma:idt}\bref{part:idt-deg}).

But putting $\deg(0) = -\infty$ also makes intuitive sense. At least for
polynomials over $\R$, the degree of a nonzero polynomial tells us how fast
it grows: when $x$ is large, $f(x)$ behaves roughly like
$x^{\deg(f)}$. What about the zero polynomial? Well, whether or not $x$ is
large, $0(x) = 0$. And $x^{-\infty}$ can sensibly be interpreted as
$\lim_{r \to -\infty} x^r = 0$, so it makes sense to put $\deg(0) =
-\infty$.
\end{digression}

\begin{lemma}
\label{lemma:idt}
Let $R$ be an integral domain. Then:
\begin{enumerate}
\item
\label{part:idt-deg}
$\deg(fg) = \deg(f) + \deg(g)$ for all $f, g \in R[t]$;

\item
\label{part:idt-id}
$R[t]$ is an integral domain.
\end{enumerate}
\end{lemma}

\begin{proof}
This was proved in Honours Algebra (Section~3.3).
\end{proof}

\begin{example}
\label{eg:poly-multi-id}
For any integral domain $R$, the ring $R[t_1, \ldots, t_n]$ of polynomials
over $R$ in $n$ variables is also an integral domain, by
Lemma~\ref{lemma:idt}\bref{part:idt-id} and induction. In particular, this
is true when $R$ is a field.
\end{example}

\begin{ex}{ex:pth-root-trans}
Let $p$ be a prime and consider the field $\F_p(t)$ of rational expressions
over $\F_p$. Show that $t$ has no $p$th root in $\F_p(t)$. (Hint: consider
degrees of polynomials.) 
\end{ex}

The one and only polynomial of degree $-\infty$ is the zero polynomial. The
polynomials of degree $0$ are the nonzero constants. The polynomials of
degree $> 0$ are, therefore, the nonconstant polynomials.

\begin{lemma}
\label{lemma:fdt}
Let $K$ be a field. Then:
\begin{enumerate}
\item
\label{part:fdt-units}
the units in $K[t]$ are the nonzero constants; 

\item
\label{part:fdt-irred}
$f \in K[t]$ is irreducible if and only if $f$ is nonconstant and cannot
be expressed as a product of two nonconstant polynomials.
\end{enumerate}
\end{lemma}

\begin{proof}
Part~\bref{part:fdt-units} was also in Honours Algebra
(Section~3.3), and part~\bref{part:fdt-irred} follows from the general
definition of irreducible element of a ring.
\end{proof}

\section{Factorizing polynomials}

Every nonzero integer can be expressed as a product of primes in an
essentially unique way. But the analogous statement is not true in all
rings, or even all integral domains. Some rings have elements that can't be
expressed as a product of irreducibles at all. In other rings,
factorizations into irreducibles exist but are not unique. (By `not unique'
I mean more than just changing the order of the factors or multiplying them
by units.)

The big theorem of this section is that, happily, every polynomial over a
field \emph{can} be factorized into irreducibles, essentially uniquely.

We begin with a result on division of polynomials from Section~3.3 of
Honours Algebra.

\begin{propn}
\label{propn:div-alg}
Let $K$ be a field and $f, g \in K[t]$ with $g \neq 0$. Then there is
exactly one pair of polynomials $q, r \in K[t]$ such that $f = qg + r$ and
$\deg(r) < \deg(g)$.  \qed
\end{propn}

We use this to prove an extremely useful fact:

\begin{propn}
\label{propn:Kt-pid}
Let $K$ be a field. Then $K[t]$ is a principal ideal domain.
\end{propn}

\begin{proof}
First, $K[t]$ is an integral domain, by
Lemma~\ref{lemma:idt}\bref{part:idt-id}. 

Now let $I \idl K[t]$. If $I = \{0\}$ then $I = \idlgen{0}$. Otherwise, put
$d = \min \{ \deg(f) \such 0 \neq f \in I \}$ and choose $g \in I$ such
that $\deg(g) = d$.

I claim that $I = \idlgen{g}$. To prove this, let $f \in I$; we must show
that $g \dvd f$. By Proposition~\ref{propn:div-alg}, $f = qg + r$ for some
$q, r \in K[t]$ with $\deg(r) < d$. Now $r = f - qg \in I$ since $f, g \in
I$, so the minimality of $d$ implies that $r = 0$. Hence $f = qg$, as
required. 
\end{proof}

If you struggled with Exercise~\ref{ex:Z-pid}, that proof should give you a
clue.

\begin{warning}{wg:not-pid}
We just saw that $K[t]$ is a principal ideal domain, and we saw in
Example~\ref{eg:poly-multi-id} that $K[t_1,
\ldots, t_n]$ is an \emph{integral} domain. But it is not a
\emph{principal ideal} domain if $n > 1$. For example, the ideal
\[
\idlgen{t_1, t_2} 
=
\{ f(t_1, t_2) \in \Q[t_1, t_2] \such 
f \text{ has constant term } 0 \}
\]
of $\Q[t_1, t_2]$ is not principal. 

Also, Proposition~\ref{propn:Kt-pid} really needed the hypothesis that $K$
is a \emph{field}; it's not enough for it to be a principal ideal
domain. For example, $\Z$ is a principal ideal domain, but in $\Z[t]$, the
ideal
\[
\idlgen{2, t}
=
\{ f(t) \in \Z[t]
\such
\text{the constant term of $f$ is even} \}
\]
is not principal.
\end{warning}

\begin{ex}{ex:not-pids}
Prove that the ideals in Warning~\ref{wg:not-pid} are indeed not
principal. 
\end{ex}

\video{Exercise~\ref{ex:not-pids}: a non-principal ideal}%
At the end of Chapter~\ref{ch:garf}, I promised I'd give you a way of
manufacturing lots of new fields. Here it is!

\begin{cor}
\label{cor:qt-by-irr}
Let $K$ be a field and let $0 \neq f \in K[t]$. Then
\[
f \text{ is irreducible}
\iff
K[t]/\idlgen{f} \text{ is a field}.
\]
\end{cor}

\begin{proof}
This follows from Propositions~\ref{propn:pid-irr} and~\ref{propn:Kt-pid}.
\end{proof}

To make new fields using Corollary~\ref{cor:qt-by-irr}, we'll need a
way of knowing which polynomials are irreducible. That's the topic of
Section~\ref{sec:irr-polys}. But for now, let's stick to our mission:
proving that every polynomial factorizes into irreducibles in an
essentially unique way.

To achieve our mission, we'll need two more lemmas.

\begin{lemma}
\label{lemma:some-irr}
Let $K$ be a field and let $f(t) \in K[t]$ be a nonconstant
polynomial. Then $f(t)$ is divisible by some irreducible in $K[t]$.
\end{lemma}

\begin{proof}
Let $g$ be a nonconstant polynomial of smallest possible degree such that
$g \dvd f$. (For this to make sense, there must be at least one nonconstant
polynomial dividing $f$, and there is: $f$.) I claim that $g$ is
irreducible. Proof: if $g = g_1 g_2$ then each $g_i$ divides $f$, so by 
minimality of $\deg(g)$, each $g_i$ has degree $0$ or
$\deg(g)$. They cannot both have degree $\deg(g)$, since $\deg(g_1) +
\deg(g_2) = \deg(g) > 0$. So at least one has degree $0$, which by
Lemma~\ref{lemma:fdt}\bref{part:fdt-units} means that it is a unit. 
\end{proof}

\begin{lemma}
\label{lemma:irr-prime}
Let $K$ be a field and $f, g, h \in K[t]$. Suppose that $f$ is irreducible
and $f \dvd g h$. Then $f \dvd g$ or $f \dvd h$.
\end{lemma}

This behaviour is familiar in the integers: if a prime $p$ divides some
product $ab$, then $p \dvd a$ or $p \dvd b$. In fact, our proof works in
any principal ideal domain.

\begin{proof}
Suppose that $f \ndvd g$. Since $f$ is irreducible, $f$ and $g$ are
coprime. Since $K[t]$ is a principal ideal domain,
Proposition~\ref{propn:bezout} implies that there are $p, q \in K[t]$
satisfying 
\[
pf + qg = 1.
\]
Multiplying both sides by $h$ gives
\[
pfh + qgh = h.
\]
But $f \dvd pfh$ and $f \dvd gh$, so $f \dvd h$.
\end{proof}

\begin{bigthm}
\label{thm:fact-polys}
Let $K$ be a field and $0 \neq f \in K[t]$. Then 
\[
f = af_1 f_2 \cdots f_n
\]
for some $n \geq 0$, $a \in K$ and monic irreducibles $f_1, \ldots, f_n \in
K[t]$. Moreover, $n$ and $a$ are uniquely determined by $f$, and $f_1,
\ldots, f_n$ are uniquely determined up to reordering.
\end{bigthm}

In the case $n = 0$, the product $f_1 \cdots f_n$ should be interpreted as
$1$ (as in Digression~\ref{dig:defn-id}). \demph{Monic} means that the
leading coefficient is $1$.

\begin{proof}
First we prove that such a factorization exists, by induction on
$\deg(f)$. If $\deg(f) = 0$ then $f$ is a constant $a$ and we take $n =
0$. Now suppose that $\deg(f) > 0$ and assume the result for polynomials of
smaller degree. By Lemma~\ref{lemma:some-irr}, there is an irreducible $g$
dividing $f$, and we can assume that $g$ is monic by dividing by a constant
if necessary. Then $f/g$ is a nonzero polynomial of smaller degree than
$f$, so by inductive hypothesis, 
\[
f/g = a h_1 \cdots h_m
\]
for some $a \in K$ and monic irreducibles $h_1, \ldots, h_m$. Rearranging
gives
\[
f = a h_1 \cdots h_m g,
\]
completing the induction.

Now we prove uniqueness, again by induction on $\deg(f)$. If $\deg(f) = 0$
then $f$ is a constant $a$ and the only possible factorization is the one
with $n = 0$. Now suppose that $\deg(f) > 0$, and take two factorizations 
\begin{align}
\label{eq:two-factns}
af_1 \cdots f_n = f = bg_1 \cdots g_m
\end{align}
where $a, b \in K$ and $f_i, g_j$ are monic irreducible. Since $\deg(f) >
0$, we have $n, m \geq 1$. Now $f_n \dvd b g_1 \cdots g_m$, so by
Lemma~\ref{lemma:irr-prime}, $f_n \dvd g_j$ for some $j$. By
rearranging, we can assume that $j = m$.  But $g_m$ is also
irreducible, so $f_n = c g_m$ for some nonzero $c \in K$, and both $f_n$
and $g_m$ are monic, so $c = 1$. Hence $f_n = g_m$. Cancelling
in~\eqref{eq:two-factns} (which we can do as $K[t]$ is an integral domain)
gives
\[
af_1 \cdots f_{n - 1} = bg_1 \cdots g_{m - 1}.
\]
By inductive hypothesis, $n - 1 = m - 1$, $a = b$, and the lists $f_1,
\ldots, f_{n - 1}$ and $g_1, \ldots, g_{m - 1}$ are the same up to
reordering. This completes the induction.
\end{proof}

One way to find an irreducible factor of a polynomial $f(t) \in K[t]$ is to
find a \demph{root} (an element $a \in K$ such that $f(a) = 0$):

\begin{lemma}
\label{lemma:root-irr}
Let $K$ be a field, $f(t) \in K[t]$ and $a \in K$. Then
\[
f(a) = 0 \iff (t - a) \dvd f(t).
\]
\end{lemma}

\begin{proof}
$\Rightarrow$: suppose that $f(a) = 0$. By Proposition~\ref{propn:div-alg},
\begin{align}
\label{eq:ri-da}
f(t) = (t - a)q(t) + r(t)
\end{align}
for some $q, r \in K[t]$ with $\deg(r) < 1$. Then $r$ is a constant, so
putting $t = a$ in~\eqref{eq:ri-da} gives $r = 0$.

$\Leftarrow$: if $f(t) = (t - a) q(t)$ for some polynomial $q$ then $f(a) = 0$.
\end{proof}

A field is \demph{algebraically closed} if every nonconstant polynomial has
at least one root. For example, $\C$ is algebraically closed (the
fundamental theorem of algebra). A straightforward induction shows:

\begin{lemma}
\label{lemma:ac-splits}
Let $K$ be an algebraically closed field and $0 \neq f \in K[t]$. Then
\[
f(t) = c(t - a_1)^{m_1} \cdots (t - a_k)^{m_k},
\]
where $c$ is the leading coefficient of $f$, and $a_1, \ldots, a_k$ are the
distinct roots of $f$ in $K$, and $m_1, \ldots, m_k \geq 1$.  \qed
\end{lemma}

\section{Irreducible polynomials}
\label{sec:irr-polys}

Determining whether an integer is prime is generally hard, so it's no
surprise that determining whether a polynomial is irreducible is hard
too. This section presents a few techniques for doing so.

Let's begin with the simplest cases. Recall
Lemma~\ref{lemma:fdt}\bref{part:fdt-irred}: a polynomial over a field is
irreducible if and only if it is nonconstant (has degree $> 0$) and cannot
be expressed as a product of two nonconstant polynomials.

\begin{lemma}
\label{lemma:irr-triv}
Let $K$ be a field and $f \in K[t]$.
\begin{enumerate}
\item 
\label{part:it-0}
If $f$ is constant then $f$ is not irreducible.

\item
\label{part:it-1}
If $\deg(f) = 1$ then $f$ is irreducible.

\item
\label{part:it-root}
If $\deg(f) \geq 2$ and $f$ has a root then $f$ is reducible.

\item
\label{part:it-23}
If $\deg(f) \in \{2, 3\}$ and $f$ has no root then $f$ is irreducible.
\end{enumerate}
\end{lemma}

\begin{proof}
Parts~\bref{part:it-0} and~\bref{part:it-1} follow from what we just
recalled, and~\bref{part:it-root} follows from
Lemma~\ref{lemma:root-irr}. For~\bref{part:it-23}, suppose for a
contradiction that $f = gh$ with $\deg(g), \deg(h) \geq 1$. We have
$\deg(g) + \deg(h) \in \{2, 3\}$, so without loss of generality,
$\deg(g) = 1$. Also without loss of generality, $g$ is monic, say $g(t) = t
+ a$; but then $f(-a) = 0$, a contradiction.
\end{proof}

\begin{warning}{wg:roots-irred}
To show a polynomial is irreducible, it's generally \emph{not} enough to
show it has no root. The converse of~\bref{part:it-root} is false! For
instance, $(t^2 + 1)^2 \in \Q[t]$ has no root but is reducible.
\end{warning}

\begin{warning}{wg:roots-irred-bis}
Make sure you've digested Warning~\ref{wg:roots-irred}!

This is an extremely common mistake.\\[-1ex]
\end{warning}

\begin{examples}
\begin{enumerate}
\item 
Let $p$ be a prime. Then $f(t) = 1 + t + \cdots + t^{p - 1} \in \F_p[t]$ is
reducible, since $f(1) = 0$.

\item
Let $f(t) = t^3 - 10 \in \Q[t]$. Then $\deg(f) = 3$ and $f$ has no root in
$\Q$, so $f$ is irreducible by part~\bref{part:it-23} of the lemma.

\item
Over $\C$ or any other algebraically closed field, the irreducibles are
exactly the polynomials of degree $1$. 
\end{enumerate}
\end{examples}

\begin{ex}{ex:quad-irr}
If I gave you a quadratic over $\Q$, how would you decide whether it was
reducible or irreducible?
\end{ex}

From now on we focus on $K = \Q$. Any polynomial over $\Q$ can be
multiplied by a nonzero integer to get a polynomial over $\Z$,
and that's often a helpful move, so we'll look at $\Z[t]$ too.

\begin{defn}
A polynomial over $\Z$ is \demph{primitive} if its coefficients have no
common divisor except for $\pm 1$.
\end{defn}

For example, $15 + 6t + 10t^2$ is primitive but $15 + 6t + 30t^2$ is not.

\begin{lemma}
\label{lemma:const-prim}
Let $f(t) \in \Q[t]$. Then there exist a primitive polynomial $F(t) \in
\Z[t]$ and $\alpha \in \Q$ such that $f = \alpha F$.
\end{lemma}

\begin{proof}
Write $f(t) = \sum_i (a_i/b_i) t^i$, where $a_i \in \Z$ and $0 \neq b_i \in
\Z$. Take any common multiple $b$ of the $b_i$s; then writing $c_i = a_i
b/b_i \in \Z$, we have $f(t) = (1/b) \sum c_i t^i$. Now let $c$ be
the greatest common divisor of the $c_i$s, put $d_i = c_i/c \in \Z$, and
put $F(t) = \sum d_i t^i$. Then $F(t)$ is primitive and $f(t) = (c/b)
F(t)$.
\end{proof}

If the coefficients of a polynomial $f(t) \in \Q[t]$ happen to all be
integers, the word `irreducible' could mean two things: irreducibility in
the ring $\Q[t]$ or in the ring $\Z[t]$. We say that $f$ is irreducible
\demph{over} $\Q$ or $\Z$ to distinguish between the two.

Suppose we have a polynomial over $\Z$ that's irreducible over $\Z$. In
principle it could still be reducible over $\Q$: although there's no
nontrivial way of factorizing it over $\Z$, perhaps it can be factorized
when you give yourself the freedom of non-integer coefficients. But the
next result tells us that you can't.

\begin{lemma}[Gauss]
\label{lemma:gauss}
\begin{enumerate}
\item 
\label{part:gl-prim}
The product of two primitive polynomials over $\Z$ is primitive.

\item
\label{part:gl-irr}
If a nonconstant polynomial over $\Z$ is irreducible over $\Z$, it is
irreducible over $\Q$. 
\end{enumerate}
\end{lemma}

\begin{proof}
For~\bref{part:gl-prim}, let $f$ and $g$ be primitive polynomials over
$\Z$. Let $p$ be a prime number. (We're going to show that $p$ doesn't
divide all the coefficients of $fg$.) Write $\pi \from \Z \to \Z/p\Z =
\F_p$ for the canonical homomorphism, which induces a homomorphism $\pi_*
\from \Z[t] \to \F_p[t]$ as in Definition~\ref{defn:ind-hom}. 

Since $f$ is primitive, $p$ does not divide all the coefficients of $f$.
Equivalently, $\pi_*(f) \neq 0$. Similarly, $\pi_*(g) \neq 0$. But
$\F_p[t]$ is an integral domain, so 
\[
\pi_*(fg) = \pi_*(f)\pi_*(g) \neq 0,
\]
so $p$ does not divide all the coefficients of
$fg$. This holds for all primes $p$, so $fg$ is primitive.

For~\bref{part:gl-irr}, let $f \in \Z[t]$ be a nonconstant polynomial
irreducible over $\Z$. Let $g, h \in \Q[t]$ with $f = gh$. By
Lemma~\ref{lemma:const-prim}, $g = \alpha G$ and $h = \beta H$ for some
$\alpha, \beta \in \Q$ and primitive $G, H \in \Z[t]$. Then $\alpha\beta =
m/n$ for some coprime integers $m$ and $n$, giving
\[
n f = m G H.
\]
(All three of these polynomials are over $\Z$.)  Now $n$ divides every
coefficient of $n f$, hence every coefficient of $mGH$. Since $m$ and $n$
are coprime, $n$ divides every coefficient of $GH$. But $GH$ is primitive
by~\bref{part:gl-prim}, so $n = \pm 1$, so $f = \pm m GH$. Since $f$ is
irreducible over $\Z$, either $G$ or $H$ is constant, so $g$ or $h$ is
constant, as required.
\end{proof}

Gauss's lemma quickly leads to a test for irreducibility. It involves
taking a polynomial over $\Z$ and reducing it mod $p$, for some prime
$p$. This means applying the map $\pi_* \from \Z[t] \to \F_p[t]$ from the
last proof. As we saw after Definition~\ref{defn:ind-hom}, if $f(t) = \sum
a_i t^i$ then $\pi_*(f)(t) = \sum \pi(a_i) t^i$, where $\pi(a_i)$ is the
congruence class of $a_i$ mod $p$. I'll write $\pi(a)$ as $\dem{\ovln{a}}$
and $\pi_*(f)$ as $\dem{\ovln{f}}$. That is, $\ovln{f}$ is `$f$ mod $p$'.

\begin{propn}[Mod $p$ method]
\label{propn:mod-p-method}
Let $f(t) = a_0 + a_1 t + \cdots + a_n t^n \in \Z[t]$. If there is some
prime $p$ such that $p \ndvd a_n$ and $\ovln{f} \in \F_p[t]$ is
irreducible, then $f$ is irreducible over $\Q$.
\end{propn}

I'll give some examples first, then the proof.

\begin{examples}
\label{egs:mpm}
\begin{enumerate}
\item 
\label{eg:mpm-9}
Let's use the mod $p$ method to show that $f(t) = 9 + 14t - 8t^3$ is
irreducible over $\Q$. Take $p = 7$: then $\ovln{f}(t) = 2 - t^3 \in
\F_7[t]$, so it's enough to show that $2 - t^3$ is irreducible over
$\F_7$. Since this has degree $3$, it's enough to show that $t^3 = 2$ has
no solution in $\F_7$ (by Lemma~\ref{lemma:irr-triv}\bref{part:it-23}). And
you can easily check this by computing $0^3$, $(\pm 1)^3$, $(\pm 2)^3$ and
$(\pm 3)^3$ mod $7$.

\item
The condition in Proposition~\ref{propn:mod-p-method} that $p \ndvd a_n$
can't be dropped. For instance, consider $f(t) = 6t^2 + t$ and $p = 2$.
\end{enumerate}
\end{examples}

\begin{warning}{wg:mod-p}
Take $f(t)$ as in Example~\ref{egs:mpm}\bref{eg:mpm-9}, but this time take
$p = 3$. Then $\ovln{f}(t) = -t + t^3 \in \F_3[t]$, which is reducible. But
that doesn't mean $f$ is reducible! The mod $p$ method only ever lets you
show that a polynomial is \emph{irreducible} over $\Q$, not reducible.
\end{warning}

\begin{pfof}{Proposition~\ref{propn:mod-p-method}}
Take a prime $p$ satisfying the stated conditions.  

First suppose that $f$ is primitive.  By Gauss's lemma, it is enough to
prove that $f$ is irreducible over $\Z$.

Since $\ovln{f}$ is irreducible, $\deg(\ovln{f}) > 0$, so $\deg(f) > 0$.

Let $f = gh$ in $\Z[t]$. We have $\ovln{f} = \ovln{g}\ovln{h}$ and
$\ovln{f}$ is irreducible, so without loss of generality, $\ovln{g}$ is
constant. The leading coefficient of $f$ is the product of the leading
coefficients of $g$ and $h$, and is not divisible by $p$, so
the leading coefficient of $g$ is not divisible by $p$. Hence
$\deg(g) = \deg(\ovln{g})$. But $\deg(\ovln{g}) = 0$, so $\deg(g) = 0$, so
$g \in \Z[t]$ is a constant $b \in \Z$. Finally, $f = gh = bh$ and $f$ is
primitive, so $b = \pm 1$, which is a unit in $\Z[t]$. It follows that $f$
is irreducible over $\Z$.

Now take an arbitrary $f$ satisfying the hypotheses. We have $f = cF$ where
$c \in \Z$ is the greatest common divisor of the coefficients and $F \in
\Z[t]$ is primitive. Then $\ovln{f} = \ovln{c} \ovln{F}$, and $\ovln{c}$
is a unit in $\F_p$ because $p\ndvd c$. Since $\ovln{f}$ is irreducible, this
implies that $\ovln{F}$ is irreducible, and so by what we've just proved,
$F$ is irreducible over $\Q$. But $c \neq 0$, so $c$ is a unit in $\Q$, so
$f = cF$ is also irreducible over $\Q$. 
\end{pfof}

We finish with an irreducibility test that turns out to be surprisingly
powerful.

\begin{propn}[Eisenstein's criterion]
Let $f(t) = a_0 + \cdots + a_n t^n \in \Z[t]$, with $n \geq 1$. Suppose
\marginpar{\footnotesize \ \hfill Not Einstein.\hfill \ }%
there exists a prime $p$ such that:
\begin{itemize}
\item 
$p \ndvd a_n$;

\item
$p \dvd a_i$ for all $i \in \{ 0, \ldots, n - 1 \}$;

\item
$p^2 \ndvd a_0$.
\end{itemize}
Then $f$ is irreducible over $\Q$.
\end{propn}

To prove this, we will use the concept of the \demph{codegree $\codeg(f)$}
of a polynomial $f(t) = \sum_i a_i t^i$, which is defined to be the least
$i$ such that $a_i \neq 0$ (if $f \neq 0$), or as the formal symbol
$\infty$ if $f = 0$. For polynomials $f$ and $g$ over an integral domain,
\[
\codeg(fg) = \codeg(f) + \codeg(g).
\]
Clearly $\codeg(f) \leq \deg(f)$ unless $f = 0$.

\begin{proof}
We may assume $f$ is primitive: if not, divide $f$ through by the greatest
common divisor of its coefficients, which does not affect their
divisibility by powers of $p$ or the reducibility of $f$ over $\Q$. By Gauss's
lemma, it is enough to show that $f$ is irreducible over $\Z$. Let $g, h \in
\Z[t]$ with $f = gh$. Continue to write $\ovln{f}(t) \in \F_p[t]$ for $f$
reduced mod $p$; then $\ovln{f} = \ovln{g} \ovln{h}$. Since
\[
p^2 \ndvd a_0 = f(0) = g(0)h(0),
\] 
we may assume without loss of generality that $p \ndvd g(0)$. Hence
$\codeg(\ovln{g}) = 0$. Also, $\codeg(\ovln{f}) = n$, since $p$ divides each
of $a_0, \ldots, a_{n - 1}$ but not $a_n$. So
\begin{align}
\label{eq:eis-bds}
n 
= 
\codeg(\ovln{f}) 
= 
\codeg(\ovln{g}) + \codeg(\ovln{h})
=
\codeg(\ovln{h})
\leq
\deg(\ovln{h})
\leq
\deg(h),
\end{align}
giving $n \leq \deg(h)$. But $f = gh$ with $\deg(f) = n$, so $\deg(h) = n$
and $\deg(g) = 0$. Hence $g$ is constant. Since $f$ is primitive, $g = \pm
1$, so $g$ is a unit in $\Z[t]$.
\end{proof}

\begin{ex}{ex:eis-step}
The last step in~\eqref{eq:eis-bds} was `$\deg(\ovln{h}) \leq
\deg(h)$'. Why is that true? And when does equality hold?
\end{ex}

\begin{example}
Let 
\[
g(t) = \frac{2}{9}t^5 - \frac{5}{3}t^4 + t^3 + \frac{1}{3} \in \Q[t].
\]
Then $g$ is irreducible over $\Q$ if and only if 
\[
9g(t) = 2t^5 - 15t^4 + 9t^3 + 3
\]
is irreducible over $\Q$, which it is by Eisenstein's criterion with $p =
3$. 
\video{Testing for irreducibility}%
\end{example}

\begin{ex}{ex:eis-every}
Use Eisenstein's criterion to show that for every $n \geq 1$, there is an
irreducible polynomial over $\Q$ of degree $n$.
\end{ex}

I'll give you one more example, and it's an important one.

\begin{example}
\label{eg:cyclo-irr}
Let $p$ be a prime. The \demph{$p$th cyclotomic polynomial} is
\begin{align}
\label{eq:cyclo-p}
\dem{\Phi_p(t)} = 1 + t + \cdots + t^{p - 1} = \frac{t^p - 1}{t - 1}.
\end{align}
I claim that $\Phi_p$ is irreducible. We can't apply Eisenstein to $\Phi_p$
as it stands, because whichever prime we choose (whether it's $p$ or another
one) doesn't divide any of the coefficients. However, we saw on
p.~\pageref{p:subst-irr} that $\Phi_p(t)$ is irreducible if and only if
$\Phi_p(t - c)$ is irreducible, for any $c \in \Q$. We'll take $c = -1$. We
have 
\begin{align*}
\Phi_p(t + 1)   &
=
\frac{(t + 1)^p - 1}{(t + 1) - 1}       \\
&
=
\frac{1}{t}\sum_{i = 1}^p \binom{p}{i} t^i       \\
&
=
p + \binom{p}{2} t + \cdots + \binom{p}{p - 1} t^{p - 2} + t^{p - 1}.
\end{align*}
So $\Phi_p(t + 1)$ is irreducible by Eisenstein's criterion and
Lemma~\ref{lemma:p-binom}, hence $\Phi_p(t)$ is irreducible too.
\end{example}

\begin{digression}{dig:cyclo}
I defined the $p$th cyclotomic polynomial $\Phi_p$ only when $p$ is
prime. The definition of $\Phi_n$ for general $n \geq 1$ is \emph{not} the
obvious generalization of~\eqref{eq:cyclo-p}. Instead, it's this:
\[
\dem{\Phi_n(t)} = \prod_\zeta (t - \zeta),
\]
where the product runs over all primitive $n$th roots of unity $\zeta$. (In
this context, `primitive' means that $n$ is the smallest number satisfying
$\zeta^n = 1$; it's a different usage from `primitive polynomial'.) 

Many surprising things are true. It's not obvious that the coefficients of
$\Phi_n$ are real, but they are. Even given that they're real, it's not
obvious that they're rational, but they are. Even given that they're
rational, it's not obvious that they're integers, but they are (\href{https://www.maths.ed.ac.uk/~tl/galois}{Workshop~4,
question~14}). The degree of $\Phi_n$ is $\phi(n)$, the number of integers
between $1$ and $n$ that are coprime with $n$ (Euler's function). It's also
true that the polynomial $\Phi_n$ is irreducible for \emph{all} $n$, not
just primes.

Some of these things are quite hard to prove, and results from Galois
theory help. We won't get into all of this, but you can
\href{https://en.wikipedia.org/wiki/Cyclotomic_polynomial}{read more here}.
\end{digression}

\chapter{Field extensions}
\label{ch:exts}

Roughly speaking, an `extension' of a field $K$ is a field $M$ that
contains $K$ as a subfield. It's not much of an exaggeration to say that
field extensions are the central objects of Galois theory, in much the same
way that vector spaces are the central objects of linear algebra.
\video{Introduction to Week~4}

It will be a while before it becomes truly clear why field extensions are
so important, but here are a couple of indications:
\begin{itemize}
\item
For any polynomial $f$ over $\Q$, we can take the smallest subfield $M$ of
$\C$ that contains all the complex roots of $f$, and that's an extension of
$\Q$.

\item
For any irreducible polynomial $f$ over a field $K$, the quotient ring
$M = K[t]/\idlgen{f}$ is a field. The constant polynomials form a subfield of
$M$ isomorphic to $K$, so $M$ is an extension of $K$.
\end{itemize}
It's important to distinguish between these two types of example. The first
extends $\Q$ by \emph{all} the roots of $f$, whereas the second extends $K$
by just \emph{one} root of $f$---as we'll see.

\section{Definition and examples}

Before we do anything else, we need to think about some set theory. What
follows might seem trivial, but it's worth taking the time to get it
straight. 

Given a set $A$ and a subset $B \sub A$, there is an \demph{inclusion}
function $\iota \from B \to A$ defined by $\iota(b) = b$ for all $b \in
B$. (That's a Greek letter iota.) Remember that by definition, every
function has a specified domain and codomain, so this is not the same as
the identity on $B$. The inclusion $\iota$ is injective.

On the other hand, given any injective function between sets, say $\phi
\from X \to A$, the image $\im\phi$ is a subset of $A$, and there is a
bijection $\phi' \from X \to \im\phi$ given by $\phi'(x) = \phi(x)$ ($x \in
X$). Hence the set $X$ is isomorphic to (in bijection with) the subset
$\im\phi$ of $A$.

So given any subset of $A$, we get an injection into $A$, and vice
versa. These two back-and-forth processes are mutually inverse (up to
isomorphism), so subsets and injections are more or less the same thing.

Now here's an example to show you that the concept of subset is not as
clear-cut as it might seem---at least when you look at what mathematicians
actually \emph{do}, rather than what we claim we do. 
\begin{itemize}
\item
It's common to define the set $\C$ as $\R^2$.

\item
Everyone treats $\R$ as a subset of $\C$.

\item 
But almost no one would say that $\R$ is a subset of $\R^2$. (If you think
it is, and you agree that $\R$ has an element called $6$, then you must think
that $\R^2$ has an element called $6$---which you probably don't.)
\end{itemize}
So, is $\R$ a subset of $\C$ or not? In truth, while we almost always `know
what we mean', the common conventions are inconsistent.

Probably you're thinking that this all seems rather distant from `real
mathematics'. Nothing important should depend on whether $\R$ is
\emph{literally} a subset of $\C$. I agree! But the challenge is to set up
the formal definitions so that we never have to worry about
irrelevant-seeming questions like this again. And the solution is to work
with injections rather than subsets.

So: we intuitively want to define an `extension' of a field $K$ as a field
$M$ that contains $K$ as a subfield. But if we defined it that way, we'd
run into the annoying question of whether $\C$ really is an extension of
$\R$. So instead, we define an extension of $K$ to be a field $M$ together
with an injective homomorphism $K \to M$. Lemma~\ref{lemma:fd-homm} tells
us that \emph{every} homomorphism between fields is injective, so our
actual definition is as follows.

\begin{defn}
\label{defn:ext}
Let $K$ be a field. An \demph{extension} of $K$ is a field $M$ together
with a homomorphism $\iota \from K \to M$. 
\end{defn}

Often we blur the distinction between injections and subsets, speaking as
if $K$ is literally a subfield of $M$ and $\iota$ is the inclusion. We then
write \demph{$M: K$} (read `$M$ over $K$') to mean that $M$ is an extension
of $K$, not bothering to mention $\iota$.

\begin{examples}
\label{egs:ext}
\begin{enumerate}
\item 
The field $\C$, together with the inclusion $\iota \from \Q \to \C$, is an
extension of $\Q$. We write it as $\C: \Q$. Similarly, there are field
extensions $\C : \R$ and $\R : \Q$.

\item
\label{eg:ext-sqrt2}
Let 
\[
\Q(\sqrt{2}) = \{ a + b\sqrt{2} \such a, b \in \Q\}.
\]
Then $\Q(\sqrt{2})$ is a subring of $\C$ (easily), and in fact it's a
subfield: for if $(a, b) \neq (0, 0)$ then
\[
\frac{1}{a + b\sqrt{2}} 
=
\frac{a - b\sqrt{2}}{a^2 - 2b^2}
\]
(noting that the denominators are not $0$ because $\sqrt{2}$ is
irrational). So we have an extension $\C: \Q(\sqrt{2})$. Also, because $\Q
\sub \Q(\sqrt{2})$, we have another extension $\Q(\sqrt{2}): \Q$.

\item
\label{eg:ext-sqrt2i}
Write
\[
\Q(\sqrt{2}, i) = 
\{
a + b \sqrt{2} + ci + d\sqrt{2}i
\such
a, b, c, d \in \Q.
\}
\]
By direct calculation or later theory (which will make it much easier),
$\Q(\sqrt{2}, i)$ is also a subfield of $\C$, so we have extensions $\C:
\Q(\sqrt{2}, i)$ and $\Q(\sqrt{2}, i): \Q$. 

\item
\label{eg:ext-rat}
Let $K$ be a field, and consider the field $K(t)$ of rational expressions
over $K$ (Example~\ref{eg:rat}). There is a homomorphism $\iota \from K
\to K(t)$ given by $\iota(a) = a/1$ ($a \in K$). In other words, $K(t)$
contains a copy of $K$ as the constant rational expressions. So, we have a
field extension $K(t) : K$.

\item
\label{eg:ext-conj}
There is a homomorphism $\kappa\from \C \to \C$ defined by $\kappa(z) =
\ovln{z}$. So $\C$ together with $\kappa$ is an extension of $\C$! You
might feel that this example obeys the letter but not the spirit of
Definition~\ref{defn:ext}, but it \emph{is} an example.
\end{enumerate}
\end{examples}

\begin{ex}{ex:eg-int}
Find two examples of fields $K$ such that $\Q \subsetneqq K \subsetneqq
\Q(\sqrt{2}, i)$. (The symbol $\subsetneqq$ means proper subset.)
\end{ex}

Sometimes we fix a field $K$ and think about fields that contain
it---extensions of $K$. Other times, we fix a field $K$ and think about
fields it contains---subfields of $K$. It may be that we are given a mere
sub\emph{set} $X$ of $K$ and want to generate a subfield from it. Recalling
the top-down/bottom-up distinction of Digression~\ref{dig:top-bottom}, we
define this as follows.

\begin{defn}
Let $K$ be a field and $X$ a subset of $K$. The subfield of $K$
\demph{generated by} $X$ is the intersection of all the subfields of $K$
containing $X$.
\end{defn}

Let $F$ be the subfield of $K$ generated by $X$. Since any intersection of
subfields is a subfield, $F$ really is a subfield of $K$. It contains
$X$. By definition of intersection, $F$ is the \emph{smallest} subfield of
$K$ containing $X$, in the sense that any subfield of $K$ containing $X$
contains $F$.

\begin{ex}{ex:int-sub}
Check the truth of all the statements in the previous paragraph.
\end{ex}

\begin{examples}
\label{egs:sub-gen}
\begin{enumerate}
\item 
The subfield of $K$ generated by $\emptyset$ is the prime subfield of~$K$.

\item
\label{eg:sub-gen-i}
Let $L$ be the subfield of $\C$ generated by $\{i\}$. I claim that
\[
L = \{ a + b i \such a, b \in \Q \}.
\]
To prove this, we have to show that $L$ is the smallest subfield of $\C$
containing $i$. First, it \emph{is} a subfield of $\C$ (by an argument
similar to Example~\ref{egs:ext}\bref{eg:ext-sqrt2}) and it contains $0 +
1i = i$. Now let $L'$ be any subfield of $\C$ containing $i$. Then $L'$
contains the prime subfield of $\C$ (by definition of prime subfield),
which is $\Q$. So whenever $a, b \in \Q$, we have $a, b, i \in L'$ and so
$a + bi \in L'$. Hence $L \sub L'$, as required.

\item
\label{eg:sub-gen-sqrt2}
A very similar argument shows that the subfield of $\C$ generated by
$\sqrt{2}$ is what we have been calling $\Q(\sqrt{2})$.
\end{enumerate}
\end{examples}

\begin{ex}{ex:sub-gen-triv}
What is the subfield of $\C$ generated by $\{7/8\}$? By $\{2 + 3i\}$? By
$\R \cup \{i\}$?
\end{ex}

We will be \emph{very} interested in chains of fields
\[
K \sub L \sub M
\]
in which $K$ and $M$ are regarded as fixed and $L$ as variable. You can think
of $K$ as the floor, $M$ as the ceiling, and $L$ as varying in between. 

\begin{defn}
\label{defn:adj}
Let $M : K$ be a field extension and $Y \sub M$. We write \demph{$K(Y)$}
for the subfield of $M$ generated by $K \cup Y$. We call it $K$ with $Y$
\demph{adjoined}, or the subfield of $M$ \demph{generated by $Y$ over $K$}.
\end{defn}

So, $K(Y)$ is the smallest subfield of $M$ containing both $K$ and $Y$.

When $Y$ is a finite set $\{\alpha_1, \ldots, \alpha_n\}$, we write
$K(\{\alpha_1, \ldots, \alpha_n\})$ as \demph{$K(\alpha_1, \ldots,
\alpha_n)$}.

\begin{examples}
\label{egs:adj}
\begin{enumerate}
\item
\label{eg:adj-sqrt2}
Take $M : K$ to be $\C : \Q$ and $Y = \{\sqrt{2}\}$. By definition, $K(Y)$
is the smallest subfield of $\C$ containing $\Q \cup \{\sqrt{2}\}$. But
\emph{every} subfield of $\C$ contains $\Q$: that's what it means for $\Q$
to be the prime subfield of $\C$. So, $K(Y)$ is the smallest subfield of
$\C$ containing $\sqrt{2}$. By
Example~\ref{egs:sub-gen}\bref{eg:sub-gen-sqrt2}, that's exactly what we've
been calling $\Q(\sqrt{2})$ all along. We refer to $\Q(\sqrt{2})$ as `$\Q$
with $\sqrt{2}$ adjoined'.

\item
\label{eg:adj-sqrt2i}
Similarly, $\Q$ with $i$ adjoined is
\[
\Q(i) = \{a + bi \such a, b \in \Q\}
\]
(Example~\ref{egs:sub-gen}\bref{eg:sub-gen-i}), and $\Q$ with $\{\sqrt{2},
i\}$ adjoined is the subfield denoted by $\Q(\sqrt{2}, i)$ in
Example~\ref{egs:ext}\bref{eg:ext-sqrt2i}. 

\item 
Let $M$ be a field and $X \sub M$. Write $K$ for the prime subfield of
$M$. Then $K(X)$ is the smallest subfield of $M$ containing $K$ and
$X$. But \emph{every} subfield of $M$ contains $K$, by definition of prime
subfield. So $K(X)$ is the smallest subfield of $M$ containing $X$; that
is, it's the subfield of $M$ generated by $X$.

We already saw this argument in~\bref{eg:adj-sqrt2}, in the case $M = \C$
and $X = \{\sqrt{2}\}$.

\item
Let $K$ be any field and let $M$ be the field $K(t)$ of rational
expressions over $K$, which is an extension of $K$. You might worry that
there's some ambiguity in the notation: $K(t)$ could \emph{either} mean the
field of rational expressions over $K$ (as defined in
Example~\ref{egs:ext}\bref{eg:ext-rat}) \emph{or} the subfield of $K(t)$
obtained by adjoining the element $t$ of $K(t)$ to $K$ (as in
Definition~\ref{defn:adj}). 

In fact, they're the same. In other words, the smallest subfield of $K(t)$
containing $K$ and $t$ is $K(t)$ itself. Or equivalently, the \emph{only}
subfield of $K(t)$ containing $K$ and $t$ is $K(t)$ itself. To see this,
let $L$ be any such subfield. For any polynomial $f(t) = \sum a_i t^i$
over $K$, we have $f(t) \in L$, since $a_i, t \in L$ and $L$ is closed under
multiplication and addition. Hence for any polynomials $f(t), g(t)$ over
$K$ with $g(t) \neq 0$, we have $f(t), g(t) \in L$, so $f(t)/g(t) \in L$
as $L$ is closed under division by nonzero elements. So $L = K(t)$. 
\end{enumerate}
\end{examples}

\begin{warning}{wg:adj-not-quad}
It is \emph{not} true in general that 
\begin{align}
\label{eq:adj-not-quad}
K(\alpha) = \{a + b\alpha \such a, b \in K\}
\qquad
\text{\color{Red!100}(false!)}
\end{align}
Examples like $\Q(\sqrt{2})$ and $\Q(i)$ do satisfy this, but that's only
because $\sqrt{2}$ and $i$ satisfy \emph{quadratic} equations. Certainly
the right-hand side is a \emph{subset} of $K(\alpha)$, but in general it's
much smaller, and isn't a subfield.

You've just seen an example: the field $K(t)$ of rational expressions 
is much bigger than the set $\{a + bt \such a, b \in K\}$ of polynomials of
degree $\leq 1$. And that set of polynomials isn't closed under
multiplication. 

Another example: let $\xi$ be the real cube root of $2$. You can show
that $\xi^2$ cannot be expressed as $a + b\xi$ for any $a, b \in \Q$
(a fact we'll come back to in
Example~\ref{egs:min-poly}\bref{eg:min-poly-cbrt2}). But $\xi \in
\Q(\xi)$, so $\xi^2 \in \Q(\xi)$, so~\eqref{eq:adj-not-quad} fails
in this case. In fact,
\[
\Q(\xi) = \{a + b\xi + c\xi^2 \such a, b, c \in \Q\}.
\]
We'll see why next week.
\end{warning}

\begin{ex}{ex:adjn-two}
Let $M: K$ be a field extension. Show that $K(Y \cup Z) = (K(Y))(Z)$
whenever $Y, Z \sub M$. (For example, $K(\alpha, \beta) =
(K(\alpha))(\beta)$ whenever $\alpha, \beta \in M$.)
\end{ex}

\begin{remark}
For a field extension $M: K$, I'll generally use small Greek letters
$\alpha, \beta, \ldots$ for elements of $M$ and small English letters $a,
b, \ldots$ for elements of $K$. 
\end{remark}

\section{Algebraic and transcendental elements}
\label{sec:alg-trans}

A complex number $\alpha$ is said to be `algebraic' if
\[
a_0 + a_1 \alpha + \cdots + a_n \alpha^n = 0
\]
for some rational numbers $a_i$, not all zero. (You may have seen this
definition with `integer' instead of `rational number'. It makes no
difference, as you can always clear the denominators.) This concept
generalizes to arbitrary field extensions:

\begin{defn}
Let $M: K$ be a field extension and $\alpha \in M$. Then $\alpha$ is
\demph{algebraic} over $K$ if there exists $f \in K[t]$ such that
$f(\alpha) = 0$ but $f \neq 0$, and \demph{transcendental} otherwise.
\end{defn}

\begin{ex}{ex:alg-base}
Show that every element of $K$ is algebraic over $K$.\\
\end{ex}

\begin{examples}
\label{egs:alg-trans}
\begin{enumerate}
\item 
Let $n \geq 1$. Then $e^{2\pi i/n} \in \C$ is algebraic over $\Q$, since
$f(t) = t^n - 1$ is a nonzero polynomial such that $f(e^{2\pi i/n}) = 0$.

\item
The numbers $\pi$ and $e$ are both transcendental over $\Q$. Both
statements are hard to prove (and we won't prove them). By
Exercise~\ref{ex:alg-base}, any complex number transcendental over $\Q$ is
irrational. Proving the irrationality of $\pi$ and $e$ is already a
challenge; proving they're transcendental is even harder.

\item
Although $\pi$ is transcendental over $\Q$, it is algebraic over $\R$, since
it's an \emph{element} of $\R$. (Again, we're using
Exercise~\ref{ex:alg-base}.) Moral: you shouldn't say an element of a field
is just `algebraic' or `transcendental'; you should say it's
`algebraic/transcendental 
\emph{over $K$}', specifying your $K$. Or
at least, you should do this when there's any danger of confusion.

\item
\label{eg:at-rat}
Take the field $K(t)$ of rational expressions over a field $K$. Then $t \in
K(t)$ is transcendental over $K$, since $f(t) = 0 \iff f = 0$. 
\end{enumerate}
\end{examples}

The set of complex numbers algebraic over $\Q$ is written as
\demph{$\ovln{\Q}$}. It's a fact that $\ovln{\Q}$ is a subfield of $\C$,
but this is extremely hard to prove by elementary means. Next week I'll
show you that with a surprisingly small amount of abstract algebra, you can
transform this from a very hard problem into an easy one
(Proposition~\ref{propn:Qbar-subfd}).

So that you appreciate the miracle later, I give you this unusual
exercise now.

\begin{ex}{ex:alg-fd-doomed}
Attempt to prove any part of the statement that $\ovln{\Q}$ is a subfield
of $\C$. For example, try to show that $\ovln{\Q}$ is closed under
addition, or multiplication, or reciprocals. I have no idea how to do any
of these using only our current tools, but it's definitely worth a few
minutes of doomed effort to get a sense of the difficulties.
\end{ex}

\begin{digression}{dig:ac}
The field $\ovln{\Q}$ is, in fact, algebraically closed, as you'll see in
\href{https://www.maths.ed.ac.uk/~tl/galois}{Workshop~3, question~8}. So
you might ask whether it's possible for \emph{every} field $K$ to build an
algebraically closed field containing $K$. It turns out that it is. Better
still, there is a unique `smallest' algebraically closed field containing
$K$, called its \demph{algebraic closure $\ovln{K}$}. For example, the
algebraic closure of $\Q$ is $\ovln{\Q}$. We won't have time to do
algebraic closure properly, but you can read about it in most Galois theory
texts.
\end{digression}

Let $M: K$ be a field extension and $\alpha \in M$. An \demph{annihilating
polynomial} of $\alpha$ is a polynomial $f \in K[t]$ such that $f(\alpha)
= 0$. So, $\alpha$ is algebraic if and only if it has some nonzero
annihilating polynomial. 

It is natural to ask not only \emph{whether} $\alpha$ is annihilated by
some nonzero polynomial, but \emph{which} polynomials annihilate it. The
situation is pleasantly simple: 

\begin{lemma}
Let $M: K$ be a field extension and $\alpha \in M$. Then there is a
polynomial $m(t) \in K[t]$ such that
\begin{align}
\label{eq:ann}
\idlgen{m} 
= 
\{\text{annihilating polynomials of }\alpha\text{ over }K\}.
\end{align}
If $\alpha$ is transcendental over $K$ then $m = 0$. If $\alpha$ is
algebraic over $K$ then there is a unique monic polynomial $m$
satisfying~\eqref{eq:ann}. 
\end{lemma}

\begin{proof}
By the universal property of polynomial rings
(Proposition~\ref{propn:univ-poly}), there is a unique homomorphism
\[
\theta \from K[t] \to M
\]
such that $\theta(a) = a$ for all $a \in K$ and $\theta(t) = \alpha$. (Here
we're taking the `$\phi$' of Proposition~\ref{propn:univ-poly} to be the
inclusion $K \to M$.) Then
\[
\theta \Bigl( \sum a_i t^i \Bigr) 
=
\sum a_i \alpha^i
\]
for all $\sum a_i t^i \in K[t]$, so 
\[
\ker \theta 
= 
\{\text{annihilating polynomials of }\alpha\text{ over }K\}.
\]
But $\ker\theta$ is an ideal of the principal ideal domain $K[t]$ (using
Proposition~\ref{propn:Kt-pid}), so $\ker\theta = \idlgen{m}$ for some $m
\in K[t]$.

If $\alpha$ is transcendental then $\ker\theta = \{0\}$, so $m = 0$. 

If $\alpha$ is algebraic then $m \neq 0$. Multiplying a polynomial by a
nonzero constant does not change the ideal it generates (by
Exercise~\ref{ex:associates} and
Lemma~\ref{lemma:fdt}\bref{part:fdt-units}), so we can assume that $m$ is
monic. It remains to prove that $m$ is the \emph{only} monic polynomial
such that $\idlgen{m} = \ker\theta$. If $\twid{m}$ is another monic
polynomial such that $\idlgen{\twid{m}} = \ker\theta$ then $\twid{m} = cm$
for some nonzero constant $c$ (again by Exercise~\ref{ex:associates} and
Lemma~\ref{lemma:fdt}\bref{part:fdt-units}), and both are monic, so $c = 1$
and $\twid{m} = m$.
\end{proof}

\begin{defn}
Let $M: K$ be a field extension and let $\alpha \in M$ be algebraic over
$K$. The \demph{minimal polynomial} of $\alpha$ is the unique monic
polynomial $m$ satisfying~\eqref{eq:ann}.
\end{defn}

\begin{warning}{wg:trans-min}
We do not define the minimal polynomial of a transcendental element. So for
an arbitrary field extension $M: K$, some elements of $M$ may have no
minimal polynomial.
\end{warning}

\begin{ex}{ex:min-poly-triv}
What is the minimal polynomial of an element of $K$?\\
\end{ex}

This is an important definition, so we give some equivalent conditions.

\begin{lemma}
\label{lemma:min-poly-tfae}
Let $M: K$ be a field extension, let $\alpha \in M$ be algebraic over $K$,
and let $m \in K[t]$ be a monic polynomial. The following are equivalent:
\begin{enumerate}
\item 
\label{part:mpt-min}
$m$ is the minimal polynomial of $\alpha$ over $K$;

\item
\label{part:mpt-exp}
$m(\alpha) = 0$, and $m \dvd f$ for all annihilating polynomials $f$ of
$\alpha$ over $K$;

\item 
\label{part:mpt-deg}
$m(\alpha) = 0$, and $\deg(m) \leq \deg(f)$ for all nonzero annihilating
polynomials $f$ of $\alpha$ over $K$;

\item
\label{part:mpt-irr}
$m(\alpha) = 0$ and $m$ is irreducible over $K$.
\end{enumerate}
\end{lemma}

Part~\bref{part:mpt-deg} says the minimal polynomial is a monic
annihilating polynomial of least degree.

\begin{proof}
\bref{part:mpt-min}$\implies$\bref{part:mpt-exp} follows from the
definition of minimal polynomial.

\bref{part:mpt-exp}$\implies$\bref{part:mpt-deg} because if $m \dvd f \neq
0$ then $\deg(m) \leq \deg(f)$. 

\bref{part:mpt-deg}$\implies$\bref{part:mpt-irr}:
assume~\bref{part:mpt-deg}. First, $m$ is not constant: for if $m$ is
constant then $m = 1$ (since $m$ is monic); but $m(\alpha) = 0$, so $1 = 0$
in $K$, a contradiction. Next, suppose that $m = fg$ for some $f, g \in
K[t]$. Then $0 = m(\alpha) = f(\alpha)g(\alpha)$, so without loss of
generality, $f(\alpha) = 0$. By~\bref{part:mpt-deg}, $\deg(f) \geq
\deg(m)$, so $\deg(f) = \deg(m)$ and $\deg(g) = 0$. This
proves~\bref{part:mpt-irr}. 

\bref{part:mpt-irr}$\implies$\bref{part:mpt-min}:
assume~\bref{part:mpt-irr}, and write $m_\alpha$ for the minimal polynomial
of $\alpha$. We have $m_\alpha \dvd m$ by definition of $m_\alpha$ and
since $m(\alpha) = 0$. But $m$ is irreducible and $m_\alpha$ is not
constant, so $m$ is a nonzero constant multiple of $m_\alpha$. Since both
are monic, $m = m_\alpha$, proving~\bref{part:mpt-min}.
\end{proof}

\begin{examples}
\label{egs:min-poly}
\begin{enumerate}
\item 
\label{eg:min-poly-sqrt2}
The minimal polynomial of $\sqrt{2}$ over $\Q$ is $t^2 - 2$. There are
several ways to see this. 

One argument: $t^2 - 2$ is a monic annihilating polynomial of $\sqrt{2}$,
and no nonzero polynomial of degree $\leq 1$ over $\Q$ annihilates
$\sqrt{2}$ since it is irrational. Then use
Lemma~\ref{lemma:min-poly-tfae}\bref{part:mpt-deg}.

Another: $t^2 - 2$ is an irreducible monic annihilating polynomial. It is
irreducible because $t^2 - 2$ has degree $2$ and has no rational
roots (using Lemma~\ref{lemma:irr-triv}\bref{part:it-23}). Then use
Lemma~\ref{lemma:min-poly-tfae}\bref{part:mpt-irr}.

\item
\label{eg:min-poly-cbrt2}
The minimal polynomial of $\sqrt[3]{2}$ over $\Q$ is $t^3 - 2$. This will
follow from Lemma~\ref{lemma:min-poly-tfae}\bref{part:mpt-irr} as long as
$t^3 - 2$ is irreducible, which you can show using either
Lemma~\ref{lemma:irr-triv}\bref{part:it-23} or Eisenstein.

But unlike in~\bref{eg:min-poly-sqrt2}, it's \emph{not} so easy to show
directly that $t^3 - 2$ is the annihilating polynomial of least degree. Try
proving with your bare hands that $\sqrt[3]{2}$ satisfies no quadratic
equation over $\Q$, i.e.\ that the equation
\[
\sqrt[3]{2}^2 = a\sqrt[3]{2} + b
\]
has no solution for $a, b \in \Q$. It's not impossible, but it's a
mess. (You naturally begin by cubing both sides, but look what happens
next\ldots) So the theory really gets us something here.
\video{Two traps}

\item
\label{eg:min-poly-cyclo}
Let $p$ be a prime number, and put $\omega = e^{2\pi i/p} \in \C$. Then
$\omega$ is a root of $t^p - 1$, but that is not the minimal polynomial of
$\omega$, since it is reducible:
\[
t^p - 1 = (t - 1)m(t)
\]
where
\[
m(t) = t^{p - 1} + \cdots + t + 1.
\]
Since $\omega^p - 1 = 0$ but $\omega - 1 \neq 0$, we must have $m(\omega) =
0$. By Example~\ref{eg:cyclo-irr}, $m$ is irreducible over $\Q$. Hence $m$
is the minimal polynomial of $\omega$ over $\Q$.
\end{enumerate}
\end{examples}

\section{Simple extensions}
\label{sec:simple}

Suppose I give you a field $K$ and a nonconstant polynomial $f$ over
$K$. Can you find an extension of $K$ containing a root of $f$?

If $K = \Q$, it's easy. The
fundamental theorem of algebra guarantees that $f$ has a root $\alpha$
in $\C$, so you can take your extension to be $\C$. Or, if you're feeling
economical, you can take $\Q(\alpha)$ as your extension, that being the
\emph{smallest} subfield of $\C$ containing your root $\alpha$.

But what if $K$ is not $\Q$?

It's a bit like this. Say you want to go rock-climbing. If you live next to
Arthur's Seat, no problem: just walk out of your door and get
started. There's a ready-made solution. But if you live in the middle of
the fields in the Netherlands, you're going to have to build your own
climbing wall.

When $K = \Q$, we have a ready-made algebraically closed field $\C$
containing $K$, so it's easy to find an extension of $K$ containing a root
of $f$. For a general $K$, it's not so easy. We're going to have to
build an extension of our own. But it's not so hard either! 

Rather than taking a general polynomial $f$, we will just consider
irreducibles. That's fine, because Theorem~\ref{thm:fact-polys}
guarantees that $f$ has some irreducible factor $m$, and any root of $m$ is
automatically a root of $f$. We will also restrict to \emph{monic}
irreducibles, which makes no real difference to anything.

So, we have a field $K$ and a monic irreducible polynomial $m \in K[t]$. We
are trying to construct an extension $M$ of $K$ and an element $\alpha \in
M$ such that $m(\alpha) = 0$. By Lemma~\ref{lemma:min-poly-tfae}, $m$ will
then be the minimal polynomial of $\alpha$.

This construction can be done as follows. By Corollary~\ref{cor:qt-by-irr},
the quotient $K[t]/\idlgen{m}$ is a field. We have ring homomorphisms
\begin{align}
\label{eq:qt-ext}
K \to K[t] \toby{\pi} K[t]/\idlgen{m},
\end{align}
where the first homomorphism sends $a \in K$ to the constant polynomial $a
\in K[t]$ and $\pi$ is the canonical homomorphism. Their composite is a
homomorphism of fields $K \to K[t]/\idlgen{m}$.  So, we have a field
extension $\bigl(K[t]/\idlgen{m}\bigr): K$. And one of the elements of
$K[t]/\idlgen{m}$ is $\pi(t)$, which I will call $\alpha$. 

For a polynomial $\sum a_i t^i \in K[t]$,
\begin{align}
\label{eq:qt-alpha}
\pi\biggl( \sum_i a_i t^i\biggr)
=
\sum_i a_i \alpha^i.
\end{align}
Since $\pi$ is surjective, every element of $K[t]/\idlgen{m}$ is of the
form $\sum a_i \alpha^i$. You can think of $K[t]/\idlgen{m}$ as the ring of
polynomials over $K$, but with two polynomials seen as equal if they differ
by a multiple of $m$. (Compare how you think of $\Z/\idlgen{p}$.)

Part~\bref{part:sb-alg} of the following lemma says that $\alpha$ is a root
of $m$, and that if we're looking for an extension of $K$ containing a root
of $m$, then $K[t]/\idlgen{m}$ is an economical choice: it's no bigger
than it needs to be.

Part~\bref{part:sb-trans} answers an analogous but easier question: what if
we start with a field $K$ and want to extend it by an element that
satisfies \emph{no} nonzero polynomial over $K$?

\begin{lemma}
\label{lemma:simple-build}
Let $K$ be a field.
\begin{enumerate}
\item 
\label{part:sb-alg}
Let $m \in K[t]$ be monic and irreducible. Write $\alpha \in
K[t]/\idlgen{m}$ for the image of $t$ under the canonical homomorphism
$K[t] \to K[t]/\idlgen{m}$. Then $\alpha$ has minimal polynomial $m$ over
$K$, and $K[t]/\idlgen{m}$ is generated by $\alpha$ over $K$.

\item
\label{part:sb-trans}
The element $t$ of the field $K(t)$ of rational expressions over $K$ is
transcendental over $K$, and $K(t)$ is generated by $t$ over $K$.
\end{enumerate}
\end{lemma}

In part~\bref{part:sb-alg}, we are viewing $K[t]/\idlgen{m}$ as
an extension of $K$, as in~\eqref{eq:qt-ext}.

\begin{proof}
For~\bref{part:sb-alg}, write $M =
K[t]/\idlgen{m}$. Equation~\eqref{eq:qt-alpha} implies that the set of
annihilating polynomials of $\alpha$ over $K$ is $\ker\pi$, which is
$\idlgen{m}$. So $m$ is by definition the minimal polynomial of $\alpha$
over $K$. 

Any subfield $L$ of $M$ containing $K$ and $\alpha$ contains every
polynomial in $\alpha$ over $K$, so $L = M$. Hence $M$ is generated by
$\alpha$ over $K$.

For~\bref{part:sb-trans}, we have already seen that $t$ is transcendental
over $K$ (Example~\ref{egs:alg-trans}\bref{eg:at-rat}). 

Let $L$ be a subfield of $K(t)$ containing $K$ and $t$. Then any
polynomials $f, g \in K[t]$ are in $L$, so if $g \neq 0$ then $f/g \in
L$. Hence $L = M$, and $M$ is generated by $t$ over $K$.
\end{proof}

So far, we've seen that given a monic irreducible polynomial $m$ over a
field $K$, we can build an extension of $K$ containing a root of $m$. In
fact, there are many such extensions. For instance:

\begin{example}
\label{eg:simple-ext-lots}
If $K = \Q$ and $m(t) = t^2 - 2$ then all three of the extensions
$\Q(\sqrt{2})$, $\R$ and $\C$ contain a root of $m$.
\end{example}

But $K[t]/\idlgen{m}$ is the \emph{canonical} or \emph{minimal}
choice. In fact, $K[t]/\idlgen{m}$ has a universal property. To
express it, we need a definition.

\begin{defn}
\label{defn:homm-over}
Let $K$ be a field, and let $\iota\from K \to M$ and $\iota' \from K \to
M'$ be extensions of $K$. A homomorphism $\phi \from M \to M'$ is said to
be a \demph{homomorphism over $K$} if 
\[
\xymatrix{
M \ar[rr]^\phi &      
        &
M'      \\
        &
K \ar[lu]^\iota \ar[ru]_{\iota'}
}
\]
commutes. 
\end{defn}

For the triangle to commute means that $\phi(\iota(a)) = \iota'(a)$ for all
$a \in K$. Very often, the homomorphisms $\iota$ and $\iota'$ are thought
of as inclusions, and we write both $\iota(a)$ and $\iota'(a)$ as just
$a$. Then for $\phi$ to be a homomorphism over $K$ means that $\phi(a) = a$
for all $a \in K$.

\begin{example}
\label{eg:conj-over}
Define $\kappa \from \C \to \C$ by $\kappa(z) = \ovln{z}$. Then $\kappa$ is
a homomorphism, and it is a homomorphism over $\R$ since $\ovln{a} = a$ for
all $a \in \R$. 
\end{example}

\begin{ex}{ex:homm-minpoly}
Let $M: K$ and $L: K$ be field extensions, and let $\phi \from M \to L$ be
a homomorphism over $K$. Show that if $\alpha \in M$ has minimal polynomial
$m$ over $K$ then $\phi(\alpha) \in L$ also has minimal polynomial $m$ over
$K$. 
\end{ex}

Here's an extremely useful lemma about homomorphisms over a field.

\begin{lemma}
\label{lemma:gen-epic}
Let $M$ and $M'$ be extensions of a field $K$, and let $\phi, \psi\from M
\to M'$ be homomorphisms over $K$. Let $Y$ be a subset of $M$ such that $M
= K(Y)$. If $\phi(\alpha) = \psi(\alpha)$ for all $\alpha \in Y$ then $\phi
= \psi$.
\end{lemma}

\begin{proof}
We have $\phi(a) = a = \psi(a)$ for all $a \in K$, since $\phi$ and $\psi$
are homomorphisms over $K$. But we are assuming that $\phi(\alpha) =
\psi(\alpha)$ for all $\alpha \in Y$, so $K \cup Y$ is a subset of the
equalizer $\Eq\{\phi, \psi\}$ (Definition~\ref{defn:equalizer}). Hence by
Lemma~\ref{lemma:eq-subfd}, $\Eq\{\phi, \psi\}$ is a subfield of $M$
containing $K \cup Y$. But $K(Y)$ is the smallest subfield of $M$
containing $K \cup Y$, so $\Eq\{\phi, \psi\} = K(Y) = M$. Hence $\phi =
\psi$. 
\end{proof}

Now we can formulate the universal property of $K[t]/\idlgen{m}$, and
similarly that of $K(t)$.

\begin{propn}[Universal properties of {$K[t]/\idlgen{m}$} and \protect$K(t)$]
\label{propn:univ-simp}
Let $K$ be a field.
\begin{enumerate}
\item 
\label{part:us-alg}
Let $m \in K[t]$ be monic and irreducible, let $L: K$ be an extension
of $K$, and let $\beta \in L$ with minimal polynomial
$m$. 
Write $\alpha$
for the image of $t$ under the canonical homomorphism $K[t] \to
K[t]/\idlgen{m}$. 
Then there is exactly one homomorphism $\phi \from K[t]/\idlgen{m} \to L$
over $K$ such that $\phi(\alpha) = \beta$.

\item
\label{part:us-trans}
Let $L: K$ be an extension of $K$, and let $\beta \in L$ be
transcendental. Then there is exactly one homomorphism $\phi \from K(t) \to
L$ over $K$ such that $\phi(t) = \beta$.
\end{enumerate}
\end{propn}

Diagram for~\bref{part:us-alg}:
\[
\xymatrix@R=2ex{
        &
        &
\beta   \\
\alpha \ar@{|.>}[rru]
        &       
        &
L       \\
K[t]/\idlgen{m} \ar@{.>}[rru]^\phi      &
        &
        \\
        &
        &
        \\
        &
K \ar[luu] \ar[ruuu]
}
\]
I've drawn $L$ higher than $K[t]/\idlgen{m}$ to convey the idea that $L$
may be bigger. Before I give the proof, here's an example.

\begin{example}
\label{eg:univ-simp-sqrt2}
Let $K = \Q$ and $m(t) = t^2 - 2$. Let $L = \C$ and let $\beta = -\sqrt{2}
\in \C$. Proposition~\ref{propn:univ-simp}\bref{part:us-alg} tells us that
there is a unique field homomorphism 
\[
\phi \from \Q[t]/\idlgen{t^2 - 2} \to \C
\]
over $\Q$ mapping the equivalence class of $t$ to $-\sqrt{2}$. 
\end{example}

\begin{pfof}{Proposition~\ref{propn:univ-simp}}
For~\bref{part:us-alg}, first we show there is \emph{at least} one
homomorphism $\phi\from K[t]/\idlgen{m} \to L$ over $K$ such that
$\phi(\alpha) = \beta$. By the universal property of polynomial rings
(Proposition~\ref{propn:univ-poly}), there is exactly one homomorphism
$\theta \from K[t] \to L$ such that $\theta(a) = a$ for all $a \in K$ and
$\theta(t) = \beta$. Then $\theta(m(t)) = m(\beta) = 0$, so $\idlgen{m}
\sub \ker\theta$. Hence by the universal property of quotients
(p.~\pageref{p:univ-qt}), there is exactly one homomorphism $\phi \from
K[t]/\idlgen{m} \to L$ such that
\[
\xymatrix{
K[t] \ar[d]_\pi \ar[dr]^\theta  &       \\
K[t]/\idlgen{m} \ar@{.>}[r]_-\phi   &
L
}
\]
commutes. Then $\phi$ is a homomorphism over $K$, since for all $a \in K$
we have 
\[
\phi(a) = \phi(\pi(a)) = \theta(a) = a.
\] 
Moreover, 
\[
\phi(\alpha) = \phi(\pi(t)) = \theta(t) = \beta, 
\]
so $\phi(\alpha) = \beta$. 

Now we show there is \emph{at most} one homomorphism $K[t]/\idlgen{m} \to
L$ over $K$ such that $\alpha \mapsto \beta$. Let $\phi$ and $\phi'$ be two
such. Then $\phi(\alpha) = \phi'(\alpha)$, and $\alpha$ generates
$K[t]/\idlgen{m}$ over $K$
(Lemma~\ref{lemma:simple-build}\bref{part:sb-alg}), so $\phi = \phi'$ by
Lemma~\ref{lemma:gen-epic}.  

For~\bref{part:us-trans}, first we show there is \emph{at least} one
homomorphism $\phi\from K(t) \to L$ over $K$ such that $\phi(t) =
\beta$. Every element of $K(t)$ can be represented as $f/g$ where $f, g \in
K[t]$ with $g \neq 0$. Since $\beta$ is transcendental over $K$, we have
$g(\beta) \neq 0$, and so $f(\beta)/g(\beta)$ is a well-defined element of
$L$. One can check that this gives a well-defined homomorphism
\[
\begin{array}{cccc}
\phi\from       &K(t)                   &
\to             &L      \\[1ex]
                &\displaystyle\frac{f(t)}{g(t)}      &
\mapsto         &\displaystyle\frac{f(\beta)}{g(\beta)}.
\end{array}
\]
Evidently $\phi$ is a homomorphism over $K$ (that is, $\phi(a) = a$ for all
$a \in K$), and evidently $\phi(t) = \beta$.

The proof that there is \emph{at most} one homomorphism $K(t) \to L$ over
$K$ such that $t \mapsto \beta$ is similar to the uniqueness proof in
part~\bref{part:us-alg}.
\end{pfof}

\begin{ex}{ex:unique-trans}
Fill in the details of the last paragraph of that proof.\\
\end{ex}

We know now that for a monic irreducible polynomial $m$ over $K$, the
extension $K[t]/\idlgen{m}$ contains a root of $m$ and is generated by that
root. As we're about to see, Proposition~\ref{propn:univ-simp} implies that
$K[t]/\idlgen{m}$ is the \emph{only} extension of $K$ with this property.
But the word `only' has to be interpreted in an up-to-isomorphism sense
(as you're used to from statements like `there is only one group of order
$5$'). The appropriate notion of isomorphism is as follows.

Let $M$ and $M'$ be extensions of a field $K$. A homomorphism $\phi \from M
\to M'$ is an \demph{isomorphism over $K$} if it is a homomorphism over $K$
and an isomorphism of fields. (You can check that $\phi^{-1}$ is then also
a homomorphism over $K$.) If such a $\phi$ exists, we say that $M$ and $M'$
are \demph{isomorphic over $K$}.

\begin{warning}{wg:iso-over}
Let $M$ and $M'$ be extensions of a field $K$. It can happen that $M$ and
$M'$ are isomorphic, but not isomorphic over $K$. In other words, just
because it's possible to find an isomorphism $\phi\from M \to M'$, it
doesn't mean you can find one making the triangle in
Definition~\ref{defn:homm-over}
commute. \href{https://www.maths.ed.ac.uk/~tl/galois}{Workshop~3,
question~16} leads you through a counterexample.
\end{warning}

\begin{cor}
\label{cor:simp-iso}
Let $K$ be a field.
\begin{enumerate}
\item 
\label{part:si-alg}
Let $m \in K[t]$ be monic and irreducible, let $L: K$ be an extension of
$K$, and let $\beta \in L$ with minimal polynomial $m$ and with $L =
K(\beta)$. Write $\alpha$
for the image of $t$ under the canonical homomorphism $K[t] \to
K[t]/\idlgen{m}$. 
Then there is exactly one isomorphism $\phi \from
K[t]/\idlgen{m} \to L$ over $K$ such that $\phi(\alpha) = \beta$.

\item
\label{part:si-trans}
Let $L: K$ be an extension of $K$, and let $\beta \in L$ be transcendental
with $L = K(\beta)$. Then there is exactly one isomorphism $\phi \from K(t)
\to L$ over $K$ such that $\phi(t) = \beta$.
\end{enumerate}
\end{cor}

(Spot the differences between this corollary and
Proposition~\ref{propn:univ-simp}\ldots)

\begin{proof}
For~\bref{part:si-alg}, Proposition~\ref{propn:univ-simp}\bref{part:us-alg}
implies that there is a unique homomorphism $\phi \from K[t]/\idlgen{m} \to
L$ over $K$ such that $\phi(\alpha) = \beta$. So we only have to show that
$\phi$ is an isomorphism. Since homomorphisms of fields are injective, we
need only show that $\phi$ is surjective. Now by
Lemma~\ref{lemma:subfd-im-inv}\bref{part:sii-im}, $\im\phi$ is a subfield
of $L$, and it contains both $K$ (since $\phi$ is a homomorphism over $K$)
and $\beta$ (since $\phi(\alpha) = \beta$). But $L = K(\beta)$, so $\im\phi
= L$.

The proof of~\bref{part:si-trans} is similar.
\end{proof}

\begin{examples}
\begin{enumerate}
\item
Let $m$ be a monic irreducible polynomial over $\Q$. Choose a complex root
$\beta$ of $m$. Then the subfield $\Q(\beta)$ of $\C$ is an extension of
$\Q$ generated by $\beta$. So by Corollary~\ref{cor:simp-iso}\bref{part:si-alg},
$\Q[t]/\idlgen{m} \iso \Q(\beta)$.

\item
Let $\beta$ be a transcendental complex number. Then by
Corollary~\ref{cor:simp-iso}\bref{part:si-trans}, the field $\Q(t)$ of
rational expressions is isomorphic to $\Q(\beta) \sub \C$.
\end{enumerate}
\end{examples}

Field extensions generated by a single element have a special name.

\begin{defn}
A field extension $M: K$ is \demph{simple} if there exists $\alpha \in M$
such that $M = K(\alpha)$.
\end{defn}

\begin{examples}
\label{egs:simple}
\begin{enumerate}
\item
\label{eg:simple-Q23}
Surprisingly many extensions are simple. For instance, $\Q(\sqrt{2},
\sqrt{3}) : \Q$ is a simple extension (despite appearances), because in
fact $\Q(\sqrt{2}, \sqrt{3}) = \Q(\sqrt{2} + \sqrt{3})$.

\item
$K(t): K$ is simple, where $K(t)$ is the field of rational expressions
over $K$.
\end{enumerate}
\end{examples}

\begin{ex}{ex:sqrt2plus3}
Prove that $\Q(\sqrt{2}, \sqrt{3}) = \Q(\sqrt{2} + \sqrt{3})$. Hint: begin
by finding $(\sqrt{2} + \sqrt{3})^3$.
\end{ex}

We've now shown that simple extensions can be
classified completely:
\video{How to understand simple algebraic extensions}%

\begin{bigthm}[Classification of simple extensions]
\label{thm:class-simp}
Let $K$ be a field.
\begin{enumerate}
\item 
\label{part:cs-alg}
Let $m \in K[t]$ be a monic irreducible polynomial. Then there exist an
extension $M: K$ and an algebraic element $\alpha \in M$ such that $M =
K(\alpha)$ and $\alpha$ has minimal polynomial $m$ over $K$.

Moreover, if $(M, \alpha)$ and $(M', \alpha')$ are two such pairs, there is
exactly one isomorphism $\phi\from M \to M'$ over $K$ such that
$\phi(\alpha) = \alpha'$.

\item
\label{part:cs-trans}
There exist an extension $M: K$ and a transcendental element $\alpha \in M$
such that $M = K(\alpha)$.

Moreover, if $(M, \alpha)$ and $(M', \alpha')$ are two such pairs, there is
exactly one isomorphism $\phi\from M \to M'$ over $K$ such that $\phi(\alpha)
= \alpha'$.
\end{enumerate}
\end{bigthm}

\begin{proof}
For~\bref{part:cs-alg}, we can take $M = K[t]/\idlgen{m}$ and $\alpha$ to
be the image of $t$ under the canonical homomorphism $K[t] \to
M$. Lemma~\ref{lemma:simple-build}\bref{part:sb-alg} implies that $\alpha$
has minimal polynomial $m$ over $K$ and that $M = K(\alpha)$, and
Corollary~\ref{cor:simp-iso}\bref{part:si-alg} gives `Moreover'.

Part~\bref{part:cs-trans} follows from
Lemma~\ref{lemma:simple-build}\bref{part:sb-trans} and
Corollary~\ref{cor:simp-iso}\bref{part:si-trans} in the same way.
\end{proof}

Conclusion: given any field $K$ (not necessarily $\Q$!) and any
monic irreducible $m(t) \in K[t]$, we can say the words `\demph{adjoin to
$K$ a root $\alpha$ of $m$}', and this unambiguously defines an extension
$K(\alpha): K$. (At least, unambiguously up to isomorphism over $K$---but
who could want more?) Similarly, we can unambiguously adjoin to $K$ a
transcendental element.

\begin{examples}
\label{egs:adjoin}
\begin{enumerate}
\item
Let $K$ be any field not containing a square root of $2$. Then $t^2 - 2$ is
irreducible over $K$. So we can adjoin to $K$ a root of $t^2 - 2$, giving
an extension $K(\sqrt{2}): K$.

We have already seen this example many times when $K = \Q$, in which case
$K(\sqrt{2})$ can be seen as a subfield of $\C$. But the construction works
for \emph{any} $K$. For instance, $2$ has no square root in $\F_3$, so there
is an extension $\F_3(\sqrt{2})$ of $\F_3$. It can be constructed as
$\F_3[t]/\idlgen{t^2 - 2}$.

\item
\label{eg:adjoin-F2}
The polynomial $m(t) = 1 + t + t^2$ is irreducible over $\F_2$, so we may
adjoin to $\F_2$ a root $\alpha$ of $m$. Then $\F_2(\alpha) =
\F_2[t]/\idlgen{1 + t + t^2}$.
\end{enumerate}
\end{examples}

\begin{ex}{ex:F3_2}
How many elements does the field $\F_3(\sqrt{2})$ have? What about
$\F_2(\alpha)$, where $\alpha$ is a root of $1 + t + t^2$? 
\end{ex}

\begin{warning}{wg:sub-abs}
Take $K = \Q$ and $m(t) = t^3 - 2$, which is irreducible. Write $\alpha_1,
\alpha_2, \alpha_3$ for the roots of $m$ in $\C$. Then $\Q(\alpha_1)$,
$\Q(\alpha_2)$ and $\Q(\alpha_3)$ are all different \emph{as subsets of
\,$\C$}. For example, one of the $\alpha_i$ is the real cube root of $2$
(say $\alpha_1$), which implies that $\Q(\alpha_1) \sub \R$, whereas the
other two are not real, so $\Q(\alpha_i) \not\subseteq \R$ for $i \neq
1$. However, $\Q(\alpha_1): \Q$, \, $\Q(\alpha_2): \Q$ and $\Q(\alpha_3): \Q$
are all isomorphic \emph{as abstract field extensions} of $\Q$. This
follows from Theorem~\ref{thm:class-simp}, since all the $\alpha_i$ have
the same minimal polynomial, $m$.

You're already very familiar with this kind of situation in other branches
of algebra. For instance, in linear algebra, take three vectors $\vc{v}_1,
\vc{v}_2, \vc{v}_3$ in $\R^2$, none a scalar multiple of any other. Then
$\spn(\vc{v}_1)$, $\spn(\vc{v}_2)$ and $\spn(\vc{v}_3)$ are all different
\emph{as subsets of \,$\R^2$}, but they are all isomorphic \emph{as
abstract vector spaces} (since they're all $1$-dimensional). A similar
example could be given with a group containing several subgroups that are
all isomorphic.
\end{warning}

You've seen that Galois theory involves aspects of group theory and ring
theory. In the next chapter, you'll see how linear algebra enters the
picture too. 

\chapter{Degree}
\label{ch:deg}

We've seen that if you adjoin to $\Q$ a \emph{square} root of $2$,
then each element of the resulting field can be specified using \emph{two}
rational numbers, $a$ and $b$:
\[
\Q(\sqrt{2}) = \bigl\{ a + b\sqrt{2} \such a, b \in \Q\bigr\}.
\]
I also mentioned that if you adjoin to $\Q$ a \emph{cube} root of $2$, then
it takes \emph{three} rational numbers to specify each element of the
resulting field:
\video{Introduction to Week~5}
\[
\Q(\sqrt[3]{2}) = 
\Bigl\{ a + b\sqrt[3]{2} + c\sqrt[3]{2}^2 \such a, b, c \in \Q \Bigr\}
\]
(Warning~\ref{wg:adj-not-quad}). 
This might lead us to suspect that
$\Q(\sqrt[3]{2}): \Q$ is in some sense a `bigger' extension than
$\Q(\sqrt{2}): \Q$. 

The first thing we'll do in this chapter is to make this intuition
rigorous. We'll define the `degree' of an extension and see that
$\Q(\sqrt{2}): \Q$ and $\Q(\sqrt[3]{2}): \Q$ have degrees $2$ and $3$,
respectively. 

The concept of degree is incredibly useful, and not only in Galois
theory. In fact, I'll show you how it can be used to solve three problems
that remained unsolved for literally millennia, since the time of
the ancient Greeks.

\section{The degree of an extension}
\label{sec:degree}

Let $M: K$ be a field extension. Then $M$ is a vector space over $K$ in a
natural way. Addition and subtraction in the vector space $M$ are the same
as in the field $M$. Scalar multiplication in the vector space is just
multiplication of elements of $M$ by elements of $K$, which makes sense
because $K$ is embedded as a subfield of $M$.

This little observation is amazingly useful. It is an excellent
illustration of a powerful mathematical technique: forgetting. When we view
$M$ as a vector space over $K$ rather than a field extension of $K$, we are
forgetting how to multiply together elements of $M$ that aren't in $K$.

\begin{defn}
\label{defn:deg-ext}
The \demph{degree $[M: K]$} of a field extension $M: K$ is the
dimension of $M$ as a vector space over $K$.
\end{defn}

If $M$ is a \emph{finite-dimensional} vector space over $K$, it's clear
what this means. If $M$ is infinite-dimensional over $K$, we write $[M: K]
= \infty$, where \demph{$\infty$} is a formal symbol which we give the
properties
\[
n < \infty, 
\qquad
n \cdot \infty = \infty \quad (n \geq 1),
\qquad
\infty \cdot \infty = \infty
\]
for integers $n$. An extension $M: K$ is
\demph{finite}\label{p:fin-deg} if $[M: K] < \infty$.

\begin{digression}{dig:inf-dim}
You know that whenever $V$ is a finite-dimensional vector space, (i)~there
exists a basis of $V$, and (ii)~there is a bijection between any two
bases. This makes it possible to define the dimension of a vector space as
the number of elements in a basis. In fact, both~(i) and~(ii) are true for
\emph{every} vector space, not just the finite-dimensional ones. So we can
define the dimension of an arbitrary vector space as the `number' of
elements in a basis, where now `number' means cardinal, i.e.\ isomorphism
class of sets. 

We could interpret Definition~\ref{defn:deg-ext} using this general
definition of dimension. For instance, suppose we had one field extension
$M: K$ such that $M$ had a countably infinite basis over $K$, and another,
$M': K$, such that $M'$ had an uncountably infinite basis over
$K$. Then $[M: K]$ and $[M': K]$ would be different. 

However, we'll lump all the infinite-dimensional extensions together
and say that their degrees are all $\infty$. We'll mostly be dealing with
finite extensions anyway, and won't need to distinguish between sizes of
$\infty$. It's a bit like the difference between a house that costs a
million pounds and a house that costs ten million: although the
difference is vast, most of us would lump them together in a
single category called `unaffordable'.
\end{digression}

\begin{examples}
\label{egs:deg-ext}
\begin{enumerate}
\item 
\label{eg:de-triv}
Every field $M$ contains at least one nonzero element, namely, $1$. So $[M:
K] \geq 1$ for every field extension $M: K$. 

If $M = K$ then $\{1\}$ is a
basis, so $[M: K] = 1$. On the other hand, if $[M: K] = 1$ then the
one-element linearly independent set $\{1\}$ must be a basis, which implies
that every element of $M$ is equal to $a \cdot 1 = a$ for some $a \in K$,
and so $M = K$. Hence
\[
[M: K] = 1 \iff M = K.
\]

\item
Every element of $\C$ is equal to $x + yi$ for a unique pair $(x, y)$ of
elements of $\R$. That is, $\{1, i\}$ is a basis of $\C$ over $\R$. Hence
$[\C: \R] = 2$.

\item
\label{eg:de-rat}
Let $K$ be a field and $K(t)$ the field of rational expressions over
$K$. Then $1, t, t^2, \ldots$ are linearly independent over $K$, so $[K(t):
K] = \infty$.
\end{enumerate}
\end{examples}

\begin{warning}{wg:deg-not-0}
The degree $[K: K]$ of $K$ over itself is $1$, not $0$.  Degrees of
extensions are never $0$.  See
Example~\ref{egs:deg-ext}\bref{eg:de-triv}.\\
\end{warning}

\begin{bigthm}
\label{thm:simple-basis}
Let $K(\alpha): K$ be a simple extension.
\begin{enumerate}
\item
\label{part:sba-alg}
Suppose that $\alpha$ is algebraic over
$K$. Write $m \in K[t]$ for the minimal polynomial of $\alpha$ and $n =
\deg(m)$. Then
\[
1, \alpha, \ldots, \alpha^{n - 1}
\]
is a basis of $K(\alpha)$ over $K$. In particular, $[K(\alpha): K] =
\deg(m)$.

\item
\label{part:sba-trans}
Suppose that $\alpha$ is transcendental over $K$. Then $1, \alpha,
\alpha^2, \ldots$ are linearly independent over $K$. In particular,
$[K(\alpha): K] = \infty$.
\end{enumerate}
\end{bigthm}

\begin{proof}
For~\bref{part:sba-alg}, to show that $1, \alpha, \ldots, \alpha^{n - 1}$
is a basis of $K(\alpha)$ over $K$, we will show that every element of
$K(\alpha)$ can be expressed as a $K$-linear combination of $1, \alpha,
\ldots, \alpha^{n - 1}$ in a unique way.

By Lemma~\ref{lemma:simple-build}\bref{part:sb-alg} and
Theorem~\ref{thm:class-simp}\bref{part:cs-alg}, we might as well take
$K(\alpha) = K[t]/\idlgen{m}$ and $\alpha = \pi(t)$, where $\pi\from K[t]
\to K[t]/\idlgen{m}$ is the canonical homomorphism.

Since $\pi$ is surjective, every element of $K(\alpha)$ is equal to
$\pi(f)$ for some $f \in K[t]$. By Proposition~\ref{propn:div-alg}, there
are unique $q, r \in K[t]$ such that $f = qm + r$ and $\deg(r) < n$.  In
particular, there is a unique polynomial $r \in K[t]$ such that $f - r \in
\idlgen{m}$ and $\deg(r) < n$. Equivalently, there are unique $a_0, \ldots,
a_{n - 1} \in K$ such that
\[
f(t) - \bigl(a_0 + a_1 t + \cdots + a_{n - 1}t^{n - 1}\bigr) 
\in \idlgen{m}.
\]
Equivalently, there are unique $a_0, \ldots, a_{n - 1} \in K$ such that
\[
\pi(f) = \pi\bigl(a_0 + a_1 t + \cdots + a_{n - 1} t^{n - 1}\bigr).
\]
Equivalently (since $\pi(t) = \alpha$), there are unique $a_0, \ldots, a_{n
- 1} \in K$ such that 
\[
\pi(f) = a_0 + a_1 \alpha + \cdots a_{n - 1} \alpha^{n - 1},
\]
as required.

For~\bref{part:sba-trans}, 
Theorem~\ref{thm:class-simp}\bref{part:cs-trans} implies that
$K(\alpha)$ is isomorphic over $K$ to the
field $K(t)$ of rational expressions. The result now follows from 
Example~\ref{egs:deg-ext}\bref{eg:de-rat}.
\end{proof}

\begin{examples}
\label{egs:alg-basis}
\begin{enumerate}
\item
Let $\alpha \in \C$ be an algebraic number over $\Q$ whose minimal
polynomial is quadratic. Then by
Theorem~\ref{thm:simple-basis}\bref{part:sba-alg},
\[
\Q(\alpha) = \{ a + b\alpha \such a, b \in \Q\}.
\]
We've already seen this in many examples, such as $\alpha = \sqrt{2}$ and
$\alpha = i$.

\item
\label{eg:ab-cyclo}
Let $p$ be a prime. We saw in
Example~\ref{egs:min-poly}\bref{eg:min-poly-cyclo} that $e^{2\pi i/p}$ has
minimal polynomial $1 + t + \cdots + t^{p - 1}$. This has degree $p - 1$,
so $[\Q(e^{2\pi i/p}): \Q] = p - 1$.
\end{enumerate}
\end{examples}

\begin{warning}{wg:deg-root-unity}
For a prime $p$, the degree of $e^{2\pi i/p}$ over $\Q$ is $p - 1$, not
$p$! \\
\end{warning}

\begin{example}
\label{eg:ff-4}
Apart from finite fields of the form $\F_p$, the simplest finite field is
$\F_2(\alpha)$, where $\alpha$ is a root of the irreducible polynomial $1 +
t + t^2$ over $\F_2$ (Example~\ref{egs:adjoin}\bref{eg:adjoin-F2}). By
Theorem~\ref{thm:simple-basis}\bref{part:sba-alg},
\[
\F_2(\alpha)
=
\{ a + b\alpha \such a, b \in \F_2\}
=
\{ 0, 1, \alpha, 1 + \alpha\}.
\]
Since $1 + \alpha + \alpha^2 = 0$ and $\F_2(\alpha)$ has characteristic $2$, 
\[
\alpha^2 = 1 + \alpha,
\qquad
(1 + \alpha)^2 = \alpha.
\]
So the Frobenius automorphism of $\F_2(\alpha)$ interchanges $\alpha$ and
$1 + \alpha$. Like all automorphisms, it fixes $0$ and $1$.
\end{example}

\begin{ex}{ex:F4-mult}
Write out the addition and multiplication tables of $\F_2(\alpha)$. 
\end{ex}

Theorem~\ref{thm:simple-basis}\bref{part:sba-alg} implies that when $\alpha
\in M$ is algebraic over $K$, with minimal polynomial of degree $n$, the
subset $\{a_0 + a_1 \alpha + \cdots + a_{n - 1} \alpha^{n - 1} \such a_i
\in K\}$ is a subfield of $M$. This isn't particularly obvious: for
instance, why is it closed under taking reciprocals? But it's true.

For a field extension $M: K$ and $\alpha \in M$, the \demph{degree} of
$\alpha$ over $K$ is $[K(\alpha): K]$. We write it as
\demph{$\deg_K(\alpha)$}. 
Theorem~\ref{thm:simple-basis} immediately implies:

\begin{cor}
\label{cor:fin-deg-alg}
Let $M: K$ be a field extension and $\alpha \in M$. Then
\[
\deg_K(\alpha) < \infty
\iff
\alpha \text{ is algebraic over } K.
\]
\qed
\end{cor}

If $\alpha$ is algebraic over $K$ then by
Theorem~\ref{thm:simple-basis}\bref{part:sba-alg}, the degree of $\alpha$
over $K$ is the degree of the minimal polynomial of $\alpha$ over $K$.

\begin{example}
\label{eg:sb-cbrt2}
Let $\xi$ be the real cube root of $2$. By
Example~\ref{egs:min-poly}\bref{eg:min-poly-cbrt2}, the minimal polynomial
of $\xi$ over $\Q$ is $t^3 - 2$, so $\deg_\Q(\xi) = 3$. It follows
that $\Q(\xi) \neq \{a + b\xi \such a, b \in \Q\}$, since otherwise 
the two-element set $\{1, \xi\}$ would span the
three-dimensional vector space $\Q(\xi)$. So we have another
proof that $2^{2/3}$ cannot be written as a $\Q$-linear
combination of $1$ and $2^{1/3}$. As observed in
Example~\ref{egs:min-poly}\bref{eg:min-poly-cbrt2}, this is messy to prove
directly. 
\end{example}

Theorem~\ref{thm:simple-basis} is powerful. Here are two more
of its corollaries.

\begin{cor}
\label{cor:deg-ineqs}
Let $M: L: K$ be field extensions and $\beta \in M$. Then $[L(\beta): L]
\leq [K(\beta): K]$.
\end{cor}

Informally, I think of Corollary~\ref{cor:deg-ineqs} as in
Figure~\ref{fig:mental-deg-ineq}. The degree of $\beta$ over $K$ measures
how far $\beta$ is from being in $K$. Since $L$ contains $K$, it might be
that $\beta$ is closer to $L$ than to $K$ (i.e.\ $[L(\beta): L] <
[K(\beta): K]$), and it's certainly no further away.

\begin{proof}
If $[K(\beta): K] = \infty$ then the inequality
is clear. Otherwise, $\beta$ is algebraic over $K$ (by
Corollary~\ref{cor:fin-deg-alg}), with minimal polynomial $m \in K[t]$,
say. Then $m$ is an annihilating polynomial for $\beta$ over $L$, so the
minimal polynomial of $\beta$ over $L$ has degree $\leq \deg(m)$. The
result follows from Theorem~\ref{thm:simple-basis}\bref{part:sba-alg}.
\end{proof}

\begin{ex}{ex:strict-triv}
Give an example to show that the inequality in
Corollary~\ref{cor:deg-ineqs} can be strict. Your example
can be as trivial as you like.
\end{ex}

\begin{figure}
\setlength{\fboxsep}{0mm}%
\setlength{\unitlength}{1mm}%
\begin{picture}(136,30)
\cell{68}{15}{c}{\includegraphics[height=30\unitlength]{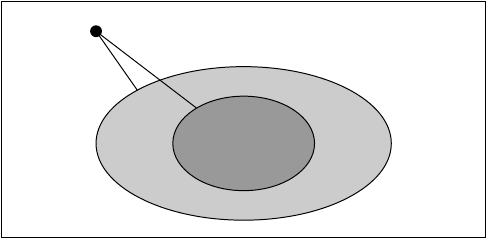}}
\cell{102}{28}{c}{$M$}
\cell{85}{20}{c}{$L$}
\cell{78}{16}{c}{$K$}
\cell{47.5}{27.5}{c}{$\beta$}
\cell{46}{21}{c}{$\scriptstyle{[L(\beta): L]}$}
\cell{60}{24.5}{c}{$\scriptstyle{[K(\beta): K]}$}
\end{picture}%
\caption{Visualization of Corollary~\ref{cor:deg-ineqs}
(not to be taken too seriously).}
\label{fig:mental-deg-ineq}
\end{figure}

\begin{cor}
\label{cor:alg-ext-poly}
Let $M: K$ be a field extension. Let $\alpha_1, \ldots, \alpha_n \in
M$, with $\alpha_i$ algebraic over $K$ of degree
$d_i$. Then every element $\alpha \in K(\alpha_1, \ldots, \alpha_n)$ can be
expressed as a polynomial in $\alpha_1, \ldots, \alpha_n$ over $K$. More
exactly, 
\[
\alpha
=
\sum_{r_1, \ldots, r_n} 
c_{r_1, \ldots, r_n} \alpha_1^{r_1} \cdots \alpha_n^{r_n}
\]
for some $c_{r_1, \ldots, r_n} \in K$, where $r_i$ ranges over $0,
\ldots, d_i - 1$. 
\end{cor}

For example, here's what this says in the case $n = 2$. Let $M: K$ be a
field extension, and take algebraic elements $\alpha_1, \alpha_2$ of
$M$. Write $d_1$ and $d_2$ for their degrees over $K$.  Then every element
of $K(\alpha_1, \alpha_2)$ is equal to
\[
\sum_{r = 0}^{d_1 - 1} \sum_{s = 0}^{d_2 - 1} c_{rs} \alpha_1^r \alpha_2^s
\]
for some coefficients $c_{rs} \in K$. A fundamental point in the proof is
that a polynomial in two variables can be seen as a polynomial in one
variable whose coefficients are themselves polynomials in one variable, and
similarly for more than two variables.

\begin{proof}
When $n = 0$, this is trivial. Now let $n \geq 1$ and suppose inductively
that the result holds for $n - 1$. Let
\[
\alpha 
\in 
K(\alpha_1, \ldots, \alpha_n) 
= 
\bigl(K(\alpha_1, \ldots, \alpha_{n - 1})\bigr)(\alpha_n).
\]
By Theorem~\ref{thm:simple-basis}\bref{part:sba-alg} applied to the
extension $(K(\alpha_1, \ldots, \alpha_{n - 1}))(\alpha_n): K(\alpha_1,
\ldots, \alpha_{n - 1})$, noting that $\deg_{K(\alpha_1, \ldots, \alpha_{n
- 1})}(\alpha_n) \leq \deg_K(\alpha_n) = d_n$, we have
\begin{align}
\label{eq:aep-1}
\alpha = \sum_{r = 0}^{d_n - 1} c_r \alpha_n^r
\end{align}
for some $c_0, \ldots, c_{d_n - 1} \in K(\alpha_1, \ldots, \alpha_{n -
1})$.  By inductive hypothesis, for each $r$ we have
\begin{align}
\label{eq:aep-2}
c_r 
=
\sum_{r_1, \ldots, r_{n - 1}}
c_{r_1, \ldots, r_{n - 1}, r} 
\alpha_1^{r_1} \cdots \alpha_{n - 1}^{r_{n - 1}}
\end{align}
for some $c_{r_1, \ldots, r_{n - 1}, r} \in K$, where $r_i$ ranges over $0,
\ldots, d_i - 1$. Substituting~\eqref{eq:aep-2} into~\eqref{eq:aep-1}
completes the induction.
\end{proof}

\begin{example}
\label{eg:aep-sqrt2i}
Back in Example~\ref{egs:adj}\bref{eg:adj-sqrt2i}, I claimed that 
\[
\Q(\sqrt{2}, i)
=
\{ a + b\sqrt{2} + ci + d\sqrt{2}i
\such a, b, c, d \in \Q\}.
\]
Corollary~\ref{cor:alg-ext-poly} applied to $\Q(\sqrt{2}, i): \Q$
proves this, since $\deg_\Q(\sqrt{2})
= \deg_\Q(i) = 2$. 
\end{example}

\begin{ex}{ex:trans-poly}
Let $M: K$ be a field extension and $\alpha$ a transcendental element of
$M$. Can every element of $K(\alpha)$ be represented as a polynomial in
$\alpha$ over $K$?
\end{ex}

For extensions obtained by adjoining several elements, the following result
is invaluable.

\begin{bigthm}[Tower law]
\label{thm:tower}
Let $M: L: K$ be field extensions. 
\begin{enumerate}
\item
\label{part:tower-basis}
If $(\alpha_i)_{i \in I}$ is a basis of
$L$ over $K$ and $(\beta_j)_{j \in J}$ is a basis of $M$ over $L$, then
$(\alpha_i \beta_j)_{(i, j) \in I \times J}$ is a basis of $M$ over $K$.

\item
\label{part:tower-fin}
$M: K$ is finite $\iff$ $M: L$ and $L: K$ are finite.

\item
\label{part:tower-dim}
$[M: K] = [M: L][L: K]$.
\end{enumerate}
\end{bigthm}%
\marginpar{\hspace*{8mm}%
\raisebox{3mm}{%
\includegraphics[height=39mm]{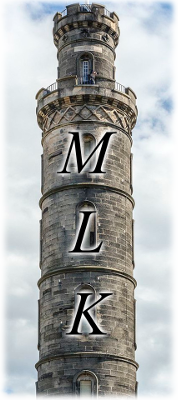}%
}}

The sets $I$ and $J$ here could be infinite. I'll say that a family
$(a_i)_{i \in I}$ of elements of a field is \demph{finitely supported} if
the set $\{ i \in I \such a_i \neq 0\}$ is finite.

\begin{proof}
To prove~\bref{part:tower-basis}, we show that $(\alpha_i \beta_j)_{(i, j) \in I
\times J}$ is a linearly independent spanning set of $M$ over $K$.  

For linear independence, let $(c_{ij})_{(i, j) \in I \times J}$ be a
finitely supported family of elements of $K$ such that $\sum_{i, j} c_{ij}
\alpha_i \beta_j = 0$. Then $\sum_j (\sum_i c_{ij} \alpha_i) \beta_j = 0$,
with $\sum_i c_{ij} \alpha_i \in L$ for each $j \in J$. Since $(\beta_j)_{j
\in J}$ is linearly independent over $L$, we have $\sum_i c_{ij} \alpha_i =
0$ for each $j \in J$. But $(\alpha_i)_{i \in I}$ is linearly
independent over $K$, so $c_{ij} = 0$ for each $i \in I$ and $j \in
J$.

To show $(\alpha_i \beta_j)_{(i, j) \in I \times J}$ spans $M$ over $K$,
let $e \in M$. Since $(\beta_j)_{j \in J}$ spans $M$ over $L$, we have $e =
\sum_j d_j \beta_j$ for some finitely supported family $(d_j)_{j \in J}$ of
elements of $L$. Since $(\alpha_i)_{i \in I}$ spans $L$ over $K$, for each
$j \in J$ we have $d_j = \sum_i c_{ij} \alpha_i$ for some finitely
supported family $(c_{ij})_{i \in I}$ of $K$. Hence $e = \sum_{i, j} c_{ij}
\alpha_i \beta_j$, as required.

Parts~\bref{part:tower-fin} and~\bref{part:tower-dim} follow.
\end{proof}

\begin{example}
\label{eg:deg-23}
What is $[\Q(\sqrt{2}, \sqrt{3}): \Q]$? The tower law gives
\begin{align*}
\bigl[\Q(\sqrt{2}, \sqrt{3}): \Q\bigr]    &
=
\bigl[\Q(\sqrt{2}, \sqrt{3}): \Q(\sqrt{2})\bigr] \bigl[\Q(\sqrt{2}): \Q\bigr]       \\
&
=
2\bigl[\Q(\sqrt{2}, \sqrt{3}): \Q(\sqrt{2})\bigr].
\end{align*}
Now on the one hand,
\[
\bigl[\Q(\sqrt{2}, \sqrt{3}): \Q(\sqrt{2})\bigr] 
\leq 
\bigl[\Q(\sqrt{3}): \Q\bigr] 
= 
2 
\]
by Corollary~\ref{cor:deg-ineqs}. On the other,
$\sqrt{3} \not\in \Q(\sqrt{2})$, so $\Q(\sqrt{2}, \sqrt{3}) \neq
\Q(\sqrt{2})$, so $[\Q(\sqrt{2}, \sqrt{3}): \Q(\sqrt{2})] > 1$ by
Example~\ref{egs:deg-ext}\bref{eg:de-triv}. So $[\Q(\sqrt{2}, \sqrt{3}):
\Q(\sqrt{2})] = 2$, giving the answer: $[\Q(\sqrt{2}, \sqrt{3}): \Q] = 4$.

By the same argument as in Example~\ref{eg:aep-sqrt2i}, $\{1, \sqrt{2},
\sqrt{3}, \sqrt{6}\}$ spans $\Q(\sqrt{2}, \sqrt{3})$ over $\Q$. But we have
just shown that $\Q(\sqrt{2}, \sqrt{3})$ has dimension $4$ over $\Q$. Hence
this spanning set is a basis. That is, for every element $\alpha \in
\Q(\sqrt{2}, \sqrt{3})$, there is one \emph{and only one} 4-tuple $(a, b,
c, d)$ of rational numbers such that
\[
\alpha = a + b\sqrt{2} + c\sqrt{3} + d\sqrt{6}.
\]
\end{example}

\begin{cor}
\label{cor:tower-div}
Let $M: L': L: K$ be field extensions. If $M: K$ is finite then $[L': L]$
divides $[M: K]$.
\end{cor}

\begin{proof}
By the tower law twice, $[M: K] = [M: L'][L': L][L: K]$.
\end{proof}

That result might remind you of Lagrange's theorem on group orders. The
resemblance is no coincidence, as we'll see.

\begin{ex}{ex:prime-simple}
Show that a field extension whose degree is a prime number must be simple. 
\end{ex}

\emph{That} result might remind you of the fact that a group of prime order
must be cyclic, and that's no coincidence either! 

A second corollary of the tower law:

\begin{cor}
\label{cor:tower-prod}
Let $M: K$ be a field extension and $\alpha_1, \ldots, \alpha_n \in
M$. Then
\[
[K(\alpha_1, \ldots, \alpha_n) : K]
\leq
[K(\alpha_1): K] \cdots [K(\alpha_n): K].
\]
\end{cor}

\begin{proof}
By the tower law and then Corollary~\ref{cor:deg-ineqs},
\begin{align*}
&
[K(\alpha_1, \ldots, \alpha_n) : K]     \\
&\qquad
=
[K(\alpha_1, \ldots, \alpha_n) : K(\alpha_1, \ldots, \alpha_{n - 1})]
\cdots
[K(\alpha_1, \alpha_2): K(\alpha_1)] [K(\alpha_1): K]   \\
&
\qquad
\leq
[K(\alpha_n): K] \cdots [K(\alpha_2): K] [K(\alpha_1): K].
\end{align*}
\end{proof}

\begin{example}
\label{eg:tower-126}
What is $[\Q(12^{1/4}, 6^{1/15}): \Q]$? You can check (hint, hint) that
$\deg_\Q(12^{1/4}) = 4$ and $\deg_\Q(6^{1/15}) = 15$. So by
Corollary~\ref{cor:tower-div}, $[\Q(12^{1/4}, 6^{1/15}): \Q]$ is divisible
by $4$ and $15$. But also, Corollary~\ref{cor:tower-prod} implies that
$[\Q(12^{1/4}, 6^{1/15}): \Q] \leq 4 \times 15 = 60$. Since $4$ and $15$
are coprime, the answer is $60$.
\end{example}

\begin{ex}{ex:126-gen}
Generalize Example~\ref{eg:tower-126}. In other words, what general result
does the argument of Example~\ref{eg:tower-126} prove, not involving the
particular numbers chosen there?
\end{ex}

\section{Algebraic extensions}

We defined a field extension $M: K$ to be finite if $[M: K] < \infty$, that
is, $M$ is finite-dimensional as a vector space over $K$. Here are two
related conditions.

\begin{defn}
A field extension $M: K$ is \demph{finitely generated} if $M = K(Y)$ for
some finite subset $Y \sub M$.
\end{defn}

\begin{defn}
A field extension $M: K$ is \demph{algebraic} if every element of $M$ is
algebraic over $K$.
\end{defn}

Recall from Corollary~\ref{cor:fin-deg-alg} that $\alpha$ is algebraic over
$K$ if and only if $K(\alpha): K$ is finite. So for a field extension to be
algebraic is also a kind of finiteness condition.

\begin{examples}
\begin{enumerate}
\item 
For any field $K$, the extension $K(t): K$ is finitely
generated (take the `$Y$' above to be $\{t\}$) but not finite, by
Corollary~\ref{cor:fin-deg-alg}.

\item
In Section~\ref{sec:alg-trans} you met the set $\ovln{\Q}$ of complex
numbers algebraic over $\Q$. We'll very soon prove that it's a subfield of
$\C$. It is algebraic over $\Q$, by definition. But you'll show in
\href{https://www.maths.ed.ac.uk/~tl/galois}{Workshop~3, question~13} that
it is not finite over $\Q$. 
\end{enumerate}
\end{examples}

Our three finiteness conditions are related as follows
(Figure~\ref{fig:fin-condns}). 

\begin{figure}
\setlength{\fboxsep}{0mm}%
\setlength{\unitlength}{1mm}%
\begin{picture}(136,30)
\cell{68}{15}{c}{\includegraphics[height=30\unitlength]{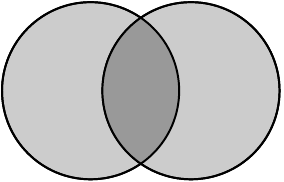}}
\cell{37}{20}{c}{algebraic}
\cell{92}{21.5}{l}{finitely}
\cell{92}{18}{l}{generated}
\cell{68}{15}{c}{finite}
\end{picture}%
\caption{Finiteness conditions on a field extension}
\label{fig:fin-condns}
\end{figure}

\begin{propn}
\label{propn:three-fin}
The following conditions on a field extension $M: K$ are equivalent:
\begin{enumerate}
\item 
\label{part:tf-fin}
$M: K$ is finite;

\item
\label{part:tf-fga}
$M: K$ is finitely generated and algebraic;

\item
\label{part:tf-seq}
$M = K(\alpha_1, \ldots, \alpha_n)$ for some finite set $\{\alpha_1,
\ldots, \alpha_n\}$ of elements of $M$ algebraic over $K$.
\end{enumerate}
\end{propn}

\begin{proof}
\bref{part:tf-fin}$\implies$\bref{part:tf-fga}: suppose that $M: K$ is
finite.

To show that $M: K$ is finitely generated, take a basis $\alpha_1, \ldots,
\alpha_n$ of $M$ over $K$. Every subfield $L$ of $M$ containing $K$ is a
$K$-linear subspace of $M$, so if $\alpha_1, \ldots, \alpha_n \in L$ then
$L = M$. This proves that the only subfield of $M$ containing $K \cup
\{\alpha_1, \ldots, \alpha_n\}$ is $M$ itself; that is, $M = K(\alpha_1,
\ldots, \alpha_n)$. So $M: K$ is finitely generated.

To show that $M: K$ is algebraic, let $\alpha \in M$. Then by
part~\bref{part:tower-fin} of the tower law
(Theorem~\ref{thm:tower}), $K(\alpha): K$ is finite, so by
Corollary~\ref{cor:fin-deg-alg}, $\alpha$ is algebraic over $K$.

\bref{part:tf-fga}$\implies$\bref{part:tf-seq} is immediate from the
definitions. 

\bref{part:tf-seq}$\implies$\bref{part:tf-fin}: suppose that $M =
K(\alpha_1, \ldots, \alpha_n)$ for some $\alpha_i \in M$ algebraic over
$K$. Then 
\[
[M: K] \leq [K(\alpha_1): K] \cdots [K(\alpha_n): K]
\]
by Corollary~\ref{cor:tower-prod}. For each $i$, we have $[K(\alpha_i): K] < \infty$  since $\alpha_i$ is algebraic over $K$ (using
Corollary~\ref{cor:fin-deg-alg} again). So $[M: K] < \infty$.
\end{proof}

We already saw that when $M = K(\alpha_1, \ldots, \alpha_n)$ with each
$\alpha_i$ algebraic, every element of $M$ is a polynomial in $\alpha_1,
\ldots, \alpha_n$ (Corollary~\ref{cor:alg-ext-poly}). So for any finite
extension $M: K$, there is some finite set of elements such that everything
in $M$ can be expressed as a polynomial over $K$ in these elements.

\begin{ex}{ex:subfd-subsp}
Let $M: K$ be a field extension and $K \sub L \sub M$.  In the proof of
Proposition~\ref{propn:three-fin}, I said that \emph{if} $L$ is a subfield
of $M$ \emph{then} $L$ is a $K$-linear subspace of $M$. Why is that true?
And is the converse also true? Give a proof or a counterexample.
\end{ex}

\begin{cor}
\label{cor:simple-alg}
Let $K(\alpha): K$ be a simple extension. The following are equivalent:
\begin{enumerate}
\item 
\label{part:sa-fin}
$K(\alpha): K$ is finite;

\item
\label{part:sa-alg-ext}
$K(\alpha): K$ is algebraic;

\item
\label{part:sa-alg-elt}
$\alpha$ is algebraic over $K$.
\end{enumerate}
\end{cor}

\begin{proof}
\bref{part:sa-fin}$\implies$\bref{part:sa-alg-ext} follows from
\bref{part:tf-fin}$\implies$\bref{part:tf-fga} of
Proposition~\ref{propn:three-fin}.

\bref{part:sa-alg-ext}$\implies$\bref{part:sa-alg-elt} is immediate from
the definitions.

\bref{part:sa-alg-elt}$\implies$\bref{part:sa-fin} follows from 
\bref{part:tf-seq}$\implies$\bref{part:tf-fin} of
Proposition~\ref{propn:three-fin}.
\end{proof}

Here's a spectacular application of Corollary~\ref{cor:simple-alg}.

\begin{propn}
\label{propn:Qbar-subfd}
$\ovln{\Q}$ is a subfield of $\C$.
\end{propn}

\begin{proof}
By Corollary~\ref{cor:simple-alg},
\[
\ovln{\Q} = \{ \alpha \in \C \such [\Q(\alpha): \Q] < \infty\}.
\]
For all $\alpha, \beta \in \ovln{\Q}$, 
\[
[\Q(\alpha, \beta): \Q] 
\leq 
[\Q(\alpha): \Q] [\Q(\beta): \Q]
< 
\infty
\]
by Corollary~\ref{cor:tower-prod}. Now $\alpha + \beta \in \Q(\alpha,
\beta)$, so $\Q(\alpha + \beta) \sub \Q(\alpha, \beta)$, so
\[
[\Q(\alpha + \beta): \Q]
\leq
[\Q(\alpha, \beta): \Q]
< 
\infty,
\]
giving $\alpha + \beta \in \ovln{\Q}$. Similarly, $\alpha \cdot \beta \in
\ovln{\Q}$. For all $\alpha \in \ovln{\Q}$,
\[
[\Q(-\alpha): \Q]
=
[\Q(\alpha): \Q]
< 
\infty,
\]
giving $-\alpha \in \ovln{\Q}$. Similarly, $1/\alpha \in \ovln{\Q}$ (if
$\alpha \neq 0$). And clearly $0, 1 \in \ovln{\Q}$.
\end{proof}

If you did Exercise~\ref{ex:alg-fd-doomed}, you'll appreciate how hard that
result is to prove from first principles, and how amazing it is that the
proof above is so clean and simple.

\begin{ex}{ex:alg-subfd}
Let $M: K$ be a field extension, and write $L$ for the set of elements of
$M$ algebraic over $K$. By imitating the proof of 
Proposition~\ref{propn:Qbar-subfd}, prove that $L$ is a subfield of $M$.
\end{ex}

\section{Ruler and compass constructions}
\label{sec:ruler-compass}

This section is a truly wonderful application of the algebra we've
developed so far. Using it, we will solve problems that lay unsolved for
thousands of years.

The arguments here have a lot in common with those we'll use in
Chapter~\ref{ch:sol} for the problem of solving polynomials by
radicals. It's well worth getting used to these arguments now, since the
polynomial problem involves some extra subtleties that will need your
full attention then. In other words, treat this as a warm-up.

The ancient Greeks developed planar geometry to an extraordinary degree,
discovering how to perform a very wide range of constructions using only
ruler and compasses. But there were three particular constructions that
they couldn't figure out how to do using only these instruments:
\begin{itemize}
\item \textbf{Trisect the angle:} given an angle $\theta$, construct the
angle $\theta/3$. 

\item \textbf{Duplicate the cube:} given a length, construct a new length
whose cube is twice the cube of the original. That is, given two points
distance $L$ apart, construct two points distance $\sqrt[3]{2}L$ apart.

\item \textbf{Square the circle:} given a circle, construct a square with
the same area. That is, given two points distance $L$ apart, construct two
points distance $\sqrt{\pi}L$ apart.
\end{itemize}
The challenge of finding constructions lay unanswered for millennia. And it
wasn't for lack of attention. My Galois
theory lecture notes from when I was an undergraduate contain the following
words:
\begin{quote}
Thomas Hobbes claimed to have solved these. John Wallis
disagreed. A 17th century pamphlet war ensued.
\end{quote}
Twitter users may conclude that human nature has not changed.

It turns out that the reason why no one could find a way to do these
constructions is that they're impossible. We'll prove it using field
theory.

In order to prove that you \emph{can't} do these things using ruler and
compasses, it's necessary to know that you \emph{can} do certain other
things using ruler and compasses. 
\video{Ruler and compass constructions}
I'll take some simple constructions for granted (but there's a video if you
want the details). 

\begin{digression}{dig:rc}
The standard phrase is `ruler and compass constructions', but it's slightly
misleading. A ruler has distance markings on it, whereas for the problems
of ancient Greece, you're supposed to use only a `straight edge': a ruler
without markings (and no, you're not allowed to mark it). As Stewart
explains (Section~7.1), with a marked or markable straight edge, you
\emph{can} solve all three problems. Also, for what it's worth, an
instrument for drawing circles is strictly speaking a \emph{pair} of
compasses.
But like everyone else, we'll say `ruler and compass'---
\[
\includegraphics[height=25mm]{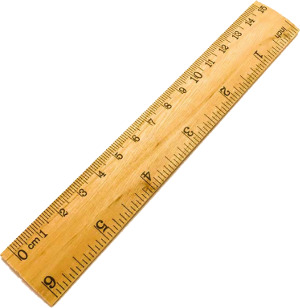}
\qquad
\includegraphics[height=25mm]{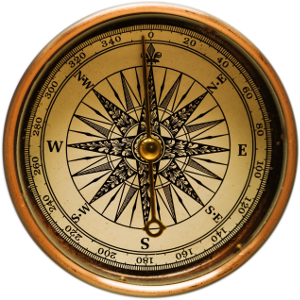}
\]
---when we really mean `straight edge and compasses'---
\[
\includegraphics[height=20mm]{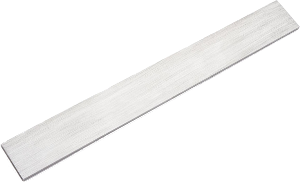}
\qquad
\includegraphics[height=20mm]{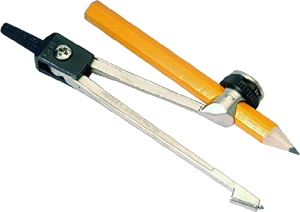}
\]
\end{digression}

The problems as stated above are maybe not quite precise; let's 
formalize them.

Starting from a subset $\Sigma$ of the plane, our instruments allow the
following constructions:
\begin{itemize}
\item 
given two distinct points $A, B$ of $\Sigma$, draw the (infinite) line
through $A$ and $B$;

\item
given two distinct points $A, B$ of $\Sigma$, draw the circle with centre
$A$ passing through $B$.
\end{itemize}
A point in the plane is \demph{immediately constructible} from $\Sigma$
if it is a point of intersection between two distinct lines, or two
distinct circles, or a line and a circle, of the form above. A point $C$ in
the plane is \demph{constructible} from $\Sigma$ if there is a finite
sequence $C_1, \ldots, C_n = C$ of points such that $C_i$ is immediately
constructible from $\Sigma \cup \{C_1, \ldots, C_{i - 1}\}$ for each
$i$. Broadly, the question is: which points are constructible from which?

The key idea of the solution is that when you're doing Greek-style
geometry, then in terms of coordinates, you're repeatedly solving linear or
quadratic equations. (The Greeks didn't use coordinates, but we
will.) This is because equations of lines and circles are linear or
quadratic.

For instance, suppose we start with the points $(0, 0)$ and $(1, 0)$. Draw
the circle with centre $(0, 0)$ passing through $(1, 0)$, and vice
versa. The intersection points of the two circles are $(1/2, \pm
\sqrt{3}/2)$, where the square root came from solving a quadratic. If we
do further ruler and compass constructions, we might end up with
coordinates like $\sqrt{\sqrt{2} + \sqrt{3}}$.  But we can never get
coordinates like $\sqrt[3]{2}$, because where would a cube root come from?
We're only solving quadratics here.

To translate this idea into field terms, we make the following
definition. Take a subfield $K \sub \R$. We say that the extension $K: \Q$
is \demph{iterated quadratic} if there is some finite sequence
of subfields
\[
\Q = K_0 \sub K_1 \sub \cdots \sub K_n = K
\]
such that $[K_i: K_{i - 1}] = 2$ for all $i \in \{1, \ldots, n\}$. 

\begin{example}
$\Q\Bigl(\sqrt{\sqrt{2} + \sqrt{3}}\Bigr)$ is an iterated
quadratic extension of $\Q$, because we have a chain of subfields
\[
\Q 
\sub \Q\bigl(\sqrt{2}\bigr)
\sub \Q\bigl(\sqrt{2}, \sqrt{3}\bigr) 
= \Q\bigl(\sqrt{2} + \sqrt{3}\bigr)
\sub \Q\biggl(\sqrt{\sqrt{2} + \sqrt{3}}\biggr)
\]
where each has degree $2$ over the last. (For the equality, see
Exercise~\ref{ex:sqrt2plus3}.) 
\end{example}

There is an iterated quadratic extension of $\Q$ containing $\sqrt{\sqrt{2}
+ \sqrt{3}}$, and by the same argument, there is one containing
$\sqrt{\sqrt{5} + \sqrt{7}}$. Is there one containing both? We will prove a
general result guaranteeing that there is. The following terminology will
be useful.

\begin{defn}
\label{defn:compm}
Let $L$ and $L'$ be subfields of a field $M$. The \demph{compositum $LL'$} of
$L$ and $L'$ is the subfield of $M$ generated by $L \cup L'$.
\end{defn}

That is, $LL'$ is the smallest subfield of $M$ containing both $L$ and
$L'$. In the notation of Definition~\ref{defn:adj}, we could also write $LL'$
as either $L(L')$ or $L'(L)$.

\begin{example}
The compositum of the subfields $\Q(\sqrt{2})$ and $\Q(\sqrt{3})$ of $\R$
is $\Q(\sqrt{2}, \sqrt{3})$. 
\end{example}

\begin{warning}{wg:compm}
Despite the notation,
\[
LL'
\ {\color{Red!100} \neq}\  
\{ \alpha\alpha' \such \alpha \in L, \alpha' \in L'\}.
\]
But it is true that $LL'$ is the subfield of $M$ \emph{generated} by the
right-hand side. (Why?)
\end{warning}

To show that any two iterated quadratic extensions of $\Q$ can be merged
into one, we first consider extensions of degree~$2$. 

\begin{lemma}
\label{lemma:2-join}
Let $M: K$ be a field extension and let $L, L'$ be subfields of $M$
containing $K$. 
\marginpar{\quad
$\xymatrix@C=3mm@R=3mm{%
        &
\scriptstyle
M \ar@{-}[d]   &       \\
        &
\scriptstyle
LL' \ar@{-}[ld] \ar@{-}[rd]    &       \\
\scriptstyle
L \ar@{-}[rd]   &       &
\scriptstyle
L' \ar@{-}[ld] \\
        &
\scriptstyle
K      &       
}$%
}%
If $[L: K] = 2$ then $[LL': L'] \in \{1, 2\}$.
\end{lemma}

Actually, it is true more generally that $[LL': L'] \leq [L: K]$
(\href{https://www.maths.ed.ac.uk/~tl/galois}{Workshop~3, question~18}),
but we will not need this fact.

\begin{proof}
Choose some $\beta \in L \without K$. By applying the tower law to $L:
K(\beta): K$ and using the hypothesis that $[L: K] = 2$, we see that
$K(\beta) = L$. 

Next we show that $LL' = L'(\beta)$. Certainly $L'(\beta) \sub LL'$, since
$L' \sub LL'$ and $\beta \in L \sub LL'$. Conversely, $L'(\beta)$ is a
subfield of $M$ that contains both $K(\beta) = L$ and $L'$, so it contains
$LL'$. Hence $LL' = L'(\beta)$, as claimed.

It follows that $[LL': L'] = [L'(\beta): L'] \leq [K(\beta): K] = 2$, where
the inequality comes from Corollary~\ref{cor:deg-ineqs}. 
\end{proof}

\begin{ex}{ex:diamond-2}
Find an example of Lemma~\ref{lemma:2-join} where $[LL': L'] = 2$, and
another where $[LL': L'] = 1$.
\end{ex}

\begin{lemma}
\label{lemma:iq-join}
Let $K$ and $L$ be subfields of $\R$ such that the extensions $K: \Q$ and
$L: \Q$ are iterated quadratic. Then there is some subfield $M$ of $\R$
such that the extension $M: \Q$ is iterated quadratic and $K, L \sub M$.
\end{lemma}

\begin{proof}
Take subfields 
\[
\Q = K_0 \sub K_1 \sub \cdots \sub K_n = K \sub \R,
\qquad
\Q = L_0 \sub L_1 \sub \cdots \sub L_m = L \sub \R
\]
with $[K_i: K_{i - 1}] = 2 = [L_j:
L_{j - 1}]$ for all $i, j$. Consider the chain of subfields
\begin{align}
\label{eq:ij-chain}
\Q = K_0 \sub K_1 \sub \cdots \sub K_n = K
= K L_0 \sub K L_1 \sub \cdots \sub K L_m = KL
\end{align}
of $\R$. It is enough to show that $KL$ is an iterated quadratic extension
of $K$.

In the chain~\eqref{eq:ij-chain}, $[K_i: K_{i - 1}] = 2$ for all
$i$. Moreover, for each $j$ we have $[L_j: L_{j - 1}] = 2$, so 
Lemma~\ref{lemma:2-join} implies that $[KL_j: KL_{j - 1}] \in \{1, 2\}$
(taking the `$K$' of that lemma to be $L_{j - 1}$). Hence
in~\eqref{eq:ij-chain}, all the successive degrees are $1$ or $2$. An
extension of degree $1$ is an equality, so by ignoring repeats, we see that
$KL: \Q$ is an iterated quadratic extension.
\end{proof}

The general theory of ruler and compass constructibility starts with any
set $\Sigma \sub \R^2$ of given points. But for simplicity, we will stick to
the case where $\Sigma$ consists of just two points, and we'll choose our
coordinate axes so that they have coordinates $(0, 0)$ and $(1, 0)$. This
will still enable us to solve the notorious problems of ancient Greece.

\begin{propn}
\label{propn:iq}
Let $(x, y) \in \R^2$. If $(x, y)$ is constructible from $\{(0, 0), (1,
0)\}$ then there is an iterated quadratic extension of $\Q$ containing
$x$ and $y$.
\end{propn}

\begin{proof}
Suppose that $(x, y)$ is constructible from $\{(0, 0), (1, 0)\}$ in $n$
steps. If $n = 0$ then $(x, y)$ is $(0, 0)$ or $(1, 0)$, so $x, y \in \Q$,
and $\Q$ is trivially an iterated quadratic extension of $\Q$.

Now let $n \geq 1$. Suppose inductively that each coordinate of each
point constructible from $\{(0, 0), (1, 0)\}$ in $< n$ steps lies
in some iterated quadratic extension of $\Q$. By definition, $(x, y)$ is an
intersection point of two distinct lines/circles through points
constructible in $< n$ steps. By inductive hypothesis, each coordinate of
each of those points lies in some iterated quadratic extension of $\Q$, so
by Lemma~\ref{lemma:iq-join}, there is an iterated quadratic extension $L$
of $\Q$ containing all the points' coordinates. The coefficients in the
equations of the lines/circles then also lie in~$L$.

We now show that $\deg_L(x) \in \{1, 2\}$.

If $(x, y)$ is the intersection point of two distinct lines, then $x$ and
$y$ satisfy two linearly independent equations
\begin{align*}
ax + by + c     &= 0,   \\
a'x + b'y + c'  &= 0
\end{align*}
with $a, b, c, a', b', c' \in L$. Solving gives $x \in L$. (In more detail,
$x$ is a rational function of $a, b, \ldots$---write it down if you
want!---and so $x \in L$.)

If $(x, y)$ is an intersection point of a line and a circle, then
\begin{align*}
ax + by + c     &= 0,   \\
x^2 + y^2 + dx + ey + f &= 0
\end{align*}
with $a, b, c, d, e, f \in L$. If $b = 0$ then $a \neq 0$ and $x = -c/a \in
L$. Otherwise, we can eliminate $y$ to give a quadratic over $L$ satisfied
by $x$, so that $\deg_L(x) \in \{1, 2\}$. 

If $(x, y)$ is an intersection point of two circles, then
\begin{align*}
x^2 + y^2 + dx + ey + f &=0,    \\
x^2 + y^2 + d'x + e'y + f'      &=0
\end{align*}
with $d, e, f, d', e', f' \in L$. Subtracting, we reduce to the case of a
line and a circle, again giving $\deg_L(x) \in \{1, 2\}$. 

Hence $\deg_L(x) \in \{1, 2\}$. If $\deg_L(x) = 1$ then $x \in L$, and $L$ is
an iterated quadratic extension of $\Q$. If $\deg_L(x) = 2$,
i.e.\ $[L(x): L] = 2$, then $L(x)$ is an iterated quadratic extension of
$\Q$. In either case, $x$ lies in some iterated quadratic
extension of $\Q$. The same is true of $y$. Hence by
Lemma~\ref{lemma:iq-join}, there is an iterated quadratic extension of
$\Q$ containing $x$ and $y$. This completes the induction.
\end{proof}

\begin{bigthm}
\label{thm:rc-deg}
Let $(x, y) \in \R^2$. If $(x, y)$ is constructible from $\{(0, 0), (1,
0)\}$ then $x$ and $y$ are algebraic over $\Q$, and their degrees over $\Q$
are powers of $2$.
\end{bigthm}

\begin{proof}
By Proposition~\ref{propn:iq}, there is an iterated quadratic extension $M$
of $\Q$ with $x \in M$. Then $[M: \Q] = 2^n$ for some $n \geq 0$, by the
tower law. But then $\deg_{\Q}(x) = [\Q(x): \Q]$ divides $2^n$ by
Corollary~\ref{cor:tower-div}, and is therefore a power of $2$. And
similarly for $y$.
\end{proof}

Now we solve the problems of ancient Greece. 

\begin{propn}
The angle cannot be trisected by ruler and compass.
\end{propn}

\begin{proof}
Suppose it can be. Construct an equilateral triangle with $(0, 0)$ and $(1,
0)$ as two of its vertices (which can be done by ruler and compass;
Figure~\ref{fig:trisect}). Trisect the angle of the triangle at $(0,
0)$. Plot the point $(x, y)$ where the trisector meets the circle with
centre $(0, 0)$ through $(1, 0)$. Then $x = \cos(\pi/9)$, so by
Theorem~\ref{thm:rc-deg}, $\deg_\Q(\cos(\pi/9))$ is a power of $2$. But you
showed in \href{https://www.maths.ed.ac.uk/~tl/galois}{Assignment~2} that
$\deg_\Q(\cos(\pi/9)) = 3$, a contradiction.
\end{proof}
\begin{figure}
\setlength{\fboxsep}{0mm}%
\setlength{\unitlength}{1mm}%
\begin{picture}(136,32)
\cell{68}{15}{c}{\includegraphics[height=30\unitlength]{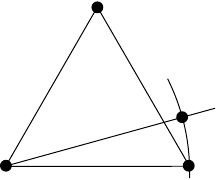}}
\cell{45}{2}{c}{$(0, 0)$}
\cell{87}{2}{c}{$(1, 0)$}
\cell{86}{9.5}{c}{$(x, y)$}
\end{picture}%
\caption{The impossibility of trisecting $60^\circ$.}
\label{fig:trisect}
\end{figure}

\begin{propn}
The cube cannot be duplicated by ruler and compass.
\end{propn}

\begin{proof}
Suppose it can be. Since $(0, 0)$ and $(1, 0)$ are distance $1$ apart, we
can construct from them two points $A$ and $B$ distance $\sqrt[3]{2}$
apart. From $A$ and $B$ we can construct, using ruler and compass, the
point $(\sqrt[3]{2}, 0)$.  So $\deg_\Q(\sqrt[3]{2})$ is a power of $2$, by
Theorem~\ref{thm:rc-deg}. But $\deg_\Q(\sqrt[3]{2}) = 3$ by
Example~\ref{eg:sb-cbrt2}, a contradiction.
\end{proof}

\begin{propn}
The circle cannot be squared by ruler and compass.
\end{propn}

This one is the most outrageously impossible, yet the hardest to prove.

\begin{proof}
Suppose it can be. Since the circle with centre $(0, 0)$ through $(1, 0)$
has area $\pi$, we can construct by ruler and compass a square with
side-length $\sqrt{\pi}$, and from that, we can construct by ruler and
compass the point $(\sqrt{\pi}, 0)$. So by Theorem~\ref{thm:rc-deg},
$\sqrt{\pi}$ is algebraic over $\Q$ with degree a power of $2$. Since
$\ovln{\Q}$ is a subfield of $\C$, it follows that $\pi$ is algebraic over
$\Q$. But it is a (hard) theorem that $\pi$ is transcendental over $\Q$.
\end{proof}

\begin{digression}{dig:rc-complex}
Stewart has a nice alternative approach to all this, in his Chapter~7. He
treats the plane as the \emph{complex} plane, and he shows that the set of
all points in $\C$ constructible from $0$ and $1$ is a subfield. In fact,
it is the smallest subfield of $\C$ closed under taking square roots. He
calls it $\Q^\text{py}$, the `Pythagorean closure' of $\Q$. It can also be
described as the set of complex numbers contained in some iterated
quadratic extension of $\Q$.
\end{digression}

There is one more famous ruler and compass problem: for which integers $n$
is the regular $n$-sided polygon constructible, starting from just a pair
of points in the plane?

The answer has to do with \demph{Fermat primes}, which are prime numbers of
the form $2^u + 1$ for some $u \geq 1$. A little exercise in number theory
shows that if $2^u + 1$ is prime then $u$ must itself be a power of
$2$. The only known Fermat primes are
\[
2^{2^0} + 1 = 3,
\quad
2^{2^1} + 1 = 5,
\quad
2^{2^2} + 1 = 17,
\quad
2^{2^3} + 1 = 257,
\quad
2^{2^4} + 1 = 65537.
\]
Whether there are any others is a longstanding open question. In any case,
it can be shown that the regular $n$-sided polygon is constructible if and
only if
\[
n = 2^r p_1 \cdots p_k
\]
for some $r, k \geq 0$ and distinct Fermat primes $p_1, \ldots, p_k$. 

We will not do the proof, but it involves cyclotomic polynomials. A glimpse
of the connection: let $p$ be a prime such that the regular $p$-sided
polygon is constructible. Consider the regular $p$-sided polygon inscribed
in the unit circle in $\C$, with one of its vertices at $1$. Then another
vertex is at $e^{2\pi i/p}$, and from constructibility, it follows that
$\deg_\Q(e^{2\pi i/p})$ is a power of $2$. But we saw in
Example~\ref{egs:alg-basis}\bref{eg:ab-cyclo} that $\deg_\Q(e^{2\pi i/p}) =
p - 1$. So $p - 1$ is a power of $2$, that is, $p$ is a Fermat prime.
Field theory, number theory and Euclidean geometry come together!

\chapter{Splitting fields}
\label{ch:split}

In Chapter~\ref{ch:overview}, we met a definition of the symmetry group of
a polynomial over $\Q$. It was phrased in terms of conjugate tuples, it was
possibly a little mysterious, and it was definitely difficult to work with
(e.g.\ we couldn't compute the symmetry group of $1 + t + t^2 + t^3 +
t^4$).
\video{Introduction to Week~6}

In this chapter, we're going to give a different but equivalent definition
of the symmetry group of a polynomial. It's a two-step process:
\begin{itemize}
\item[1.]
We show how every polynomial $f$ over $K$ gives rise to an extension of
$K$, called the `splitting field' of $f$.

\item[2.]
We show how every field extension has a symmetry group.
\end{itemize}
The symmetry group, or `Galois group', of a polynomial is then defined to
be the symmetry group of its splitting field extension.

How does these two steps work? 
\begin{itemize}
\item[1.]  When $K = \Q$, the splitting field of $f$ is the smallest
subfield of $\C$ containing all the complex roots of $f$.  For a general
field $K$, it's constructed by adding the roots of $f$ one at a time, using
simple extensions, until we obtain an extension of $K$ in which $f$ splits
into linear factors.

\item[2.]
The symmetry group of a field extension $M: K$ is defined as the group of
automorphisms of $M$ over $K$. This is the same idea you've seen many
times before, for symmetry groups of other mathematical objects.
\end{itemize}
Why bother? Why not define the symmetry group of $f$ directly, as in
Chapter~\ref{ch:overview}? 
\begin{itemize}
\item 
Because this strategy works over every field $K$, not just $\Q$.

\item
Because there are field extensions that do not arise from a
polynomial, and their symmetry groups are sometimes important. For
example, an important structure in number theory, somewhat mysterious
to this day, is the symmetry group of the algebraic numbers $\ovln{\Q}$
over $\Q$. 

\item
Because using abstract algebra means you can cut down on explicit
calculations with polynomials. (By way of analogy, you've seen how abstract
linear algebra with vector spaces and linear maps allows you to cut down on
calculations with matrices.) It also reveals connections with other parts of
mathematics.
\end{itemize}

\section{Extending homomorphisms}
\label{sec:ext-homm}

In your degree so far, you'll have picked up the general principle that for
many kinds of mathematical \emph{object} (such as groups, rings, fields,
vector spaces, modules, metric spaces, topological spaces, measure spaces,
\ldots), it's important to consider the appropriate notion of
\emph{mapping} between them (such as homomorphisms, linear maps, continuous
maps, \ldots).  And since Chapter~\ref{ch:exts}, you've known that the basic
objects of Galois theory are field extensions.

So it's no surprise that sooner or later, we have to think about
mappings from one field extension to another. That moment is now. We'll
need what's in this section in order to establish fundamental facts about
splitting fields.

When we think about a field extension $M: K$, we generally regard the field
$K$ as our starting point and $M$ as a field that extends it.  Similarly,
we might start with a \emph{homomorphism} $\psi \from K \to K'$ between
fields, together with extensions $M$ of $K$ and $M'$ of $K'$, and look for
a \emph{homomorphism} $M \to M'$ that extends $\psi$. The language is as
follows.

\begin{defn}
Let $\iota \from K \to M$ and $\iota' \from K' \to M'$ be field
extensions. Let $\psi \from K \to K'$ be a homomorphism of fields. A
homomorphism $\phi \from M \to M'$ \demph{extends} $\psi$ if 
the square
\video{Extension problems}%
\[
\xymatrix{
M \ar[r]^\phi   &
M'      \\
K \ar[u]^\iota \ar[r]_\psi      &
K' \ar[u]_{\iota'}
}
\]
commutes ($\phi \of \iota = \iota' \of \psi$). 
\end{defn}

Here I've used the strict definition of a field extension as a
homomorphism $\iota$ of fields (Definition~\ref{defn:ext}). Most of the
time we view $K$ as a subset of $M$ and $K'$ as a subset of $M'$, with
$\iota$ and $\iota'$ being the inclusions. In that case, for $\phi$ to
extend $\psi$ just means that
\[
\phi(a) = \psi(a) \text{ for all } a \in K.
\]

\begin{examples}
\begin{enumerate}
\item
Let $M$ and $M'$ be two extensions of a field $K$. For a homomorphism $\phi
\from M \to M'$ to extend $\id_K$ means that $\phi$ is a homomorphism over
$K$. 

\item
The conjugation homomorphism $\C \to \C$ extends the conjugation
homomorphism $\Q(i) \to \Q(i)$.
\end{enumerate}
\end{examples}

The basic questions about extending homomorphisms are: given the two field
extensions and the homomorphism $\psi$, is there some $\phi$ that extends
$\psi$? If so, how many extensions $\phi$ are there?

We'll get to these questions later. In this section, we simply prove two
general results about extensions of field homomorphisms.

Recall that any ring homomorphism $\psi \from R \to S$ induces a
homomorphism $\psi_* \from R[t] \to S[t]$
(Definition~\ref{defn:ind-hom}). To reduce clutter, I'll write $\psi_*(f)$
as \demph{$\psi_* f$}.
\video{Explanation of Lemma~\ref{lemma:ext-zeros}}%

\begin{lemma}
\label{lemma:ext-zeros}
Let $M: K$ and $M': K'$ be field extensions, let $\psi \from K \to K'$ be a
homomorphism, and let $\phi \from M \to M'$ be a homomorphism extending
$\psi$. Let $\alpha \in M$ and $f(t) \in K[t]$. Then 
\marginpar{%
$\xymatrix@C=3mm@R=4mm{%
\scriptstyle \alpha \ \ \ M \ar[r]^-\phi          &
\scriptstyle M' \ \ \ \phi(\alpha)                \\
\scriptstyle f \ \ \ K \ar[r]_-\psi \ar@<-2.5mm>[u] &
\scriptstyle K' \ar@<5mm>[u] \ \ \ \psi_* f
}$%
}%
\[
f(\alpha) = 0 \iff (\psi_* f)(\phi(\alpha)) = 0.
\]
\end{lemma}

\begin{proof}
Write $f(t) = \sum_i a_i t^i$, where $a_i \in
K$. Then $(\psi_* f)(t) = \sum_i \psi(a_i) t^i \in K'[t]$, so
\[
(\psi_* f)(\phi(\alpha)) 
= 
\sum_i \psi(a_i) \phi(\alpha)^i
=
\sum_i \phi(a_i) \phi(\alpha)^i
=
\phi(f(\alpha)),
\]
where the second equality holds because $\phi$ extends $\psi$. Since $\phi$
is injective (Lemma~\ref{lemma:fd-homm}), the result follows.
\end{proof}

\begin{example}
\label{eg:ext-zeros-id}
Let $M$ and $M'$ be extensions of a field $K$, and let $\phi \from M \to
M'$ be a homomorphism over $K$. Then the annihilating polynomials of an
element $\alpha \in M$ are the same as those of $\phi(\alpha)$. This is the
case $\psi = \id_K$ of Lemma~\ref{lemma:ext-zeros}.
\end{example}

\begin{ex}{ex:ind-inj}
Show that if a ring homomorphism $\psi$ is injective then so is $\psi_*$,
and if $\psi$ is an isomorphism then so is $\psi_*$.
\end{ex}

An isomorphism between fields, rings, groups, vector spaces, etc., can be
understood as simply a renaming of the elements. For example, if I tell you
that the ring $R$ is left Noetherian but not right Artinian, and that $S$
is isomorphic to $R$, then you can deduce that $S$ is left Noetherian but
not right Artinian \emph{without having the slightest idea what those words
mean}. Just as long as they don't depend on the names of the elements of the
ring concerned (which such definitions never do), you're fine.

\begin{propn}
\label{propn:simp-unique}
Let $\psi \from K \to K'$ be an isomorphism of fields. Let $K(\alpha): K$
be a simple extension where $\alpha$ has minimal polynomial $m$ over $K$,
and let $K'(\alpha'): K'$ be a simple extension where $\alpha'$ has minimal
polynomial $\psi_* m$ over $K'$. Then there is exactly one isomorphism $\phi
\from K(\alpha) \to K'(\alpha')$ that extends $\psi$ and satisfies
$\phi(\alpha) = \alpha'$. 
\end{propn}

Diagram:
\[
\xymatrix{
K(\alpha) \ar@{.>}[r]^\phi_\iso   &
K'(\alpha')     \\
K \ar[u] \ar[r]_\psi^\iso &
K' \ar[u]       
}
\]
We often use a dotted arrow to denote a map whose existence is part of the
conclusion of a theorem.

\begin{proof}
View $K'(\alpha')$ as an extension of $K$ via the composite homomorphism $K
\toby{\psi} K' \to K'(\alpha')$. Then the minimal polynomial of $\alpha'$
over $K$ is $m$. (If this isn't intuitively clear to you, think of the
isomorphism $\psi$ as renaming.) Hence by the classification of simple
extensions, Theorem~\ref{thm:class-simp}, there is exactly one isomorphism
$\phi \from K(\alpha) \to K'(\alpha')$ over $K$ such that $\phi(\alpha) =
\alpha'$. The result follows.
\end{proof}

\section{Existence and uniqueness of splitting fields}

Let $f$ be a polynomial over a field $K$. Informally, a splitting field for
$f$ is an extension of $K$ where $f$ has all its roots, and which is no
bigger than it needs to be. 

\begin{warning}{wg:one-all}
If $f$ is irreducible, we know how to create an extension of $K$ where $f$
has at least \emph{one} root: take the simple extension $K[t]/\idlgen{f}$,
in which the equivalence class of $t$ is a root of $f$
(Lemma~\ref{lemma:simple-build}\bref{part:sb-alg}). 

But $K[t]/\idlgen{f}$ is not usually a splitting field for $f$. For
example, take $K = \Q$ and $f(t) = t^3 - 2$, as in
Warning~\ref{wg:sub-abs}. Write $\xi$ for the real cube root of
$2$. (Half the counterexamples in Galois theory involve the real cube root
of $2$.) Then $\Q[t]/\idlgen{f}$ is isomorphic to the subfield $\Q(\xi)$
of $\R$, which only contains \emph{one} root of $f$: the other two are
non-real, hence not in $\Q(\xi)$.
\end{warning}

\begin{defn}
Let $f$ be a polynomial over a field $M$. Then $f$ \demph{splits} in
$M$ if
\[
f(t) = \beta(t - \alpha_1) \cdots (t - \alpha_n)
\]
for some $n \geq 0$ and $\beta, \alpha_1, \ldots, \alpha_n \in M$.
\end{defn}

Equivalently, $f$ splits in $M$ if all its irreducible factors in $M[t]$
are linear.

\begin{examples}
\label{egs:splits}
\begin{enumerate}
\item 
A field $M$ is algebraically closed if and only if every polynomial over
$M$ splits in $M$.

\item
\label{eg:splits-some}
Let $f(t) = t^4 - 4t^2 - 5$. Then $f$ splits in $\Q(i, \sqrt{5})$, since
\begin{align*}
f(t)    &
=
(t^2 + 1)(t^2 - 5)      \\
&
=
(t - i)(t + i)(t - \sqrt{5})(t + \sqrt{5}).
\end{align*}
But $f$ does not split in $\Q(i)$, as its factorization into irreducibles
in $\Q(i)[t]$ is
\[
f(t) = (t - i)(t + i)(t^2 - 5),
\]
which contains a nonlinear factor. 
\end{enumerate}
\end{examples}

\begin{warning}{wg:splits-roots}
As Example~\ref{egs:splits}\bref{eg:splits-some} shows, a polynomial 
over $M$ may have one root or
even several roots in $M$, but still not split in $M$.
\\
\end{warning}

\begin{example}
\label{eg:splits-F2}
Let $M = \F_2(\alpha)$, where $\alpha$ is a root of $f(t) = 1 + t + t^2$,
as in Example~\ref{egs:adjoin}\bref{eg:adjoin-F2}. We have
\[
f(1 + \alpha) 
= 
1 + (1 + \alpha) + (1 + 2\alpha + \alpha^2) 
=
1 + \alpha + \alpha^2
=
0,
\]
so $f$ has two distinct roots in $M$, giving
\[
f(t) = (t - \alpha)(t - (1 + \alpha))
\]
in $M[t]$. Hence $f$ splits in $M$. 

In \emph{this} example, adjoining one root of $f$ gave us a second root for
free. But this doesn't typically happen (Warning~\ref{wg:one-all}).
\end{example}

\begin{defn}
Let $f$ be a nonzero polynomial over a field $K$. A \demph{splitting field}
of $f$ over $K$ is an extension $M$ of $K$ such that:
\begin{enumerate}
\item 
$f$ splits in $M$;

\item
\label{condn:sf-gen}
$M = K(\alpha_1, \ldots, \alpha_n)$, where $\alpha_1, \ldots, \alpha_n$ are
the roots of $f$ in $M$.
\end{enumerate}
\end{defn}

\begin{ex}{ex:sf-gen-eqv}
Show that~\bref{condn:sf-gen} can equivalently be replaced by:
`if $L$ is a subfield of $M$ containing $K$, and $f$ splits in $L$, then $L =
M$'. 
\end{ex}

\begin{examples}
\label{egs:sf}
\begin{enumerate}
\item 
\label{eg:sf-Q}
Let $0 \neq f \in \Q[t]$. Write $\alpha_1, \ldots, \alpha_n$ for the
complex roots of $f$. Then $\Q(\alpha_1, \ldots, \alpha_n)$, the smallest
subfield of $\C$ containing $\alpha_1, \ldots, \alpha_n$, is a splitting
field of $f$ over $\Q$.

Splitting fields over $\Q$ are easy because we have a ready-made
algebraically closed field containing $\Q$, namely, $\C$.

\item
\label{eg:sf-triv}
If a polynomial $f \in K[t]$ splits in $K$ then $K$ itself is a splitting
field of $f$ over $K$. For instance, since $\C$ is algebraically closed, it
is a splitting field of every nonzero polynomial over $\C$.

\item
\label{eg:sf-cbrt2}
Let $f(t) = t^3 - 2 \in \Q[t]$. Its complex roots are $\xi$,
$\omega\xi$ and $\omega^2\xi$, where $\xi$ is the real cube root
of $2$ and $\omega = e^{2\pi i/3}$. Hence a splitting field of $f$ over
$\Q$ is 
\[
\Q(\xi, \omega\xi, \omega^2\xi) = \Q(\xi, \omega).
\]
Now $\deg_\Q(\xi) = 3$ as $f$ is irreducible, and $\deg_\Q(\omega) = 2$ as
$\omega$ has minimal polynomial $1 + t + t^2$. By an argument like that in
Example~\ref{eg:tower-126}, it follows that $[\Q(\xi, \omega): \Q] = 6$. On
the other hand, $[\Q(\xi): \Q] = 3$. So again, the extension we get by
adjoining \emph{all} the roots of $f$ is bigger than the one we get by
adjoining just \emph{one} root of $f$.

\item
Take $f(t) = 1 + t + t^2 \in \F_2[t]$, as in Example~\ref{eg:splits-F2}. By
Theorem~\ref{thm:simple-basis}\bref{part:sba-alg}, $\{1, \alpha\}$ is a
basis of $\F_2(\alpha)$ over $\F_2$, so
\begin{align*}
\F_2(\alpha)    &
= 
\{0, 1, \alpha, 1 + \alpha\}    \\
&
=
\F_2 \cup \{\text{the roots of $f$ in $\F_2(\alpha)$}\}.
\end{align*}
Hence $\F_2(\alpha)$ is a splitting field of $f$ over $\F_2$.
\end{enumerate}
\end{examples}

\begin{ex}{ex:cbrt2-gen}
In Example~\ref{egs:sf}\bref{eg:sf-cbrt2}, I said that $\Q(\xi,
\omega\xi, \omega^2\xi) = \Q(\xi, \omega)$. Why is that true?
\end{ex}

Our mission for the rest of this section is to show that every nonzero
polynomial $f$ has exactly one splitting field. So that's actually two
tasks: first, show that $f$ has \emph{at least} one splitting field, then,
show that $f$ has \emph{only} one splitting field. The first task is easy,
and in fact we prove a little bit more:

\begin{lemma}
\label{lemma:sf-exists}
Let $f \neq 0$ be a polynomial over a field $K$. Then there exists a splitting
field $M$ of $f$ over $K$ such that $[M: K] \leq \deg(f)!$.
\end{lemma}

\begin{proof}
We prove this by induction on $\deg(f)$, for all fields $K$
simultaneously. 

If $\deg(f) = 0$ then $K$ is a splitting field of $f$ over $K$, and the
result holds trivially. 

Now suppose that $\deg(f) \geq 1$. We may choose an irreducible factor $m$
of $f$. By Theorem~\ref{thm:class-simp}, there is an extension $K(\alpha)$
of $K$ with $m(\alpha) = 0$. Then $(t - \alpha) \dvd f(t)$ in
$K(\alpha)[t]$, giving a polynomial $g(t) = f(t)/(t - \alpha)$ over
$K(\alpha)$.

We have $\deg(g) = \deg(f) - 1$, so by inductive hypothesis, there is a
splitting field $M$ of $g$ over $K(\alpha)$ with $[M: K(\alpha)] \leq
\deg(g)!$. Then $M$ is a splitting field of $f$ over $K$. (Check that you
understand why.) Also, by the tower law,
\[
[M: K] 
= 
[M: K(\alpha)][K(\alpha): K]
\leq
(\deg(f) - 1)! \cdot \deg(m)
\leq
\deg(f)!,
\]
completing the induction.
\end{proof}

Proving that every polynomial has \emph{only} one splitting field is
harder. As ever, `only one' has to be understood up to isomorphism: after
all, if you're given a splitting field, you can always rename its elements
to get an isomorphic copy that's not literally identical to the original
one. But isomorphism is all that matters.

Our proof of the uniqueness of splitting fields depends on the following
result, which will also be useful for other purposes as we head
towards the fundamental theorem of Galois theory.

\begin{propn}
\label{propn:ext-iso}
Let $\psi \from K \to K'$ be an isomorphism of fields, let $0 \neq f \in K[t]$,
let $M$ be a splitting field of $f$ over $K$, and let $M'$ be a splitting
field of $\psi_* f$ over $K'$. Then:
\begin{enumerate}
\item 
\label{part:eh-exists}
there exists an isomorphism $\phi \from M \to M'$ extending $\psi$;%
\marginpar{%
$\xymatrix@C=5mm@R=4mm{%
\scriptstyle \ \ \ \ M \ar@{.>}[r]^-\phi        &  
\scriptstyle M' \ \ \ \phantom{\psi_* f}                       \\
\scriptstyle f \ \ \ K \ar[r]_-\psi \ar@<-2.5mm>[u] &
\scriptstyle K' \ar@<5mm>[u] \ \ \ \psi_* f
}$%
}%

\item
\label{part:eh-upr}
there are at most $[M: K]$ such extensions $\phi$.
\end{enumerate}
\end{propn}

We'll often use this result in the case where $K' = K$ and $\psi =
\id_K$. (What does it say then?)

\begin{proof}
We prove both statements by induction on $\deg(f)$. If $\deg(f) = 0$ then
both field extensions are trivial, so there is exactly one
isomorphism $\phi$ extending~$\psi$.

Now suppose that $\deg(f) \geq 1$. We can choose a monic irreducible factor $m$
of $f$. Then $m$ splits in $M$ since $f$ does and $m \dvd f$; choose a root
$\alpha \in M$ of $m$. We have $f(\alpha) = 0$, so $(t - \alpha) \dvd f(t)$
in $K(\alpha)[t]$, giving a polynomial $g(t) = f(t)/(t - \alpha)$ over
$K(\alpha)$. Then $M$ is a splitting field of $g$ over $K(\alpha)$, and
$\deg(g) = \deg(f) - 1$.

Also, $\psi_* m$ splits in $M'$ since $\psi_* f$
does and $\psi_* m \dvd \psi_* f$. Write $\alpha'_1, \ldots, \alpha'_s$
for the distinct roots of $\psi_* m$ in $M'$. 
Note that 
\begin{align}
\label{eq:ei-ub}
1 \leq s \leq \deg(\psi_* m) = \deg(m).
\end{align}
\video{Counting isomorphisms: the proof of Proposition~\ref{propn:ext-iso}}%
Since $\psi_*$ is an isomorphism, $\psi_* m$ is monic and irreducible, and
is therefore the minimal polynomial of $\alpha'_j$ for each $j \in \{1,
\ldots, s\}$. Hence by Proposition~\ref{propn:simp-unique}, for each $j$,
there is a unique isomorphism $\theta_j \from K(\alpha) \to K'(\alpha'_j)$
that extends $\psi$ and satisfies $\theta_j(\alpha) = \alpha'_j$. (See diagram
below.) 

For each $j \in \{1, \ldots, s\}$, we have a polynomial
\[
\theta_{j*}(g)
=
\frac{\theta_{j*}(f)}{\theta_{j*}(t - \alpha)}
=
\frac{\psi_* f}{t - \alpha'_j}
\]
over $K'(\alpha'_j)$, and $M'$ is a splitting field of $\psi_* f$ over
$K'$, so $M'$ is also a splitting field of $\theta_{j*}(g)$ over
$K'(\alpha'_j)$.

To prove that there is at least one isomorphism $\phi$ extending $\psi$,
choose any $j \in \{1, \ldots, s\}$ (as we may since $s \geq 1$). By
applying the inductive hypothesis to $g$ and $\theta_j$, there is an
isomorphism $\phi$ extending $\theta_j$:
\[
\xymatrix{
M \ar@{.>}[r]^\phi &
M'      \\
K(\alpha) \ar[u] \ar[r]^{\theta_j}      &
K'(\alpha'_j) \ar[u]    \\
K \ar[u] \ar[r]_\psi &
K' \ar[u]
}
\]
But then $\phi$ also extends $\psi$, as required.

To prove there are at most $[M: K]$ isomorphisms $\phi\from M \to M'$
extending $\psi$, first note that any such $\phi$ satisfies
$(\psi_* m)(\phi(\alpha)) = 0$ (by Lemma~\ref{lemma:ext-zeros}), so
$\phi(\alpha) = \alpha'_j$ for some $j \in \{1, \ldots, s\}$. Hence
\begin{multline*}
(\text{number of isos $\phi$ extending $\psi$})
= \\
\sum_{j = 1}^s
(\text{number of isos $\phi$ extending $\psi$ such that $\phi(\alpha) = \alpha'_j$}).
\end{multline*}
If $\phi$ extends $\psi$ then $\phi K = \psi K = K'$, and if also
$\phi(\alpha) = \alpha'_j$ then $\phi(K(\alpha)) = K'(\alpha'_j)$. Since
homomorphisms of fields are injective, $\phi$ then restricts to an
isomorphism $K(\alpha) \to K'(\alpha'_j)$ satisfying $\alpha \mapsto
\alpha'_j$. By the uniqueness part of Proposition~\ref{propn:simp-unique},
this restricted isomorphism must be $\theta_j$. Thus, $\phi$ extends
$\theta_j$. Hence
\[
(\text{number of isos $\phi$ extending $\psi$})
=
\sum_{j = 1}^s
(\text{number of isos $\phi$ extending $\theta_j$}).
\]
For each $j$, the number of isomorphisms $\phi$ extending $\theta_j$ is
$\leq [M: K(\alpha)]$, by inductive hypothesis. So, using the tower law
and~\eqref{eq:ei-ub}, 
\[
(\text{number of isos $\phi$ extending $\psi$})
\leq
s \cdot [M: K(\alpha)]
=
s \cdot \frac{[M: K]}{\deg(m)}
\leq 
[M: K],
\]
completing the induction.
\end{proof}

\begin{ex}{ex:ei-eq}
Why does the proof of Proposition~\ref{propn:ext-iso} not show that there
are \emph{exactly} $[M: K]$ isomorphisms $\phi$ extending $\psi$? How could
you strengthen the hypotheses in order to obtain that conclusion? (The
second question is a bit harder, and we'll see the answer next week.)
\end{ex}

This brings us to the foundational result on splitting fields. Recall that
an \demph{automorphism} of an object $X$ is an isomorphism $X \to X$.

\begin{bigthm}
\label{thm:sf}
Let $f$ be a nonzero polynomial over a field $K$. Then:
\begin{enumerate}
\item 
\label{part:sf-exists}
there exists a splitting field of $f$ over $K$;

\item
\label{part:sf-unique}
any two splitting fields of $f$ are isomorphic over $K$;

\item
\label{part:sf-ineqs}
when $M$ is a splitting field of $f$ over $K$,
\[
(\text{number of automorphisms of $M$ over $K$})
\leq
[M: K]
\leq 
\deg(f)!.
\]
\end{enumerate}
\end{bigthm}

\begin{proof}
Part~\bref{part:sf-exists} is immediate from Lemma~\ref{lemma:sf-exists}, and
part~\bref{part:sf-unique} follows from Proposition~\ref{propn:ext-iso} by
taking $K' = K$ and $\psi = \id_K$. The first inequality
in~\bref{part:sf-ineqs} follows from Proposition~\ref{propn:ext-iso} by
taking $K' = K$, $M' = M$ and $\psi = \id_K$, and the second follows from
Lemma~\ref{lemma:sf-exists}. 
\end{proof}

Up to now we have been saying `a' splitting
field. Parts~\bref{part:sf-exists} and~\bref{part:sf-unique} of
Theorem~\ref{thm:sf} give us the right to
speak of \emph{the} splitting field of a given polynomial $f$ over a given
field $K$. We write it as \demph{$\SF_K(f)$}.

We finish with a left over lemma that will be useful later. 

\begin{lemma}
\label{lemma:sf-adj}
\begin{enumerate}
\item 
\label{part:sfa-gen}
Let $M: S: K$ be field extensions, $0 \neq f \in K[t]$, and $Y \sub M$. Suppose
that $S$ is the splitting field of $f$ over $K$. Then $S(Y)$ is the splitting
field of $f$ over $K(Y)$.

\item
\label{part:sfa-fds}
Let $f \neq 0$ be a polynomial over a field $K$, and let $L$ be a subfield of
$\SF_K(f)$ containing $K$ (so that $\SF_K(f): L: K$). Then $\SF_K(f)$ is
the splitting field of $f$ over $L$.
\end{enumerate}
\end{lemma}

\begin{proof}
For~\bref{part:sfa-gen}, $f$ splits in $S$, hence in $S(Y)$. Writing $X$
for the set of roots of $f$ in $S$, we have $S = K(X)$ and so $S(Y) =
K(X)(Y) = K(X \cup Y) = K(Y)(X)$; that is, $S(Y)$ is generated over $K(Y)$
by $X$. This proves~\bref{part:sfa-gen}, and~\bref{part:sfa-fds} follows by
taking $M = \SF_K(f)$ and $Y = L$.
\end{proof}

\section{The Galois group}
\label{sec:gal-gp}

Before you get stuck into this section, you may want to review
Section~\ref{sec:actions}, especially the parts about homomorphisms $G \to
\Sym(X)$. We'll need all of it.

What gives Galois theory its special flavour is the use of groups to study
fields and polynomials. Here is the central definition.

\begin{defn}
The \demph{Galois group $\Gal(M: K)$} of a field extension $M: K$ is the
group of automorphisms of $M$ over $K$, with composition as the group
operation.
\end{defn}

\begin{ex}{ex:Gal-sensical}
Check that this really does define a group.\\
\end{ex}

In other words, an element of $\Gal(M: K)$ is an isomorphism $\theta \from
M \to M$ such that $\theta(a) = a$ for all $a \in K$.

\begin{examples}
\label{egs:gal-ext}
\begin{enumerate}
\item 
\label{eg:gal-ext-CR}
What is $\Gal(\C: \R)$? Certainly the identity is an automorphism of $\C$
over $\R$. So is complex conjugation $\kappa$, as implicitly shown in the
first proof of Lemma~\ref{lemma:indist-R}. So $\{\id, \kappa\} \sub
\Gal(\C: \R)$. I claim that $\Gal(\C: \R)$ has no other elements. For let
$\theta \in \Gal(\C: \R)$. Then
\[
(\theta(i))^2 = \theta(i^2) = \theta(-1) = -\theta(1) = -1
\]
as $\theta$ is a homomorphism, so $\theta(i) = \pm i$. If $\theta(i) = i$
then $\theta = \id$, by Lemma~\ref{lemma:gen-epic} and the fact that $\C =
\R(i)$. Similarly, if $\theta(i) = -i$ then $\theta = \kappa$. So $\Gal(\C:
\R) = \{\id, \kappa\} \iso C_2$.

\item
\label{eg:gal-ext-cbrt2}
Let $\xi$ be the real cube root of $2$. For each $\theta \in
\Gal(\Q(\xi): \Q)$, we have 
\[
(\theta(\xi))^3 = \theta(\xi^3) = \theta(2) = 2
\]
and $\theta(\xi) \in \Q(\xi) \sub \R$, so $\theta(\xi) = \xi$. It follows
from Lemma~\ref{lemma:gen-epic} that $\theta = \id$. Hence
$\Gal(\Q(\xi): \Q)$ is trivial.
\end{enumerate}
\end{examples}

\begin{ex}{ex:gal-gp-3u}
Prove that $\Gal(\Q(e^{2\pi i/3}): \Q) = \{\id, \kappa\}$, where $\kappa(z) =
\ovln{z}$. (Hint: imitate Example~\ref{egs:gal-ext}\bref{eg:gal-ext-CR}.)
\end{ex}

The Galois group of a polynomial is defined to be the Galois group of its
splitting field extension:

\begin{defn}
\label{defn:gal-gp-poly}
Let $f$ be a nonzero polynomial over a field $K$. The \demph{Galois group
$\Gal_K(f)$} of $f$ over $K$ is $\Gal(\SF_K(f): K)$. 
\end{defn}

So the definitions fit together like this:
\[
\label{p:peg}
\text{polynomial } 
\longmapsto 
\text{ field extension } 
\longmapsto 
\text{ group.}
\]
We will soon prove that Definition~\ref{defn:gal-gp-poly} is equivalent to
the definition of Galois group in Chapter~\ref{ch:overview}, where we went
straight from polynomials to groups.

Theorem~\ref{thm:sf}\bref{part:sf-ineqs} says that 
\begin{align}
\label{eq:gal-poly-bounds}
|\Gal_K(f)| \leq [\SF_K(f): K] \leq \deg(f)!.
\end{align}
In particular, $\Gal_K(f)$ is always a \emph{finite} group.

\begin{examples}
\label{egs:gal-gp-first}
\begin{enumerate}
\item
\label{eg:ggf-triv}
If $f \in K[t]$ splits in $K$ then $\SF_K(f) = K$
(Example~\ref{egs:sf}\bref{eg:sf-triv}), so $\Gal_K(f)$ is trivial. In
particular, the Galois group of any polynomial over an algebraically closed
field is trivial.

\item 
$\Gal_\Q(t^2 + 1) = \Gal(\Q(i): \Q) = \{\id, \kappa\} \iso C_2$, where
$\kappa$ is complex conjugation on $\Q(i)$. The second equality is proved
by the same argument as in Example~\ref{egs:gal-ext}\bref{eg:gal-ext-CR},
replacing $\C: \R$ by $\Q(i): \Q$.

\item
\label{eg:ggf-conj}
Generally, let $f \in \Q[t]$. We can view $\SF_\Q(f)$ as the subfield of
$\C$ generated by the complex roots of $f$, and if $\alpha \in \C$ is a
root of $f$ then so is $\ovln{\alpha}$. Hence complex conjugation, as
an automorphism of $\C$, restricts to an automorphism $\kappa$ of
$\SF_\Q(f)$. 

If all the complex roots of $f$ are real then $\kappa = \id \in
\Gal_\Q(f)$. Otherwise, $\kappa$ is an element of $\Gal_\Q(f)$ of order
$2$.

\item
\label{eg:ggf-klein}
\video{Calculating the Galois group with bare hands, part 1}%
\video{Calculating the Galois group with bare hands, part 2}%
Let $f(t) = (t^2 + 1)(t^2 - 2)$. Then $\Gal_\Q(f)$ is the group of
automorphisms of $\Q(i, \sqrt{2})$ over $\Q$. Similar arguments to those in
Examples~\ref{egs:gal-ext} show that every $\theta \in \Gal_\Q(f)$ must
satisfy $\theta(i) = \pm i$ and $\theta(\sqrt{2}) = \pm \sqrt{2}$, and that
the two choices of sign determine $\theta$ completely. And one can show
that all four choices are possible, so that $|\Gal_\Q(f)| = 4$. There are
two groups of order four, $C_4$ and $C_2 \times C_2$. But each element of
$\Gal_\Q(f)$ has order $1$ or $2$, so $\Gal_\Q(f)$ is not $C_4$, so
$\Gal_\Q(f) \iso C_2 \times C_2$.

I've been sketchy with the details here, because it's not really sensible
to try to calculate Galois groups until we have a few more tools at our
disposal. We start to assemble them now.
\end{enumerate}
\end{examples}

By definition, $\Gal_K(f)$ acts on $\SF_K(f)$
(Example~\ref{egs:gp-actions}\bref{eg:ga-aut}). The action is
\[
(\theta, \alpha) \mapsto \theta(\alpha)
\]
($\theta \in \Gal_K(f)$, $\alpha \in \SF_K(f)$). In the examples
so far, we've seen that if $\alpha$ is a root of $f$ then so is
$\theta(\alpha)$ for every $\theta \in \Gal_K(f)$. This is true in
general: the action of $\Gal_K(f)$ on $\SF_K(f)$ restricts
to an action on the set of roots. In a slogan: \emph{the Galois group
permutes the roots.}

\begin{lemma}
\label{lemma:gal-acts}
Let $f$ be a nonzero polynomial over a field $K$. Then the action of
$\Gal_K(f)$ on $\SF_K(f)$ restricts to an action on the set of roots of
$f$ in $\SF_K(f)$. 
\end{lemma}

Terminology: given a group $G$ acting on a set $X$ and a subset $A \sub X$,
the action \demph{restricts} to $A$ if $ga \in A$ for all $g \in G$ and $a
\in A$.

\begin{proof}
We have to show that if $\theta \in \Gal_K(f)$ and $\alpha$ is a root of
$f$ in $\SF_K(f)$ then $\theta(\alpha)$ is also a root. This follows from
Example~\ref{eg:ext-zeros-id}. 
\end{proof}
\video{The action of the Galois group}%

Better still, the Galois group acts \emph{faithfully} on the roots:

\begin{lemma}
\label{lemma:gal-acts-faithfully}
Let $f$ be a nonzero polynomial over a field $K$. Then the action of
$\Gal_K(f)$ on the roots of $f$ is faithful.
\end{lemma}

\begin{proof}
Write $X$ for the set of roots of $f$ in $\SF_K(f)$. Then $\SF_K(f) =
K(X)$. Hence by Lemma~\ref{lemma:gen-epic}, if $\theta \in \Gal_K(f)$ with
$\theta(x) = x$ for all $x \in X$, then $\theta = \id$.
\end{proof}

In other words, an element of the Galois group of $f$ is completely
determined by how it permutes the roots of $f$. So you can view elements of
the Galois group as \emph{being} permutations of the roots. 

However, not every permutation of the roots belongs to the Galois group. To
understand the situation, recall Remark~\ref{rmk:faith-indices}, which
tells us the following. Suppose that $f \in K[t]$ has distinct roots
$\alpha_1, \ldots, \alpha_k$ in its splitting field. For each $\theta \in
\Gal_K(f)$, there is a permutation $\sigma_\theta \in S_k$ defined by
\[
\theta(\alpha_i) = \alpha_{\sigma_\theta(i)}
\]
($i \in \{1, \ldots, k\}$). Then $\Gal_K(f)$ is isomorphic to the subgroup
$\{\sigma_\theta \such \theta \in \Gal_K(f)\}$ of $S_k$ (and this is indeed
a subgroup). The isomorphism is given by $\theta \mapsto \sigma_\theta$. 

All this talk of the Galois group as a subgroup of $S_k$ may have set
your antennae tingling. Back in Chapter~1, we provisionally
\emph{defined} the Galois group to be a certain subgroup of $S_k$
(Definition~\ref{defn:gg-conc}). We can now show that the two definitions
are equivalent. 

That definition was in terms of conjugacy. Let's now make the
concept of conjugacy official, also generalizing from $\Q$ to an arbitrary
field.

\begin{defn}
\label{defn:conj-tuples}
Let $M: K$ be a field extension, let $k \geq 0$, and let $(\alpha_1,
\ldots, \alpha_k)$ and $(\alpha'_1, \ldots, \alpha'_k)$ be $k$-tuples of
elements of $M$. Then $(\alpha_1, \ldots, \alpha_k)$ and $(\alpha'_1, \ldots,
\alpha'_k)$ are \demph{conjugate} over $K$ if for all $p \in K[t_1, \ldots,
t_k]$, 
\[
p(\alpha_1, \ldots, \alpha_k) = 0 \iff 
p(\alpha'_1, \ldots, \alpha'_k) = 0.
\]
In the case $k = 1$, we omit the brackets and say that $\alpha$ and
$\alpha'$ are conjugate to mean that $(\alpha)$ and $(\alpha')$ are.
\end{defn}

We now show that the two definitions of the Galois group of $f$ are
equivalent. 

\begin{propn}
\label{propn:gal-defns-same}
Let $f$ be a nonzero polynomial over a field $K$, with distinct roots $\alpha_1,
\ldots, \alpha_k$ in $\SF_K(f)$. Then
\begin{align}
\label{eq:gal-perm-subgp}
\{ \sigma \in S_k \such (\alpha_1, \ldots, \alpha_k) \text{ and }
(\alpha_{\sigma(1)}, \ldots, \alpha_{\sigma(k)}) \text{ are conjugate over
} K\}
\end{align}
is a subgroup of $S_k$ isomorphic to $\Gal_K(f)$. 
\end{propn}

\begin{proof}
As above, each $\theta \in \Gal_K(f)$ gives rise to a permutation
$\sigma_\theta \in S_k$, defined by $\theta(\alpha_i) =
\alpha_{\sigma_\theta(i)}$. For the purposes of this proof, let us say that
a permutation $\sigma \in S_k$ is `good' if it belongs to the
set~\eqref{eq:gal-perm-subgp}. By Remark~\ref{rmk:faith-indices}, it
suffices to show that a permutation $\sigma$ is good if and only if $\sigma =
\sigma_\theta$ for some $\theta \in \Gal_K(f)$. 

First suppose that $\sigma = \sigma_\theta$ for some $\theta \in
\Gal_K(f)$. For every $p \in K[t_1, \ldots, t_k]$, 
\[
p(\alpha_{\sigma(1)}, \ldots, \alpha_{\sigma(k)})
=
p(\theta(\alpha_1), \ldots, \theta(\alpha_k))
=
\theta(p(\alpha_1, \ldots, \alpha_k)),
\]
where the first equality is by definition of $\sigma_\theta$ and
the second is because $\theta$ is a homomorphism over $K$. But $\theta$ is
an isomorphism, so it follows that
\[
p(\alpha_{\sigma(1)}, \ldots, \alpha_{\sigma(k)}) = 0
\iff
p(\alpha_1, \ldots, \alpha_k) = 0.
\]
Hence $\sigma$ is good.

Conversely, suppose that $\sigma$ is good. By
Corollary~\ref{cor:alg-ext-poly}, every element of $\SF_K(f)$ can be
expressed as $p(\alpha_1, \ldots, \alpha_k)$ for some $p \in K[t_1, \ldots,
t_k]$. Now for $p, q \in K[t_1, \ldots, t_k]$, we have
\[
p(\alpha_1, \ldots, \alpha_k) = q(\alpha_1, \ldots, \alpha_k)   
\iff
p(\alpha_{\sigma(1)}, \ldots, \alpha_{\sigma(k)})
=
q(\alpha_{\sigma(1)}, \ldots, \alpha_{\sigma(k)})
\]
(by applying Definition~\ref{defn:conj-tuples} of conjugacy with $p - q$ as
the `$p$'). So there is a well-defined, injective function $\theta \from
\SF_K(f) \to \SF_K(f)$ satisfying
\begin{align}
\label{eq:constr-theta}
\theta(p(\alpha_1, \ldots, \alpha_k)) 
=
p(\alpha_{\sigma(1)}, \ldots, \alpha_{\sigma(k)})
\end{align}
for all $p \in K[t_1, \ldots, t_k]$. Moreover, $\theta$ is surjective
because $\sigma$ is a permutation, and $\theta(a) = a$ for all $a \in K$ (by
taking $p = a$ in~\eqref{eq:constr-theta}), and $\theta(\alpha_i) =
\alpha_{\sigma(i)}$ for all $i$ (by taking $p = t_i$
in~\eqref{eq:constr-theta}). You can check that $\theta$ is a homomorphism
of fields. Hence $\theta \in \Gal_K(f)$ with $\sigma_\theta = \sigma$, as
required. 
\end{proof}

\begin{ex}{ex:two-gal}
I skipped two small bits in that proof: `$\theta$ is surjective because
$\sigma$ is a permutation' (why?), and `You can check that $\theta$ is a
homomorphism of fields'. Fill in the gaps.
\end{ex}

It's important in Galois theory to be able to move between
fields. For example, you might start with a polynomial whose coefficients
belong to one field $K$, but later decide to interpret the coefficients as
belonging to some larger field $L$. Here's what happens to the Galois group
when you do that.

\begin{cor}
\label{cor:gal-emb}
Let $L: K$ be a field extension and $0 \neq f \in K[t]$. Then $\Gal_L(f)$
is isomorphic to a subgroup of $\Gal_K(f)$.
\end{cor}

\begin{proof}
This follows from Proposition~\ref{propn:gal-defns-same} together with the
observation that if two $k$-tuples are conjugate over $L$, they are
conjugate over $K$.
\end{proof}

\begin{example}
Let's find the Galois group of $f(t) = (t^2 + 1)(t^2 - 2)$ over $\Q$, $\R$
and $\C$ in turn. 

In Example~\ref{egs:gal-gp-first}\bref{eg:ggf-klein}, we saw that
$\Gal_\Q(f) \iso C_2 \times C_2$. 

Since both roots of $t^2 - 2$ are real, $\SF_\R(f) = \SF_\R(t^2 + 1) =
\C$. So $\Gal_\R(f) = \Gal(\C: \R) \iso C_2$, where the last step is by
Example~\ref{egs:gal-ext}\bref{eg:gal-ext-CR}. 

Finally, $\Gal_\C(f)$ is trivial since $\C$ is algebraically closed
(Example~\ref{egs:gal-gp-first}\bref{eg:ggf-triv}). 

So as Corollary~\ref{cor:gal-emb} predicts, $\Gal_\C(f)$ is isomorphic to a
subgroup of $\Gal_\R(f)$, which is isomorphic to a subgroup of
$\Gal_\Q(f)$. 
\end{example}

\begin{cor}
\label{cor:gal-poly-distinct}
Let $f$ be a nonzero polynomial over a field $K$, with $k$ distinct roots in
$\SF_K(f)$. Then $\left|\Gal_K(f)\right|$ divides $k!$.
\end{cor}

\begin{proof}
By Proposition~\ref{propn:gal-defns-same}, $\Gal_K(f)$ is isomorphic to a
subgroup of $S_k$, which has $k!$ elements. The result follows from
Lagrange's theorem.
\end{proof}

The inequalities~\eqref{eq:gal-poly-bounds} already gave us $|\Gal_K(f)|
\leq \deg(f)!$. Corollary~\ref{cor:gal-poly-distinct} improves on this in
two respects. First, it implies that $|\Gal_K(f)| \leq k!$. It's always the
case that $k \leq \deg(f)$ in all cases, and $k < \deg(f)$ if $f$ has
repeated roots in its splitting field. A trivial example: if $f(t) = t^2$
then $k = 1$ and $\deg(f) = 2$. Second, it tells us that $|\Gal_K(f)|$ is
not only less than or equal to $k!$, but a factor of it.

Galois theory is about the interplay between field extensions and
groups. In the next chapter, we'll see that just as every field extension
gives rise to a group of automorphisms (its Galois group), every group of
automorphisms gives rise to a field extension. We'll also go deeper into
the different types of field extension: normal extensions (the mirror image
of normal subgroups) and separable extensions (which have to do with
repeated roots). All of that will lead us towards the fundamental theorem of
Galois theory.

\chapter{Preparation for the fundamental theorem}

Very roughly, the fundamental theorem of Galois theory says that you can
tell a lot about a field extension by looking at its Galois group. A bit
\video{Introduction to Week~7}%
more specifically, it says that the subgroups and quotients of $\Gal(M:
K)$, and their orders, give us information about the subfields of $M$
containing $K$, and their degrees. For example, one part of the fundamental
theorem is that
\[
[M: K] = |\Gal(M: K)|.
\]
The theorem doesn't hold for all extensions, just those that are `nice
enough'. Crucially, this includes splitting field extensions $\SF_\Q(f):
\Q$ of polynomials $f$ over $\Q$---the starting point of classical Galois
theory. 

Let's dip our toes into the water by thinking about why it might be true
that $[M: K] = |\Gal(M: K)|$, at least for extensions that are nice
enough. 

The easiest nontrivial extensions are the simple algebraic extensions, $M
= K(\alpha)$. Write $m$ for the minimal polynomial of $\alpha$ over $K$ and
$\alpha_1, \alpha_2, \ldots, \alpha_s$ for the distinct roots of $m$ in
$M$. For every element $\phi$ of $\Gal(M: K)$, we have $m(\phi(\alpha)) =
0$ by Example~\ref{eg:ext-zeros-id}, and so $\phi(\alpha) = \alpha_j$ for
some $j \in \{1, \ldots, s\}$.  On the other hand, for each $j \in \{1,
\ldots, s\}$, there is exactly one $\phi \in \Gal(M: K)$ such that
$\phi(\alpha) = \alpha_j$, by Proposition~\ref{propn:simp-unique}. So
$|\Gal(M: K)| = s$.

On the other hand, $[M: K] = \deg(m)$. So $[M: K] = |\Gal(M: K)|$ if and
only if $\deg(m)$ is equal to $s$, the number of distinct roots of $m$ in
$M$. Certainly $s \leq \deg(m)$. But are $s$ and $\deg(m)$ equal?

There are two reasons why they might not be. First, $m$ might not split in
$M$. For instance, if $K = \Q$ and $\alpha = \sqrt[3]{2}$ then $m(t) = t^3
- 2$, which has only one root in $\Q(\sqrt[3]{2})$, so
$|\Gal(\Q(\sqrt[3]{2}): \Q)| = 1 < 3 = \deg(m)$.  An algebraic extension is
called `normal' if this problem doesn't occur, that is, if the minimal
polynomial of every element does split. That's what Section~\ref{sec:nml}
is about.

Second, we might have $s < \deg(m)$ because some of the roots of $m$ in $M$
are repeated. If they are, the number $s$ of \emph{distinct} roots will be
less then $\deg(m)$. An algebraic extension is called `separable' if this
problem doesn't occur, that is, if the minimal polynomial of every element has
no repeated roots in its splitting field. That's what Section~\ref{sec:sep}
is about.

If we take any finite extension $M: K$ (not necessarily simple) that is both
normal and separable, then it is indeed true that
$|\Gal(M: K)| = [M: K]$.  And in fact, these conditions are enough to make
the whole fundamental theorem work, as we'll see next week.

I hesitated before putting normality and separability into the same
chapter, because you should think of them in quite different
ways:
\begin{itemize}
\item
Normality has a clear conceptual meaning, and its importance was recognized
by Galois himself. Despite the name, most field extensions aren't
normal. Normality isn't something to be taken for granted.

\item
In contrast, Galois never considered separability, because it holds
automatically over $\Q$ (his focus), and in fact over any field of
characteristic $0$, as well as any finite field. It takes some work to find
an extension that \emph{isn't} separable. You can view separability as more
of a technicality.
\end{itemize}

There's one more concept in this chapter: the `fixed field' of a group of
automorphisms (Section~\ref{sec:fixed}). Every Galois theory text I've seen
contains at least one proof that makes you ask `how did anyone think of
that?' I would argue that the proof of Theorem~\ref{thm:fixed} is the one
and only truly ingenious argument in this course: maybe not the hardest,
but the most ingenious. This is not a compliment.

\section{Normality}
\label{sec:nml}

\begin{defn}
An algebraic field extension $M: K$ is \demph{normal} if for all $\alpha
\in M$, the minimal polynomial of $\alpha$ splits in $M$.
\end{defn}

We also say \demph{$M$ is normal over $K$} to mean that $M: K$ is normal.

\begin{lemma}
\label{lemma:nml-irr}
Let $M: K$ be an algebraic extension. Then $M: K$ is normal if and only if
every irreducible polynomial over $K$ either has no roots in $M$ or
splits in~$M$.
\end{lemma}

Put another way, normality means that any irreducible polynomial over $K$
with at least \emph{one} root in $M$ has \emph{all} its roots in $M$.

\begin{proof}
Suppose that $M: K$ is normal, and let $f$ be an irreducible polynomial over
$K$. If $f$ has a root $\alpha$ in $M$ then the minimal polynomial of
$\alpha$ is $f/c$, where $c \in K$ is the leading coefficient of $f$. Since
$M: K$ is normal, $f/c$ splits in $M$, so $f$ does too.

Conversely, suppose that every irreducible polynomial over $K$ either has
no roots in $M$ or splits in $M$. Let $\alpha \in M$. Then the minimal
polynomial of $\alpha$ has at least one root in $M$ (namely, $\alpha$), so
it splits in $M$. 
\end{proof}

\begin{examples}
\label{egs:nml-ext}
\begin{enumerate}
\item 
\label{eg:ne-cbrt2}
Let $\xi = \sqrt[3]{2} \in \R$, and consider $\Q(\xi): \Q$. The minimal
polynomial of $\xi$ over $\Q$ is $t^3 - 2$, whose roots in $\C$ are $\xi
\in \R$ and $\omega\xi, \omega^2\xi \in \C \without \R$, where $\omega =
e^{2\pi i/3}$. Since $\Q(\xi) \sub \R$, the minimal polynomial $t^3 - 2$
does not split in $\Q(\xi)$. Hence $\Q(\xi)$ is not normal over $\Q$.

Alternatively, using the equivalent condition in
Lemma~\ref{lemma:nml-irr}, $\Q(\xi): \Q$ is not normal because
$t^3 - 2$ is an irreducible polynomial over $\Q$ that has a root in
$\Q(\xi)$ but does not split there.

One way to think about the non-normality of $\Q(\xi): \Q$ is as
follows.  The three roots of $t^3 - 2$ are conjugate (`indistinguishable')
over $\Q$, since they have the same minimal polynomial. But if they're
\video{What does it mean to be normal?}%
indistinguishable, it would be strange for an extension to contain some but
not all of them---that would be making a distinction between elements that
are supposed to be indistinguishable. In this sense, $\Q(\xi): \Q$
is `abnormal'.

\item
Let $f$ be a nonzero polynomial over a field $K$. Then $\SF_K(f): K$ is
always normal, as we shall see (Theorem~\ref{thm:sf-fin-nml}).

\item
Every extension of degree $2$ is normal (just as, in group theory, every
subgroup of index $2$ is normal). You'll be asked to show this in
\href{https://www.maths.ed.ac.uk/~tl/galois}{Workshop~4, question~4}, but
you also know enough to prove it now.
\end{enumerate}
\end{examples}

\begin{ex}{ex:nml-irr}
What happens if you drop the word `irreducible' from
Lemma~\ref{lemma:nml-irr}? Is it still true?
\end{ex}

Normality of field extensions is intimately related to normality of
subgroups, and conjugacy in field extensions is also related to conjugacy
in groups. (The video \href{https://www.maths.ed.ac.uk/~tl/galois}{`What
does it mean to be normal?}'\ explains both kinds of normality and
conjugacy in intuitive terms.)

Here's the first of our two theorems about normal extensions. It describes
which extensions arise as splitting field extensions.

\begin{bigthm}
\label{thm:sf-fin-nml}
Let $M: K$ be a field extension. Then
\[
M = \SF_K(f) \text{ for some nonzero } f \in K[t]      
\iff
M: K \text{ is finite and normal}.
\]
\end{bigthm}

\begin{proof}
For $\impliedby$, suppose that $M: K$ is finite and normal. By finiteness,
there is a basis $\alpha_1, \ldots, \alpha_n$ of $M$ over $K$, and each
$\alpha_i$ is algebraic over $K$ (by
Proposition~\ref{propn:three-fin}). For each $i$, let $m_i$ be the minimal
polynomial of $\alpha_i$ over $K$; then by normality, $m_i$ splits in
$M$. Hence $f = m_1 m_2 \cdots m_n \in K[t]$ splits in $M$. The set of
roots of $f$ in $M$ contains $\{\alpha_1, \ldots, \alpha_n\}$, and $M =
K(\alpha_1, \ldots, \alpha_n)$, so $M$ is generated over $K$ by the set of
roots of $f$ in $M$. Thus, $M$ is a splitting field of $f$ over $K$.

For $\implies$, take a nonzero $f \in K[t]$ such that $M = \SF_K(f)$.  Write
$\alpha_1, \ldots, \alpha_n$ for the roots of $f$ in $M$. Then $M =
K(\alpha_1, \ldots, \alpha_n)$. Each $\alpha_i$ is algebraic over $K$
(since $f \neq 0$), so by Proposition~\ref{propn:three-fin}, $M: K$ is
finite.

We now show that $M: K$ is normal, which is the most substantial part of the
proof (Figure~\ref{fig:sf-nml}).
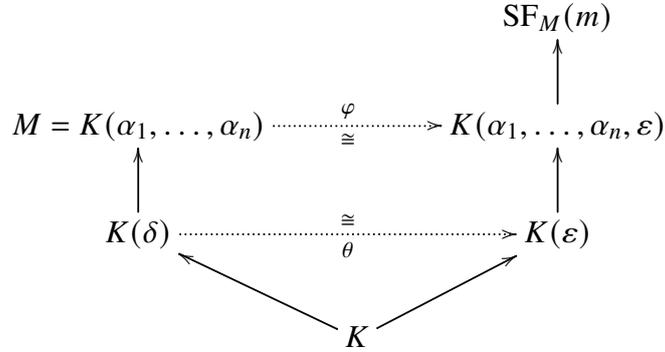
\begin{figure}
\[
\xymatrix{
        &
        &
\SF_M(m)        \\
M = K(\alpha_1, \ldots, \alpha_n)
\ar@{.>}[rr]^\phi_\iso  &
        &
K(\alpha_1, \ldots, \alpha_n, \epsln)
\ar[u]  \\
K(\delta) 
\ar@{.>}[rr]^\iso_\theta 
\ar[u]  &
        &
K(\epsln)
\ar[u]  \\
        &
K
\ar[lu] \ar[ru]
}
\]
\caption{Maps used in the proof that splitting field extensions are normal.}
\label{fig:sf-nml}
\end{figure}
Let $\delta \in M$, with minimal polynomial $m \in K[t]$. Certainly
$m$ splits in $\SF_M(m)$, so to show that $m$ splits in $M$, it is enough
to show that every root $\epsln$ of $m$ in $\SF_M(m)$ lies in $M$.%
\video{Splitting field extensions are normal}%

Since $m$ is a monic irreducible annihilating polynomial of $\epsln$ over
$K$, it is the minimal polynomial of $\epsln$ over $K$. Hence by
Theorem~\ref{thm:class-simp}, there is an isomorphism $\theta\from
K(\delta) \to K(\epsln)$ over $K$ such that $\theta(\delta) = \epsln$. Now
observe that:
\begin{itemize}
\item
$M = \SF_{K(\delta)}(f)$, by Lemma~\ref{lemma:sf-adj}\bref{part:sfa-fds};

\item
$K(\alpha_1, \ldots, \alpha_n, \epsln) = \SF_{K(\epsln)}(f)$, by
Lemma~\ref{lemma:sf-adj}\bref{part:sfa-gen};

\item
$\theta_* f = f$, since $f \in K[t]$ and $\theta$ is a homomorphism over
$K$. 
\end{itemize}
So we can apply Proposition~\ref{propn:ext-iso}, which implies that there
is an isomorphism $\phi \from M \to K(\alpha_1, \ldots, \alpha_n,
\epsln)$ extending $\theta$. It is an isomorphism over $K$, since
$\theta$ is.

Since $\delta \in K(\alpha_1, \ldots, \alpha_n)$ and $\phi$ is a
homomorphism over $K$, we have $\phi(\delta) \in K(\phi(\alpha_1), \ldots,
\phi(\alpha_n))$. Now $\phi(\delta) = \theta(\delta) = \epsln$, so $\epsln
\in K(\phi(\alpha_1), \ldots, \phi(\alpha_n))$. Moreover, for each $i$ we
have $f(\phi(\alpha_i)) = 0$ (by Example~\ref{eg:ext-zeros-id}) and so
$\phi(\alpha_i) \in \{\alpha_1, \ldots, \alpha_n\}$. Hence $\epsln \in
K(\alpha_1, \ldots, \alpha_n) = M$, as required.
\end{proof}

\begin{cor}
\label{cor:top-nml}
Let $M: L: K$ be field extensions. If $M: K$ is finite and normal then so
is $M: L$.
\end{cor}

\begin{proof}
Follows from Theorem~\ref{thm:sf-fin-nml} and
Lemma~\ref{lemma:sf-adj}\bref{part:sfa-fds}. 
\end{proof}

\begin{warning}{wg:btm-not-nml}
It does \emph{not} follow that $L: K$ is normal. For
instance, consider $\Q(\sqrt[3]{2}, e^{2\pi i/3}): \Q(\sqrt[3]{2}):
\Q$. The first field is the splitting field of $t^3 - 2$ over $\Q$, and
therefore normal over $\Q$, but $\Q(\sqrt[3]{2})$ is not
(Example~\ref{egs:nml-ext}\bref{eg:ne-cbrt2}). 
\end{warning}

Theorem~\ref{thm:sf-fin-nml} is the first of two theorems about
normality. The second is to do with the action of the Galois group of an
extension. 

\begin{warning}{wg:two-actions}
By definition, the Galois group $\Gal(M: K)$ of an extension $M: K$ acts on
$M$. But if $M$ is the splitting field of some polynomial $f$ over $K$ then
the action of $\Gal(M: K)$ on $M$ restricts to an action on the roots of
$f$ (a finite set), as we saw in Section~\ref{sec:gal-gp}. So there are
\emph{two} actions of the Galois group in play, one the
restriction of the other. Both are important.
\end{warning}

When a group acts on a set, a basic question is: what are the orbits? For
$\Gal(M: K)$ acting on $M$, the answer is: the conjugacy classes of $M$
over $K$. Or at least, that's the case when $M: K$ is finite and normal:

\begin{propn}
\label{propn:orbit-cc}
Let $M: K$ be a finite normal extension and $\alpha, \alpha' \in M$. Then
\[
\alpha \text{ and } \alpha' \text{ are conjugate over } K 
\iff    
\alpha' = \phi(\alpha) \text{ for some } \phi \in \Gal(M: K).
\]
\end{propn}

\begin{proof}
For $\impliedby$, let $\phi \in \Gal(M: K)$ with $\alpha' =
\phi(\alpha)$. Then $\alpha$ and $\alpha'$ are conjugate over $K$, by
Example~\ref{eg:ext-zeros-id}.

For $\implies$, suppose that $\alpha$ and $\alpha'$ are conjugate over
$K$. Since $M: K$ is finite, both are algebraic over $K$, and since they
are conjugate over $K$, they have the same minimal polynomial $m \in
K[t]$. By Theorem~\ref{thm:class-simp}, there is an isomorphism $\theta
\from K(\alpha) \to K(\alpha')$ over $K$ such that $\theta(\alpha) =
\alpha'$ (see diagram below). 

By Theorem~\ref{thm:sf-fin-nml}, $M$ is the splitting field of some
polynomial $f$ over $K$. Hence $M$ is also the splitting field of $f$ over
both $K(\alpha)$ and $K(\alpha')$, by
Lemma~\ref{lemma:sf-adj}\bref{part:sfa-fds}. Moreover, $\theta_* f = f$
since $\theta$ is a homomorphism over $K$ and $f$ is a polynomial over
$K$. So by Proposition~\ref{propn:ext-iso}\bref{part:eh-exists}, there is
an automorphism $\phi$ of $M$ extending $\theta$:
\[
\xymatrix{
M
\ar@{.>}[rr]^\phi_\iso  &
        &
M       \\
K(\alpha) 
\ar@{.>}[rr]^\iso_\theta 
\ar[u]  &
        &
K(\alpha')
\ar[u]  \\
        &
K
\ar[lu] \ar[ru]
}
\]
Then $\phi \in \Gal(M: K)$ with $\phi(\alpha) =
\theta(\alpha) = \alpha'$, as required.
\end{proof}

\begin{example}
Consider the finite normal extension $\C: \R$. Let $\alpha, \alpha' \in
\C$. Lemma~\ref{lemma:indist-R} states that $\alpha$ and $\alpha'$ are
conjugate over $\R$ if and only if $\alpha'$ is either $\alpha$ or
$\ovln{\alpha}$. Example~\ref{egs:gal-ext}\bref{eg:gal-ext-CR} states that
$\Gal(\C: \R) = \{\id, \kappa\}$, where $\kappa$ is complex
conjugation. This confirms Proposition~\ref{propn:orbit-cc} in the case
$\C: \R$.
\end{example}

Proposition~\ref{propn:orbit-cc} is about the action of $\Gal(M: K)$ on the
whole field $M$, but it has a powerful corollary involving the action of
the Galois group on just the roots of an irreducible polynomial $f$, in the
case $M = \SF_K(f)$:

\begin{cor}
\label{cor:act-trans}
Let $f$ be an irreducible polynomial over a field $K$. Then the action of
$\Gal_K(f)$ on the roots of $f$ in $\SF_K(f)$ is transitive.  
\end{cor}

Recall what \demph{transitive} means, for an action of a group $G$ on a set
$X$: for all $x, x' \in X$, there exists $g \in G$ such that $gx = x'$.

\begin{proof}
Since $f$ is irreducible, the roots of $f$ in $\SF_K(f)$ all have the same
minimal polynomial, namely, $f$ divided by its leading coefficient. So they
are all conjugate over $K$. Since $\SF_K(f): K$ is finite and normal (by
Theorem~\ref{thm:sf-fin-nml}), the result follows from
Proposition~\ref{propn:orbit-cc}.
\end{proof}

\begin{ex}{ex:trans-irr}
Show by example that Corollary~\ref{cor:act-trans} becomes false if you
drop the word `irreducible'.
\end{ex}

\begin{example}
\label{eg:spectacular}
Let $f(t) = 1 + t + \cdots + t^{p - 1} \in \Q[t]$, where $p$ is
prime. Since $(1 - t)f(t) = 1 - t^p$, the roots of $f$ in $\C$ are $\omega,
\omega^2, \ldots, \omega^{p - 1}$, where $\omega = e^{2\pi i/p}$. By
Example~\ref{eg:cyclo-irr}, $f$ is irreducible over $\Q$. Hence by
Corollary~\ref{cor:act-trans}, for each $i \in \{1, \ldots, p - 1\}$, there
is some $\phi \in \Gal_\Q(f)$ such that $\phi(\omega) = \omega^i$.

\textcolor{red}{\emph{This is spectacular!}} Until now, we've been unable
to prove such things without a huge amount of explicit checking, which,
moreover, only works on a case-by-case basis. For example, if you watched
the video \href{https://www.maths.ed.ac.uk/~tl/galois}{`Calculating Galois groups with bare hands, part~2'}, you'll have
seen how much tedious calculation went into the single case $p =
5$, $i = 2$:
\[
\hspace*{-32mm}\includegraphics[height=50mm,width=200mm]{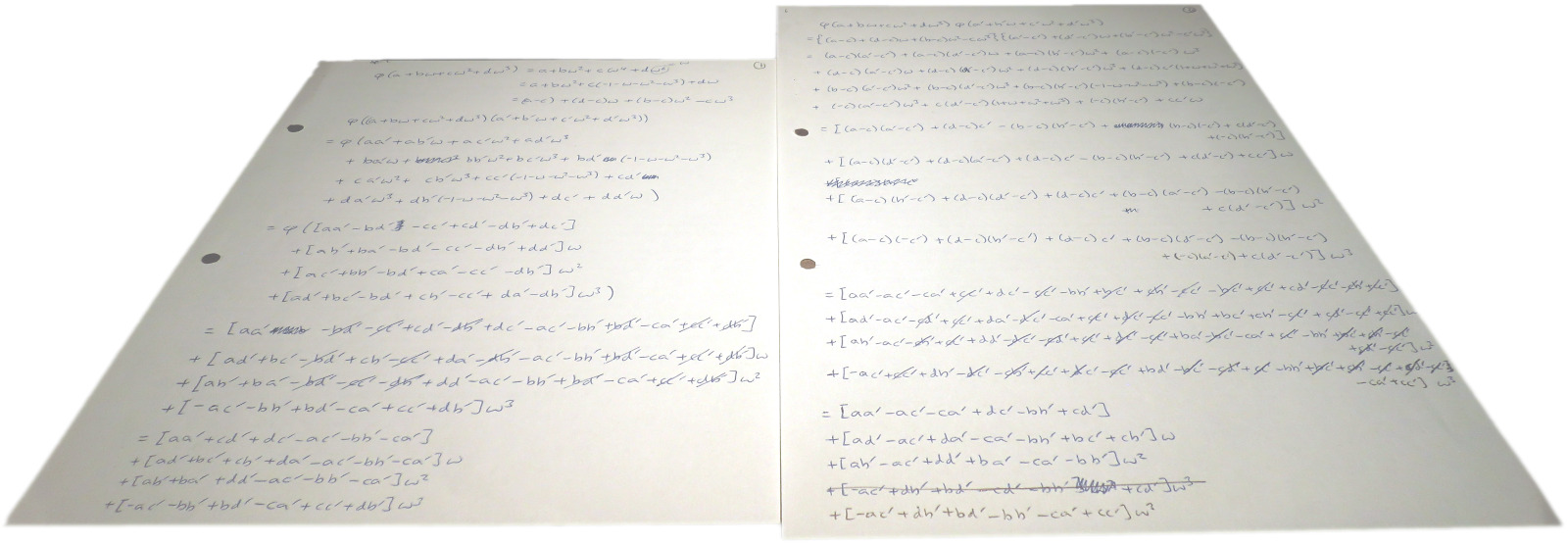}
\]
But the theorems we've proved make all this unnecessary.

In fact, for each $i \in \{1, \ldots, p - 1\}$, there's exactly \emph{one}
element $\phi_i$ of $\Gal_\Q(f)$ such that $\phi_i(\omega) = \omega^i$. For
since $\SF_\Q(f) = \Q(\omega)$, two elements of $\Gal_\Q(f)$ that take the
same value on $\omega$ must be equal. Hence
\[
\Gal_\Q(f) = \{\phi_1, \ldots, \phi_{p - 1}\}.
\]
In fact, $\Gal_\Q(f) \iso C_{p - 1}$
(\href{https://www.maths.ed.ac.uk/~tl/galois}{Workshop~4, question~13}).
\end{example}

\begin{example}
\label{eg:gal-cbrt2}
Let's calculate $G = \Gal_\Q(t^3 - 2)$. Since $t^3 - 2$ has $3$ distinct
roots in $\C$, it has $3$ distinct roots in its splitting field. By
Proposition~\ref{propn:gal-defns-same}, $G$ is isomorphic to a subgroup of
$S_3$. Now $G$ acts transitively on the $3$ roots, so it has at least $3$
elements, so it is isomorphic to either $A_3$ or $S_3$. Since two of the
roots are non-real complex conjugates, one of the elements of $G$ is
complex conjugation, which has order $2$
(Example~\ref{egs:gal-gp-first}\bref{eg:ggf-conj}). Hence $2$ divides
$|G|$, forcing $G \iso S_3$.
\end{example}

We now show how a normal field extension gives rise to a normal
subgroup. Whenever in life you meet a normal subgroup, you should
immediately want to form the quotient, so we do that too.

\begin{bigthm}
\label{thm:nml-ftgt}
Let $M: L: K$ be field extensions with $M: K$ finite and normal. 
\begin{enumerate}
\item 
\label{part:nf-fix}
$L: K$ is a normal extension $\iff$ $\phi L = L$ for all $\phi \in
\Gal(M: K)$.

\item
\label{part:nf-main}
If $L: K$ is a normal extension then $\Gal(M: L)$ is a normal subgroup of
$\Gal(M: K)$ and
\[
\frac{\Gal(M: K)}{\Gal(M: L)} \iso \Gal(L: K).
\]
\end{enumerate}
\end{bigthm}

Before the proof, here's some context and explanation. 

Part~\bref{part:nf-fix} answers the question implicit in
Warning~\ref{wg:btm-not-nml}: we know from Corollary~\ref{cor:top-nml} that
$M: L$ is normal, but when is $L: K$ normal? The notation
\demph{$\phi L$} means $\{\phi(\alpha) \such \alpha \in L\}$. For $\phi L$ to be
equal to $L$ means that $\phi$ fixes $L$ \emph{as a set} (in other words,
permutes it within itself), not that $\phi$ fixes each element of $L$.

In part~\bref{part:nf-main}, it's true for all $M: L: K$ that
$\Gal(M: L)$ is a \emph{subset} of $\Gal(M: K)$, since
\begin{align*}
\Gal(M: L)      &
=
\{ \text{automorphisms $\phi$ of $M$ such that $\phi(\alpha) = \alpha$ 
for all $\alpha \in L$} \}  \\
&
\sub
\{ \text{automorphisms $\phi$ of $M$ such that $\phi(\alpha) = \alpha$ 
for all $\alpha \in K$} \}  \\
&
=
\Gal(M: K).
\end{align*}
And it's always true that $\Gal(M: L)$ is a \emph{subgroup} of $\Gal(M:
K)$, as you can easily check. But part~\bref{part:nf-main} tells us
something much more substantial: it's a \emph{normal} subgroup when $L: K$
is a normal extension.

\begin{pfof}{Theorem~\ref{thm:nml-ftgt}}
For~\bref{part:nf-fix}, first suppose that $L$ is normal over $K$, and let
$\phi \in \Gal(M: K)$. For all $\alpha \in L$,
Proposition~\ref{propn:orbit-cc} implies that $\alpha$ and $\phi(\alpha)$
are conjugate over $K$, so they have the same minimal polynomial, so
$\phi(\alpha) \in L$ by normality. Hence $\phi L \sub L$. The same argument
with $\phi^{-1}$ in place of $\phi$ gives $\phi^{-1} L \sub L$, and
applying $\phi$ to each side then gives $L \sub \phi L$. So $\phi L = L$.

Conversely, suppose that $\phi L = L$ for all $\phi \in \Gal(M: K)$. Let
$\alpha \in L$ with minimal polynomial $m$. Since $M: K$ is normal, $m$
splits in $M$. Each root $\alpha'$ of $m$ in $M$ is conjugate to $\alpha$
over $K$, so by Proposition~\ref{propn:orbit-cc}, $\alpha' = \phi(\alpha)$
for some $\phi \in \Gal(M: K)$, giving $\alpha' \in \phi L = L$. Hence $m$
splits in $L$ and $L: K$ is normal.

For~\bref{part:nf-main}, suppose that $L: K$ is normal. To
prove that $\Gal(M: L)$ is a normal subgroup of $\Gal(M: K)$, let $\phi \in
\Gal(M: K)$ and $\theta \in \Gal(M: L)$. We show that $\phi^{-1} \theta
\phi \in \Gal(M: L)$, or equivalently, 
\[
\phi^{-1} \theta \phi(\alpha) = \alpha \text{ for all } \alpha \in L,
\]
or equivalently,
\[
\theta \phi(\alpha) = \phi(\alpha) \text{ for all } \alpha \in L.
\]
But by~\bref{part:nf-fix}, $\phi(\alpha) \in L$ for all $\alpha
\in L$, so $\theta(\phi(\alpha)) = \phi(\alpha)$ since $\theta \in \Gal(M:
L)$. This completes the proof that $\Gal(M: L) \nsub \Gal(M: K)$.

Finally, we prove the statement on quotients (still supposing that $L: K$
is a normal extension). Every automorphism $\phi$ of $M$ over $K$ satisfies
$\phi L = L$ (by~\bref{part:nf-fix}), and therefore restricts to an
automorphism $\hat{\phi}$ of $L$. The function
\[
\begin{array}{cccc}
\nu\from        &\Gal(M: K)     &\to            &\Gal(L: K)     \\
                &\phi           &\mapsto        &\hat{\phi}
\end{array}
\]
is a group homomorphism, since it preserves composition. Its kernel is
$\Gal(M: L)$, by definition. If we can prove that $\nu$ is surjective then
the last part of the theorem will follow from the first isomorphism theorem.

To prove that $\nu$ is surjective, we must show that each automorphism
$\psi$ of $L$ over $K$ extends to an automorphism $\phi$ of $M$:
\[
\xymatrix{
M
\ar@{.>}[rr]^\phi_\iso  &
        &
M       \\
L
\ar[rr]^\iso_\psi 
\ar[u]  &
        &
L
\ar[u]  \\
        &
K
\ar[lu] \ar[ru]
}
\]
The argument is similar to the second half of the proof of
Proposition~\ref{propn:orbit-cc}. By Theorem~\ref{thm:sf-fin-nml}, $M$ is
the splitting field of some $f \in K[t]$. Then $M$ is also the splittting
field of $f$ over $L$. Also, $\psi_* f = f$ since $\psi$ is a homomorphism
over $K$ and $f$ is a polynomial over $K$. So by
Proposition~\ref{propn:ext-iso}\bref{part:eh-exists}, there is an
automorphism $\phi$ of $M$ extending $\psi$, as required.
\end{pfof}

\begin{example}
\label{eg:nf-cbrt2}
Take $M: L: K$ to be 
\[
\Q\bigl(\xi, \omega\bigr): \Q(\omega): \Q,
\]
where $\xi = \sqrt[3]{2}$ and $\omega = e^{2\pi i/3}$. As you will
recognize by now, $\Q(\xi, \omega)$ is the splitting field of $t^3 - 2$
over $\Q$, so it is a finite normal extension of $\Q$ by
Theorem~\ref{thm:sf-fin-nml}.

Also, $\Q(\omega)$ is the splitting field of $t^2 + t + 1$ over $\Q$, so it
too is a normal extension of $\Q$. Part~\bref{part:nf-fix} of
Theorem~\ref{thm:nml-ftgt} implies that every element of $\Gal_\Q(t^3 - 2)$
restricts to an automorphism of $\Q(\omega)$. 

Part~\bref{part:nf-main} implies that
\[
\Gal\bigl(\Q\bigl(\xi, \omega\bigr): \Q(\omega)\bigr)
\nsub
\Gal\bigl(\Q\bigl(\xi, \omega\bigr): \Q\bigr)
\]
and that
\begin{align}
\label{eq:cbrt2-qt-iso}
\frac{\Gal\bigl(\Q\bigl(\xi, \omega\bigr): \Q\bigr)}
{\Gal\bigl(\Q\bigl(\xi, \omega\bigr): \Q(\omega)\bigr)}
\iso
\Gal(\Q(\omega): \Q).
\end{align}
What does this say explicitly? We showed in Example~\ref{eg:gal-cbrt2} that
$\Gal(\Q(\xi, \omega): \Q) \iso S_3$. That is, each element of the
Galois group permutes the three roots
\[
\xi, \, \omega\xi, \, \omega^2 \xi
\]
of $t^3 - 2$, and all six permutations are realized by some element of the
Galois group. An element of $\Gal(\Q(\xi, \omega): \Q)$ that fixes
$\omega$ is determined by which of the three roots $\xi$ is mapped
to, so $\Gal(\Q(\xi, \omega): \Q(\omega)) \iso A_3$. Finally,
$\Gal(\Q(\omega): \Q) \iso C_2$ by Example~\ref{eg:spectacular}. So in this
case, the isomorphism~\eqref{eq:cbrt2-qt-iso} states that
\[
\frac{S_3}{A_3}
\iso
C_2.
\]
\end{example}

\begin{ex}{ex:nf-cbrt2-pic}
Draw a diagram showing the three roots of $t^3 - 2$ and the elements of
$H = \Gal(\Q(\xi, \omega): \Q(\omega))$ acting on them. There is a
simple geometric description of the elements of $\Gal(\Q(\xi,
\omega): \Q)$ that belong to the subgroup $H$. What is it?
\end{ex}

\section{Separability}
\label{sec:sep}

Theorem~\ref{thm:sf} implies that $|\Gal(M: K)| \leq [M: K]$ whenever $M:
K$ is a splitting field extension. Why is this an inequality, not an
equality? The answer can be traced back to the proof of
Proposition~\ref{propn:ext-iso} on extension of isomorphisms. There, we had
an irreducible polynomial called $\psi_* m$, and we wrote $s$ for the
number of distinct roots of $\psi_* m$ in its splitting field.
Ultimately, the source of the inequality was the
fact that $s \leq \deg(\psi_* m)$.

But is this last inequality actually an equality? That is, does an
irreducible polynomial of degree $d$ always have $d$ distinct roots in its
splitting field?  Certainly it has $d$ roots when counted \emph{with
multiplicity}. But there will be fewer than $d$ \emph{distinct} roots if
any of the roots are repeated (have multiplicity $\geq 2$). The
question is whether this can ever happen.

Formally, for a polynomial $f(t) \in K[t]$ and a root $\alpha$ of $f$ in
some extension $M$ of $K$, we say that $\alpha$ is a \demph{repeated} root
if $(t - \alpha)^2 \dvd f(t)$ in $M[t]$.

\begin{ex}{ex:sep-futile}
Try to find an example of an irreducible polynomial of degree $d$ with
fewer than $d$ distinct roots in its splitting field. Or if you can't, see
if you can prove that this is impossible over $\Q$: that is, an
irreducible over $\Q$ has no repeated roots in $\C$. Both are quite hard,
but ten minutes spent trying may help you to appreciate what's to come.
\end{ex}

\begin{defn}
\label{defn:irr-sep}
An irreducible polynomial over a field is \demph{separable} if it has no
repeated roots in its splitting field.
\end{defn}

Equivalently, an irreducible polynomial $f \in K[t]$ is separable if it
splits into \emph{distinct} linear factors in $\SF_K(f)$:
\[
f(t) = a(t - \alpha_1) \cdots (t - \alpha_n)
\]
for some $a \in K$ and \emph{distinct} $\alpha_1, \ldots, \alpha_n \in
\SF_K(f)$. Put another way, an irreducible $f$ is separable if and only if
it has $\deg(f)$ distinct roots in its splitting field.

\begin{example}
$t^3 - 2 \in \Q[t]$ is separable, since it has $3$ distinct roots in $\C$, 
hence in its splitting field.
\end{example}

\begin{example}
\label{eg:insep-poly}
This is an example of an irreducible polynomial that's
\emph{in}separable. It's a little bit complicated, but it's
the simplest example there is.

Let $p$ be a prime number. We will consider the field $K = \F_p(u)$ of rational
expressions over $\F_p$ in an indeterminate (variable symbol)
$u$. Put $f(t) = t^p - u \in K[t]$. We will show that $f$ is an inseparable
irreducible polynomial.

By definition, $f$ has at least one root $\alpha$ in its splitting
field. But the roots of $f$ are the $p$th roots of $u$, and in fields of
characteristic $p$, each element has at most one $p$th root
(Corollary~\ref{cor:pth-roots}\bref{part:pr-fd}). So $\alpha$ is the
\emph{only} root of $f$ in $\SF_K(f)$, despite $f$ having degree $p >
1$. Alternatively, we can argue like this:
\[
f(t)
=
t^p - u
=
t^p - \alpha^p
=
(t - \alpha)^p,
\]
where the last step comes from the Frobenius map of $\SF_K(f)$ being a
homomorphism (Proposition~\ref{propn:frob}\bref{part:frob-ring}). 

We now show that $f$ is irreducible over $K$. Suppose it is reducible. The
unique factorization of $f$ into irreducible polynomials over $\SF_K(f)$ is
$f(t) = (t - \alpha)^p$, so any nontrivial factorization of $f$ in $K[t]$
is of the form
\[
f(t) = (t - \alpha)^i (t - \alpha)^{p - i}
\]
where $0 < i < p$ and both factors belong to $K[t]$. The coefficient of
$t^{i - 1}$ in $(t - \alpha)^i$ is $-i\alpha$, so $-i\alpha \in K$. But $i$
is invertible in $K$, so $\alpha \in K$. Hence $u$ has a $p$th root in
$K = \F_p(u)$, contradicting Exercise~\ref{ex:pth-root-trans}.
\end{example}

\begin{warning}{wg:sep-irred}
Definition~\ref{defn:irr-sep} is only a definition of separability for
\emph{irreducible} polynomials. There \emph{is} a definition of
separability for arbitrary polynomials, but it's not simply
Definition~\ref{defn:irr-sep} with the word `irreducible' deleted. We won't
need it, but here it is: an arbitrary polynomial is called separable if
each of its irreducible factors is separable. So $t^2$ is separable, even
though it has a repeated root.
\end{warning}

In real analysis, we can test whether a root is repeated by asking whether
the derivative is $0$ there:
\[
\includegraphics[height=30mm]{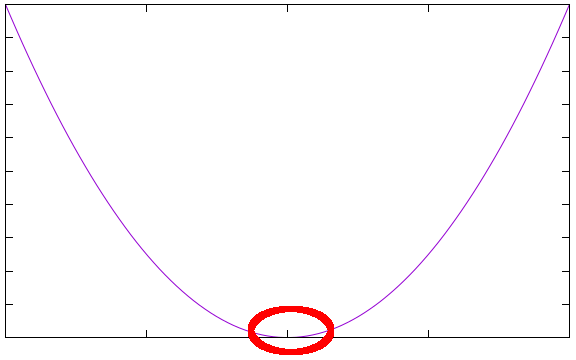}
\]
Over an arbitrary field, there's no general definition of the derivative of
a function, as there's no meaningful notion of limit. But even
without limits, we can differentiate polynomials in the following sense.

\begin{defn}
Let $K$ be a field and let $f(t) = \sum_{i = 0}^n a_i t^i \in K[t]$. The
\demph{formal derivative} of $f$ is 
\[
\dem{(Df)(t)} = \sum_{i = 1}^n i a_i t^{i - 1} \in K[t].
\]
\end{defn}

We use $Df$ rather than $f'$ to remind ourselves not to take the familiar
properties of differentiation for granted. Nevertheless, the usual basic
laws hold:

\begin{lemma}
\label{lemma:fdiff-elem}
Let $K$ be a field. Then 
\[
D(f + g) = Df + Dg,
\qquad
D(fg) = f\cdot Dg + Df \cdot g,
\qquad
Da = 0
\]
for all $f, g \in K[t]$ and $a \in K$.
\qed
\end{lemma}

\begin{ex}{ex:fdiff-check}
Check one or two of the properties in Lemma~\ref{lemma:fdiff-elem}.\\
\end{ex}

The real analysis test for repetition of roots has an algebraic analogue:

\begin{lemma}
\label{lemma:rep-tfae}
Let $f$ be a nonzero polynomial over a field $K$. The following are
equivalent: 
\begin{enumerate}
\item 
\label{part:rt-rep}
$f$ has a repeated root in $\SF_K(f)$;

\item
\label{part:rt-shared}
$f$ and $Df$ have a common root in $\SF_K(f)$;

\item
\label{part:rt-factor}
$f$ and $Df$ have a nonconstant common factor in $K[t]$.
\end{enumerate}
\end{lemma}

\begin{proof}
\bref{part:rt-rep}$\implies$\bref{part:rt-shared}: suppose that $f$ has a
repeated root $\alpha$ in $\SF_K(f)$. Then $f(t) = (t -
\alpha)^2 g(t)$ for some $g(t) \in (\SF_K(f))[t]$. Hence
\[
(Df)(t) 
= 
(t - \alpha)
\bigl\{
2g(t) + (t - \alpha) \cdot (Dg)(t)
\bigr\},
\]
so $\alpha$ is a common root of $f$ and $Df$ in $\SF_K(f)$.

\bref{part:rt-shared}$\implies$\bref{part:rt-factor}: suppose that $f$ and
$Df$ have a common root $\alpha$ in $\SF_K(f)$. Then $\alpha$ is algebraic
over $K$ (since $f \neq 0$), and the minimal polynomial of $\alpha$ over
$K$ is then a nonconstant common factor of $f$ and $Df$ in $K[t]$.

\bref{part:rt-factor}$\implies$\bref{part:rt-shared}: if $f$ and $Df$ have
a nonconstant common factor $g$ then $g$ splits in $\SF_K(f)$, and any root
of $g$ in $\SF_K(f)$ is a common root of $f$ and $Df$.

\bref{part:rt-shared}$\implies$\bref{part:rt-rep}: suppose that $f$ and $Df$
have a common root $\alpha \in \SF_K(f)$. Then $f(t) = (t - \alpha)g(t)$
for some $g \in (\SF_K(f))[t]$, giving
\[
(Df)(t) = g(t) + (t - \alpha)\cdot (Dg)(t).
\]
But $(Df)(\alpha) = 0$, so $g(\alpha) = 0$, so $g(t) = (t - \alpha) h(t)$
for some $h \in (\SF_K(f))[t]$. Hence $f(t) = (t - \alpha)^2 h(t)$, and
$\alpha$ is a repeated root of $f$ in its splitting field.
\end{proof}

The point of Lemma~\ref{lemma:rep-tfae} is that
condition~\bref{part:rt-factor} allows us to test for repetition of roots
in $\SF_K(f)$ without ever leaving $K[t]$, or even
knowing what $\SF_K(f)$ is.

\begin{propn}
\label{propn:sep-deriv-zero}
Let $f$ be an irreducible polynomial over a field. Then $f$ is
inseparable if and only if $Df = 0$.
\end{propn}

\begin{proof}
This follows from \bref{part:rt-rep}$\iff$\bref{part:rt-factor} in
Lemma~\ref{lemma:rep-tfae}. Since $f$ is
irreducible, $f$ and $Df$ have a nonconstant
common factor if and only if $f$ divides $Df$; but $\deg(Df) < \deg(f)$, so
$f \dvd Df$ if and only if $Df = 0$.
\end{proof}

\begin{cor}
\label{cor:sep-chars}
Let $K$ be a field.
\begin{enumerate}
\item 
If $\chr K = 0$ then every irreducible polynomial over $K$ is separable.

\item
If $\chr K = p > 0$ then an irreducible polynomial $f \in K[t]$ is
inseparable if and only if
\[
f(t) = b_0 + b_1 t^p + \cdots + b_r t^{rp}
\]
for some $b_0, \ldots, b_r \in K$.
\end{enumerate}
\end{cor}

In other words, the only irreducible polynomials that are inseparable are
the polynomials in $t^p$ in characteristic $p$. Inevitably,
Example~\ref{eg:insep-poly} is of this form.

\begin{proof}
Let $f(t) = \sum a_i t^i$ be an irreducible polynomial. Then $f$ is
inseparable if and only if $Df = 0$, if and only if $i a_i = 0$ for all
$i \geq 1$. If $\chr K = 0$, this implies that $a_i = 0$ for all $i \geq
1$, so $f$ is constant, which contradicts $f$ being irreducible. If $\chr K
= p$, then $i a_i = 0$ for all $i \geq 1$ is equivalent to $a_i = 0$
whenever $p \ndvd i$. 
\end{proof}

\begin{remark}
\label{rmk:sep-ff}
In the final chapter we will show that every irreducible polynomial over a
finite field is separable. So, it is only over infinite fields of
characteristic $p$ that you have to worry about inseparability.
\end{remark}

We now build up to showing that $|\Gal(M: K)| = [M: K]$ whenever $M: K$ is
a finite normal extension in which the minimal polynomial of every element
of $M$ is separable. First, some terminology:

\begin{defn}
Let $M: K$ be an algebraic extension. An element of $M$ is
\demph{separable} over $K$ if its minimal polynomial over $K$ is
separable. The extension $M: K$ is \demph{separable} if every element of
$M$ is separable over $K$.
\end{defn}

\begin{examples}
\label{egs:sep-ext}
\begin{enumerate}
\item 
\label{eg:se-zero}
Every algebraic extension of fields of characteristic $0$ is separable, by
Corollary~\ref{cor:sep-chars}. 

\item
\label{eg:se-fin}
Every algebraic extension of a finite field is separable, by
Remark~\ref{rmk:sep-ff}. 

\item
\label{eg:se-insep}
The splitting field of $t^p - u$ over $\F_p(u)$ is inseparable. Indeed, the
element denoted by $\alpha$ in Example~\ref{eg:insep-poly} is inseparable
over $\F_p(u)$, since its minimal polynomial is the inseparable polynomial
$t^p - u$.
\end{enumerate}
\end{examples}

\begin{ex}{ex:alg-int}
Let $M: L: K$ be field extensions. Show that if $M: K$ is algebraic then so
are $M: L$ and $L: K$.
\end{ex}

\begin{lemma}
\label{lemma:sep-int}
Let $M: L: K$ be field extensions, with $M: K$ algebraic. If $M: K$ is
separable then so are $M: L$ and $L: K$.
\end{lemma}

\begin{proof}
Both $M: L$ and $L: K$ are algebraic by Exercise~\ref{ex:alg-int}, so it
does make sense to ask whether they are separable. (We only defined what it
means for an \emph{algebraic} extension to be separable.) That $L: K$ is
separable is immediate from the definition. To show that $M: L$ is
separable, let $\alpha \in M$. Write $m_L$ and $m_K$ for the minimal
polynomials of $\alpha$ over $L$ and $K$, respectively. Then $m_K$ is an
annihilating polynomial of $\alpha$ over $L$, so $m_L \dvd m_K$ in
$L[t]$. Since $M: K$ is separable, $m_K$ splits into distinct linear
factors in $\SF_K(m_K)$. Since $m_L \dvd m_K$, so does $m_L$. Hence $m_L
\in L[t]$ is separable, so $\alpha$ is separable over $L$.
\end{proof}

As hinted in the introduction to this section, we will prove that $|\Gal(M:
K)| = [M: K]$ by refining Proposition~\ref{propn:ext-iso}.

\begin{propn}
\label{propn:ext-sep}
Let $\psi \from K \to K'$ be an isomorphism of fields, let $0 \neq f \in
K[t]$, let $M$ be a splitting field of $f$ over $K$, and let $M'$ be a
\marginpar{%
$\xymatrix@C=5mm@R=4mm{%
\scriptstyle \ \ \ \ M \ar@{.>}[r]^-\phi        &  
\scriptstyle M' \ \ \ \phantom{\psi_* f}                       \\
\scriptstyle f \ \ \ K \ar[r]_-\psi \ar@<-2.5mm>[u] &
\scriptstyle K' \ar@<5mm>[u] \ \ \ \psi_* f
}$%
}%
splitting field of $\psi_* f$ over $K'$. Suppose that the extension $M':
K'$ is separable. Then there are exactly $[M: K]$ isomorphisms $\phi \from
M \to M'$ extending $\psi$.
\end{propn}

\begin{proof}
This is almost the same as the proof of Proposition~\ref{propn:ext-iso},
but with the inequality $s \leq \deg(\psi_* m)$ replaced by an equality,
which holds by separability. For the inductive hypothesis to
go through, we need the extension $M': K'(\alpha'_j)$ to be separable, and
this follows from the separability of $M': K'$ by
Lemma~\ref{lemma:sep-int}.
\end{proof}

\begin{bigthm}
\label{thm:gal-size}
$|\Gal(M: K)| = [M: K]$ for every finite normal separable extension $M: K$.
\end{bigthm}

\begin{proof}
By Theorem~\ref{thm:sf-fin-nml}, $M = \SF_K(f)$ for some $f \in K[t]$. The
result follows from Proposition~\ref{propn:ext-sep}, taking $M' = M$, $K' =
K$, and $\psi = \id_K$.
\end{proof}

\begin{examples}
\label{egs:order-degree}
\begin{enumerate}
\item 
$|\Gal_K(f)| = [\SF_K(f): K]$ for any nonzero polynomial $f$ over a field
$K$ of characteristic $0$.

For instance, if $f(t) = t^3 -2$ then $[\SF_\Q(f): \Q] = 6$ by a similar
argument to Example~\ref{eg:tower-126}, using that $\SF_\Q(f)$ contains
elements of degree $2$ and $3$ over $\Q$. Hence $|\Gal_\Q(f)| = 6$. But
$\Gal_\Q(f)$ embeds into $S_3$ by Proposition~\ref{propn:gal-defns-same},
so $\Gal_\Q(f) \iso S_3$. We already proved this in
Example~\ref{eg:nf-cbrt2}, by a different argument.

\item
\label{eg:od-insep}
Consider $K = \F_p(u)$ and $M = \SF_K(t^p - u)$. With notation as in
Example~\ref{eg:insep-poly}, we have $M = K(\alpha)$, so $[M: K] =
\deg_K(\alpha) = p$. On the other hand, $|\Gal(M: K)| = 1$ by
Corollary~\ref{cor:gal-poly-distinct}. So Theorem~\ref{thm:gal-size} fails
if we drop the separability hypothesis.
\end{enumerate}
\end{examples}

\begin{digression}{dig:sep-sep}
With some effort, one can show that in any algebraic extension $M: K$, the
separable elements form a subfield of $M$. (See Stewart, Theorem~17.22.)
It follows that a finite extension $K(\alpha_1, \ldots, \alpha_n): K$ is
separable if and only if each $\alpha_i$ is. Hence a splitting field
extension $\SF_K(f): K$ is separable if and only if every root of $f$ is
separable in $\SF_K(f)$, which itself is equivalent to $f$ being separable
in the sense of Warning~\ref{wg:sep-irred}.

So: $\SF_K(f)$ is separable over $K$ if and only if $f$ is separable over
$K$. Thus, the different meanings of `separable' interact nicely.
\end{digression}

\begin{digression}{dig:primitive}
It's a stunning fact that every finite separable extension is simple. This
is called the theorem of the primitive element. For instance, whenever
$\alpha_1, \ldots, \alpha_n$ are complex numbers algebraic over $\Q$, there
is some $\alpha \in \C$ (a `primitive element') such that $\Q(\alpha_1,
\ldots, \alpha_n) = \Q(\alpha)$. We saw one case of this in
Example~\ref{egs:simple}\bref{eg:simple-Q23}: $\Q(\sqrt{2}, \sqrt{3}) =
\Q(\sqrt{2} + \sqrt{3})$.

The theorem of the primitive element was at the heart of most early
accounts of Galois theory, and is used in many modern treatments too, but
not this one.
\end{digression}

\section{Fixed fields}
\label{sec:fixed}

Write \demph{$\Aut(M)$} for the group of automorphisms of a field $M$. Then
$\Aut(M)$ acts naturally on $M$
(Example~\ref{egs:gp-actions}\bref{eg:ga-aut}). Given a subset $S$ of
$\Aut(M)$, we can consider the set $\Fix(S)$ of elements of $M$ fixed by
$S$ (Definition~\ref{defn:fixed-set}).

\begin{lemma}
\label{lemma:fix-subfd}
$\Fix(S)$ is a subfield of $M$, for any $S \sub \Aut(M)$.
\end{lemma}

\begin{proof}
$\Fix(S)$ is the equalizer $\Eq(S \cup \{\id_M\})$, which is a
subfield of $M$ by Lemma~\ref{lemma:eq-subfd}. 
\end{proof}

For this reason, we call $\Fix(S)$ the \demph{fixed field} of
$S$. 

\begin{ex}{ex:}
Using Lemma~\ref{lemma:fix-subfd}, show that every automorphism of a field
is an automorphism over its prime subfield. In other words, $\Aut(M) =
\Gal(M: K)$ whenever $M$ is a field with prime subfield $K$.
\end{ex}

Here's the big, ingenious, result about fixed fields. It will play a
crucial part in the proof of the fundamental theorem of Galois theory.

\begin{bigthm}
\label{thm:fixed}
Let $M$ be a field and $H$ a finite subgroup of $\Aut(M)$. Then $[M:
\Fix(H)] \leq |H|$.
\end{bigthm}

So the smaller $|H|$ is, the smaller $[M: \Fix(H)]$ must be, which means
that $\Fix(H)$ must be bigger. In other words, the smaller $|H|$ is, the
more of $M$ must be fixed by $H$. 

\begin{proof}
Write $n = |H|$. It is enough to prove that any $n + 1$ elements $\alpha_0,
\ldots, \alpha_n$ of $M$ are linearly dependent over $\Fix(H)$.

Write
\[
W =
\bigl\{ (x_0, \ldots, x_n) \in M^{n + 1} \such
x_0 \theta(\alpha_0) + \cdots + x_n \theta(\alpha_n) = 0
\text{ for all } \theta \in H \bigr\}.
\]
\video{The size of fixed fields}%
Then $W$ is defined by $n$ homogeneous linear equations in $n + 1$
variables, so it is a nontrivial $M$-linear subspace of $M^{n + 1}$.

\emph{Claim:} Let $(x_0, \ldots, x_n) \in W$ and $\phi \in H$. Then $(\phi(x_0),
\ldots, \phi(x_n)) \in W$.

\emph{Proof:} For all $\theta \in H$, we have 
\[
x_0 (\phi^{-1} \of \theta)(\alpha_0) + \cdots + x_n (\phi^{-1} \of
\theta)(\alpha_n) = 0,
\]
since $\phi^{-1} \of \theta \in H$. Applying $\phi$ to both sides gives
that for all $\theta \in H$,
\[
\phi(x_0) \theta(\alpha_0) + \cdots + \phi(x_n)\theta(\alpha_n) = 0,
\]
proving the claim.

Define the \emph{length} of a nonzero vector $\vc{x} = (x_0, \ldots, x_n)$
to be the unique number $\ell \in \{0, \ldots, n\}$ such that $x_\ell \neq
0$ but $x_{\ell + 1} = \cdots = x_n = 0$. Since $W$ is nontrivial, we can
choose an element $\vc{x}$ of $W$ of minimal length, $\ell$. Since $W$ is
closed under scalar multiplication by $M$, we may assume that $x_\ell =
1$. Thus, $\vc{x}$ is of the form $(x_0, \ldots, x_{\ell - 1}, 1, 0,
\ldots, 0)$. Since $\vc{x}$ is a nonzero element of $W$ of \emph{minimal}
length, the only element of $W$ of the form $(y_0, \ldots, y_{\ell - 1}, 0,
0, \ldots, 0)$ is $\vc{0}$.

We now show that $x_i \in \Fix(H)$ for all $i$. Let $\phi \in H$. By the
claim, $(\phi(x_0), \ldots, \phi(x_n)) \in W$. Put
\[
\vc{y} = (\phi(x_0) - x_0, \ldots, \phi(x_n) - x_n).
\]
Since $W$ is a linear subspace, $\vc{y} \in W$. Now $\phi(x_i) - x_i =
\phi(0) - 0 = 0$ for all $i > \ell$ and $\phi(x_\ell) - x_\ell = \phi(1) -
1 = 0$, so by the last sentence of the previous paragraph, $\vc{y} =
\vc{0}$. In other words, $\phi(x_i) = x_i$ for all $i$. This holds for all
$\phi \in H$, so $x_i \in \Fix(H)$ for all $i$.

We have shown that $W$ contains a nonzero vector $\vc{x} \in
\Fix(H)^{n + 1}$. But taking $\theta = \id$ in the definition of $W$ gives
$\sum x_i \alpha_i = 0$. Hence $\alpha_0, \ldots, \alpha_n$ are linearly
dependent over $\Fix(H)$.
\end{proof}

\begin{example}
Write $\kappa \from \C \to \C$ for complex conjugation. Then $H = \{\id,
\kappa\}$ is a subgroup of $\Aut(\C)$, and Theorem~\ref{thm:fixed} predicts
that $[\C: \Fix(H)] \leq 2$. Since $\Fix(H) = \R$, this is true.
\end{example}

\begin{ex}{ex:fix-size-eg}
Find another example of Theorem~\ref{thm:fixed}.\\
\end{ex}

\begin{digression}{dig:chars-li}
In fact, Theorem~\ref{thm:fixed} is an equality: $[M: \Fix(H)] = |H|$. This
is proved directly in many Galois theory books (e.g.\ Stewart,
Theorem~10.5). In our approach, it will be a \emph{consequence} of the
fundamental theorem of Galois theory rather than a step on the way to
proving it.

The reverse inequality, $[M: \Fix(H)] \geq |H|$, is closely related to the
result called `linear independence of characters'. (A good reference is
Lang, \emph{Algebra}, 3rd edition, Theorem 4.1.) Another instance of linear
independence of characters is that the functions $x \mapsto e^{2\pi i n x}$
($n \in \Z$) on $\R$ are linearly independent, a fundamental fact in the
theory of Fourier series.
\end{digression}

We finish by adding a further connecting strand between the concepts of
normal extension and normal subgroup, complementary to the strands in
Theorem~\ref{thm:nml-ftgt}. 

\begin{propn}
\label{propn:fix-nml}
Let $M: K$ be a finite normal extension and $H$ a normal subgroup of
$\Gal(M: K)$.  Then $\Fix(H)$ is a normal extension of $K$.
\end{propn}

\begin{proof}
Since every element of $H$ is an automorphism over $K$, the subfield
$\Fix(H)$ of $M$ contains $K$. For each $\phi \in \Gal(M: K)$, we have 
\[
\phi\Fix(H) = \Fix(\phi H \phi^{-1}) = \Fix(H),
\]
where the first equality holds by Lemma~\ref{lemma:act-fix-conj} and the second
because $H$ is normal in $\Gal(M: K)$. Hence by
Theorem~\ref{thm:nml-ftgt}\bref{part:nf-fix}, $\Fix(H): K$ is a normal
extension. 
\end{proof}

The stage is now set for the central result of the course: the fundamental
theorem of Galois theory.

\chapter{The fundamental theorem of Galois theory}
\label{ch:ftgt}

We've been building up to this moment all semester. Let's do it!%
\video{Introduction to Week~8}

\section{Introducing the Galois correspondence}
\label{sec:gal-corr}

Let $M: K$ be a field extension, with $K$ viewed as a subfield of $M$, as
usual. 

An \demph{intermediate field} of $M: K$ is a subfield of $M$
containing $K$. Write
\[
\dem{\FF} = \{ \text{intermediate fields of } M: K \}.
\]
For $L \in \FF$, we draw diagrams like this:
\[
\xymatrix@R-1ex{
M \ar@{-}[d] \\
L \ar@{-}[d] \\
K,
}
\]
with the bigger fields higher up. 

Also write
\[
\dem{\GG} = \{ \text{subgroups of } \Gal(M: K) \}.
\]
For $H \in \GG$, we draw diagrams like this:
\[
\xymatrix@R-1ex{
1 \ar@{-}[d] \\
H \ar@{-}[d] \\
\Gal(M: K).
}
\]
Here $1$ denotes the trivial subgroup and the bigger groups are
\emph{lower down}. It will become clear soon why we're using opposite
conventions.

For $L \in \FF$, the group $\Gal(M: L)$ consists of all automorphisms $\phi$
of $M$ that fix each element of $L$. Since $K \sub L$, any
such $\phi$ certainly fixes each element of $K$. Hence $\Gal(M: L)$ is a
subgroup of $\Gal(M: K)$. This process defines a function
\[
\begin{array}{cccc}
\dem{\Gal(M: -)}\from   &
\FF                     &\to            &\GG    \\
                &L      &\mapsto        &\Gal(M: L).
\end{array}
\]
In the expression $\Gal(M: -)$, the symbol $-$ should be seen as a blank
space into which arguments can be inserted.

\begin{warning}{wg:which-gal}
The group we're associating with $L$ is $\Gal(M: L)$, not $\Gal(L: K)$!
Both groups matter, but only one is a subgroup of $\Gal(M: K)$, which is
what we're interested in here.

We showed just now that $\Gal(M: L)$ is a subgroup of $\Gal(M: K)$. If you
wanted to show that $\Gal(L: K)$ is (isomorphic to) a subgroup of $\Gal(M:
K)$---which it isn't---then you'd probably do it by trying to prove that
every automorphism of $L$ over $K$ extends uniquely to $M$. And that's
false. For instance, when $L = K$, the identity on $L$ typically has
\emph{many} extensions to $M$: they're the elements of $\Gal(M: K)$.

Although $\Gal(L: K)$ isn't a subgroup of $\Gal(M: K)$, it is a
\emph{quotient} of it, at least when both extensions are finite and
normal. We saw this in Theorem~\ref{thm:nml-ftgt}, and we'll come back to
it in Section~\ref{sec:ftgt}.
\end{warning}

In the other direction, for $H \in \GG$, the subfield $\Fix(H)$ of $M$
contains $K$. Indeed, $H \sub \Gal(M: K)$, and by definition, every element
of $\Gal(M: K)$ fixes every element of $K$, so $\Fix(H) \supseteq K$. Hence
$\Fix(H)$ is an intermediate field of $M: K$. This process defines a
function
\[
\begin{array}{cccc}
\dem{\Fix}\from &\GG    &\to            &\FF    \\
                &H      &\mapsto        &\Fix(H).
\end{array}
\]
We have now defined functions
\[
\xymatrix@C+2em{
\FF \ar@<.5ex>[r]^{\Gal(M: -)}  &
\GG. \ar@<.5ex>[l]^{\Fix}
}
\]
The fundamental theorem of Galois theory tells us how these functions
behave: how the concepts of Galois group and fixed field interact. It
will bring together most of the big results we've proved so far, and
will assume that the extension is finite, normal and separable. But first,
let's say the simple things that are true for all extensions:

\begin{lemma}
\label{lemma:gc-simple}
Let $M: K$ be a field extension, and define $\FF$ and $\GG$ as above.
\begin{enumerate}
\item 
\label{part:gcs-rev}
For $L_1, L_2 \in \FF$,
\[
L_1 \sub L_2
\implies
\Gal(M: L_1) \supseteq \Gal(M: L_2)
\]
(Figure~\ref{fig:gcs-rev}). For $H_1, H_2 \in \GG$,
\[
H_1 \sub H_2
\implies
\Fix(H_1) \supseteq \Fix(H_2).
\]

\item
\label{part:gcs-adj}
For $L \in \FF$ and $H \in \GG$,
\[
L \sub \Fix(H) \iff H \sub \Gal(M: L).
\]

\item
\label{part:gcs-units}
For all $L \in \FF$,
\[
L \sub \Fix(\Gal(M: L)).
\]
For all $H \in \GG$,
\[
H \sub \Gal(M: \Fix(H)).
\]
\end{enumerate}
\end{lemma}

\begin{figure}
\[
\xymatrix@R-1ex{
M \vphantom{1}
\ar@{-}[d] \\
L_2 \vphantom{{}}\vphantom{\Gal(M: L_2)}
\ar@{-}[d] \\
L_1 \vphantom{{}}\vphantom{\Gal(M: L_1)}
\ar@{-}[d] \\
K \vphantom{{}}\vphantom{\Gal(M: K)}
}
\qquad\qquad\qquad
\xymatrix@R-1ex{
1 \ar@{-}[d] \\
\Gal(M: L_2) \ar@{-}[d] \\
\Gal(M: L_1) \ar@{-}[d] \\
\Gal(M: K)
}
\]
\caption{The function $L \mapsto \Gal(M: L)$ is order-reversing
(Lemma~\ref{lemma:gc-simple}\bref{part:gcs-rev}).} 
\label{fig:gcs-rev}
\end{figure}

\begin{warning}{wg:rev-pres}
In part~\bref{part:gcs-rev}, the functions $\Gal(M: -)$ and
$\Fix$ \emph{reverse} inclusions. The bigger you make $L$, the smaller you
make $\Gal(M: L)$, because it gets harder for an automorphism to fix
everything in $L$. And the bigger you make $H$, the smaller you make $\Fix(H)$,
because it gets harder for an element of $M$ to be fixed by everything in
$H$. That's why the field and group diagrams are opposite ways up.
\end{warning}

\begin{proof}
\bref{part:gcs-rev}: I leave the first half as an exercise. For the second,
suppose that $H_1 \sub H_2$, and let $\alpha \in \Fix(H_2)$. Then
$\theta(\alpha) = \alpha$ for all $\theta \in H_2$, so $\theta(\alpha) =
\alpha$ for all $\theta \in H_1$, so $\alpha \in \Fix(H_1)$.

\bref{part:gcs-adj}: both sides are equivalent to the
statement that $\theta(\alpha) = \alpha$ for all $\theta \in H$ and $\alpha
\in L$.

\bref{part:gcs-units}: the first statement follows from the $\impliedby$
direction of~\bref{part:gcs-adj} by taking $H = \Gal(M: L)$, and the second
follows from the $\implies$ direction of~\bref{part:gcs-adj} by taking $L =
\Fix(H)$. (Or, they can be proved directly.)
\end{proof}

\begin{ex}{ex:gcs-rev-pf}
Prove the first half of Lemma~\ref{lemma:gc-simple}\bref{part:gcs-rev}.\\
\end{ex}

\begin{ex}{ex:gcs-rev-diag}
Draw a diagram like Figure~\ref{fig:gcs-rev} for the second half of
Lemma~\ref{lemma:gc-simple}\bref{part:gcs-rev}. 
\end{ex}

\begin{digression}{dig:gal-conn}
If you've done some algebraic geometry, the formal structure of
Lemma~\ref{lemma:gc-simple} might seem familiar. Given a field $K$ and a
natural number $n$, we can form the set $\FF$ of subsets of $K^n$ and the
set $\GG$ of ideals of $K[t_1, \ldots, t_n]$, and there are functions $\FF
\oppairu \GG$ defined by taking the annihilating ideal of a subset of $K^n$
and the zero-set of an ideal of $K[t_1, \ldots, t_n]$. The analogue of 
Lemma~\ref{lemma:gc-simple} holds.

In general, a pair of ordered sets $\FF$ and $\GG$ equipped with functions
$\FF \oppairu \GG$ satisfying the properties in Lemma~\ref{lemma:gc-simple}
is called a \demph{Galois connection}. This in turn is a special case of
the category-theoretic notion of adjoint functors.
\end{digression}

The functions
\[
\xymatrix@C+2em{
\FF \ar@<.5ex>[r]^{\Gal(M: -)}  &
\GG. \ar@<.5ex>[l]^{\Fix}
}
\]
are called the \demph{Galois correspondence} for $M: K$. This terminology
is mostly used in the case where the functions are \demph{mutually inverse},
meaning that
\[
L = \Fix(\Gal(M: L)),
\qquad
H = \Gal(M: \Fix(H))
\]
for all $L \in \FF$ and $H \in \GG$. We saw in
Lemma~\ref{lemma:gc-simple}\bref{part:gcs-units} that in both cases, the
left-hand side is a subset of the right-hand side. But they are not always
equal: 

\begin{example}
Let $M: K$ be $\Q(\sqrt[3]{2}): \Q$. Since $[M: K]$ is $3$, which is a
prime number, the tower law implies that there are no nontrivial
intermediate fields: $\FF = \{M, K\}$. We saw in
Example~\ref{egs:gal-ext}\bref{eg:gal-ext-cbrt2} that $\Gal(M: K)$ is
trivial, so $\GG = \{\Gal(M: K)\}$. Hence $\FF$ has two elements and $\GG$
has only one. This makes it impossible for there to be mutually inverse
functions between $\FF$ and $\GG$. Specifically, what goes wrong is that
\[
\Fix\bigl(\Gal\bigl(\Q\bigl(\sqrt[3]{2}\bigr): \Q\bigr)\bigr) 
= 
\Fix\bigl(\bigl\{\id_{\Q(\sqrt[3]{2})}\bigr\}\bigr) 
=
\Q\bigl(\sqrt[3]{2}\bigr) 
\neq 
\Q.
\]
\end{example}

\begin{ex}{ex:not-inv-sep}
Let $p$ be a prime number, let $K = \F_p(u)$, and let $M$ be the splitting
field of $t^p - u$ over $K$, as in Examples~\ref{eg:insep-poly}
and~\ref{egs:order-degree}\bref{eg:od-insep}. Prove that $\Gal(M: -)$ and
$\Fix$ are not mutually inverse.
\end{ex}

If $\Gal(M: -)$ and $\Fix$ \emph{are} mutually inverse then they set up a
one-to-one correspondence between the set $\FF$ of intermediate fields of
$M: K$ and the set $\GG$ of subgroups of $\Gal(M: K)$. The fundamental
theorem of Galois theory tells us that this dream comes true when $M: K$ is
finite, normal and separable. And it tells us more besides.

\section{The theorem}
\label{sec:ftgt}

The moment has come.

\begin{megathm}[Fundamental theorem of Galois theory]
\label{thm:ftgt}
Let $M: K$ be a finite normal separable extension. Write
\begin{align*}
\FF     &
=
\{ \text{intermediate fields of } M: K \},      \\
\GG     &
=
\{ \text{subgroups of } \Gal(M: K) \}.
\end{align*}
\begin{enumerate}
\item
\label{part:ftgt-inv}
The functions $\xymatrix@1@C+1.5em{
\FF \ar@<.5ex>[r]^{\Gal(M: -)}  &
\GG \ar@<.5ex>[l]^{\Fix}
}$ are mutually inverse.

\item
\label{part:ftgt-od}
$|\Gal(M: L)| = [M: L]$ for all $L \in \FF$, and $[M: \Fix(H)] = |H|$ for
all $H \in \GG$.

\item
\label{part:ftgt-nml}
Let $L \in \FF$. Then 
\begin{align*}
&
L \text{ is a normal extension of } K   \\
\iff    &
\Gal(M: L) \text{ is a normal subgroup of } \Gal(M: K),
\end{align*}
and in that case,
\[
\frac{\Gal(M: K)}{\Gal(M: L)} \iso \Gal(L: K).
\]
\end{enumerate}
\end{megathm}

\begin{proof}
First note that for each $L \in \FF$, the extension $M: L$ is finite and
normal (by Corollary~\ref{cor:top-nml}) and separable (by
Lemma~\ref{lemma:sep-int}). Also, the group $\Gal(M: K)$ is finite (by
Theorem~\ref{thm:gal-size}), so every subgroup is finite too.

We prove~\bref{part:ftgt-inv} and~\bref{part:ftgt-od} together. First let
$H \in \GG$. We have
\begin{align}
\label{eq:ftgt-ineqs}
|H|
\leq 
|\Gal(M: \Fix(H))|
=
[M: \Fix(H)]
\leq
|H|,
\end{align}
where the first inequality holds because $H \sub \Gal(M: \Fix(H))$
(Lemma~\ref{lemma:gc-simple}\bref{part:gcs-units}), the equality follows from
Theorem~\ref{thm:gal-size} (since
$M: \Fix(H)$ is finite, normal and separable), and the second inequality
follows from Theorem~\ref{thm:fixed} (since $H$ is finite). So equality
holds throughout~\eqref{eq:ftgt-ineqs}, giving
\[
H = \Gal(M: \Fix(H)),
\qquad
[M: \Fix(H)] = |H|.
\]

Now let $L \in \FF$. We have
\[
[M: \Fix(\Gal(M: L))] 
=
|\Gal(M: L)| 
=
[M: L],
\]
where the first equality follows from the previous paragraph by taking $H =
\Gal(M: L)$, and the second follows from Theorem~\ref{thm:gal-size}. But $L
\sub \Fix(\Gal(M:L))$ by Lemma~\ref{lemma:gc-simple}\bref{part:gcs-units},
so $L = \Fix(\Gal(M:L))$ by
\href{https://www.maths.ed.ac.uk/~tl/galois}{Workshop~3, question~3}. This
completes the proof of~\bref{part:ftgt-inv} and~\bref{part:ftgt-od}.

We have already proved most of~\bref{part:ftgt-nml} as
Theorem~\ref{thm:nml-ftgt}\bref{part:nf-main}. It only remains to show that
whenever $L$ is an intermediate field such that $\Gal(M: L)$ is a normal
subgroup of $\Gal(M: K)$, then $L$ is a normal extension of $K$. By
Proposition~\ref{propn:fix-nml}, $\Fix(\Gal(M: L)): K$ is normal. But
by~\bref{part:ftgt-inv}, $\Fix(\Gal(M: L)) = L$, so $L: K$ is normal, as
required.
\end{proof}

The fundamental theorem of Galois theory is about field extensions that are
finite, normal and separable. Let's take a moment to think about what those
conditions mean.

An extension $M: K$ is finite and normal if and only if $M$ is the
splitting field of some polynomial over $K$
(Theorem~\ref{thm:sf-fin-nml}). So, the theorem can be understood as a
result about splitting fields of polynomials.

Not every splitting field extension is separable
(Example~\ref{egs:sep-ext}\bref{eg:se-insep}). However, we know of two
settings where separability is guaranteed. The first is fields of
characteristic zero (Example~\ref{egs:sep-ext}\bref{eg:se-zero}). The most
important of these is $\Q$, which is our focus in this chapter: we'll
consider examples in which $M: K$ is the splitting field extension of a
polynomial over $K = \Q$. The second is where the fields are finite
(Example~\ref{egs:sep-ext}\bref{eg:se-fin}). We'll come to finite fields in
the final chapter.

\begin{digression}{dig:ftgt-fin}
Normality and separability are core requirements of Galois theory, but
there are extensions of the fundamental theorem (well beyond this course)
in which the finiteness condition on $M: K$ is relaxed.

The first level of relaxation replaces `finite' by `algebraic'. Then
$\Gal(M: K)$ is no longer a finite group, but it does acquire an
interesting topology. One example is where $M$ is the algebraic closure
$\ovln{K}$ of $K$, and $\Gal(\ovln{K}: K)$ is called the \demph{absolute
Galois group} of $K$ (at least when $\chr K = 0$). It contains \emph{all}
splitting fields of polynomials over $K$, so to study it is to study all
polynomials over $K$ at once.

Going further, we can even drop the condition that the extension is
algebraic. In this realm, we need the notion of `transcendence degree',
which counts how many algebraically independent elements can be found in
the extension. You may have met this if you're taking Algebraic Geometry.
\end{digression}

You'll want to see some examples!  Section~\ref{sec:ftgt-eg} is devoted to
a single example of the fundamental theorem, showing every aspect of the
theorem in all its glory. I'll give a couple of simpler examples in
a moment, but before that, it's helpful to review some of what we
did earlier:

\begin{remark}
When working out the details of the Galois correspondence for a polynomial
$f \in K[t]$, it's not only the fundamental theorem that's useful. Some of
our earlier results also come in handy, such as the following.
\begin{enumerate}
\item 
Lemmas~\ref{lemma:gal-acts} and~\ref{lemma:gal-acts-faithfully} state that
$\Gal_K(f)$ acts faithfully on the set of roots
of $f$ in $\SF_K(f)$. That is, an element of the Galois group can be
understood as a permutation of the roots. 

\item
Corollary~\ref{cor:gal-poly-distinct} states that $|\Gal_K(f)|$ divides
$k!$, where $k$ is the number of distinct roots of $f$ in its splitting
field.

\item
Let $\alpha$ and $\beta$ be roots of $f$ in $\SF_K(f)$. Then there is an
element of the Galois group mapping $\alpha$ to $\beta$ if and only if
$\alpha$ and $\beta$ are conjugate over $K$ (have the same minimal
polynomial). This follows from Proposition~\ref{propn:orbit-cc}.

\item
In particular, when $f$ is irreducible, the action of the Galois group on
the roots is transitive (Corollary~\ref{cor:act-trans}). See
Example~\ref{eg:spectacular} for an illustration of the power of this
principle.
\end{enumerate}
\end{remark}

\begin{example}
\label{eg:ftgt-prime}
Let $M: K$ be a normal separable extension of prime degree~$p$. 
By the fundamental theorem, $|\Gal(M: K)| = [M: K] = p$. Every
group of prime order is cyclic, so $\Gal(M: K) \iso C_p$. By the
tower law, $M: K$ has no nontrivial intermediate fields, and by Lagrange's
theorem, $\Gal(M: K)$ has no nontrivial subgroups. So $\FF = \{M, K\}$ and
$\GG = \{1, \Gal(M: K)\}$:
\[
\xymatrix@R-1ex@C+1em{
M \ar@{-}[d]    &       &1 \ar@{-}[d]   \\
K               &       &\Gal(M: K)
}
\]
Both $M$ and $K$ are normal extensions of $K$, and both $1$ and $\Gal(M:
K)$ are normal subgroups of $\Gal(M: K)$.
\end{example}

\begin{example}
\label{eg:ftgt-2i}
Let $f(t) = (t^2 + 1)(t^2 - 2) \in \Q[t]$. Put $M = \SF_\Q(f) =
\Q(\sqrt{2}, i)$ and $G = \Gal(M: K) = \Gal_\Q(f)$. Then $M: K$ is a finite
normal separable extension, so the fundamental theorem applies.  We already
calculated $G$ in a sketchy way in
Example~\ref{egs:gal-gp-first}\bref{eg:ggf-klein}. Let's do it again in
full, using what we now know.

First,
\[
[M: K] 
=
\bigl[\Q\bigl(\sqrt{2}, i\bigr): \Q\bigl(\sqrt{2}\bigr)\bigr]
\bigl[\Q\bigl(\sqrt{2}\bigr): \Q\bigr] = 2 \times 2 = 4
\]
(much as in Example~\ref{eg:deg-23}). 

Now consider how $G$ acts on the set $\{\pm \sqrt{2}, \pm i\}$ of
roots of $f$. The conjugacy class of $\sqrt{2}$ is $\{\sqrt{2},
-\sqrt{2}\}$, so for each $\phi \in G$ we have $\phi(\sqrt{2}) = \pm
\sqrt{2}$. Similarly, $\phi(i) = \pm i$ for each $\phi \in G$. The two
choices of sign determine $\phi$ entirely, so $|G| \leq 4$. But by the
fundamental theorem, $|G| = [M: K] = 4$, so each of the four possibilities
does in fact occur. So $G = \{\id, \phi_{+-}, \phi_{-+}, \phi_{--}\}$,
where
\begin{align*}
\phi_{+-}\bigl(\sqrt{2}\bigr)     &= \sqrt{2},    &
\phi_{-+}\bigl(\sqrt{2}\bigr)     &= -\sqrt{2},   &
\phi_{--}\bigl(\sqrt{2}\bigr)     &= -\sqrt{2},   \\
\phi_{+-}(i)            &= -i,          &
\phi_{-+}(i)            &= i,           &
\phi_{--}(i)            &= -i.
\end{align*}
The only two groups of order $4$ are $C_4$ and $C_2 \times C_2$, and 
each element of $G$ has order $1$ or $2$, so $G \iso C_2 \times C_2$. 

The subgroups of $G$ are
\begin{align}
\label{eq:2i-gps}
\begin{array}{c}
\xymatrix@R-1ex{
        &1      &       \\
\gpgen{\phi_{+-}} \ar@{-}[ru] \ar@{-}[rd]       &
\gpgen{\phi_{-+}} \ar@{-}[u] \ar@{-}[d]         &
\gpgen{\phi_{--}} \ar@{-}[lu] \ar@{-}[ld]       \\
        &G
}
\end{array}
\end{align}
where lines indicate inclusions.  Here $\gpgen{\phi_{+-}}$ is the subgroup
generated by $\phi_{+-}$, which is $\{\id, \phi_{+-}\}$, and similarly for
$\phi_{-+}$ and $\phi_{--}$.

What are the fixed fields of these subgroups? The fundamental theorem
implies that $\Fix(G) = \Q$. Also, $\phi_{+-}(\sqrt{2}) = \sqrt{2}$, so
$\Q(\sqrt{2}) \sub \Fix(\gpgen{\phi_{+-}})$. But
\[
\bigl[\Q\bigl(\sqrt{2}, i\bigr): \Q\bigl(\sqrt{2}\bigr)\bigr]
=
2
=
|\gpgen{\phi_{+-}}|
=
\bigl[\Q\bigl(\sqrt{2}, i\bigr): \Fix(\gpgen{\phi_{+-}})\bigr]
\]
(where the last step is by the fundamental theorem), so
$\Q(\sqrt{2}) = \Fix(\gpgen{\phi_{+-}})$.

Let's reflect for a moment on the argument in the last paragraph, because
it's one you'll need to master. We have a subgroup
$H$ of $\Gal(M: K)$ (here, $H = \gpgen{\phi_{+-}}$) and we want to find its
fixed field. We do this in three steps:
\begin{itemize}
\item
First, we spot some elements $\alpha_1, \ldots,
\alpha_r$ fixed by $H$. (Here, $r = 1$ and $\alpha_1 = \sqrt{2}$.) It
follows that $K(\alpha_1, \ldots, \alpha_r) \sub \Fix(H)$.

\item
Next, we check that $[M: K(\alpha_1, \ldots, \alpha_r)] = |H|$. If they're
not equal, that means we didn't spot enough elements fixed by $H$.

\item
Finally, we apply a standard argument:
\[
[M: K(\alpha_1, \ldots, \alpha_r)]
=
|H|
=
[M: \Fix(H)]
\]
(using the fundamental theorem), so by the tower law,
\[
[\Fix(H): K(\alpha_1, \ldots, \alpha_r)]
=
\frac{[M: K(\alpha_1, \ldots, \alpha_r)]}{[M: \Fix(H)]}
=
1,
\]
giving $\Fix(H) = K(\alpha_1, \ldots, \alpha_r)$.
\end{itemize}
The strategy is similar to one you've met in linear algebra: to prove that
two subspaces of a vector space are equal, show that one is a subspace of
the other and that they have the same dimension.

Similar arguments apply to $\phi_{-+}$ and
$\phi_{--}$, so the fixed fields of the groups in diagram~\eqref{eq:2i-gps}
are
\begin{align}
\label{eq:2i-fds}
\begin{array}{c}
\xymatrix@R-1ex{
        &\Q\bigl(\sqrt{2}, i\bigr)      &       \\
\Q\bigl(\sqrt{2}\bigr) \ar@{-}[ru] \ar@{-}[rd]       &
\Q(i) \ar@{-}[u] \ar@{-}[d]         &
\Q\bigl(\sqrt{2}i\bigr) \ar@{-}[lu] \ar@{-}[ld]       \\
        &\Q
}
\end{array}
\end{align}
Equivalently, the groups in~\eqref{eq:2i-gps} are the Galois groups
of $\Q(\sqrt{2}, i)$ over the fields in~\eqref{eq:2i-fds}. For
instance, $\Gal(\Q(\sqrt{2}, i): \Q(i)) = \gpgen{\phi_{-+}}$. 

Since the overall Galois group $G \iso C_2 \times C_2$ is abelian, every
subgroup is normal. Hence all the extensions in diagram~\eqref{eq:2i-fds}
are normal too.
\end{example}

\begin{ex}{ex:ptc-nml}
In this particular example, one can also see more directly that all the
extensions in~\eqref{eq:2i-fds} are normal. How?
\end{ex}

Like any big theorem, the fundamental theorem of Galois theory has some
important corollaries. Here's one.

\begin{cor}
Let $M: K$ be a finite normal separable extension. Then for every $\alpha
\in M \without K$, there is some automorphism $\phi$ of $M$ over $K$ such
that $\phi(\alpha) \neq \alpha$. 
\end{cor}

\begin{proof}
Theorem~\ref{thm:ftgt}\bref{part:ftgt-inv} implies that $\Fix(\Gal(M: K)) =
K$. Now $\alpha \not\in K$, so $\alpha \not\in \Fix(\Gal(M: K))$, which is what
had to be proved.
\end{proof}

\begin{example}
For any $f \in \Q[t]$ and irrational $\alpha \in \SF_\Q(f)$,
there is some $\phi \in \Gal_\Q(f)$ that does not fix $\alpha$. This is
clear if $\alpha \not\in \R$, as we can take $\phi$ to be complex
conjugation restricted to $\SF_\Q(f)$. But it is not so obvious otherwise.
\end{example}

\section{A specific example}
\label{sec:ftgt-eg}

Chapter~13 of Stewart's book opens with these words:
\begin{quote}
The extension that we discuss is a favourite with writers on Galois theory,
because of its archetypal quality. A simpler example would be too small to
illustrate the theory adequately, and anything more complicated would be
unwieldy. The example is the Galois group of the splitting field of $t^4 -
2$ over $\Q$.
\end{quote}
We go through the same example here. My presentation of it is different
from Stewart's, so you can consult his book if anything that follows is
unclear.

Write $f(t) = t^4 - 2 \in \Q[t]$, which is irreducible by Eisenstein's
criterion. Write $G = \Gal_\Q(f)$.

\paragraph*{Splitting field} Write $\xi$ for the unique real positive root
of $f$. Then the roots of $f$ are $\pm \xi$ and $\pm \xi i$
(Figure~\ref{fig:big-eg-roots}). So $\SF_\Q(f) = \Q(\xi, \xi i) = \Q(\xi,
i)$. We have
\[
[\Q(\xi, i): \Q]
=
[\Q(\xi, i): \Q(\xi)] [\Q(\xi): \Q]
=
2 \times 4 
=
8,
\]
where the first factor is $2$ because $\Q(\xi) \sub \R$ and the second
factor is $4$ because $f$ is the minimal polynomial of $\xi$ over $\Q$ (being
irreducible). By the fundamental theorem, $|G| = 8$.

\begin{figure}
\centering
\setlength{\unitlength}{1mm}
\setlength{\fboxsep}{0mm}
\begin{picture}(100,40)
\cell{50}{20}{c}{\includegraphics[height=40\unitlength]{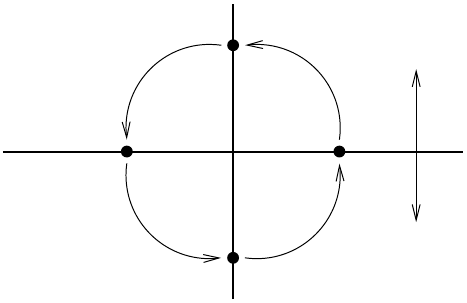}}
\cell{38}{33}{c}{$\rho$}
\cell{76}{22}{c}{$\kappa$}
\cell{62}{22}{c}{$\xi$}
\cell{39}{22}{c}{$-\xi$}
\cell{52}{31}{c}{$\xi i$}
\cell{53}{8.5}{c}{$-\xi i$}
\end{picture}%
\caption{The roots of $f$, and the effects on them of $\rho, \kappa \in
\Gal_\Q(f)$.}
\label{fig:big-eg-roots}
\end{figure}

\paragraph*{Galois group} We now look for the $8$ elements of the Galois
group. We'll use the principle that if $\phi, \theta \in G$ with $\phi(\xi)
= \theta(\xi)$ and $\phi(i) = \theta(i)$ then $\phi = \theta$ (by
Lemma~\ref{lemma:gen-epic}). 

Complex conjugation on $\C$ restricts to a nontrivial automorphism $\kappa$ of
$\Q(\xi, i)$ over $\Q$, giving an element $\kappa \in G$ of order $2$.

I now claim that $G$ has an element $\rho$ satisfying $\rho(\xi) = \xi i$
and $\rho(i) = i$. In that case, $\rho$ will act on the roots of $f$ as
follows:
\[
\xi \mapsto \xi i \mapsto -\xi \mapsto -\xi i \mapsto \xi
\]
(Figure~\ref{fig:big-eg-roots}). This element $\rho$ will have order $4$.

Proof of claim: since $f$ is irreducible, $G$ acts transitively on the
roots of $f$ in $\SF_\Q(f)$, so there is some $\phi \in G$ such that
$\phi(\xi) = \xi i$. The conjugacy class of $i$ over $\Q$ is $\{\pm i\}$,
so $\phi(i) = \pm i$. If $\phi(i) = i$ then we can take $\rho = \phi$. If
$\phi(i) = -i$ then
\[
(\phi \of \kappa)(\xi) = \phi(\xi) = \xi i,
\qquad
(\phi \of \kappa)(i) = \phi(-i) = -\phi(i) = i,
\]
so we can take $\rho = \phi \of \kappa$. 

(From now on, I will usually omit the $\of$ sign and write things like
$\phi\kappa$ instead. Of course, juxtaposition is also used to mean
multiplication, as in $\xi i$. But confusion shouldn't arise: automorphisms
are composed and numbers are multiplied.)

Figure~\ref{fig:big-eg-roots} might make us wonder if $G$ is the
dihedral group $D_4$, the symmetry group of the square. We will see that
it is.

\begin{warning}{wg:dih}
The symmetry group of a regular $n$-sided polygon has $2n$ elements: $n$
rotations and $n$ reflections. Some authors call it $D_n$ and others call it
$D_{2n}$. I will call it $D_n$, as in the Group Theory course.
\end{warning}

If $G \iso D_4$, we should have $\kappa\rho = \rho^{-1}\kappa$. (This is
one of the defining equations of the dihedral group; you saw it in
Example~3.2.12 of Group Theory.)  Let's check this. We have
\begin{align*}
\kappa\rho(\xi)         &
= \ovln{\xi i} = -\xi i,        &
\rho^{-1}\kappa(\xi)    &
= \rho^{-1}(\xi) = -\xi i,      \\
\kappa\rho(i)   &
= \kappa(i) = -i,       &
\rho^{-1}\kappa(i)      &
= \rho^{-1}(-i) = -i,
\end{align*}
so $\kappa\rho$ and $\rho^{-1}\kappa$ are equal on $\xi$ and $i$, so
$\kappa\rho = \rho^{-1}\kappa$. It follows that $\kappa\rho^r =
\rho^{-r}\kappa$ for all $r \in \Z$.

Figure~\ref{fig:big-eg-table} shows the effect of $8$ elements of $G$ on
$\xi$, $i$ and $\xi i$. Since no two of them have the same effect on both
$\xi$ and $i$, they are all \emph{distinct} elements of $G$. Since $|G| =
8$, they are the only elements of $G$. So $G \iso D_4$.

\begin{figure}
\begin{tabular}{c|ccc|c|c}
$\phi \in G$    &$\phi(\xi)$    &$\phi(i)$      &$\phi(\xi i)$  &
order   &geometric description     \\
                &               &               &               &
        &(see Warning~\ref{wg:geom-descr})      \\
\hline
$\id$           &$\xi$          &$i$            &$\xi i$        &
$1$     &identity       \\
$\rho$          &$\xi i$        &$i$            &$-\xi$         &
$4$     &rotation by $\pi/2$    \\
$\rho^2$        &$-\xi$         &$i$            &$-\xi i$       &
$2$     &rotation by $\pi$      \\
$\rho^3 = \rho^{-1}$&$-\xi i$   &$i$            &$\xi$          &
$4$     &rotation by $-\pi/2$   \\
$\kappa$        &$\xi$          &$-i$           &$-\xi i$       &
$2$     &reflection in real axis\\
$\kappa\rho = \rho^{-1}\kappa$  
                &$-\xi i$       &$-i$           &$-\xi$         &
$2$     &reflection in axis through $1 - i$     \\
$\kappa\rho^2 = \rho^2\kappa$  
                &$-\xi$         &$-i$           &$\xi i$         &
$2$     &reflection in imaginary axis   \\
$\kappa\rho^{-1} = \rho\kappa$  
                &$\xi i$        &$-i$           &$\xi$          &
$2$     &reflection in axis through $1 + i$     
\end{tabular}
\caption{The Galois group of $t^4 - 2$ over $\Q$.}
\label{fig:big-eg-table}
\end{figure}

\begin{warning}{wg:geom-descr}
The `geometric description' in Figure~\ref{fig:big-eg-table} applies only
to the roots, not the whole of the splitting field $\Q(\xi, i)$. For
example, $\rho^2$ is rotation by $\pi$ \emph{on the set of roots}, but it
is not rotation by $\pi$ on the rest of $\Q(\xi, i)$: it fixes each element
of $\Q$, for instance.
\end{warning}

\paragraph*{Subgroups of the Galois group} 
Since $|G| = 8$, any nontrivial proper subgroup of $G$ has order $2$ or
$4$. Let's look in turn at subgroups of order $2$ and $4$, also determining
which ones are normal. This is pure group theory, with no mention of
fields. 

\begin{itemize}
\item
The subgroups of order $2$ are of the form $\gpgen{\phi} = \{\id, \phi\}$
where $\phi \in G$ has order $2$. So, they are 
\[
\gpgen{\rho^2},\ 
\gpgen{\kappa},\
\gpgen{\kappa\rho},\
\gpgen{\kappa\rho^2},\
\gpgen{\kappa\rho^{-1}}.
\]
If you watched the video \href{https://www.maths.ed.ac.uk/~tl/galois}{`What
does it mean to be normal?'}, you may be able to guess which of these
subgroups are normal in $G$, the symmetry group of the square. It should be
those that can be specified without referring to particular vertices or
edges of the square. So, just the first should be normal. Let's check.

We know that $\kappa\rho^2 = \rho^2\kappa$, so $\rho^2$ commutes with both
$\kappa$ and $\rho$, which generate $G$. Hence $\rho^2$ is in the centre
of $G$ (commutes with everything in $G$). It follows that
$\gpgen{\rho^2}$ is a normal subgroup of $G$. On the other hand, for each
$r \in \Z$, the subgroup $\gpgen{\kappa\rho^r}$ is not normal, since
\[
\rho(\kappa\rho^r)\rho^{-1} = (\rho\kappa)\rho^{r - 1} =
(\kappa\rho^{-1})\rho^{r - 1} = \kappa\rho^{r - 2} 
\not\in \gpgen{\kappa\rho^r}.
\]
 
\item
The subgroups of $G$ of order $4$ are isomorphic to either $C_4$ or $C_2
\times C_2$, since these are the only groups of order $4$.

The only elements of $G$ of order $4$ are $\rho^{\pm 1}$, so the only
subgroup of $G$ isomorphic to $C_4$ is $\gpgen{\rho} = \{\id, \rho, \rho^2,
\rho^3 = \rho^{-1}\}$.

Now consider subgroups $H$ of $G$ isomorphic to $C_2 \times C_2$.

\begin{ex}{ex:not-all-refls}
Show that every such $H$ must contain $\rho^2$. (Hint: think
geometrically.)
\end{ex}

We have $\rho^2 \in H$, and both other nonidentity elements of $H$ have
order $2$, so they are of the form $\kappa\rho^r$ for some $r \in \Z$. The
two such subgroups $H$ are
\begin{align*}
\gpgen{\kappa, \rho^2}  &
=
\{ \id, \kappa, \rho^2, \kappa\rho^2\}, \\
\gpgen{\kappa\rho, \rho^2}      &
=
\{ \id, \kappa\rho, \rho^2, \kappa\rho^{-1}\}.
\end{align*}

Finally, any subgroup of index $2$ of any group is normal, so all the
subgroups of $G$ of order $4$ are normal.
\end{itemize}

Hence the subgroup structure of $G \iso D_4$ is as follows, where a box
around a subgroup means that it is normal in $G$.
\[
\xymatrix@R-1ex@C-2em{
        &       &       
*+[F]{1} 
\ar@{-}[lld] \ar@{-}[ld] \ar@{-}[d] \ar@{-}[rd] \ar@{-}[rrd]
                        &       &      &       &
\text{order }1                                          \\
\gpgen{\kappa}
        &
\gpgen{\kappa\rho^2}    
                &
*+[F]{\gpgen{\rho^2}}   &
\gpgen{\kappa\rho}              &
\gpgen{\kappa\rho^{-1}}                 &       &
\text{order }2                                          \\
*+[F]{\gpgen{\kappa, \rho^2} \iso C_2 \times C_2}
\ar@{-}[u] \ar@{-}[ru] \ar@{-}[rru]
        &       &
*+[F]{\gpgen{\rho} \iso C_4}
\ar@{-}[u]
                        &       &
*+[F]{\gpgen{\kappa\rho, \rho^2} \iso C_2 \times C_2}
\ar@{-}[u] \ar@{-}[lu] \ar@{-}[llu]
                                       &       &
\text{order }4                                          \\
        &       & 
*+[F]{G = \gpgen{\kappa, \rho} \iso D_4}
\ar@{-}[llu] \ar@{-}[u] \ar@{-}[rru]
                        &       &       &       &
\text{order }8
}
\]

\paragraph*{Fixed fields} We now find $\Fix(H)$ for each $H \in \GG$, again
considering the subgroups of orders $2$ and $4$ in turn. We'll use the same
three-step strategy as in 
Example~\ref{eg:ftgt-2i}.

\begin{itemize}
\item 
Order 2: take $\Fix\gpgen{\kappa}$ (officially
$\Fix(\gpgen{\kappa})$, but let's drop the brackets). We have $\kappa(\xi) =
\xi$, so $\xi \in \Fix\gpgen{\kappa}$, so $\Q(\xi) \sub
\Fix\gpgen{\kappa}$. But $[\Q(\xi, i): \Q(\xi)] = 2$, and by the
fundamental theorem, $[\Q(\xi, i): \Fix\gpgen{\kappa}] = |\gpgen{\kappa}| =
2$, so
$\Fix\gpgen{\kappa} = \Q(\xi)$. 

The same argument shows that for any $\phi \in G$ of order $2$, if we can
spot some $\alpha \in \Q(\xi, i)$ such that $\phi(\alpha) = \alpha$ and
\video{Finding fixed fields}%
$[\Q(\xi, i): \Q(\alpha)] \leq 2$, then $\Fix\gpgen{\phi} = \Q(\alpha)$. For
$\phi = \kappa\rho^2$, we can take $\alpha = \xi i$ (by
Figure~\ref{fig:big-eg-table}). We have $\deg_\Q(\xi i) = 4$ since $\xi i$
is a root of $f$, so $[\Q(\xi i): \Q] = 4$, or equivalently, $[\Q(\xi, i):
\Q(\xi i)] = 2$. Hence $\Fix\gpgen{\kappa\rho^2} = \Q(\xi i)$.

\begin{ex}{ex:same-arg}
I took a small liberty in the sentence beginning `The same argument',
because it included an inequality but the previous argument didn't. Prove
the statement made in that sentence.
\end{ex}

It is maybe not so easy to spot an $\alpha$ for $\kappa\rho$, but the
geometric description in Figure~\ref{fig:big-eg-table} suggests taking
$\alpha = \xi(1 - i)$. And indeed, one can check that $\kappa\rho$ fixes
$\xi(1 - i)$. One can also check that $\xi(1 - i)$ is not the root of any
nonzero quadratic over $\Q$, so $\deg_\Q(\xi(1 - i))$ is $\geq 4$ (since it
divides $8$), so $[\Q(\xi, i): \Q(\xi(1 - i))] \leq 8/4 = 2$. Hence
$\Fix\gpgen{\kappa\rho} = \Q(\xi(1 - i))$. Similarly,
$\Fix\gpgen{\kappa\rho^{-1}} = \Q(\xi(1 + i))$.

Finally, 
\[
\rho^2(\xi^2) = (\rho^2(\xi))^2 = (-\xi)^2 = \xi^2, 
\quad
\rho^2(i) = i,
\]
so $\Q(\xi^2, i) \sub \Fix\gpgen{\rho^2}$. But $[\Q(\xi, i): \Q(\xi^2,
i)] = 2$, so $\Fix\gpgen{\rho^2} = \Q(\xi^2, i)$.

\item
Order $4$: for $H = \gpgen{\kappa, \rho^2}$, note that $\xi^2$ is fixed by
both $\kappa$ and $\rho^2$, so $\xi^2 \in \Fix(H)$, so $\Q(\xi^2) \sub
\Fix(H)$. But $\xi^2 \not\in \Q$, so $[\Q(\xi^2): \Q] \geq 2$, so $[\Q(\xi,
i): \Q(\xi^2)] \leq 4$. The fundamental theorem guarantees that
\[
[\Q(\xi, i): \Fix(H)] = |H| = 4,
\]
so $\Fix(H) = \Q(\xi^2)$.

The same argument applies to the other two subgroups $H$ of order $4$: if
we can spot an element $\alpha \in \Q(\xi, i) \without \Q$ fixed by the
generators of $H$, then $\Fix(H) = \Q(\alpha)$. This gives
$\Fix\gpgen{\rho} = \Q(i)$ and $\Fix\gpgen{\kappa\rho, \rho^2} = \Q(\xi^2
i)$.
\end{itemize}

In summary, the fixed fields of the subgroups of $G$ are as follows. 
\[
\xymatrix@R-1ex@C-.5em{
        &       &       
*+[F]{\Q(\xi, i)} 
\ar@{-}[lld] \ar@{-}[ld] \ar@{-}[d] \ar@{-}[rd] \ar@{-}[rrd]
                        &       &      &       &
\text{degree }1                                         \\
\Q(\xi) &
\Q(\xi i)       &
*+[F]{\Q(\xi^2, i)}      &
\Q(\xi(1 - i))                  &
\Q(\xi(1 + i))                          &       &
\text{degree }2                                         \\
*+[F]{\Q(\xi^2)}
\ar@{-}[u] \ar@{-}[ru] \ar@{-}[rru]
        &       &
*+[F]{\Q(i)}
\ar@{-}[u]
                        &       &
*+[F]{\Q(\xi^2 i)}
\ar@{-}[u] \ar@{-}[lu] \ar@{-}[llu]
                                       &       &
\text{degree }4                                         \\
        &       & 
*+[F]{\Q}
\ar@{-}[llu] \ar@{-}[u] \ar@{-}[rru]
                        &       &       &       &
\text{degree }8
}
\]
On the right, `degree' means the degree of $\Q(\xi, i)$ over
the subfield concerned, \emph{not} the degree over $\Q$. The fundamental
theorem implies that the Galois group of $\Q(\xi, i)$ over each
intermediate field is the subgroup of $G$ in the same position in the earlier
diagram. For example, $\Gal(\Q(\xi, i): \Q(\xi^2, i)) = \gpgen{\rho^2}$.
It also implies that the intermediate fields that are
normal over $\Q$ are the boxed ones.

\paragraph*{Quotients} 
Finally, the fundamental theorem tells us that
\[
\frac{\Gal(\Q(\xi, i): \Q)}{\Gal(\Q(\xi, i): L)}
\iso
\Gal(L: \Q)
\]
whenever $L$ is an intermediate field normal over $\Q$.

For $L = \Q(\xi^2, i)$, this gives
\begin{align}
\label{eq:iso-qt-4}
G/\gpgen{\rho^2} \iso \Gal(\Q(\xi^2, i): \Q).
\end{align}
The left-hand side is the quotient of $D_4$ by a subgroup isomorphic to
$C_2$. It has order $4$, but it has no element of order $4$: for the only
elements of $G$ of order $4$ are $\rho^{\pm 1}$, whose images in
$G/\gpgen{\rho^2}$ have order $2$. Hence $G/\gpgen{\rho^2} \iso C_2 \times
C_2$. On the other hand, $\Q(\xi^2, i)$ is the splitting field over $\Q$ of
$(t^2 - 2)(t^2 + 1)$, which by Example~\ref{eg:ftgt-2i}
has Galois group $C_2 \times C_2$. This confirms the
isomorphism~\bref{eq:iso-qt-4}.

The other three intermediate fields normal over $\Q$, I leave to you:%
\video{Normal subgroups and normal extensions}

\begin{ex}{ex:qts-2}
Choose one of $\Q(\xi^2)$, $\Q(i)$ or $\Q(\xi^2 i)$, and do the same for it
as I just did for $\Q(\xi^2, i)$. 
\end{ex}

As you've now seen, it can take quite some time to work through a
particular example of the Galois correspondence. You'll get practice at
doing this in workshops. 

Beyond examples, there are at least two other uses of the
fundamental theorem. The first is to resolve the old
question on solvability of polynomials by radicals, which we met back in
Chapter~\ref{ch:overview}. The second is to work out the structure of
finite fields. We will carry out these two missions in the remaining two
weeks.  

\chapter{Solvability by radicals}
\label{ch:sol}

We began this course with a notorious old problem: can every polynomial be
solved by radicals? Theorem~\ref{thm:solv-conc} gave the answer and more:
not only is it impossible to find a \emph{general formula} that does it,
but we can tell which \emph{specific} polynomials can be solved by
radicals.%
\video{Introduction to Week~9}

Theorem~\ref{thm:solv-conc} states that a polynomial over $\Q$ is solvable
by radicals if and only if it has the right kind of Galois group---a
solvable one. In degree $5$ and higher, there are polynomials that have the
wrong kind of group. These polynomials are not, therefore, solvable by
radicals. 

We'll prove one half of this `if and only if' statement: if $f$ is solvable
by radicals then $\Gal_\Q(f)$ is solvable. This is the half that's needed
to show that some polynomials are \emph{not} solvable by radicals. The
proof of the other direction is in Chapter~18 of Stewart's book,
but we won't do it.

If you're taking Algebraic Topology, you'll already be familiar with the
idea that groups can be used to solve problems that seem to have nothing to
do with groups. You have a problem about some objects (such as topological
spaces or field extensions), you associate groups with those objects (maybe
their fundamental groups or their Galois groups), you translate your
original problem into a problem about groups, and you solve that
instead. For example, the question of whether $\R^2$ and $\R^3$ are
homeomorphic is quite difficult using only general topology; but using
algebraic topology, we can answer `no' by noticing that the fundamental
group of $\R^2$ with a point removed is not isomorphic to the fundamental
group of $\R^3$ with a point removed. In much the same way, we'll answer a
difficult question about field extensions by converting it into a question
about groups.

For this chapter, you'll need what you know about solvable
groups. At a minimum, you'll need the definition, the fact that any
quotient of a solvable group is solvable, and the fact that $S_5$ is not
solvable.

\section{Radicals}
\label{sec:rad}

We speak of square roots, cube roots, and so on, but we also speak about
roots of polynomials. To distinguish between these two related usages, we
will use the word \demph{radical} for square roots etc. (\emph{Radical}
comes from the Latin for root. A radish is a root, and a change is radical
if it gets to the root of the matter.)

Back in Chapter~\ref{ch:overview}, I said that a complex number is called
radical if `it can be obtained from the rationals using only the usual
arithmetic operations [addition, subtraction, multiplication and division]
and $k$th roots [for $k \geq 1$]'. As an example, I said that
\begin{align}
\label{eq:random-rad}
\frac{
\frac{1}{2} + \sqrt[3]{\sqrt[7]{2} - \sqrt[2]{7}}
}
{
\sqrt[4]{6 + \sqrt[5]{\frac{2}{3}}}
}
\end{align}
is radical, whichever square root, cube root, etc., we choose
(p.~\pageref{p:radical}). Let's now make this definition precise.

The first point is that the notation $\sqrt[n]{z}$ or $z^{1/n}$ is
\textcolor{red}{highly dangerous}: 

\begin{warning}{wg:bad-rad}
Let $z$ be a complex number and $n \geq 2$. Then there is \textbf{no single
number called $\sqrt[n]{z}$ or $z^{1/n}$}. There are $n$ elements $\alpha$ of
$\C$ such that $\alpha^n = z$. So, the notation $\sqrt[n]{z}$ or $z^{1/n}$
makes no sense if it is intended to denote a single complex number. It is
simply invalid.

When $z$ belongs to the set $\R^+$ of nonnegative reals, the convention is
that $\sqrt[n]{z}$ or $z^{1/n}$ denotes the unique $\alpha \in \R^+$ such
that $\alpha^n = z$. There is also a widespread convention that when $z$ is
a negative real and $n$ is odd, $\sqrt[n]{z}$ or $z^{1/n}$ denotes the
unique real $\alpha$ such that $\alpha^n = z$. In these cases, there is a
sensible and systematic way of choosing one of the $n$th roots of $z$. But
for a general $z$ and $n$, there is not. 

Complex analysis has a lot to say about different choices of $n$th
roots. But we don't need to go into that. We simply treat all the $n$th
roots of $z$ on an equal footing, not attempting to pick out any of them as
special.
\end{warning}

With this warning in mind, we define the radical numbers without using
notation like $\sqrt[n]{z}$ or $z^{1/n}$. It is a `top-down' definition, in
the sense of Section~\ref{sec:rings}. Loosely, it says that the radical
numbers form the smallest subfield of $\C$ closed under taking square
roots, cube roots, etc.

\begin{defn}
Let \demph{$\Q^\rad$} be the smallest subfield of $\C$ such that for
$\alpha \in \C$,%
\video{The definition of radical number}%
\begin{align}
\label{eq:defn-rad}
\alpha^n \in \Q^\rad \text{ for some } n \geq 1
\ \implies\ 
\alpha \in \Q^\rad.
\end{align}
A complex number is \demph{radical} if it belongs to $\Q^\rad$.
\end{defn}

So any rational number is radical; any $n$th root of a radical number is
radical; the sum, product, difference or quotient of radical numbers is
radical; and there are no more radical numbers than can be obtained by
those rules.

For the definition of $\Q^\rad$ to make sense, we need there to \emph{be} a
smallest subfield of $\C$ with the property~\eqref{eq:defn-rad}. This will
be true as long as the intersection of any family of subfields of $\C$
satisfying~\eqref{eq:defn-rad} is again a subfield of $\C$
satisfying~\eqref{eq:defn-rad}: for then $\Q^\rad$ is the intersection of
\emph{all} subfields of $\C$ satisfying~\eqref{eq:defn-rad}.

\begin{ex}{ex:defn-rad}
Check that the intersection of any family of subfields of $\C$
satisfying~\eqref{eq:defn-rad} is again a subfield of $\C$
satisfying~\eqref{eq:defn-rad}. (That any intersection of subfields is a
subfield is a fact we met back on p.~\pageref{p:int-subfds}; the new
aspect is~\eqref{eq:defn-rad}.)
\end{ex}

\begin{example}
Consider again the expression~\eqref{eq:random-rad}. It's not quite as random
as it looks. I chose it so that the various radicals are covered by one of
the two conventions mentioned in Warning~\ref{wg:bad-rad}: they're all
$n$th roots of positive reals except for $\sqrt[3]{\sqrt[7]{2} -
\sqrt[2]{7}}$, which is an odd root of a negative real. Let $z$ be the
number~\eqref{eq:random-rad}, choosing the radicals according to those
conventions. 

I claim that $z$ is radical, or equivalently that $z$ belongs to every
subfield $K$ of $\C$ satisfying~\eqref{eq:defn-rad}. 

First, $\Q \sub K$ since $\Q$ is the prime subfield of $\C$. So $2/3 \in
K$, and so $\sqrt[5]{2/3} \in K$ by~\eqref{eq:defn-rad}. Also, $6 \in K$ and
$K$ is a field, so $6 + \sqrt[5]{2/3} \in K$. But then
by~\eqref{eq:defn-rad} again, the denominator of~\eqref{eq:random-rad} is
in $K$. A similar argument shows that the numerator is in $K$. Hence $z \in
K$. 
\end{example}

\begin{defn}
A nonzero polynomial over $\Q$ is \demph{solvable by radicals} if all of its
complex roots are radical.
\end{defn}

The simplest nontrivial example of a polynomial solvable by radicals is
something of the form $t^n - a$, where $a \in \Q$. The theorem we're
heading for is that \emph{any} polynomial solvable by radicals has solvable
Galois group, and if that's true then the group $\Gal_\Q(t^n - a)$ must be
solvable. Let's consider that group now. The results we prove about it will
form part of the proof of the big theorem.

We begin with the case $a = 1$.

\begin{lemma}
\label{lemma:kummer-unity}
For all $n \geq 1$, the group $\Gal_\Q(t^n - 1)$ is abelian.
\end{lemma}

\begin{proof}
Write $\omega = e^{2\pi i/n}$. The complex roots of $t^n - 1$ are $1,
\omega, \ldots, \omega^{n - 1}$, so $\SF_\Q(t^n - 1) = \Q(\omega)$.

Let $\phi, \theta \in \Gal_\Q(t^n - 1)$. Since $\phi$ permutes the roots of
$t^n - 1$, we have $\phi(\omega) = \omega^i$ for some $i \in
\Z$. Similarly, $\theta(\omega) = \omega^j$ for some $j \in \Z$. Hence
\[
(\phi \of \theta)(\omega)
=
\phi(\omega^j)
=
\phi(\omega)^j
=
\omega^{ij},
\]
and similarly $(\theta \of \phi)(\omega) = \omega^{ij}$. So $(\phi \of
\theta)(\omega) = (\theta \of \phi)(\omega)$. Since $\SF_\Q(t^n
- 1) = \Q(\omega)$, it follows that $\theta \of \phi = \phi \of \theta$.
\end{proof}

\begin{ex}{ex:rad-follow}
In the last sentence of that proof, how exactly does it `follow'?
\end{ex}

Much more can be said about the Galois group of $t^n - 1$, and you'll see a
bit more in workshops. But this is all we need for our purposes.

Now that we've considered $t^n - 1$, let's do $t^n - a$ for an arbitrary
$a$. 

\begin{lemma}
\label{lemma:kummer-gen}
Let $K$ be a field and $n \geq 1$. Suppose that $t^n - 1$ splits in
$K$. Then $\Gal_K(t^n - a)$ is abelian for all $a \in K$.
\end{lemma}

The hypothesis that $t^n - 1$ splits in $K$ might seem so restrictive as to
make this lemma useless. For instance, it doesn't hold in $\Q$ or even $\R$
(for $n > 2$). Nevertheless, this turns out to be the key lemma
in the whole story of solvability by radicals.

\begin{proof}
If $a = 0$ then $\Gal_K(t^n - a)$ is trivial; suppose otherwise.

Choose a root $\xi$ of $t^n - a$ in $\SF_K(t^n - a)$. For any other root
$\nu$, we have $(\nu/\xi)^n = a/a = 1$ (valid since $a \neq 0$), and $t^n -
1$ splits in $K$, so $\nu/\xi \in K$.

It follows that $\SF_K(t^n - a) = K(\xi)$. Moreover, given $\phi, \theta
\in \Gal_K(t^n - a)$, we have $\phi(\xi)/\xi \in K$ (since $\phi(\xi)$ is a
root of $t^n - a$), so
\[
(\theta \of \phi)(\xi)
=
\theta \biggl( \frac{\phi(\xi)}{\xi} \cdot \xi \biggr)
=
\frac{\phi(\xi)}{\xi} \cdot \theta(\xi)
=
\frac{\phi(\xi)\theta(\xi)}{\xi}.
\]
Similarly, $(\phi \of \theta)(\xi) = \phi(\xi)\theta(\xi)/\xi$, so $(\theta
\of \phi)(\xi) = (\phi \of \theta)(\xi)$. Since $\SF_K(t^n - a) = K(\xi)$,
it follows that $\phi \of \theta = \theta \of \phi$.
\end{proof}

\begin{warning}{wg:kummer-not-ab}
For $a \in \Q$, the Galois group of $t^n - a$ over $\Q$ is \emph{not}
usually abelian. For instance, we saw in Example~\ref{eg:gal-cbrt2} that
$\Gal_\Q(t^3 - 2)$ is the nonabelian group $S_3$.
\end{warning}

\begin{ex}{ex:kummer-eig}
What does the proof of Lemma~\ref{lemma:kummer-gen} tell you about the
eigenvectors and eigenvalues of the elements of $\Gal_K(t^n - a)$?
\end{ex}

\begin{ex}{ex:kummer-sol}
Use Lemmas~\ref{lemma:kummer-unity} and~\ref{lemma:kummer-gen} to show that
$\Gal_\Q(t^n - a)$ is solvable for all $a \in \Q$. 

This is harder than most of these exercises, but I recommend it as a way
of getting into the right frame of mind for the theory that's coming in
Section~\ref{sec:sol-sol}.
\end{ex}

\begin{digression}{dig:s-by-r-fields}
We're only going to do the theory of solvability by radicals over $\Q$. It
can be done over any field, but $\Q$ has two special features. First, $\Q$
can be embedded in an algebraically closed field that we know very well:~$\C$. This makes some things easier. Second, $\chr\Q = 0$. For
fields of characteristic $p$, there are extra complications (Stewart,
Section~17.6).
\end{digression}

\section{Solvable polynomials have solvable groups}
\label{sec:sol-sol}

Here we'll prove that every polynomial over $\Q$ that is solvable by
radicals has solvable Galois group. 

You know by now that in Galois theory, we tend not to jump straight from
polynomials to groups. We go via the intermediate stage of field
extensions, as in the diagram
\[
\text{polynomial } 
\longmapsto 
\text{ field extension } 
\longmapsto 
\text{ group}
\]
that I first drew after the definition of $\Gal_K(f)$
(p.~\pageref{p:peg}). That is, we understand polynomials through their
splitting field extensions. 

So it shouldn't be a surprise that we do the same here, defining a notion
of `solvable extension' and showing (roughly speaking) that
\[
\text{solvable polynomial } 
\longmapsto
\text{ solvable extension }
\longmapsto
\text{ solvable group}.
\]
In other words, we'll define `solvable extension' in such a way that (i)~if $f
\in \Q[t]$ is a polynomial solvable by radicals then $\SF_\Q(f): \Q$ is a
solvable extension, and (ii)~if $M: K$ is a solvable extension then $\Gal(M: K)$
is a solvable group. Hence if $f$ is solvable by radicals
then $\Gal_\Q(f)$ is solvable---the result we're aiming for.%
\video{Solvable polynomials have solvable groups:\\ a map}%

\begin{defn}
\label{defn:sol-ext}
Let $M: K$ be a finite normal separable extension. Then $M: K$ is
\demph{solvable} (or $M$ is \demph{solvable over} $K$) if there exist $r
\geq 0$ and intermediate fields
\[
K = L_0 \sub L_1 \sub \cdots \sub L_r = M
\]
such that $L_i: L_{i - 1}$ is normal and $\Gal(L_i: L_{i - 1})$ is abelian
for each $i \in \{1, \ldots, r\}$.
\end{defn}

\begin{ex}{ex:comp-sol}
Let $N: M: K$ be extensions, with $N: M$, $M: K$ and $N: K$ all finite,
normal and separable. Show that if $N: M$ and $M: K$ are solvable then so
is $N: K$.
\end{ex}

We will focus on subfields of $\C$, where separability is automatic
(Example~\ref{egs:sep-ext}\bref{eg:se-zero}).

\begin{example}
\label{eg:kummer-sol}
Let $a \in \Q$ and $n \geq 1$. Then $\SF_\Q(t^n - a): \Q$ is a finite
normal separable extension, being a splitting field extension over $\Q$. I
claim that it is solvable. 

Proof: if $a = 0$ then $\SF_\Q(t^n - a) = \Q$, and $\Q: \Q$ is solvable
(taking $r = 0$ and $L_0 = \Q$ in Definition~\ref{defn:sol-ext}). Now
assume that $a \neq 0$. Choose a complex root $\xi$ of $t^n - a$ and
write $\omega = e^{2\pi i/n}$. Then the complex roots of $t^n - a$ are
\[
\xi, \omega\xi, \ldots, \omega^{n - 1}\xi.
\]
So $\SF_\Q(t^n - a)$ contains $(\omega^i\xi)/\xi = \omega^i$ for all $i$,
and so $t^n - 1$ splits in $\SF_\Q(t^n - a)$. Hence
\[
\Q \sub \SF_\Q(t^n - 1) \sub \SF_\Q(t^n - a).
\]
Now $\SF_\Q(t^n - 1): \Q$ is normal (being a splitting field extension) and
has abelian Galois group by Lemma~\ref{lemma:kummer-unity}. Also
$\SF_\Q(t^n - a): \SF_\Q(t^n - 1)$ is normal (being the splitting
field extension of $t^n - a$ over $\SF_\Q(t^n - 1)$, by
Lemma~\ref{lemma:sf-adj}\bref{part:sfa-fds}), and has abelian Galois group
by Lemma~\ref{lemma:kummer-gen}. So $\SF_\Q(t^n - a): \Q$ is a solvable
extension, as claimed.
\end{example}

The definition of solvable extension bears a striking resemblance to the
definition of solvable group. Indeed:

\begin{lemma}
\label{lemma:sole-solg}
Let $M: K$ be a finite normal separable extension. Then
\[
M: K \text{ is solvable}
\iff
\Gal(M: K) \text{ is solvable.}
\]
\end{lemma}

\begin{proof}
We will only need the $\implies$ direction, and that is all I prove
here. For the converse, see the workshop questions.

Suppose that $M: K$ is solvable. Then there are intermediate fields
\[
K = L_0 \sub L_1 \sub \cdots \sub L_r = M
\]
such that each extension $L_i: L_{i - 1}$ is normal with abelian Galois
group. For each $i \in \{1, \ldots, r\}$, the extension $M: L_{i - 1}$ is
finite, normal and separable (by Corollary~\ref{cor:top-nml} and
Lemma~\ref{lemma:sep-int}), so we can apply the fundamental theorem of
Galois theory to it. Since $L_i: L_{i - 1}$ is a normal extension, $\Gal(M:
L_i)$ is a normal subgroup of $\Gal(M: L_{i - 1})$ and
\[
\frac{\Gal(M: L_{i - 1})}{\Gal(M: L_i)}
\iso
\Gal(L_i: L_{i - 1}).
\]
By hypothesis, the right-hand side is abelian, so the left-hand side is
too. So the sequence of subgroups
\[
\Gal(M: K) = \Gal(M: L_0) \supseteq \Gal(M: L_1) \supseteq \cdots
\supseteq \Gal(M: L_r) = 1
\]
exhibits $\Gal(M: K)$ as a solvable group. 
\end{proof}

\begin{ex}{ex:solg-sole}
Prove the $\impliedby$ direction of Lemma~\ref{lemma:sole-solg}. It's a
very similar argument to the proof of $\implies$.
\end{ex}

According to the story I'm telling, solvability by radicals of a polynomial
should correspond to solvability of its splitting field extension. Thus,
the subfields of $\C$ that are solvable over $\Q$ should be exactly the
splitting fields $\SF_\Q(f)$ of polynomials $f$ that are solvable by
radicals. (This is indeed true, though we won't entirely prove it.) Now if
$f, g \in \Q[t]$ are both solvable by radicals then so is $fg$, and
$\SF_\Q(fg)$ is a solvable extension of $\Q$ containing both $\SF_\Q(f)$
and $\SF_\Q(g)$. So it should be the case that for any two subfields of
$\C$ solvable over $\Q$, there is some larger subfield, also solvable over
$\Q$, containing both. 

The following pair of lemmas proves this. They use the notion of
compositum (Definition~\ref{defn:compm}).

\begin{lemma}
\label{lemma:compm-condns}
Let $M: K$ be a field extension and let $L$ and $L'$ be intermediate
fields. 
\marginpar{\quad
$\xymatrix@C=3mm@R=3mm{%
        &
\scriptstyle
M \ar@{-}[d]   &       \\
        &
\scriptstyle
LL' \ar@{-}[ld] \ar@{-}[rd]    &       \\
\scriptstyle
L \ar@{-}[rd]   &       &
\scriptstyle
L' \ar@{-}[ld] \\
        &
\scriptstyle
K      &       
}$%
}%
\begin{enumerate}
\item 
\label{part:cc-nn}
If $L: K$ and $L': K$ are finite and normal, then so is $LL': K$.

\item
\label{part:cc-n}
If $L: K$ is finite and normal, then so is $LL': L'$.

\item
\label{part:cc-a}
If $L: K$ is finite and normal with abelian Galois group, then so is $LL': L'$.
\end{enumerate}
\end{lemma}

\begin{proof}
For~\bref{part:cc-nn}, we have $L = \SF_K(f)$ and $L' = \SF_K(f')$ for some
$f, f' \in K[t]$. Now $LL'$ is the subfield of $M$ generated by $L \cup
L'$, or equivalently by the roots of $f$ and $f'$. So $LL' = \SF_K(ff')$,
which is finite and normal over $K$.

For~\bref{part:cc-n}, we have $L = \SF_K(f)$ for some $f \in K[t]$.  Then
$LL' = \SF_{L'}(f)$ by Lemma~\ref{lemma:sf-adj}\bref{part:sfa-gen} (with $S
= L$ and $Y = L'$), so $LL'$ is finite and normal over $L'$. Now
$\Gal(LL': L') = \Gal_{L'}(f)$, which by
Corollary~\ref{cor:gal-emb} is isomorphic to a subgroup of $\Gal_K(f) =
\Gal(L: K)$. So if $\Gal(L: K)$ is abelian then so is $\Gal(LL': L')$,
proving~\bref{part:cc-a}. 
\end{proof}

\begin{lemma}
\label{lemma:sol-filt}
Let $L$ and $M$ be subfields of $\C$ such that the extensions $L: \Q$ and
$M: \Q$ are finite, normal and solvable. Then there is some subfield $N$ of
$\C$ such that $N: \Q$ is finite, normal and solvable and $L, M \sub N$.
\end{lemma}

Both the statement and the proof of this lemma should remind you of
Lemma~\ref{lemma:iq-join} on ruler and compass constructions.

\begin{proof}
Take subfields
\[
\Q = L_0 \sub \cdots \sub L_r = L,
\qquad
\Q = M_0 \sub \cdots \sub M_s = M
\]
such that $L_i: L_{i - 1}$ is normal with abelian Galois group for each
$i$, and similarly for $M_j: M_{j - 1}$. There is a chain of subfields
\begin{align}
\label{eq:filt-big}
\Q = L_0 \sub \cdots \sub L_r = L 
= LM_0 \sub \cdots \sub LM_s = LM
\end{align}
of $\C$. Put $N = LM$.
Certainly $L, M \sub N$. We show that $N: \Q$ is finite, normal
and solvable.

That $N: \Q$ is finite and normal follows from
Lemma~\ref{lemma:compm-condns}\bref{part:cc-nn}. 

To see that $N: \Q$ is solvable, we show that each successive extension
in~\eqref{eq:filt-big} is normal with abelian Galois group. For those to
the left of $L$, this is immediate. For those to the right, let $j \in \{1,
\ldots, s\}$. Since $M_j: M_{j - 1}$ is finite and normal with abelian
Galois group, so is $LM_j: LM_{j - 1}$ by
Lemma~\ref{lemma:compm-condns}\bref{part:cc-a}. 
\end{proof}

The set of radical numbers is a subfield of $\C$ closed under taking $n$th
roots. So if the story I'm telling is right, the set of complex numbers
that can be reached from $\Q$ by solvable extensions should also be a
subfield of $\C$ closed under
taking $n$th roots. That's an informal description of our next two
results. Write
\begin{align*}
\dem{\Q^\sol} &
=
\{ \alpha \in \C \such \alpha \in L \text{ for some subfield } L \sub \C
\text{ that is finite, normal and solvable}      \\ 
&
\phantom{= \{} \text{ over } \Q\}.
\end{align*}

\begin{lemma}
\label{lemma:Qsol-subfd}
$\Q^\sol$ is a subfield of $\C$. 
\end{lemma}

\begin{proof}
This is similar to the proof that the algebraic numbers form a
subfield of $\C$ (Proposition~\ref{propn:Qbar-subfd}). Let $\alpha, \beta \in
\Q^\sol$. Then $\alpha \in L$ and $\beta \in M$ for some $L, M$ that are
finite, normal and solvable over $\Q$. By Lemma~\ref{lemma:sol-filt},
$\alpha, \beta \in N$ for some $N$ that is finite, normal and solvable over
$\Q$. Then $\alpha + \beta \in N$, so $\alpha + \beta \in \Q^\sol$, and
similarly $\alpha \cdot \beta \in \Q^\sol$. This shows that $\Q^\sol$ is
closed under addition and multiplication. The other parts of the proof
(negatives, reciprocals, $0$ and $1$) are straightforward.
\end{proof}

\begin{lemma}
\label{lemma:Qsol-rad}
Let $\alpha \in \C$ and $n \geq 1$. If $\alpha^n \in \Q^\sol$ then $\alpha
\in \Q^\sol$. 
\end{lemma}

The proof (below) is slightly subtle. Here's why. 

Let $L$ be a subfield of $\C$ that's finite, normal and solvable over $\Q$,
and take $\alpha \in \C$ and $n \geq 1$ such that $\alpha^n \in L$. To find
some larger $M$ that contains $\alpha$ itself and is also solvable over
$\Q$, we could try putting $M = \SF_L(t^n - \alpha^n)$. But the problem is
that $M: \Q$ is not in general normal. And normality is part of the
definition of $\Q^\sol$, ultimately because it's essential if we want to
use the fundamental theorem of Galois theory. The basic problem is this:

\begin{warning}{wg:nml-nml}
A normal extension of a normal extension is not in general normal, just as
a normal subgroup of a normal subgroup is not in general normal. 
\end{warning}

An example should clarify.

\begin{example}
Put $\alpha = \sqrt[4]{2}$ and $n = 2$. We have $\alpha^2 = \sqrt{2}
\in \Q(\sqrt{2})$, and $\Q(\sqrt{2}): \Q$ is finite, normal and solvable
(since its Galois group is the abelian group $C_2$), so $\alpha^2 \in
\Q^\sol$. Hence, according to Lemma~\ref{lemma:Qsol-rad}, $\alpha = \sqrt[4]{2}$
should be contained in some finite normal solvable extension $M$ of $\Q$.

How can we find such an $M$? We \emph{can't} take $M =
\SF_{\Q(\sqrt{2})}(t^2 - \sqrt{2})$, since this is $\Q(\sqrt[4]{2})$, which
is not normal over $\Q$ (for the same reason that $\Q(\sqrt[3]{2})$
isn't). 

To find a bigger $M$, still finite and solvable over $\Q$ but also normal,
we have to adjoin a square root not just of $\sqrt{2}$ \emph{but also of
its conjugate}, $-\sqrt{2}$. This is the crucial point: the whole idea of
normality is that conjugates are treated equally. (Normal behaviour means
that anything you do for one element, you do for all its conjugates.) The
result is $\Q(\sqrt[4]{2}, i) = \SF_\Q(t^4 - 2)$, which is indeed a
finite, solvable and normal extension of $\Q$ containing $\sqrt[4]{2}$.
\end{example}

\begin{pfof}{Lemma~\ref{lemma:Qsol-rad}}
Write $a = \alpha^n \in \Q^\sol$.  Choose a subfield $K$ of $\C$ such that
$a \in K$ and $K: \Q$ is finite, normal and solvable.

\paragraph*{Step 1: enlarge $K$ to a field in which $t^n - 1$ splits.}
Put $L = \SF_K(t^n - 1) \sub \C$.  

Since $K: \Q$ is finite and normal, $K = \SF_\Q(f)$ for some nonzero $f \in
\Q[t]$, and then $L = \SF_\Q\bigl((t^n - 1)f(t)\bigr)$. Hence $L: \Q$ is
finite and normal. The Galois group of $L: K$ is $\Gal_K(t^n - 1)$, which
is isomorphic to a subgroup of $\Gal_\Q(t^n - 1)$ (by
Corollary~\ref{cor:gal-emb}), which is abelian (by
Lemma~\ref{lemma:kummer-unity}). Hence $L: K$ is a normal extension with
abelian Galois group. Also, $K: \Q$ is solvable. It follows from the
definition of solvable extension that $L: \Q$ is solvable.

In summary, $L$ is a subfield of $\C$ such that $a \in L$ and $L: \Q$ is
finite, normal and solvable, \emph{and}, moreover, $t^n - 1$
splits in $L$. We now forget about $K$.

\paragraph*{Step 2: adjoin the $n$th roots of the conjugates of $a$.}
Write $m \in \Q[t]$ for the minimal polynomial of $a$ over $\Q$, and put $M
= \SF_L(m(t^n)) \sub \C$. Then $\alpha \in M$, as $m(\alpha^n) = m(a) =
0$. We show that $M: \Q$ is finite, normal and solvable.

Since $L: \Q$ is finite and normal, $L = \SF_\Q(g)$ for some nonzero $g \in
\Q[t]$. Then $M = \SF_\Q(g(t)m(t^n))$, so $M: \Q$ is finite and normal.
Moreover, $M: L$ is finite and normal, being a splitting field extension.

Now to show that $M: \Q$ is solvable, it is enough to show that $M: L$ is
solvable, by Exercise~\ref{ex:comp-sol}. Since $L: \Q$ is normal and $m \in
\Q[t]$ is the minimal polynomial of $a \in L$, it follows by definition of
normality that $m$ splits in $L$, say
\[
m(t) = \prod_{i = 1}^r (t - a_i)
\]
($a_i \in L$).  Define subfields $L_0 \sub \cdots \sub L_r$ of $\C$ by
\begin{align*}
L_0     &= L    \\
L_1     &= \SF_{L_0}(t^n - a_1)         \\
L_2     &= \SF_{L_1}(t^n - a_2)         \\
        &\vdots \\
L_r     &= \SF_{L_{r - 1}}(t^n - a_r).
\end{align*}
Then
\[
L_i = 
L\bigl(\bigl\{ 
\beta \in M \such \beta^n \in \{a_1, \ldots, a_i\}
\bigr\}\bigr).
\]
In particular, $L_r = M$. For each $i \in \{1, \ldots, s\}$, the extension
$L_i: L_{i - 1}$ is finite and normal (being a splitting field extension),
and its Galois group is abelian (by Lemma~\ref{lemma:kummer-gen} and the
fact that $t^n - 1$ splits in $L \sub L_{i - 1}$). So $M: L$ is solvable.
\end{pfof}

Now we can relate the set $\Q^\rad$ of radical numbers, defined in terms of
basic arithmetic operations, to the set $\Q^\sol$, defined in terms of
field extensions.

\begin{propn}
\label{propn:Qrad-Qsol}
$\Q^\rad \sub \Q^\sol$. That is, every radical number is contained in some
subfield of $\C$ that is a finite, normal, solvable extension of $\Q$.
\end{propn}

In fact, $\Q^\rad$ and $\Q^\sol$ are equal, but we won't prove this.

\begin{proof}
By Lemmas~\ref{lemma:Qsol-subfd} and~\ref{lemma:Qsol-rad}, $\Q^\sol$ is a
subfield of $\C$ such that $\alpha^n \in \Q^\sol \implies
\alpha \in \Q^\sol$. The result follows from the definition of $\Q^\rad$.
\end{proof}

This brings us to the main result of this chapter. Notice that it doesn't
mention field extensions: it goes straight from polynomials to groups.

\begin{bigthm}
\label{thm:sol-rad}
Let $0 \neq f \in \Q[t]$. If the polynomial $f$ is solvable by radicals
then the group $\Gal_\Q(f)$ is solvable.
\end{bigthm}

\begin{proof}
Suppose that $f$ is solvable by radicals. Write $\alpha_1, \ldots, \alpha_n
\in \C$ for its roots. For each $i$, we have $\alpha_i \in \Q^\rad$ (by
definition of solvability by radicals), hence $\alpha_i \in \Q^\sol$ (by
Proposition~\ref{propn:Qrad-Qsol}). So each of $\alpha_1, \ldots, \alpha_n$
is contained in some subfield of $\C$ that is finite, normal and solvable
over $\Q$. By Lemma~\ref{lemma:sol-filt}, there is some subfield $M$ of
$\C$ that is finite, normal and solvable over $\Q$ and contains all of
$\alpha_1, \ldots, \alpha_n$. Then $\Q(\alpha_1, \ldots, \alpha_n) \sub
M$; that is, $\SF_\Q(f) \sub M$.

By Lemma~\ref{lemma:sole-solg}, $\Gal(M: \Q)$ is solvable. Now $\SF_\Q(f):
\Q$ is normal, so by the fundamental theorem of Galois theory,
$\Gal(\SF_\Q(f): \Q)$ is a quotient of $\Gal(M: \Q)$. But $\Gal(\SF_\Q(f):
\Q) = \Gal_\Q(f)$, and a quotient of a solvable group is solvable, so
$\Gal_\Q(f)$ is solvable.
\end{proof}

\begin{examples}
\begin{enumerate}
\item 
For $a \in \Q$ and $n \geq 1$, the polynomial $t^n - a$ is solvable by
radicals, so the group $\Gal_\Q(t^n - a)$ is solvable. You may already have
proved this in Exercise~\ref{ex:kummer-sol}. It also follows from
Example~\ref{eg:kummer-sol} and Lemma~\ref{lemma:sole-solg}.

\item
Let $a_1, \ldots, a_k \in \Q$ and $n_1, \ldots, n_k \geq 1$. Each of
the polynomials $t^{n_i} - a_i$ is solvable by radicals, so their product
is too. Hence $\Gal_\Q\bigl( \prod_i (t^{n_i} - a_i) \bigr)$
is a solvable group.
\end{enumerate}
\end{examples}

Theorem~\ref{thm:sol-rad} is most sensational in its contrapositive form:
if $\Gal_\Q(f)$ is \emph{not} solvable then $f$ is \emph{not} solvable by
radicals. That's the subject of the next section.

\begin{digression}{dig:sr-converse}
The converse of Theorem~\ref{thm:sol-rad} is also true: if $\Gal_\Q(f)$ is
solvable then $f$ is solvable by radicals. You can even unwind the proof to
obtain an explicit formula for the solving the quartic by radicals
(Stewart, Chapter~18). 

For this, we have to deduce properties of a field extension from
assumptions about its Galois group. A solvable group is built up from
abelian groups, and every finite abelian group is a direct sum of cyclic
groups. The key step in proving the converse of Theorem~\ref{thm:sol-rad}
has come to be known as `Hilbert's Theorem~90' (Stewart's Theorem~18.18),
which gives information about field extensions whose Galois groups are
cyclic.
\end{digression}

\begin{digression}{dig:sr-cheated}
The proof of Theorem~\ref{thm:sol-rad} might not have ended quite how you
expected. Given my explanations earlier in the chapter, you might
justifiably have imagined we were going to show that when the polynomial
$f$ is solvable by radicals, the extension $\SF_\Q(f): \Q$ is
solvable. That's not what we did. We showed that $\SF_\Q(f)$ is contained
in some larger subfield $M$ such that $M: \Q$ is solvable, then used that
to prove the solvability of the group $\Gal_\Q(f)$.

But all is right with the world: $\SF_\Q(f): \Q$ \emph{is} a solvable
extension. Indeed, its Galois group $\Gal_\Q(f)$ is
solvable, so Lemma~\ref{lemma:sole-solg} implies that $\SF_\Q(f): \Q$
is solvable too.
\end{digression}

\section{An unsolvable polynomial}
\label{sec:unsol-poly}

Here we give a specific example of a polynomial over $\Q$ that is not
solvable by radicals. By Theorem~\ref{thm:sol-rad}, our task is to
construct a polynomial whose Galois group is not solvable. The smallest
non-solvable group is $A_5$ (of order $60$). Our polynomial has
Galois group $S_5$ (of order $120$), which is also non-solvable.

Finding Galois groups is hard, and we will use a whole box of tools and
tricks, from Cauchy's theorem on groups to Rolle's theorem on
differentiable functions.

First we prove a useful general fact on the order of Galois groups.

\begin{lemma}
\label{lemma:irr-div}
Let $f$ be an irreducible polynomial over a field $K$, with $\SF_K(f): K$ 
separable. Then $\deg(f)$ divides $|\Gal_K(f)|$.
\end{lemma}

\begin{proof}
Let $\alpha$ be a root of $f$ in $\SF_K(f)$. By irreducibility, $\deg(f) =
[K(\alpha): K]$, which divides $[\SF_K(f): K]$ by the tower law, which is
equal to $|\Gal_K(f)|$ by Theorem~\ref{thm:gal-size} (using separability).
\end{proof}

Next, we need some results about the symmetric group $S_n$. I assume you
know that $S_n$ is generated by the `adjacent transpositions' $(12), (23),
\ldots, (n - 1 \ n)$. This may have been proved in Fundamentals of Pure
Mathematics, and as the Group Theory notes say (p.~58):
\begin{quote}
This is intuitively clear: suppose you have $n$ people lined up and you
want them to switch into a different order. To put them in the order you
want them, it's clearly enough to have people move up and down the line;
and each time a person moves one place, they switch places with the person
next to them.
\end{quote}
Here's a different way of generating $S_n$. 

\begin{lemma}
\label{lemma:Sn-gen}
For $n \geq 2$, the symmetric group $S_n$ is generated by $(12)$ and
$(12 \ldots n)$.
\end{lemma}

\begin{proof}
We have
\[
(12 \ldots n) (12) (12 \ldots n)^{-1} = (23),
\]
either by direct calculation or the general fact that $\sigma (a_1 \ldots
a_k) \sigma^{-1} = (\sigma(a_1) \ldots \sigma(a_k))$ for any $\sigma \in
S_n$ and cycle $(a_1 \ldots a_k)$. So any subgroup $H$ of $S_n$ containing
$(12)$ and $(12 \ldots n)$ also contains $(23)$. By the same argument, $H$
also contains $(34), \ldots, (n-1 \ n)$. But the adjacent transpositions
generate $S_n$, so $H = S_n$.
\end{proof}

\begin{lemma}
\label{lemma:Sp}
Let $p$ be a prime number, and let $f \in \Q[t]$ be an irreducible polynomial
of degree $p$ with exactly $p - 2$ real roots. Then
$\Gal_\Q(f) \iso S_p$.
\end{lemma}

\begin{proof}
Since $\chr\Q = 0$ and $f$ is irreducible, $f$ is separable and therefore
has $p$ distinct roots in $\C$.  By Proposition~\ref{propn:gal-defns-same},
the action of $\Gal_\Q(f)$ on the roots of $f$ in $\C$ defines an
isomorphism between $\Gal_\Q(f)$ and a subgroup $H$ of $S_p$. Since $f$ is
irreducible, $p$ divides $|\Gal_\Q(f)| = |H|$ (by
Lemma~\ref{lemma:irr-div}). So by Cauchy's theorem, $H$ has an element
$\sigma$ of order $p$. Then $\sigma$ is a $p$-cycle, since these are the
only elements of $S_p$ of order $p$.

The complex conjugate of any root of $f$ is also a root of $f$, so complex
conjugation restricts to an automorphism of $\SF_\Q(f)$ over $\Q$. Exactly
two of the roots of $f$ are non-real; complex conjugation transposes them
and fixes the rest. So $H$ contains a transposition $\tau$.

Without loss of generality, $\tau = (12)$. Since $\sigma$ is a $p$-cycle,
$\sigma^r(1) = 2$ for some $r \in \{1, \ldots, p - 1\}$. Since $p$ is
prime, $\sigma^r$ also has order $p$, so it is a $p$-cycle. Now without
loss of generality, $\sigma^r = (123 \ldots p)$. So $(12), (12\ldots p) \in
H$, forcing $H = S_p$ by Lemma~\ref{lemma:Sn-gen}. Hence $\Gal_\Q(f) \iso
S_p$.
\end{proof}

\begin{ex}{ex:power-order}
Explain why, in the last paragraph, $\sigma^r$ has order $p$.\\
\end{ex}

\begin{bigthm}
\label{thm:not-rad}
Not every polynomial over $\Q$ of degree $5$ is solvable by radicals.
\end{bigthm}

\begin{proof}
We show that $f(t) = t^5 - 6t + 3$ satisfies the conditions of
Lemma~\ref{lemma:Sp}. Then $\Gal_\Q(f)$ is $S_5$, which is not solvable, so
by Theorem~\ref{thm:sol-rad}, $f$ is not solvable by radicals.

Evidently $\deg(f)$ is the prime number $5$, and $f$ is irreducible by
Eisenstein's criterion with prime $3$. It remains to prove that $f$ has
exactly $3$ real roots. This is where we use some analysis,
considering $f$ as a function $\R \to \R$ (Figure~\ref{fig:quintic}).

\begin{figure}
\setlength{\fboxsep}{0mm}%
\centering\includegraphics[height=40mm]{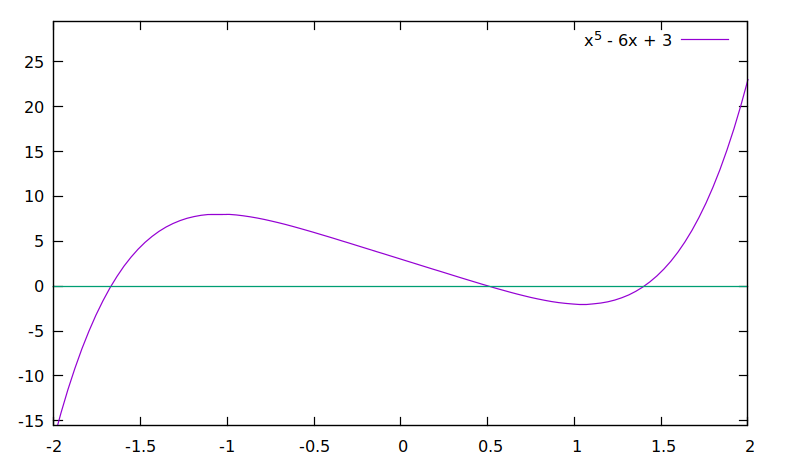}
\caption{The function $x \mapsto x^5 - 6x + 3$.}
\label{fig:quintic}
\end{figure}

We have
\[
\lim_{x \to -\infty} f(x) = -\infty,
\qquad
f(0) > 0,
\qquad
f(1) < 0,
\qquad
\lim_{x \to \infty} f(x) = \infty,
\]
and $f$ is continuous on $\R$, so by the intermediate value theorem, $f$
has at least $3$ real roots. On the other hand, $f'(x) = 5x^4 - 6$ has only
$2$ real roots ($\pm \sqrt[4]{6/5}$), so by Rolle's theorem, $f$ has at
most $3$ real roots. Hence $f$ has exactly $3$ real roots, as required.
\end{proof}

\begin{ex}{ex:more-than-5}
Prove that for every $n \geq 5$, there is some polynomial of degree $n$
that is not solvable by radicals.
\end{ex}

\begin{example}
There are also quintics with Galois group $A_5$. These are not solvable by
radicals, since $A_5$ is not a solvable group. One example, although we
won't prove it, is $t^5 + 20t + 16$.
\end{example}

\begin{digression}{dig:nml-rad}
We now know that some polynomials $f$ over $\Q$ are not solvable by radicals,
which means that not \emph{all} their complex roots are radical. 

Could it be that some of the roots are radical and others are not? Yes:
simply take a polynomial $g$ that is not solvable by radicals and put $f(t)
= tg(t)$. Then the roots of $f$ are $0$ (which is radical) together with
the roots of $g$ (which are not all radical).

But what if $f$ is irreducible? In that case, either all the roots of $f$
are radical or none of them are. This follows from the fact that the
extension $\Q^\rad: \Q$ is normal, which we will not prove.
\end{digression}

\begin{digression}{dig:rad-ruler}
There are many similarities between the theory of constructibility of
points by ruler and compass and the theory of solvability of polynomials by
radicals. In both cases, the challenge is to construct some things (points
in the plane or roots of polynomials) using only certain tools
(ruler and compass or a machine for taking $n$th roots). In both cases,
there were difficult questions of constructibility that remained open for a
very long time, and in both cases, they were solved by field theory.

The solutions have something in common too. For the geometry problem, we
used iterated quadratic extensions, and for the polynomial problem, we used
solvable extensions, which could reasonably be called iterated abelian
extensions. For the geometry problem, we showed that the coordinates of any
point constructible by ruler and compass satisfy a certain condition on
their degree over $\Q$ (Theorem~\ref{thm:rc-deg}); for the polynomial
problem, we showed that any polynomial solvable by radicals satisfies a
certain condition on its Galois group over $\Q$. There are other similarities:
compare Lemmas~\ref{lemma:iq-join} and~\ref{lemma:sol-filt}, for example,
and maybe you can find more similarities still.
\end{digression}

We have now used the fundamental theorem of Galois theory to solve a major
problem about $\Q$. What else can we do with it?

The fundamental theorem is about \emph{separable} extensions. Our two main
sources of separable extensions are:
\begin{itemize}
\item 
fields of characteristic $0$ such as $\Q$
(Example~\ref{egs:sep-ext}\bref{eg:se-zero}), which we've explored
extensively already; 

\item
finite fields (Example~\ref{egs:sep-ext}\bref{eg:se-fin}), which we've
barely touched.
\end{itemize}
In the next and final chapter, we'll use the fundamental theorem and other
results we've proved to explore the world of finite fields. In contrast to
the intricately complicated world of finite groups, finite fields are almost
shockingly simple.

\chapter{Finite fields}
\label{ch:fin}

This chapter is dessert. Through this semester, we've developed a lot of
sophisticated theory for general fields. All of it works for finite fields,
but becomes much simpler there. It's a miniature world in which life is
sweet.
\video{Introduction to Week 10}%
For example:
\begin{itemize}
\item
If we want to apply the fundamental theorem of Galois theory to a field
extension $M: K$, we first have to ask whether it is finite, and whether it
is normal, and whether it is separable. When $M$ and $K$ are finite fields, all
three conditions are automatic.

\item
There are many fields of different kinds, and to classify them all
would be a near-impossible task. But for finite fields, the classification
is very simple. We know exactly what finite fields there are.

\item
The Galois correspondence for arbitrary field extensions can also be
complicated. But again, it's simple when the fields are finite. Their
Galois groups are very easy (they're all cyclic), we know what their
subgroups are, and it's easy to describe all the subfields of any given
finite field.
\end{itemize}
So although the world of finite fields is not trivial, there's a lot about
it that's surprisingly straightforward.

We've already encountered two aspects of finite fields that may seem
counterintuitive. First, they always have positive characteristic, which
means they satisfy some equation like $1 + \cdots + 1 = 0$
(Lemma~\ref{lemma:ff-p}). Second, any element of a finite field of
characteristic $p$ has precisely one $p$th root
(Corollary~\ref{cor:pth-roots}\bref{part:pr-ff}), making finite fields
quite unlike $\C$, $\R$ or $\Q$. But the behaviour of $p$th roots and $p$th
powers is fundamental to all of finite fields' nice properties.

\section{Classification of finite fields}
\label{sec:fin-class}

If you try to write down a formula for the number of
\href{https://oeis.org/A000001}{groups} or
\href{https://oeis.org/A027623}{rings} with a given number of elements,
you'll find that it's hard and the results are quite strange. For instance,
more than 99\% of the first 50 billion groups
\href{https://doi.org/10.1090/S1079-6762-01-00087-7}{have order $1024$}.

But fields turn out to be much, much easier. We'll obtain a complete
classification of finite fields in the next two pages.

The \demph{order} of a finite field $M$ is its cardinality, or number of
elements, $|M|$.

\begin{lemma}
\label{lemma:ff-pp}
Let $M$ be a finite field. Then $\chr M$ is a prime number $p$, and $|M| =
p^n$ where $n = [M: \F_p] \geq 1$.
\end{lemma}

In particular, the order of a finite field is a prime power.

\begin{proof}
By Lemmas~\ref{lemma:char-0p} and~\ref{lemma:ff-p}, $\chr M$ is a prime
number $p$. By Lemma~\ref{lemma:char-prime-subfd}, $M$ has prime subfield
$\F_p$. Since $M$ is finite, $1 \leq [M: \F_p] < \infty$; write $n = [M:
\F_p]$. As a vector space over $\F_p$, then, $M$ is $n$-dimensional and
so isomorphic to $\F_p^n$. But $|\F_p^n| = |\F_p|^n = p^n$, so $|M|
= p^n$.
\end{proof}

\begin{example}
There is no field of order $6$, since $6$ is not a prime power.
\end{example}

\begin{warning}{wg:deg-ord}
\emph{Order} and \emph{degree} mean different things. For instance, if the
order of a field is $9$, then its degree over its prime subfield $\F_3$
is $2$.
\end{warning}

Lemma~\ref{lemma:ff-pp} prompts two questions:
\begin{itemize}
\item
Given a prime power $p^n$, \emph{is} there some field of order $p^n$?

\item
If so, how many are there?
\end{itemize}
To answer them, we need to use the Frobenius automorphism $\theta$ of a
finite field (Proposition~\ref{propn:frob}).

\begin{ex}{ex:frob-9}
Work out the values of the Frobenius
automorphism on the field $\F_3(\sqrt{2})$, which you first met in
Exercise~\ref{ex:F3_2}. 
\end{ex}

The answer to the first of these two questions is yes:

\begin{lemma}
Let $p$ be a prime number and $n \geq 1$. Then the splitting field of
$t^{p^n} - t$ over $\F_p$ has order $p^n$.
\end{lemma}

\begin{proof}
Put $f(t) = t^{p^n} - t \in \F_p[t]$ and $M = \SF_{\F_p}(f)$. Then $Df =
-1$ (since $n \geq 1$), so by
\bref{part:rt-rep}$\implies$\bref{part:rt-shared} of
Lemma~\ref{lemma:rep-tfae}, $f$ has no repeated roots in $M$. Hence $M$ has
at least $p^n$ elements.

Write $\theta$ for the Frobenius map of $M$ and $\theta^n = \theta \of
\cdots \of \theta$. Then $\theta^n(\alpha) = \alpha^{p^n}$ for all
$\alpha$, so the set $L$ of roots of $f$ in $M$ is
$\Fix\{\theta^n\}$. Since $\theta$ is a homomorphism,
Lemma~\ref{lemma:fix-subfd} implies that $L$ is a subfield of $M$. Hence by
definition of splitting field, $L = M$; that is, every element of $M$ is a
root of $f$. And since $\deg(f) = p^n$, it follows that $M$ has at most
$p^n$ elements.
\end{proof}

As for the second question, there is exactly \emph{one} field of each prime
power order. To show this, we need a lemma.

\begin{lemma}
\label{lemma:ff-thetanid}
Let $M$ be a finite field of order
$q$. Then $\alpha^q = \alpha$ for all $\alpha \in M$.
\end{lemma}

The proof uses the same argument as in Example~\ref{eg:frob-Fp}.

\begin{proof}
The multiplicative group $M^\times = M\without\{0\}$ has order $q - 1$, so
Lagrange's theorem implies that $\alpha^{q - 1} = 1$ for all $\alpha \in
M^\times$. Hence $\alpha^q = \alpha$ whenever $0 \neq \alpha \in M$, and
clearly the equation holds for $\alpha = 0$ too.
\end{proof}

\begin{ex}{ex:4th-powers}
Verify directly that $\beta^4 = \beta$ for all $\beta$ in the $4$-element
field $\F_2(\alpha)$ of Example~\ref{eg:ff-4}.
\end{ex}

\begin{lemma}
Every finite field of order $q$ is a splitting field of $t^q - t$ over
$\F_p$.
\end{lemma}

\begin{proof}
Let $M$ be a field of order $q$.  By Lemma~\ref{lemma:ff-pp}, $q = p^n$ for
some prime $p$ and $n \geq 1$, and $\chr M = p$. Hence $M$ has prime
subfield $\F_p$.  By Lemma~\ref{lemma:ff-thetanid}, every element of $M$ is
a root of $f(t) = t^{p^n} - t$. So $f$ has $|M|$ distinct roots in $M$; but
$|M| = p^n = \deg(f)$, so $f$ splits in $M$. The set of roots of $f$ in $M$
generates $M$, since it is \emph{equal} to $M$. Hence $M$ is a splitting
field of $f$.
\end{proof}

Together, these results completely classify the finite fields.

\begin{bigthm}[Classification of finite fields]
\label{thm:ff-class}
\begin{enumerate}
\item
Every finite field has order $p^n$ for some prime $p$ and integer $n \geq
1$. 

\item
For each prime $p$ and integer $n \geq 1$, there is exactly one field of
order $p^n$, up to isomorphism. It has characteristic $p$ and is a
splitting field for $t^{p^n} - t$ over $\F_p$.
\end{enumerate}
\end{bigthm}

\begin{proof}
This is immediate from the results above together with the uniqueness of
splitting fields (Theorem~\ref{thm:sf}\bref{part:sf-unique}). 
\end{proof}

When $q > 1$ is a prime power, we write \demph{$\F_q$} for the one and only
field of order~$q$.

\begin{warning}{wg:Fq}
$\F_q$ is \emph{not}\, $\Z/\idlgen{q}$ unless $q$ is a prime. It can't be,
because $\Z/\idlgen{q}$ is not a field (Example~\ref{eg:Z-fd}). To my 
knowledge, there is no description of $\F_q$ simpler than the splitting
field description.
\end{warning}

We now know exactly how many finite fields there are of each order.  But
generally in algebra, it's important to think not just about the
\emph{objects} (such as vector spaces, groups, modules, rings, fields,
\ldots), but also the \emph{maps} (homomorphisms) between objects. So now
that we've counted the finite fields, it's natural to try to count the
homomorphisms between finite fields. Field homomorphisms are injective, so
this boils down to counting subfields and automorphisms. Galois theory is
very well equipped to do that! We'll come to this in the final section. But
first, we look at another way in which finite fields are very simple.

\section{Multiplicative structure}
\label{sec:fin-mult}

The multiplicative group $K^\times$ of a finite field $K$ is as easy as can
be:

\begin{propn}
\label{propn:mult-cyc}
For an arbitrary field $K$, every finite subgroup of $K^\times$ is cyclic. In
particular, if $K$ is finite then $K^\times$ is cyclic.
\end{propn}

\begin{proof}
\video{The multiplicative group of a finite field is cyclic}%
This was Theorem~5.1.13 and Corollary~5.1.14 of Group Theory. 
\end{proof}

\begin{example}
In examples earlier in the course, we frequently used the $n$th
root of unity $\omega = e^{2\pi i/n} \in \C$, which has the property that
every other $n$th root of unity is a power of $\omega$.

Can we find an analogue of $\omega$ in an arbitrary field $K$? It's not
obvious how to generalize the formula $e^{2\pi i/n}$, since the exponential
is a concept from complex analysis. But
Proposition~\ref{propn:mult-cyc} solves our problem. For $n \geq 1$, put
\[
U_n(K) = \{ \alpha \in K \such \alpha^n = 1\}.
\]
Then $U_n(K)$ is a subgroup of $K^\times$, and is finite since
its elements are roots of $t^n - 1$. So by
Proposition~\ref{propn:mult-cyc}, $U_n(K)$ is cyclic. Let $\omega$ be a
generator of $U_n(K)$. Then every $n$th root of unity in $K$ is a power of
$\omega$, which is what we were aiming for.

Note, however, that $U_n(K)$ may have fewer than $n$ elements, or
equivalently, the order of $\omega$ may be less than $n$. For instance, if
$\chr K = p$ then $U_p(K)$ is trivial and $\omega = 1$, by
Example~\ref{egs:ff-roots}\bref{eg:ffr-unity}.
\end{example}

\begin{ex}{ex:fin-unity}
Let $K$ be a field and let $H$ be a finite subgroup of $K^\times$ of order
$n$. Prove that $H \sub U_n(K)$.
\end{ex}

\begin{example}
The group $\F_p^\times$ is cyclic, for any prime $p$.  This
means that there is some $\omega \in \{1, \ldots, p - 1\}$ such that
$\omega, \omega^2, \ldots$ runs through all elements of $\{1, \ldots, p -
1\}$ when taken mod $p$. In number theory, such an $\omega$ is called a
\demph{primitive root} mod $p$ (another usage of the word `primitive'). For
instance, you can check that $3$ is a primitive root mod $7$, but $2$ is
not, since $2^3 \equiv 1 \pmod{7}$. 

Finding the primitive roots mod $p$ is one aspect of finite fields that is
\emph{not} trivial. 
\end{example}

\begin{cor}
\label{cor:prim-elt-ff}
Every extension of one finite field over another is simple.
\end{cor}

\begin{proof}
Let $M: K$ be an extension with $M$ finite. By
Proposition~\ref{propn:mult-cyc}, the group $M^\times$ is generated by 
some element $\alpha \in M^\times$. Then $M = K(\alpha)$.
\end{proof}

This is yet another pleasant aspect of finite fields.

\begin{ex}{ex:ff-simp}
In the proof of Corollary~\ref{cor:prim-elt-ff}, once we know that the
group $M^\times$ is generated by $\alpha$, how does it follow that $M =
K(\alpha)$?
\end{ex}

\begin{digression}{dig:prim-ff}
In Digression~\ref{dig:primitive}, I mentioned the theorem of the primitive
element: every finite separable extension $M: K$ is simple. One of the
standard proofs involves splitting into two cases, according to whether $M$
is finite or infinite. We've just done the finite case.
\end{digression}

\begin{cor}
For every prime number $p$ and integer $n \geq 1$, there exists an
irreducible polynomial over $\F_p$ of degree $n$.
\end{cor}

\begin{proof}
The field $\F_{p^n}$ has prime subfield $\F_p$. By
Corollary~\ref{cor:prim-elt-ff}, the extension $\F_{p^n}: \F_p$ is simple,
say $\F_{p^n} = \F_p(\alpha)$. The minimal polynomial of $\alpha$ over
$\F_p$ is irreducible of degree $[\F_p(\alpha): \F_p] = [\F_{p^n}: \F_p] =
n$. 
\end{proof}

This is not obvious. For example, can you find an irreducible
polynomial of degree $100$ over $\F_{31}$?

\section{Galois groups for finite fields}
\label{sec:fin-gal}

Here we work out the Galois correspondence for $\F_{p^n}: \F_p$. 

\begin{warning}{wg:ffe}
The term `finite field extension' means an extension $M: K$ that's
finite in the sense defined on p.~\pageref{p:fin-deg}: $M$ is
finite-\emph{dimensional} as a vector space over $K$. It doesn't mean that
$M$ and $K$ are finite fields. But the safest policy is to avoid this
term entirely. 
\end{warning}

The three hypotheses of the fundamental theorem of Galois
theory are always satisfied when both fields in the extension are finite:

\begin{lemma}
\label{lemma:ff-fns}
Let $M: K$ be a field extension.
\begin{enumerate}
\item 
\label{part:fff-sep}
If $K$ is finite then $M: K$ is separable.

\item
\label{part:fff-fn}
If $M$ is also finite then $M: K$ is finite and normal.
\end{enumerate}
\end{lemma}

\begin{proof}
For~\bref{part:fff-sep}, we show that every irreducible polynomial $f$ over $K$
is separable. Write $p = \chr K > 0$, and suppose for a contradiction that
$f$ is inseparable. By Corollary~\ref{cor:sep-chars}, 
\[
f(t) = b_0 + b_1 t^p + \cdots + b_r t^{rp}
\]
for some $b_0, \ldots, b_r \in K$. For each $i$, there is a (unique) $p$th
root $c_i$ of $b_i$ in $K$, by
Corollary~\ref{cor:pth-roots}\bref{part:pr-ff}. Then
\[
f(t) 
= 
c_0^p + c_1^p t^p + \cdots + c_r^p t^{rp}.
\]
But by Proposition~\ref{propn:frob}\bref{part:frob-ring}, the function $g
\mapsto g^p$ is a homomorphism $K[t] \to K[t]$, so 
\[
f(t) = 
(c_0 + c_1 t + \cdots + c_r t^r)^p.
\]
This contradicts $f$ being irreducible.

For~\bref{part:fff-fn}, suppose that $M$ is finite. Write $p = \chr M >
0$. By Theorem~\ref{thm:ff-class}, $M$ is a splitting field over $\F_p$, so
by Lemma~\ref{lemma:sf-adj}\bref{part:sfa-fds}, it is also a splitting
field over $K$. Hence $M: K$ is finite and normal, by
Theorem~\ref{thm:sf-fin-nml}.
\end{proof}

Part~\bref{part:fff-sep} fulfils the promise made in
Remark~\ref{rmk:sep-ff} and Example~\ref{egs:sep-ext}\bref{eg:se-fin}, and
the lemma as a whole lets us use the fundamental theorem freely in the
world of finite fields. We now work out the Galois correspondence for the
extension $\F_{p^n}: \F_p$ of an arbitrary finite field over its prime
subfield. 

\begin{propn}
\label{propn:ff-gal}
Let $p$ be a prime and $n \geq 1$. Then $\Gal(\F_{p^n}: \F_p)$ is
cyclic of order $n$, generated by the Frobenius automorphism of $\F_{p^n}$.
\end{propn}

By an earlier workshop question, $\Gal(\F_{p^n}: \F_p)$ is the group of
\emph{all} automorphisms of $\F_{p^n}$.

\begin{proof}
Write $\theta$ for the Frobenius automorphism of $\F_{p^n}$; then $\theta
\in \Gal(\F_{p^n}: \F_p)$. First we calculate the order of $\theta$. By
Lemma~\ref{lemma:ff-thetanid}, $\alpha^{p^n} = \alpha$ for all $\alpha \in
\F_{p^n}$, or equivalently, $\theta^n = \id$. If $m$ is a positive integer
such that $\theta^m = \id$ then $\alpha^{p^m} = \alpha$ for all $\alpha \in
\F_{p^n}$, so the polynomial $t^{p^m} - t$ has $p^n$ roots in $\F_{p^n}$,
so $p^n \leq p^m$, so $n \leq m$. Hence $\theta$ has order $n$.

On the other hand, $[\F_{p^n}: \F_p] = n$, so by the fundamental theorem of
Galois theory, $|\Gal(\F_{p^n}: \F_p)| = n$. The result follows.
\end{proof}

\begin{ex}{ex:ff-ff}
What is the fixed field of $\gpgen{\theta} \sub \Gal(\F_{p^n}: \F_p)$? \\
\end{ex}

In Fundamentals of Pure Mathematics or Group Theory, you presumably saw
that the cyclic group of order $n$ has exactly one subgroup of order $k$
for each divisor $k$ of $n$. (And by Lagrange's theorem, there are no
subgroups of other orders.) 

\begin{ex}{ex:cyc-sub}
Refresh your memory by proving this fact about subgroups of cyclic groups.
\end{ex}

In the case at hand, $\Gal(\F_{p^n}: \F_p) = \gpgen{\theta} \iso C_n$, and
when $k \dvd n$, the unique subgroup of order $k$ is
$\gpgen{\theta^{n/k}}$.

\begin{propn}
\label{propn:ff-sub}
Let $p$ be a prime and $n \geq 1$. Then $\F_{p^n}$ has exactly one subfield
of order $p^m$ for each divisor $m$ of $n$, and no others. It is 
\[
\bigl\{ \alpha \in \F_{p^n} \such \alpha^{p^m} = \alpha \bigr\}.
\]
\end{propn}

\begin{proof}
The subfields of $\F_{p^n}$ are the intermediate fields of $\F_{p^n}:
\F_p$, which by the fundamental theorem of Galois theory are precisely the
fixed fields $\Fix(H)$ of subgroups $H$ of $\Gal(\F_{p^n}: \F_p)$. Any such
$H$ is of the form $\gpgen{\theta^{n/k}}$ with $k \dvd n$, and 
\[
\Fix\gpgen{\theta^{n/k}} 
=
\bigl\{ \alpha \in \F_{p^n} \such \alpha^{p^{n/k}} = \alpha \bigr\}.
\]
The tower law and the fundamental theorem give
\[
[\Fix\gpgen{\theta^{n/k}}: \F_p]
=
\frac{[\F_{p^n}: \F_p]}{[\F_{p^n}: \Fix\gpgen{\theta^{n/k}}]}
=
\frac{n}{|\gpgen{\theta^{n/k}}|}
=
\frac{n}{k},
\]
so $|\Fix\gpgen{\theta^{n/k}}| = p^{n/k}$.  As $k$ runs through the
divisors of $n$, the quotient $n/k$ also runs through the divisors of $n$, so
putting $m = n/k$ gives the result.
\end{proof}

\begin{warning}{wg:subfd-div}
The subfields of $\F_{p^n}$ are of the form $\F_{p^m}$ where $m$
\emph{divides} $n$, not $m \leq n$. For instance, $\F_8$ has no subfield
isomorphic to $\F_4$ (that is, no 4-element subfield), since $8 = 2^3$, $4
= 2^2$, and $2 \ndvd 3$.
\end{warning}

Let $m$ be a divisor of $n$. By Proposition~\ref{propn:ff-sub}, $\F_{p^n}$
has exactly one subfield isomorphic to $\F_{p^m}$.  We can therefore speak
of the extension $\F_{p^n}: \F_{p^m}$ without ambiguity. Since $\F_{p^m} =
\Fix\gpgen{\theta^m}$ (by Proposition~\ref{propn:ff-sub}) and
$\gpgen{\theta^m} \iso C_{n/m}$, it follows from the fundamental theorem
that
\begin{align}
\label{eq:gal-gp-ff-gen}
\Gal(\F_{p^n}: \F_{p^m}) \iso C_{n/m}.
\end{align}
So in working out the Galois correspondence for $\F_{p^n}: \F_p$, we have
accidentally derived the Galois group of a completely arbitrary extension
of finite fields. Another way to phrase~\eqref{eq:gal-gp-ff-gen} is:

\begin{propn}
Let $M: K$ be a field extension with $M$ finite. Then $\Gal(M: K)$ is
cyclic of order $[M: K]$.
\qed
\end{propn}

In the Galois correspondence for $\F_{p^n}: \F_p$, all the extensions and
subgroups involved are normal, either by Lemma~\ref{lemma:ff-fns} or
because cyclic groups are abelian. For $m \dvd n$, the isomorphism
\[
\frac{\Gal(\F_{p^n}: \F_p)}{\Gal(\F_{p^n}: \F_{p^m})}
\iso
\Gal(\F_{p^m}: \F_p)
\]
supplied by the fundamental theorem amounts to
\[
\frac{C_n}{C_{n/m}} 
\iso
C_m.
\]
Alternatively, substituting $k = n/m$, this is $C_n/C_k \iso C_{n/k}$.

\begin{example}
\label{eg:ff-gc-12}
Consider the Galois correspondence for $\F_{p^{12}}: \F_p$, where $p$ is
any prime. Writing $\theta$ for the Frobenius automorphism of
$\F_{p^{12}}$, the subgroups of $G = \Gal(\F_{p^{12}}: \F_p)$ are
\[
\xymatrix@!0@R-1ex@C+1.8em{
        &
\gpgen{\theta^{12}} \iso C_1 \iso 1
\ar@{-}[lddd] \ar@{-}[rdd]
                &       &       &       &
\text{order } 1                                 \\
        &       &       &       &       &       \\
        &       &
\gpgen{\theta^6} \iso C_2
\ar@{-}[lddd] \ar@{-}[rdd]
                        &       &       &
\text{order } 2                                 \\
\gpgen{\theta^4} \iso C_3 
\ar@{-}[rdd]    
        &       &       &       &       &
\text{order } 3                                 \\
        &       &       &       
\gpgen{\theta^3} \iso C_4              
\ar@{-}[lddd]                   &       &
\text{order } 4                                 \\
        &
\gpgen{\theta^2} \iso C_6
\ar@{-}[rdd]    &       &       &       &
\text{order } 6                                 \\
        &       &       &       &       &       \\
        &       &
G = \gpgen{\theta} \iso C_{12}
                        &       &       &
\text{order } 12
}
\]
Their fixed fields are
\[
\xymatrix@!0@R-1ex@C+1.8em{
        &
\F_{p^{12}}
\ar@{-}[lddd] \ar@{-}[rdd]
                &       &       &       &
\text{degree } 1                                \\
        &       &       &       &       &       \\
        &       &
\F_{p^6}
\ar@{-}[lddd] \ar@{-}[rdd]
                        &       &       &
\text{degree } 2                                \\
\F_{p^4}
\ar@{-}[rdd]    
        &       &       &       &       &
\text{degree } 3                                \\
        &       &       &       
\F_{p^3}
\ar@{-}[lddd]                   &       &
\text{degree } 4                                \\
        &
\F_{p^2}
\ar@{-}[rdd]    &       &       &       &
\text{degree } 6                                \\
        &       &       &       &       &       \\
        &       &
\F_p
                        &       &       &
\text{degree } 12                               \\
}
\]
Here, `degree' means the degree of $\F_{p^{12}}$ over the subfield, and (for
instance) the subfield of $\F_{p^{12}}$ called $\F_{p^4}$ is
\[
\bigl\{ \alpha \in \F_{p^{12}} \such \alpha^{p^4} = \alpha \bigr\}.
\]
The Galois group $\Gal(\F_{p^{12}}: \F_{p^4})$ is
$\gpgen{\theta^4} \iso C_3$, and similarly for the other subfields. 
\end{example}

\begin{ex}{ex:ff-gc-8}
What do the diagrams of Example~\ref{eg:ff-gc-12} look like for
$p^8$ in place of $p^{12}$? What about $p^{432}$? (Be systematic!)
\end{ex}
\video{Ordered sets}

In the workshop, you'll be asked to work through the Galois correspondence
for an arbitrary extension $\F_{p^n}: \F_{p^m}$ of finite fields, but
there's not much more to do: almost all the work is contained in the case
$m = 1$ that we have just done.

\bigskip

\rowofstars

\end{document}